\catcode`\@=11
\newif\if@fewtab\@fewtabtrue
{\count255=\time\divide\count255 by 60
\xdef\hourmin{\number\count255}
\multiply\count255 by-60\advance\count255 by\time
\xdef\hourmin{\hourmin:\ifnum\count255<10 0\fi\the\count255}}
\def\ps@draft{\let\@mkboth\@gobbletwo
     \def\@oddfoot{\hbox to 7 cm{\tiny \versionno
        \hfil}\hskip -7cm\hfil\rm\thepage \hfil {\tiny\draftdate}}
     \def\@oddhead{}
     \def\@evenhead{}\let\@evenfoot\@oddfoot}
\def\draftdate{\number\month/\number\day/\number\year\ \ \ \hourmin }

\def\citen#1{\if@filesw \immediate\write \@auxout {\string\citation{#1}}\fi%
\@tempcntb\m@ne \let\@h@ld\relax \def\@citea{}%
\@for \@citeb:=#1\do {\@ifundefined {b@\@citeb}%
     {\@h@ld\@citea\@tempcntb\m@ne{\bf ?}%
     \@warning {Citation `\@citeb ' on page \thepage \space undefined}}%
     {\@tempcnta\@tempcntb \advance\@tempcnta\@ne
     \setbox\z@\hbox\bgroup\ifcat0\csname b@\@citeb \endcsname \relax
     \egroup \@tempcntb\number\csname b@\@citeb \endcsname \relax
     \else \egroup \@tempcntb\m@ne \fi \ifnum\@tempcnta=\@tempcntb
     \ifx\@h@ld\relax \edef \@h@ld{\@citea\csname b@\@citeb\endcsname}%
     \else \edef\@h@ld{\hbox{--}\penalty\@highpenalty
     \csname b@\@citeb\endcsname}\fi
     \else \@h@ld\@citea\csname b@\@citeb \endcsname \let\@h@ld\relax \fi}%
\def\@citea{,\penalty\@highpenalty\hskip.13em plus.13em minus.13em}}\@h@ld}
\def\@citex[#1]#2{\@cite{\citen{#2}}{#1}}%
\def\@cite#1#2{\leavevmode\unskip\ifnum\lastpenalty=\z@\penalty\@highpenalty\fi%
   \ [{\multiply\@highpenalty 3 #1%
   \if@tempswa,\penalty\@highpenalty\ #2\fi}]}   %
\makeatother 
\catcode`\@=12

\newcommand\ad[1]  {\mathrm{ad}_{#1}}
\def\apo           {\mbox{\sc s}}
\def\apoi          {\mbox{\sc s}^{-1}}

\def\be            {\begin{equation}}
\def\bearl         {\begin{array}{l}}
\def\bearll        {\begin{array}{ll}}
\newcommand\bee[5] {\begin{eqnarray} #5 \nonumber\\[-#1.#2em]~\\[#3.#4em]~\nonumber\end{eqnarray}}

\def\bico          {_{F}}
\def\bicoa         {_\bicoaa}
\def\bicoaa        {{\triangleright\triangleleft}}
\def\bicoad        {coadjoint bimodule}
\def\bbico         {_{\!F}}
\def\bicom         {coregular bimodule}
\def\Bimod         {\mbox{-Bimod}}
\def\bohr          {{\ohr\bico}}
\def\boti          {\,{\boxtimes}\,}

\def\brho          {{\rho\bbico^{}}}
\def\C             {{\ensuremath{\mathcal C}}}
\def\catpic        {{\footnotesize \shadowbox{\C}}}
\def\catpicsm      {{\tiny \shadowbox{\C}}}
\def\CbC           {\ensuremath{\hspace{.3pt}\overline{\mathcal C}\hspace{.6pt}{\boxtimes}%
                    \hspace{1.4pt}\mathcal C}\xspace}

\def\cdo           {\,{\cdot}\,}

\def\cHbHb         {c^{}_{\!\Hb,\Hb}}
\def\cir           {\,{\circ}\,}
\newcommand\coen[1]{\int^{#1}\hspace*{-.23em}}
\newcommand\coend[1]{\int^{#1}\hspace*{-.31em}#1^\vee{\otimes}#1}
\newcommand\Coend[2]{\int^{#2}\hspace*{-.23em}#1(#2,#2)}

\def\coa           {_\triangleright}
\def\coar          {_\triangleleft}
\def\complex       {{\ensuremath{\mathbbm C}}}
\def\coop          {^{\mathrm{coop}}}
\def\CopC          {{\ensuremath{\mathcal C\op\hspace*{.8pt}{\times}\hspace*{1.5pt}\mathcal C}}}
\def\COPC          {{\ensuremath{\mathcal C\op{\times}\hspace{1.1pt}\mathcal C}}}
\def\D             {{\ensuremath{\mathcal D}}}

\def\delp          {\Delta\bico\hspace*{-9pt}{}'\hspace*{6pt}}

\def\drin          {f_Q}

\def\ee            {\end{equation}}

\def\eear          {\end{array}}
\def\End           {{\ensuremath{\mathrm{End}}}}
\def\EndC          {{\ensuremath{\mathrm{End}_\C}}}
\def\EndH          {{\ensuremath{\mathrm{End}_{H}}}}
\def\EndHH         {{\ensuremath{\mathrm{End}_{H|H}}}}

\def\eps           {\varepsilon}

\def\eq            {\,{=}\,}
\newcommand\equ[1] {\stackrel{\mbox{\tiny\erf{#1}}}=}
\newcommand\erf[1] {(\ref{#1})}

\def\Fbb           {{\ensuremath{G_{\!\Box}^H}}}
\def\Fbx           {{\ensuremath{G_{\!\otimes_\ko}^H}}}
\def\Fbxo          {{\ensuremath{G_{\!\otimes_\ko}^{H;\omega}}}}
\def\findim        {fi\-ni\-te-di\-men\-si\-o\-nal}
\def\fmap          {{{\ensuremath{\Psi}}}}

\def\H             {\ensuremath{\mathscr H}}
\def\Ha            {{\ensuremath{H^*_\triangleright}}}
\def\Haa           {{\ensuremath{H\hspace*{-6pt}H^*\bicoa}}}
\def\Haalg         {{\ensuremath{H\hspace*{-6.8pt}H^*\bicoa}}}
\def\Haasm         {{\ensuremath{H\hspace*{-4.1pt}H^*\bicoa}}}
\def\Hb            {{\ensuremath{F}}}
\def\Hbi           {{\ensuremath{F_{\!\idsm_H}}}}
\def\HBimod        {{\ensuremath{H}\mbox{-Bimod}}}
\def\Hbo           {{\ensuremath{F_{\!\omega}^{}}}}
\def\HK            {{{\ensuremath K}}}
\def\HMod          {{\ensuremath{H}\text{-Mod}}}
\def\Hom           {{\ensuremath{\mathrm{Hom}}}}

\def\HomC          {{\ensuremath{\mathrm{Hom}_\C}}}
\def\HomH          {{\ensuremath{\mathrm{Hom}_H}}}
\def\HomHH         {{\ensuremath{\mathrm{Hom}_{H|H}}}}
\def\Homk          {{\ensuremath{\mathrm{Hom}_\ko}}}

\def\Hs            {{\ensuremath{H^*}}}

\def\Hss           {{H^*_{}}}

\def\id            {\mbox{\sl id}}
\def\Id            {\mbox{\sl Id}}
\def\idA           {\ensuremath{\id_A}}

\def\idHs          {\ensuremath{\id_{{H^{\phantom:}}^{\!\!*}}}}
\def\idHsHs        {\ensuremath{\id_{{H^{\phantom:}}^{\!\!*}\otimes{H^{\phantom:}}^{\!\!*}}}}

\def\idsm          {\mbox{\footnotesize\sl id}}

\def\idXs          {\ensuremath{\id_{X^{*_{}}_{\phantom:}}}}

\def\iHaa          {i^{\bicoaa}}
\def\iHb           {\imath^{\Hb}}
\def\iHbo          {\imath^{\Hbo}}
\def\iHK           {\imath^\HK}

\def\iN            {\,{\in}\,}
\def\J             {{\ensuremath{\mathcal J}}}
\def\jHZx          {\hspace*{1.65em}j_H^Z} 

\def\ko            {{\ensuremath{\Bbbk}}}
\newcommand\labl[1]{\label{#1}\ee}
\def\Mod           {\mbox{-Mod}}

\newcommand\nxl[1] {\\[#1mm]}
\newcommand\Nxl[1] {\\[-1.3em]\\[#1mm]}

\def\ohr           {\reflectbox{$\rho$}}
\def\ohrv          {\ohr_{\scriptscriptstyle\!\vee}^{}}
\def\ohrV          {{}_{\scriptscriptstyle\vee\!}^{}\ohr}
\def\omegi         {{\omega^{\!{\scriptscriptstyle-1}}}_{}}
\def\omegis        {{\omega^{\!{\scriptscriptstyle-1}^{\scriptstyle*}}}_{}}

\def\one           {{\bf1}}

\def\op            {^{\mathrm{op}}}

\def\oti           {\,{\otimes}\,}
\def\Oti           {{\otimes}}
\def\otik          {\,{\otimes_\ko}\,}
\def\Otik          {{\otimes_\ko}}
\def\piv           {\pi}

\def\pqt           {partial monodromy trace}

\def\Qq            {\mathcal Q^{\rm l}}
\def\QQ            {\mathcal Q}
\def\qquand        {\qquad{\rm and}\qquad}
\def\rep           {representation}
\def\rhov          {\rho_{\scriptscriptstyle\!\vee}^{}}
\def\rhoV          {{}_{\scriptscriptstyle\vee\!}^{}\rho}

\def\SK            {S_\HK}
\def\slz           {\ensuremath{\mathrm{SL}(2,\zet)}}
\def\slze          {\ensuremath{\Gamma_{\!1;1}}}

\def\slzmn         {\ensuremath{\Gamma_{\!1;m+n}}}
\def\sse           {\scriptsize }
\def\tauHH         {\tau^{}_{\!H,H}}
\def\tauHHv        {\tau^{}_{\!H,H^*_{}}}
\def\tauHvH        {\tau^{}_{\!H^*_{}\!,H}}
\def\tauQ          {t'{}_{\!\!\QQ}}

\def\Times         {\,{\times}\,}
\def\TK            {T_\HK}
\def\To            {\,{\to}\,}

\def\uvi           {{t}}
\def\Vect          {\ensuremath{\mathcal V}\mbox{\sl ect}}
\def\Vectk         {\ensuremath{\mathcal V\mbox{\sl ect}_\ko}}
\def\wee           {*}  
\newcommand\wiha[1]{\Wiha(#1)}
\def\Wiha          {t_{\QQ}}
\def\Xs            {X^{*_{}}_{\phantom:}}
\def\zet           {{\ensuremath{\mathbb Z}}}

\newcommand\includepichtft[1] {{\begin{picture}(0,0)(0,0)
                    \scalebox{.304}{\includegraphics{imgs/pic_htft_#1.eps}}\end{picture}}}
\newcommand\Includepichtft[1] {{\begin{picture}(0,0)(0,0)
                    \scalebox{.38}{\includegraphics{imgs/pic_htft_#1.eps}}\end{picture}}}
\newcommand\includepichopf[1] {{\begin{picture}(0,0)(0,0)
                    \scalebox{.304}{\includegraphics{imgs/pic_hopf_#1.eps}}\end{picture}}}
\newcommand\Includepichopf[1] {{\begin{picture}(0,0)(0,0)
                    \scalebox{.38}{\includegraphics{imgs/pic_hopf_#1.eps}}\end{picture}}}
\newcommand\eqpic[4]{\begin{eqnarray}
                    \begin{picture}(#2,#3){}\end{picture}\nonumber\\
                    \raisebox{-#3pt}{ \begin{picture}(#2,#3) #4 \end{picture} }
                    \label{#1} \\~\nonumber \end{eqnarray} }
\newcommand\Eqpic[4]{\begin{eqnarray}
                    \begin{picture}(#2,#3){}\end{picture}\nonumber\\
                    \raisebox{-#3pt}{ \begin{picture}(#2,#3) #4 \end{picture} }
                    \nonumber \\[3pt]~\label{#1} \end{eqnarray} }

\documentclass[12pt]{article}
\usepackage{latexsym, amsmath, amsthm, amsfonts, enumerate, amssymb, bbm, xspace, xypic }
\usepackage{fancybox}
\usepackage[all]{xy}
\usepackage[mathscr]{eucal}
\usepackage{graphicx} \usepackage{rotating}
\usepackage{epstopdf,hyperref}

\newtheorem{thm}{Theorem}
\newtheorem{lem}[thm]{Lemma}
\newtheorem{lemma}[thm]{Lemma}

\newtheorem{prop}[thm]{Proposition}
\newtheorem{cor}[thm]{Corollary}
\theoremstyle{definition}
\newtheorem{rem}[thm]{Remark}
\newtheorem{defi}[thm]{Definition}

\begin{document}

\def\cir{\,{\circ}\,} 
\numberwithin{equation}{section}
\numberwithin{thm}{section}

\begin{flushright}
    {\sf ZMP-HH/11-9}\\
    {\sf Hamburger$\;$Beitr\"age$\;$zur$\;$Mathematik$\;$Nr.$\;$408}\\[2mm]
    June 2011
\end{flushright}
\vskip 3.5em
\begin{tabular}c \Large\bf MODULAR INVARIANT FROBENIUS ALGEBRAS \\[2mm]
                  \Large\bf FROM RIBBON HOPF ALGEBRA AUTOMORPHISMS
\end{tabular}\vskip 2.1em
\begin{center}
   ~J\"urgen Fuchs\,$^{\,a}$,~
   ~Christoph Schweigert\,$^{\,b}$,~
   ~Carl Stigner\,$^{\,a}$
\end{center}

\vskip 9mm

\begin{center}\it$^a$
   Teoretisk fysik, \ Karlstads Universitet\\
   Universitetsgatan 21, \ S\,--\,651\,88\, Karlstad
\end{center}
\begin{center}\it$^b$
   Organisationseinheit Mathematik, \ Universit\"at Hamburg\\
   Bereich Algebra und Zahlentheorie\\
   Bundesstra\ss e 55, \ D\,--\,20\,146\, Hamburg
\end{center}
\vskip 5.3em

\noindent{\sc Abstract}
\\[3pt]
For any finite-dimensional factorizable ribbon Hopf algebra $H$ and any ribbon
automorphism of $H$, we establish the existence of the following structure: 
an $H$-bimodule \Hbo\ and a bimodule morphism $Z_\omega$ from Lyubashenko's Hopf
algebra object \HK\ for the bimodule category to \Hbo. This morphism is invariant 
under the natural action of the mapping class group of the one-punctured torus
on the space of bimodule morphisms from \HK\ to \Hbo. 
We further show that the bimodule \Hbo\ can be endowed with a
natural structure of a commutative symmetric Frobenius algebra in the monoidal
category of $H$-bimodules, and that it is a special Frobenius algebra iff $H$
is semisimple.
\\
The bimodules \HK\ and \Hbo\ can both be characterized as coends of suitable
bifunctors. The morphism $Z_\omega$ is obtained by applying a monodromy
operation to the coproduct of \Hbo; a similar construction for the product of
\Hbo\ exists as well.
\\
Our results are motivated by the quest to understand the bulk state space
and the bulk partition function in two-dimensional conformal field theories
with chiral algebras that are not necessarily semisimple.

  \newpage

\section{Introduction}

One remarkable feature of complex Hopf algebras is their intimate connection with
low-dimen\-si\-o\-nal topology, including invariants of knots, links and
three-manifolds.
These connections are particularly well understood for semisimple Hopf algebras.
The representation category of a semisimple factorizable finite-dimensional
(weak) Hopf algebra is a modular tensor category \cite{nitv} and thus allows
one to construct a three-dimensional topological field theory. As a
consequence, it provides finite-dimensional projective
representations of mapping class groups of punctured surfaces.

It has been shown by Lyubashenko \cite{lyub6,lyub11} that such representations of
mapping class groups can be constructed for non-semisimple factorizable Hopf
algebras $H$ as well.
This construction is in fact purely categorical, in the sense that it only
uses the representation category as an abstract ribbon category with certain
non-degeneracy properties. In the present paper we apply this construction
not to the category of left $H$-modules, but rather to the category of
$H$-\emph{bimodules}. To this end we endow this category \HBimod\ with the
structure of a monoidal category using the coproduct of $H$ (rather than by
taking the tensor product $\otimes_H$ over $H$ as an associative algebra). With
this tensor product, the category \HBimod\ can be endowed with further structure
such that it becomes a sovereign braided monoidal category.

For our present purposes we restrict to the case that the punctured surface in
question is a one-punctured torus. Thus in the absence of punctures the mapping 
class group is the modular group \slz; if punctures are present, then the 
mapping class group has additional generators given by Dehn twists around the 
punctures and by braiding homeomorphisms \cite[Sect.\,4.3]{lyub6}. We denote 
the mapping class group of the one-punctured torus by \slze.

Specializing the results of \cite{lyub6}, we obtain a Hopf algebra object
\HK\ in the monoidal category \HBimod. For any $H$-bimodule $X$ the vector space
$\HomHH(\HK,X)$ of bimodule morphisms then carries a projective representation 
of \slze. The main result of this paper is the following assertion:

\medskip\noindent{\bf Theorem}\\[2pt] {\it
Let $H$ be a (not necessarily semisimple) finite-dimensional factorizable
ribbon Hopf algebra over an algebraically closed field of characteristic zero,
and let $\omega\colon H\To H$ be an automorphism of $H$ as a ribbon Hopf
algebra. Then there is an object \Hbo\ in the category \HBimod\ and a morphism
   $$
   Z_\omega \in \HomHH(\HK,\Hbo)
   $$
that is invariant under the natural action 
\cite{lyub6} of the mapping class group \slze\ on $\HomHH(\HK,\Hbo)$.
}

\medskip

The considerations leading to this result are inspired by structure one hopes 
to encounter in certain two-dimensional conformal field theories that are
based on non-semisimple representation categories. More information about
this motivation can be found in appendix \ref{app:cft}; here it suffices to
remark that \Hbo\ is a candidate for what in conformal field theory is called
the \emph{algebra of bulk fields}, and that the morphism $Z_\omega$ is a 
candidate for a modular invariant \emph{partition function}. Such a partition 
function should also enjoy integrality properties. As we will show elsewhere
\cite{fuSs4}, for $\omega \eq \id_H$ the relevant integers 
are closely related to the Cartan matrix of the algebra $H$.

To arrive at our result we show in fact first that the object \Hbo\ actually 
carries a lot more natural structure: \Hbo\ is a commutative symmetric Frobenius 
algebra in \HBimod. Furthermore, the Frobenius algebra \Hbo\ is a special\,%
  \footnote{~A Frobenius is called special iff, up to non-zero scalars, the
  counit is a left inverse of the unit and the coproduct is a right inverse of
  the product, see Def.\ \ref{def-spec}.}
Frobenius algebra if and only if the Hopf algebra $H$ is semisimple.
A Frobenius algebra carries a natural coalgebra structure; the invariant morphism
$Z_\omega$ is obtained by applying a monodromy operation to the coproduct of \Hbo.

\medskip

This paper is organized as follows.
In Section \ref{sec2} we introduce the relevant structure of a
monoidal category on \HBimod\ and construct, for the case $\omega \eq \id_H$,
the bimodule $\Hb \eq \Hbi$ as a Frobenius algebra in \HBimod. In Section
\ref{sec-braid} we endow the monoidal category \HBimod\ with a natural
braiding and show that with respect to this braiding the Frobenius algebra
\Hb\ is commutative. In Section \ref{sec-dual} it is established that
\Hb\ is symmetric, has trivial twist, and is special iff $H$ is semisimple.
Modular invariance of $Z_{\idsm_H}$ is proven in Section \ref{sec-modinv}.
Section \ref{sec-auto} is finally devoted to the case of a general ribbon Hopf
algebra automorphism $\omega$ of $H$, which can actually be treated by modest
modifications of the arguments of Sections \ref{sec2}\,--\,\ref{sec-modinv}.
In appendix \ref{app:A} we gather some notions from category theory and explain
how the bimodules \HK\ and \Hbo\ can be characterized as coends of suitable
bifunctors. The latter shows that the objects in our constructions are
canonically associated with the category of $H$-bimodules as an abstract
category. Appendix \ref{app:cft} contains some motivation from (logarithmic)
conformal field theory.


\section{A Frobenius algebra in the bimodule category}\label{sec2}

\subsection{Finite-dimensional ribbon Hopf algebras}

In this section we collect some basic definitions and notation for Hopf
algebras and recall that \findim\ Hopf algebras admit a canonical
Frobenius algebra structure.

Throughout this paper, \ko\ is an algebraically closed field of
characteristic zero and, unless noted otherwise,
$H$ is a \findim\ factorizable ribbon Hopf algebra over \ko.
We denote by $m$, $\eta$, $\Delta$, $\eps$ and $\apo$ the product,
unit, coproduct, counit and antipode of the Hopf algebra $H$.

There exist plenty of factorizable ribbon Hopf algebras (see e.g.\
\cite{burc}). For instance, the Drinfeld double of a \findim\ Hopf algebra $K$
is factorizable ribbon provided that \cite[Thm.\,3]{kaRad} a certain condition
for the square of the antipode of $K$ is satisfied.
Let us recall what it means that a Hopf algebra is factorizable ribbon.

\begin{defi} ~\nxl1
(a)\, A Hopf algebra $H \,{\equiv}\, (H,m,\eta,\Delta,\eps,\apo)$ is called
\emph{quasitriangular} iff it is endowed with an invertible element
$R \iN H\oti H$ (called the \emph{R-matrix}) that intertwines the coproduct
and opposite coproduct, i.e.\ $\Delta^{\!\rm op} \eq \ad R \circ \Delta$,
and satisfies
   \be
   (\Delta \oti \id_H) \circ R = R_{13}\cdot R_{23} \qquand
   (\id_H \oti \Delta) \circ R = R_{13}\cdot R_{12} \,.
   \labl{deqf-qt}
(b)\, The \emph{monodromy matrix} $Q \iN H\oti H$ of a quasitriangular 
Hopf algebra $(H,R)$ is the invertible element
   \be
   Q := R_{21}\,{\cdot}\, R \equiv (m\oti m)
   \circ (\id_H^{}\oti\tauHH\oti\id_H^{}) \circ ((\tauHH\cir R) \oti R) \,.
   \ee
(c)\, A quasitriangular Hopf algebra $(H,R)$ is called a \emph{ribbon} Hopf 
algebra iff it is endowed with a central invertible element $v\iN H$, called the
\emph{ribbon element}, that satisfies $\apo \cir v \eq v$, $\eps \cir v \eq 1$ 
and $\Delta \cir v \eq (v\oti v) \cdo Q^{-1}$.
 \\[2pt]
(d)\, A quasitriangular Hopf algebra $(H,R)$ is called \emph{factorizable}
iff the monodromy matrix can be written as $Q \eq \sum_\ell h_\ell \oti k_\ell$
with $\{h_\ell\}$ and $\{k_\ell\}$ two vector space bases of $H$.
\end{defi} \smallskip

Here and below, the symbol $\otimes$ denotes the tensor product over \ko, and
for vector spaces $V$ and $W$ the linear map $\tau_{V,W}\colon V\oti W
\,{\stackrel\simeq\to}\, W\oti V$ is the flip map which exchanges the two
tensor factors.  Also, we canonically identify $H$ with $\Homk(\ko,H)$ and think
of elements of (tensor products over \ko\ of) $H$ and $\Hs \eq \Homk(H,\ko)$ as
(multi)linear maps. This has e.g.\ the advantage that many of our considerations
still apply directly in the situation that $H$ is a Hopf algebra, with adequate
additional structure and properties, in an arbitrary \ko-linear ribbon category
instead of \Vectk. Various properties of the R-matrix and of the ribbon element,
as well as of some further distinguished elements of $H$, will be recalled later
on. Note that we do not assume the Hopf algebra $H$ to be semisimple;
in particular, the ribbon element does not need to be semisimple.

We also need a few further ingredients that are available for general \findim\
Hopf algebras, without assuming quasitriangularity, in particular the notions of
(co)integrals and of a Frobenius structure for Hopf algebras.

\begin{defi} ~\nxl1
A \emph{left integral} of a Hopf algebra $H$ is a morphism of left $H$-modules 
from the trivial $H$-mo\-du\-le $(\ko,\eps)$ to the regular $H$-module $(H,m)$,
i.e.\ an element $\Lambda \iN H$ satisfying
$m \cir (\id_H \oti \Lambda) \eq \Lambda \cir \eps$.
 \\
A \emph{right cointegral} of $H$ is a morphism of right $H$-comodules from 
$(\ko,\eta)$ to $(H,\Delta)$, i.e.\ an element $\lambda \iN \Hs$ 
satisfying $(\lambda \oti \id_H) \cir \Delta \eq \eta \cir \lambda$.
 \\
Right integrals and left cointegrals are defined analogously.
\end{defi}

Recall \cite{laSw} that for a \findim\ Hopf \ko-algebra the antipode is
invertible and that $H$ has, up to normalization, a unique non-zero left integral
$\Lambda\iN H$ and a unique non-zero right cointegral $\lambda\iN\Hs$. The number
$\lambda\cir\Lambda \iN \ko$ is invertible. A factorizable ribbon Hopf algebra
is unimodular \cite[Prop.\,3(c)]{radf13}, i.e.\ the left integral
$\Lambda$ is also a right integral, implying that $\apo\cir\Lambda \eq \Lambda$.

\smallskip

The integral and the cointegral allow one to endow Hopf \ko-algebras with more
algebraic structure. The following characterization of Frobenius algebras
will be convenient.

\begin{defi}
A \emph{Frobenius algebra} $A$ in \Vectk\ is a vector space $A$ together with
(bi)linear maps $m_A$, $\eta_A$, $\Delta_A$ and $\eps_A$ such that $(A,m_A,\eta_A)$
is an (associative, unital) algebra, $(A,\Delta_A,\eps_A)$ is a
(coassociative, counital) coalgebra and
   \be
   (m_A \oti \id_A) \circ (\id_A \oti \Delta_A) = \Delta_A \oti m_A
   = (\id_A \oti m_A) \circ (\Delta_A \oti \id_A) \,,
   \ee
i.e.\ the coproduct $\Delta_A$ is a morphism of $A$-bimodules.
\end{defi} \smallskip

We have

\begin{lem}\label{H-A}
A \findim\ Hopf \ko-algebra $(H,m,\eta,\Delta,\eps,\apo)$ carries a canonical
structure of a Frobenius algebra $A$, with the same algebra structure on $A \eq H$,
and with Frobenius coproduct and Frobenius counit given by
   \be
   \Delta_A = (m \oti \apo) \circ (\id_A \oti (\Delta\cir\Lambda) )
   \qquand \eps_A = (\lambda \cir \Lambda)^{-1}\, \lambda \,.
   \labl{DeltaA-epsA}
\end{lem}

This actually holds more generally for finitely generated projective Hopf algebras
over commutative rings (see e.g.\ \cite{pare7,kaSt2}), as well as for any Hopf
algebra in an additive ribbon category \C\ that has an invertible antipode and
a left integral $\Lambda \iN \Hom(\one,H)$ and right cointegral $\lambda
\iN \Hom(H,\one)$ such that $\lambda \cir \Lambda \in\End_\C(\one)$ is invertible
(see e.g.\ appendix A.2 of \cite{fuSc17}).

The Frobenius algebra structure given by \erf{DeltaA-epsA} is unique up to
rescaling the integral $\Lambda$ by an invertible scalar. In the sequel, for a
given choice of (non-zero) $\Lambda$, we choose the cointegral $\lambda$
such that $\lambda \cir \Lambda \eq 1$.


\subsection{\HBimod\ as a monoidal category}

Our focus in this paper is on natural structures on a distinguished
$H$-bimodule, the \bicom\ to be described below. To formulate these
we need to endow the abelian category \HBimod\ of $H$-bimodules with the structure
of a sovereign braided monoidal category.

The objects of the \ko-linear abelian category \HBimod\ of bimodules over
a Hopf \ko-algebra $H$ are triples $(X,\rho,\ohr)$ such that $(X,\rho)$ is a
left $H$-module and $(X,\ohr)$ is a right $H$-module and the left and right
actions of $H$ commute, $\rho\cir(\id_H\oti\ohr) \eq \ohr\cir(\rho\oti\id_H)$.
Morphisms are \ko-linear maps commuting with both actions.
We denote the morphism spaces of \HMod\ and \HBimod\ by
$\HomH(-{,}-)$ and $\HomHH(-{,}-)$, respectively, while
$\Hom(-{,}-) \,{\equiv}\, \Homk(-{,}-)$ is reserved for \ko-linear maps.

Just like the bimodules over any unital associative algebra, \HBimod\ carries a
monoidal structure for which the tensor product is the one over $H$, for which
the vector space underlying a tensor product bimodule $X \,{\otimes}_H Y$ is a
non-trivial quotient of the vector space tensor product $X \oti Y \,{\equiv}\,
X \otik Y$. But for our purposes, we need instead a different monoidal
structure on \HBimod\ for which also the coalgebra structure of $H$ is relevant.
This is obtained by pulling back the natural $H\oti H$-bi\-module structure on
$X \oti Y$ along the coproduct to the structure of an $H$-bi\-module. Thus if
$(X,\rho_X,\ohr_X)$ and $(Y,\rho_Y,\ohr_Y)$ are $H$-bimodules, then their
tensor product is $X \oti Y$ together with the left and right actions
   \be
   \bearl
   \rho_{X\otimes Y} := (\rho_X \oti \rho_Y) \circ (\id_H \oti \tau_{H,X} \oti
                        \id_Y) \circ (\Delta \oti \id_X \oti \id_Y)  \qquand
   \Nxl3
   \ohr_{X\otimes Y} := (\ohr_X \oti \ohr_Y) \circ (\id_X \oti \tau_{Y,H} \oti
                        \id_H) \circ (\id_X \oti \id_Y \oti \Delta) 
   \eear
   \labl{def-tp}
of $H$. The monoidal unit for this tensor product is the one-dimensional vector
space \ko\ with both left and right $H$-action given by the counit,
$\one_{H\text{-Bimod}} \eq (\ko,\eps,\eps)$.

Obviously, \erf{def-tp} is just the standard tensor product
   \be
   (X,\rho_X) \otimes^\HMod (Y,\rho_Y) = \left(\,X\oti Y ,\,
   (\rho_X \oti \rho_Y) \circ (\id_H \oti \tau_{H,X} \oti
                        \id_Y) \circ (\Delta \oti \id_X \oti \id_Y) \,\right)
   \labl{otimesHMod}
of the category \HMod\ of left $H$-modules together with the corresponding
tensor product of the category of right $H$-modules. For both monoidal structures
the ground field \ko, endowed with a left, respectively right, action via the
counit, is the monoidal unit.

If $H$ is a \emph{ribbon} Hopf algebra, then (see e.g.\ Section XIV.6 of
\cite{KAss}) \HMod\ carries the structure of a ribbon category. Analogous
further structure on \HBimod\ will become relevant later on, and we will
introduce it in due course: a braiding on \HBimod\ in Section \ref{sec-braid},
and left and right dualities and a twist in Section \ref{sec-dual}.


\subsection{The \bicom}

We now identify an object of the monoidal category \HBimod\ that is distinguished
by the fact (see Appendix \ref{bicom-coend}) that it can be determined, up to
unique isomorphism,
by a universal property formulated in \HBimod, and thus may be thought of as being
canonically associated with \HBimod\ as a rigid monoidal category. Afterwards we will
endow this object \Hb\ with the structure of a Frobenius algebra in the monoidal
category defined by the tensor product \erf{def-tp}.
As a vector space, \Hb\ is the dual \Hs\ of $H$.

\begin{defi}
The \emph{\bicom} $\Hb\iN\HBimod$ is the vector space \Hs\ endowed
with the dual of the regular left and right actions of $H$ on itself. Explicitly,
   \be
   \Hb = (\Hs,\brho,\bohr) \,,
   \ee
with $\brho \iN \Hom(H\oti\Hs,\Hs)$ and $\bohr \iN \Hom(\Hs\oti H,\Hs)$ given by
   \be
   \bearl
   \brho:= (d_H\oti\idHs) \cir (\idHs\oti m\oti\idHs) \cir (\idHs\oti\apo\oti b_H)
   \cir\tauHHv \qquand
   \nxl2
   \bohr:= (d_H\oti\idHs)\cir(\idHs\oti m\oti\idHs)
   \cir(\idHs\oti\id_H\oti\tauHvH)\cir(\idHs\oti b_H \oti\apoi) \,.
   \eear
   \labl{rhorho}
\end{defi} \smallskip

Expressions involving maps like $\brho$ and $\bohr$ tend to become unwieldy, at
least for the present authors. It is therefore convenient to resort to a
pictorial description. We depict the structure maps of the Hopf algebra $H$ as\,%
  \footnote{~It is worth stressing that these pictures refer to the category \Vectk\
  of \findim\ \ko-vector spaces. Later on, we will occasionally also work with
  pictures for morphisms in more general monoidal categories \,\C; to
  avoid confusion we will mark pictures of the latter type with the symbol
   \,\raisebox{-.55em}{\catpicsm}\,.}
   \Eqpic{graphics} {430} {18} { \put(-12,0){
   \put(0,18)       {$ m~= $}
   \put(33,7)  { \includepichopf{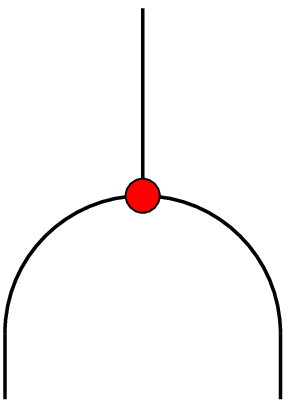}
   \put(-3.5,-8.8)  {\sse$ H $}
   \put(9.5,37.5)   {\sse$ H $}
   \put(20.2,-8.8)  {\sse$ H $}
   }
   \put(96,18)      {$ \eta~= $}
   \put(124,7) { \includepichopf{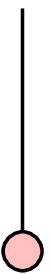}
   \put(-1.7,25.8)  {\sse$ H $}
   }
   \put(172,18)     {$ \Delta~= $}
   \put(204,7) { \includepichopf{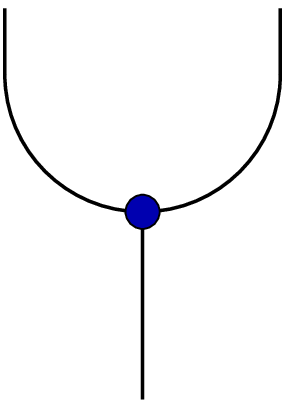}
   \put(-2.5,37.5)  {\sse$ H $}
   \put(8.1,-8.8)   {\sse$ H $}
   \put(20.8,37.5)  {\sse$ H $}
   }
   \put(268,18)     {$ \eps~= $}
   \put(295,7) { \includepichopf{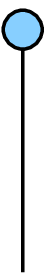}
   \put(-1.9,-8.8)  {\sse$ H $}
   }
   \put(349,18)     {$ \apo~= $}
   \put(382,9) {\Includepichtft{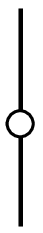}
   \put(-1.9,-8.8)  {\sse$ H $}
   \put(-1.7,27.2)  {\sse$ H $}
   }
   \put(420,18)     {$ \apoi~= $}
   \put(463,9) {\Includepichtft{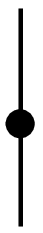}
   \put(-1.9,-8.8)  {\sse$ H $}
   \put(-1.7,27.2)  {\sse$ H $}
   } } }
the integral and cointegral as
   \eqpic{pic-int-coint} {130} {12} {
   \put(0,19)       {$ \Lambda ~= $}
   \put(37,0)  {\Includepichtft{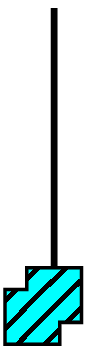}
   \put(3,41.1)     {\sse$ H $}
   }
   \put(100,19)     {$ \lambda ~= $}
   \put(137,0)  {\Includepichtft{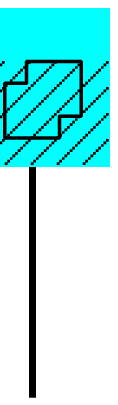}
   \put(-.7,-8.8)   {\sse$ H $}
   } }
and the evaluation and coevaluation maps, dual maps, and flip maps of \Vectk\ as
   \Eqpic{graphics2} {430} {27} { \put(-14,11){
   \put(0,18)       {$ d_V^{}~= $}
   \put(38,7)  { \includepichopf{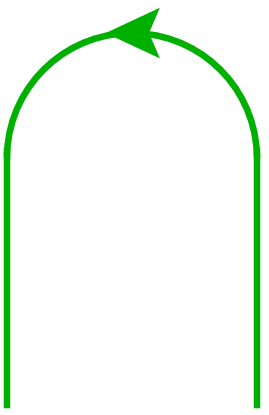}
   \put(-3.3,-8.8)  {\sse$ V^*_{} $}
   \put(19.2,-8.8)  {\sse$ V $}
   }
   \put(96,18)      {$ b_V^{}~= $}
   \put(134,-4) { \includepichopf{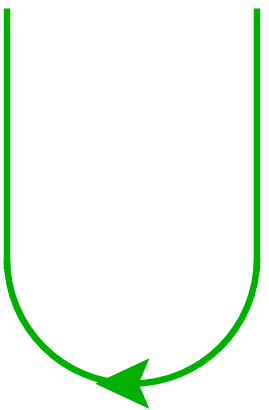}
   \put(-2.2,38.6)  {\sse$ V $}
   \put(18.9,38.6)  {\sse$ V^*_{} $}
   }
   \put(192,18)     {$ f_{}^\vee~= $}
   \put(230,0) { \includepichopf{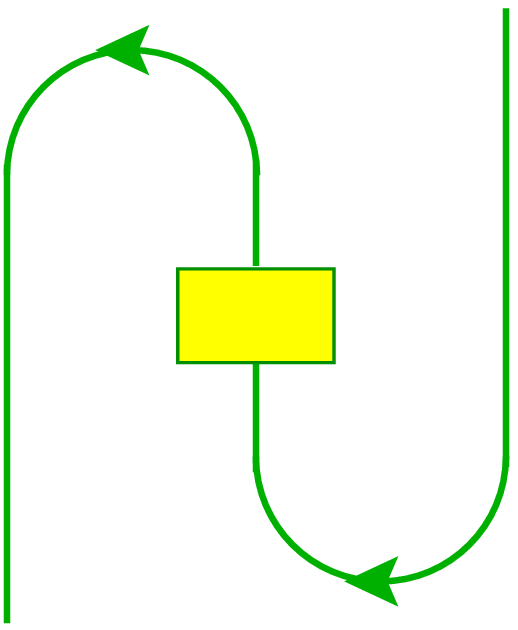}
   \put(-3.8,-8.8)  {\sse$ W^*_{} $}
   \put(20.4,25.7)  {\tiny$f$}
   \put(39.4,56.7)  {\sse$ V^*_{} $}
   }
   \put(309,18)     {$ \tau_{V,W}^{}~= $}
   \put(356,0) { \includepichopf{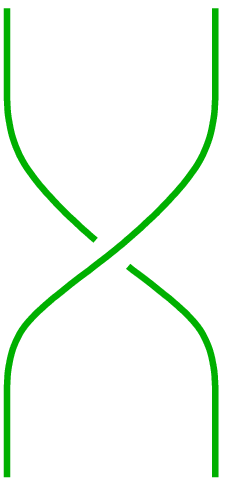}
   \put(-2.8,-8.8)  {\sse$ V $}
   \put(-2.8,45.5)  {\sse$ W $}
   \put(14.7,-8.8)  {\sse$ W $}
   \put(15.8,45.5)  {\sse$ V $}
   }
   \put(390,18)     {$ =~\tau_{W,V}^{-1}~= $}
   \put(456,0) {\includepichopf{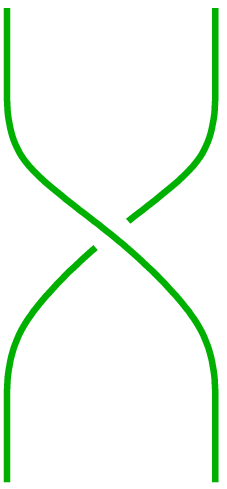}
   \put(-2.8,-8.8)  {\sse$ V $}
   \put(-2.8,45.5)  {\sse$ W $}
   \put(14.7,-8.8)  {\sse$ W $}
   \put(15.8,45.5)  {\sse$ V $}
   } } }
The left-pointing arrows in the pictures for the evaluation and
coevaluation indicate that they refer to the \emph{right} duality of \Vectk.
The evaluation and coevaluation $\tilde d$ and $\tilde b$ for the \emph{left}
duality of \Vectk\ are analogously drawn with arrows pointing to the right.
Also, for better readability we indicate the flip by either an over- or
underbraiding, even though in the present context of the \emph{symmetric}
monoidal category \Vectk\ both of them describe the same map.

\smallskip

In this graphical description the left and right actions \erf{rhorho} of $H$
on \Hb\ are given by
   \eqpic{rhoHb,ohrHb} {130} {41} {
   \put(0,0)   {\Includepichtft{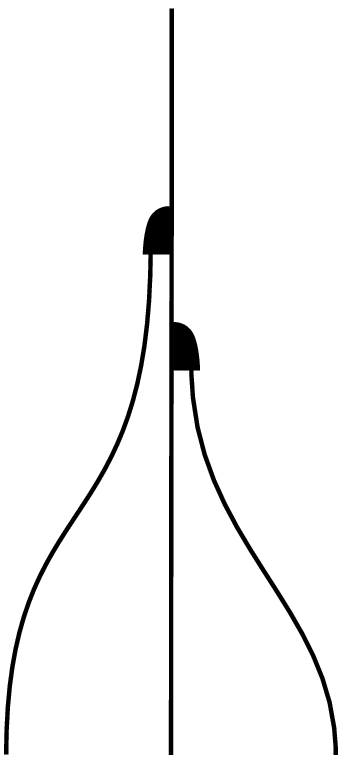}
   \put(-4,-8.8)    {\sse$ H $}
   \put(5,56)       {\sse$ \brho$}
   \put(13,-8.8)    {\sse$ \Hss $}
   \put(14.3,84.9)  {\sse$ \Hss $}
   \put(23,43)      {\sse$ \bohr$}
   \put(32,-8.8)    {\sse$ H $}
   }
   \put(58,38)      {$ =$}
   \put(88,0) { \Includepichtft{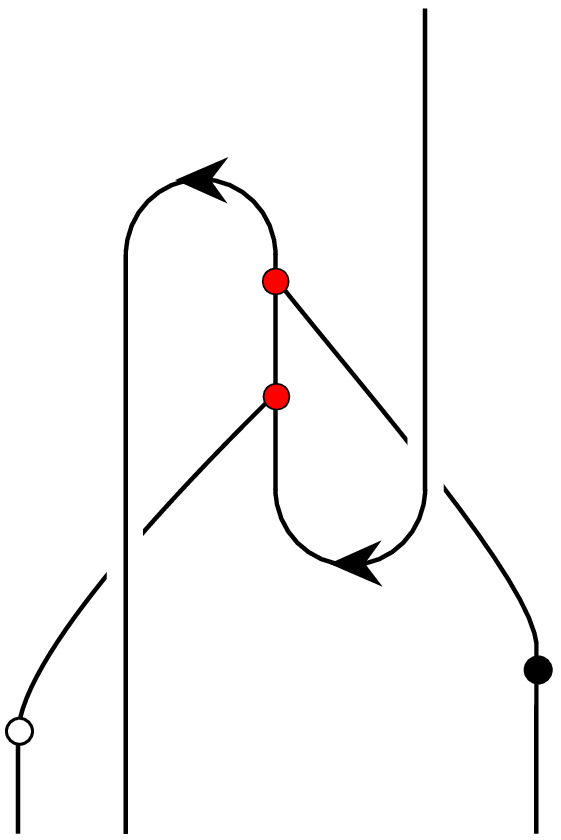}
   \put(-5,9)       {\sse$ \apo $}
   \put(-3,-8.8)    {\sse$ H $}
   \put(10,-8.8)    {\sse$ \Hss $}
   \put(42.3,93.3)  {\sse$ \Hss $}
   \put(54,-8.8)    {\sse$ H $}
   \put(61.6,16.2)  {\sse$ \apoi $}
   } }

Let us also mention that the \emph{Frobenius map} $\fmap\colon H \To \Hs$
and its inverse $\fmap^{-1}\colon \Hss \To H$ are given by
$\fmap(h) \eq \lambda \,{\leftharpoonup}\, \apo(h)$ and
$\fmap^{-1}(p) \eq \Lambda \,{\leftharpoonup}\, p$, respectively
(see e.g.\ \cite{coWe6}), i.e.
   \eqpic{def_fmap} {350} {47} {
   \put(1,40)     {$ \fmap ~= $}
   \put(39,0) {  \Includepichopf{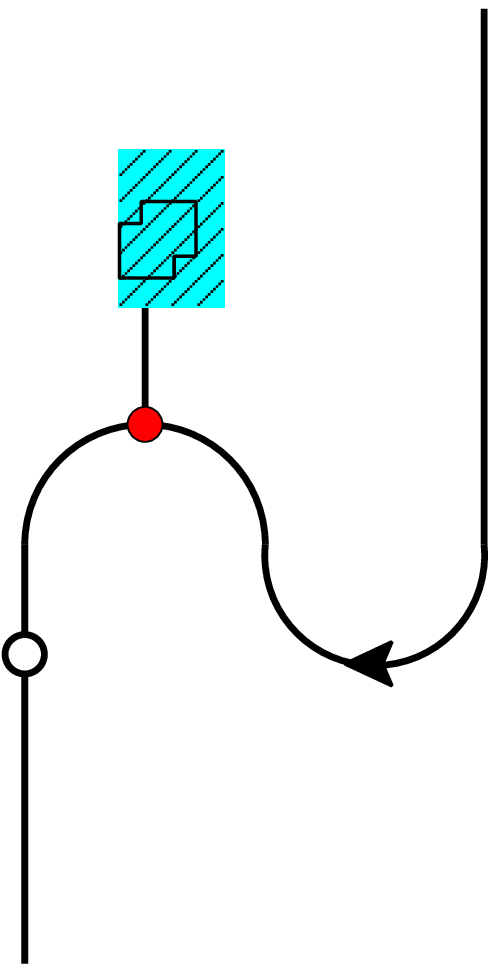}
   \put(-1.8,-8.5){\sse$ H $}
   \put(5.9,32.4) {\sse$ \apo $}
   \put(23.6,77)  {\sse$ \lambda $}
   \put(48.7,109) {\sse$ \Hss $}
   \put(107,40)   {and}
   }
  \put(226,0){
   \put(0,40)     {$ \fmap^{-1} ~= $}
   \put(49,0) {  \Includepichopf{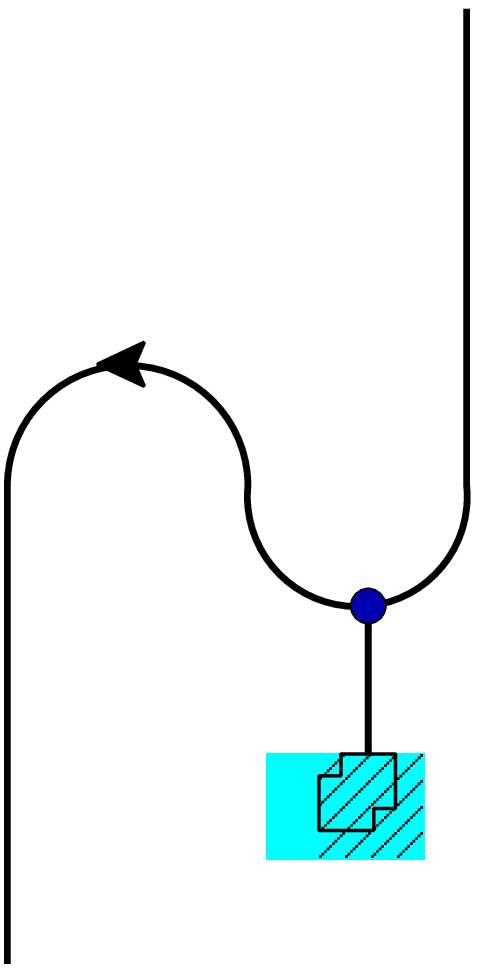}
   \put(-4.1,-8.5){\sse$ \Hss $}
   \put(41.5,10.2){\sse$ \Lambda $}
   \put(47.5,109) {\sse$ H $}
   } } }
The statement that $H$ is a Frobenius algebra (see Lemma \ref{H-A}) is equivalent
to the invertibility of $\fmap$. That the two maps \erf{def_fmap} are indeed
each others' inverses means that
   \eqpic{fmap_inv} {165} {38} { \setlength\unitlength{.8pt}
   \put(0,0)   {\includepichopf{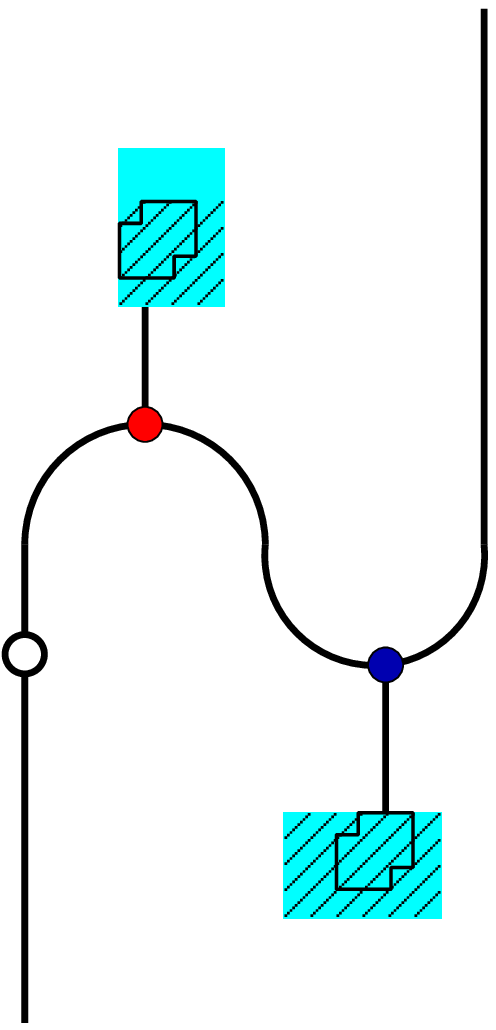}}
   \put(-5.2,39) {\sse$\apo$}
   \put(-3,-10.5) {\sse$H$}
   \put(23.1,84) {\sse$\lambda$}
   \put(47.2,17) {\sse$\Lambda$}
   \put(49,116)  {\sse$H$}
   \put(83,50)   {$=$}
   \put(120,0) {\includepichopf{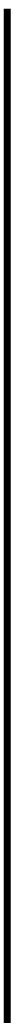}
   \put(-4.5,-10.5){\sse$H$}
   \put(-3.6,116){\sse$H$}
   }
   \put(150,50)  {$=$}
   \put(190,0) {\includepichopf{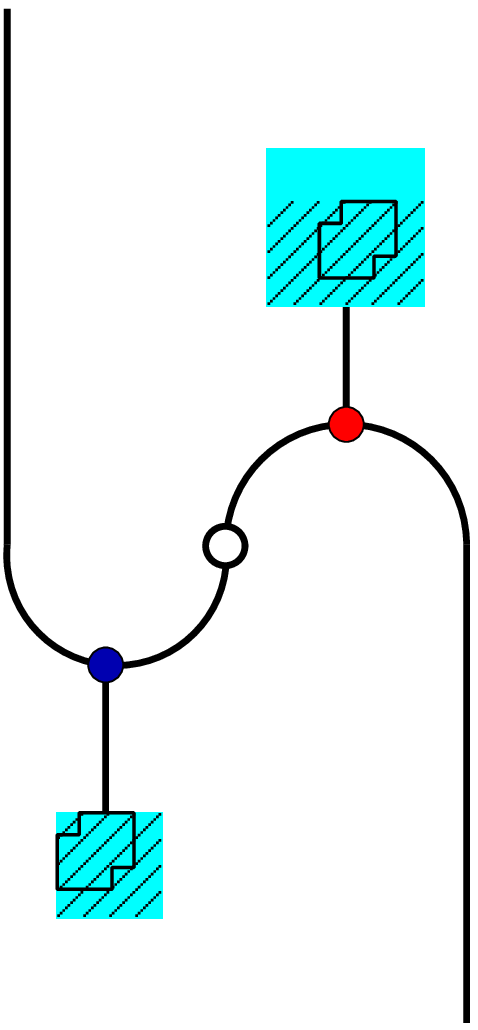}
   \put(-3.8,116){\sse$H$}
   \put(46,-10.5) {\sse$H$}
   } }


\subsection{Morphisms for algebraic structure on the bimodule \Hb}

We now introduce the morphisms that endow the object \Hb\ with
the structure of a Frobenius algebra in the monoidal category
\HBimod. Very much like the \bicom\ \Hb\ itself, the algebra structure
on \Hb\ is a consequence of the universal properties of the coend
of a functor $\Fbx\colon \HMod\op \Times \HMod \To \HBimod$
(see Appendix \ref{bicom-coend}). Analogous coends with similar properties
can be introduced for any rigid braided monoidal category, so that the Frobenius
algebra \Hb\ can be thought of as being canonically associated with the
(abstract) monoidal category \HMod\ and the functor $\Fbx$.

\begin{defi}
For $H$ a \findim\ Hopf algebra, we introduce the following linear maps
$m\bico \colon \Hs\oti\Hs\to\Hs$, $\eta\bico \colon \ko\to\Hs$,
$\Delta\bico \colon \Hs\to\Hs\oti\Hs$ and $\eps\bico \colon \Hs\to\ko$:
   \be
   \bearl
   m\bico := {\Delta^{\phantom:}}^{\!\!\!\wee_{}} \,, \qquad
   \eta\bico := \eps^\wee \,,
   \nxl4
   \Delta\bico := {[ (\id_H \oti (\lambda\cir m)) \cir (\id_H\oti\apo\oti\id_H)
   \cir (\Delta\oti\id_H) ]}^\wee \,, \qquad
   \eps\bico := {\Lambda^{\phantom:}}^{\!\!\!\wee} \,.
   \eear
   \labl{def-Hb-Frobalgebra}
\end{defi} \smallskip

Again the graphical description appears to be convenient:
   \Eqpic{pic-Hb-Frobalgebra} {440} {45} { \put(0,-3){
   \put(0,45)       {$ m\bico~= $}
   \put(48,0)  {\Includepichtft{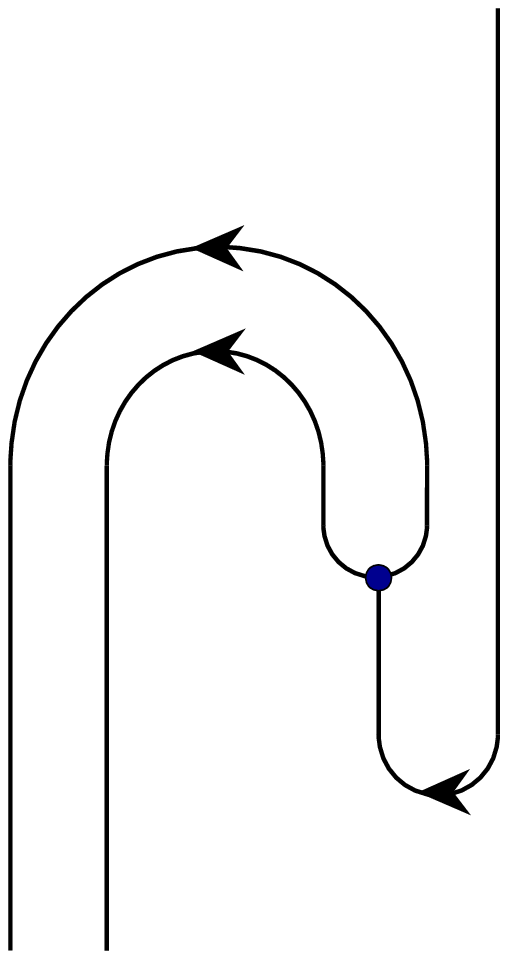}
   \put(-5.9,-8.8)  {\sse$\Hss $}
   \put(6.5,-8.8)   {\sse$\Hss $}
   \put(49.7,106.8) {\sse$\Hss $}
   }
   \put(142,45)     {$ \eta\bico~= $}
   \put(182,24) {\Includepichtft{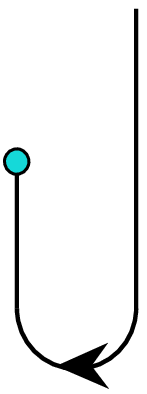}
   \put(10.7,44.1)  {\sse$\Hss $}
   }
   \put(238,45)     {$ \Delta\bico~= $}
   \put(286,0) {\Includepichtft{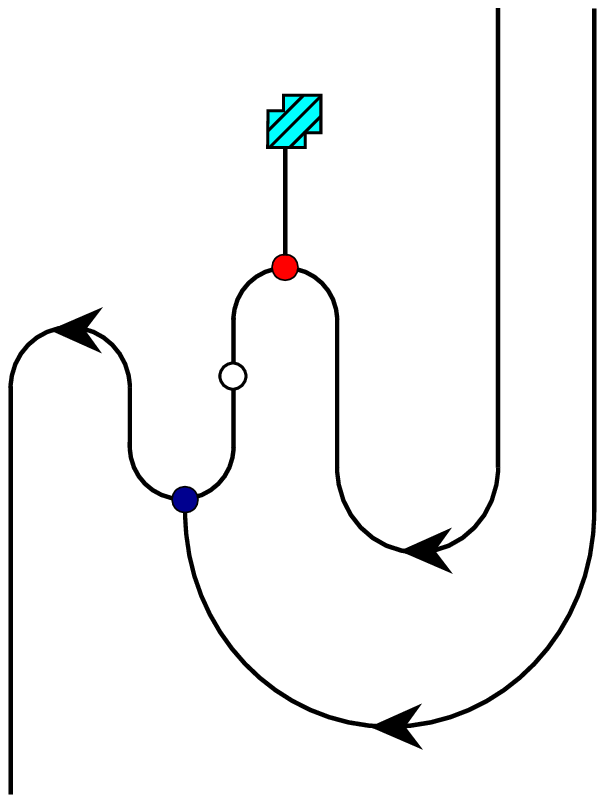}
   \put(-4.3,-8.8)  {\sse$\Hss $}
   \put(21.7,71)    {\sse$\lambda$}
   \put(48.2,89.4)  {\sse$\Hss $}
   \put(61.1,89.4)  {\sse$\Hss $}
   }
   \put(390,45)     {$ \eps\bico~= $}
   \put(430,24) {\Includepichtft{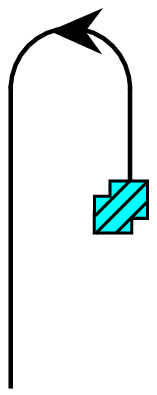}
   \put(-4.3,-8.8)  {\sse$\Hss $}
   \put(15.6,16.3)  {\sse$\Lambda$}
   } } }

We would like to interpret the maps \erf{pic-Hb-Frobalgebra} as the structural
morphisms of a Frobenius algebra in \HBimod. To this end we must first show that
these maps are actually morphisms of bimodules. We start with a few general
observations.

\begin{lem}
(i)\, For any Hopf algebra $H$ we have
   \eqpic{Hopf_Frob_trick} {430} {47} {
   \put(0,0)   {\Includepichtft{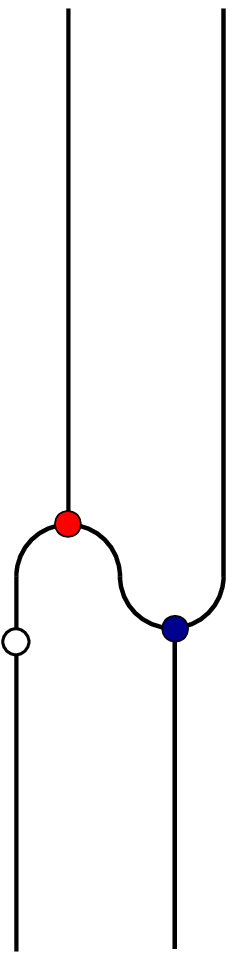}
   \put(-3,-8.5) {\sse$H$}
   \put(4,107)   {\sse$H$}
   \put(15,-8.5) {\sse$H$}
   \put(21,107)  {\sse$H$}
   }
   \put(50,47)   {$=$}
   \put(76,0) { \Includepichtft{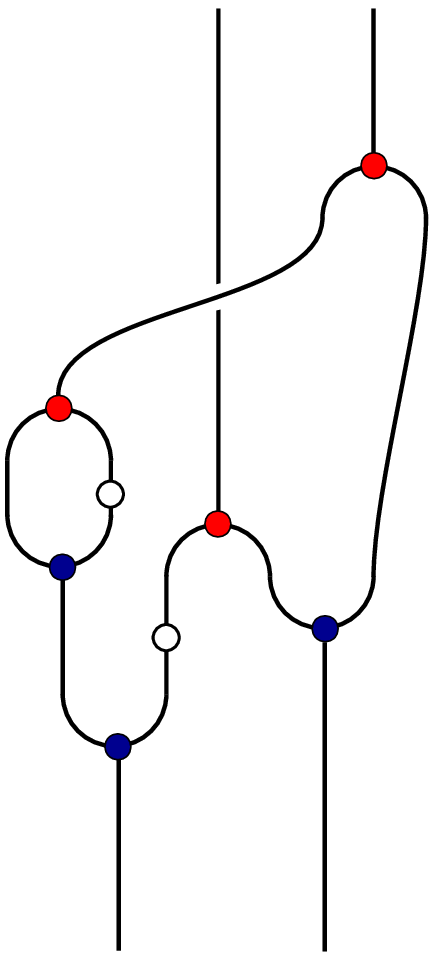}
   \put(8,-8.5)  {\sse$H$}
   \put(20,107)  {\sse$H$}
   \put(30,-8.5) {\sse$H$}
   \put(37,107)  {\sse$H$}
   }
   \put(147,47)  {$=$}
   \put(178,0)   {\Includepichtft{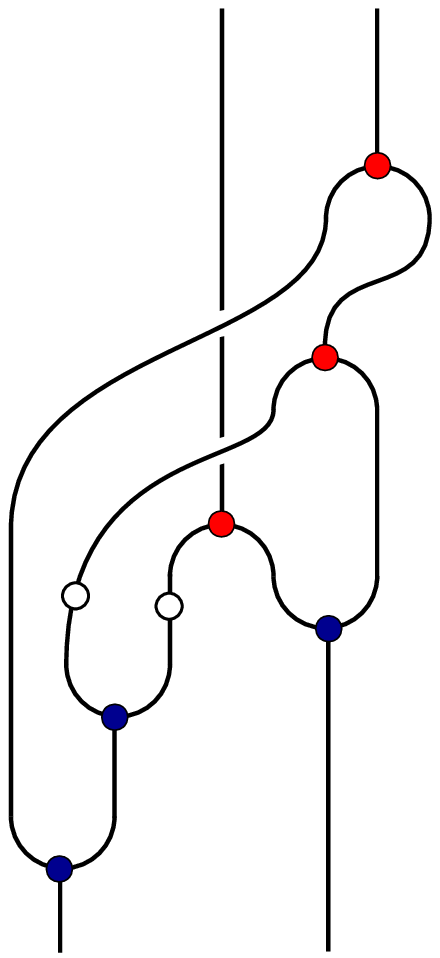}
   \put(1,-8.5)  {\sse$H$}
   \put(19.7,107){\sse$H$}
   \put(31,-8.5) {\sse$H$}
   \put(37,107)  {\sse$H$}
   }
   \put(246,47)  {$=$}
   \put(274,0) { \Includepichtft{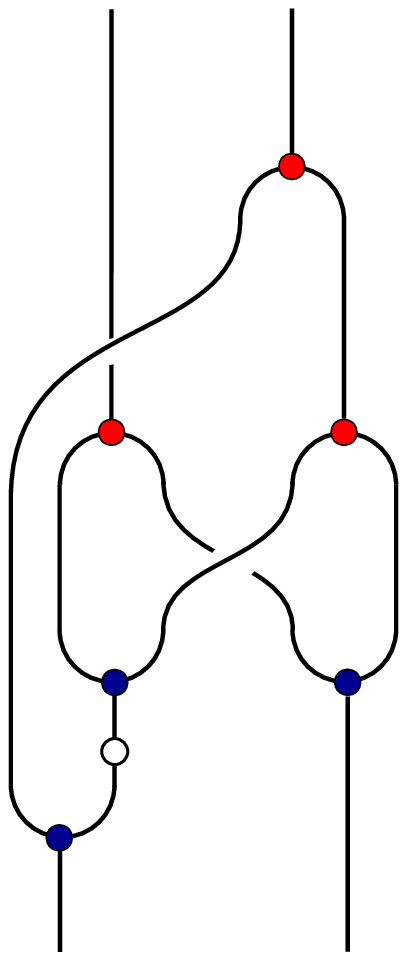}
   \put(1,-8.5)  {\sse$H$}
   \put(8,107)   {\sse$H$}
   \put(27.7,107){\sse$H$}
   \put(33,-8.5) {\sse$H$}
   }
   \put(347,47)  {$=$}
   \put(375,0) { \Includepichtft{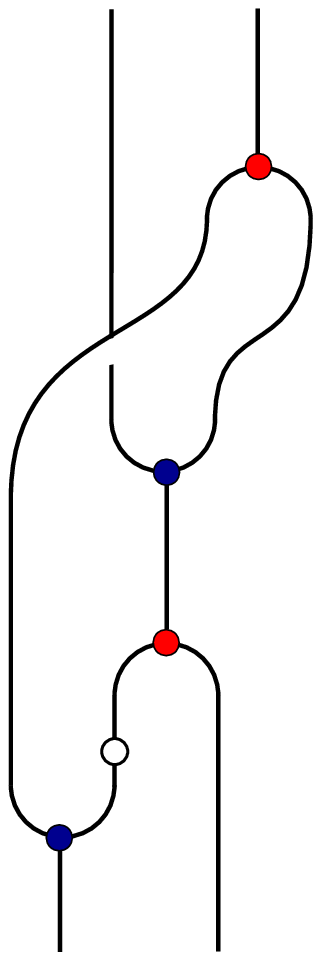}
   \put(1,-8.5)  {\sse$H$}
   \put(8,107)   {\sse$H$}
   \put(19,-8.5) {\sse$H$}
   \put(24,107)  {\sse$H$}
   } }
(ii)\, Further, if $H$ is unimodular with integral $\Lambda$, we have
   \eqpic{Hopf_Frob_trick2} {390} {39} {
   \put(0,0) {\Includepichtft{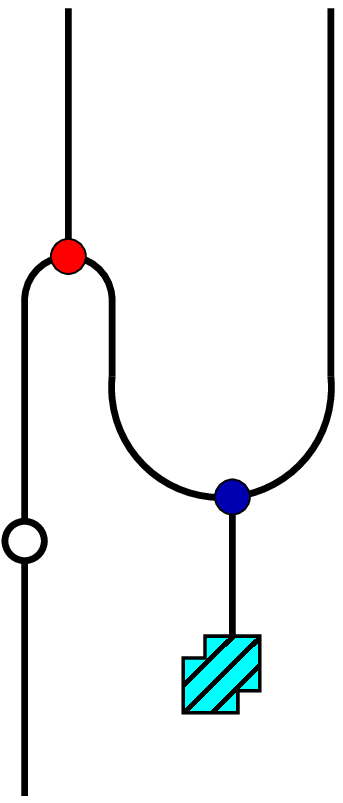} \put(2.2,0){
   \put(-4.2,-8.5) {\sse$ H $}
   \put(2.2,92)  {\sse$ H $}
   \put(27.7,10) {\sse$ \Lambda $}
   \put(31,92)   {\sse$ H $}
   } }
   \put(60,43)   {$ = $}
   \put(86,0) { \Includepichtft{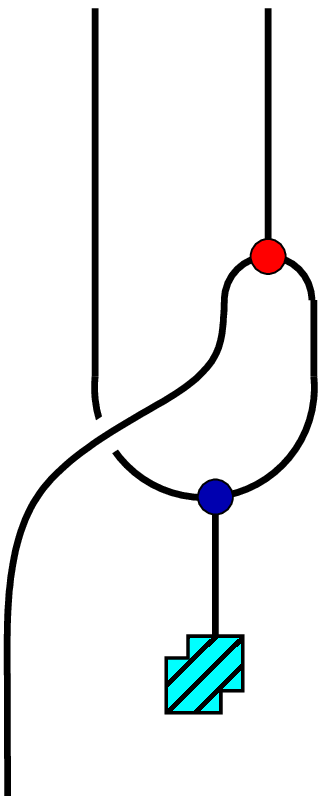}
   \put(-4.2,-8.5) {\sse$ H $}
   \put(6.6,92)  {\sse$ H $}
   \put(26.6,92) {\sse$ H $}
   \put(27.7,10) {\sse$ \Lambda $}
   }
   \put(174,43)  {and}
   \put(240,0) { \Includepichtft{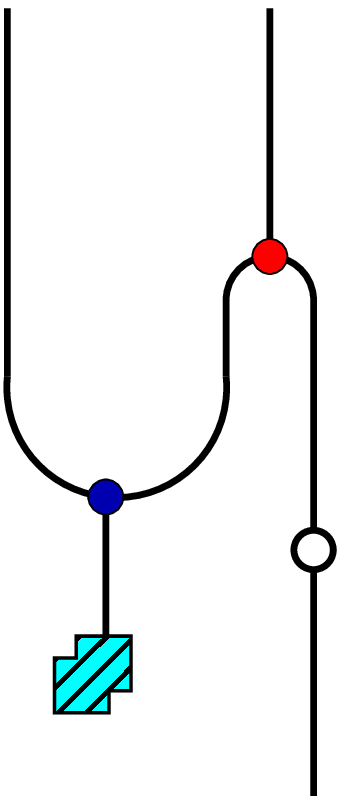}
   \put(-3.2,92) {\sse$ H $}
   \put(15.7,10) {\sse$ \Lambda $}
   \put(26.1,92) {\sse$ H $}
   \put(29.8,-8.5) {\sse$ H $}
   }
   \put(300,43)  {$ = $}
   \put(326,0) { \Includepichtft{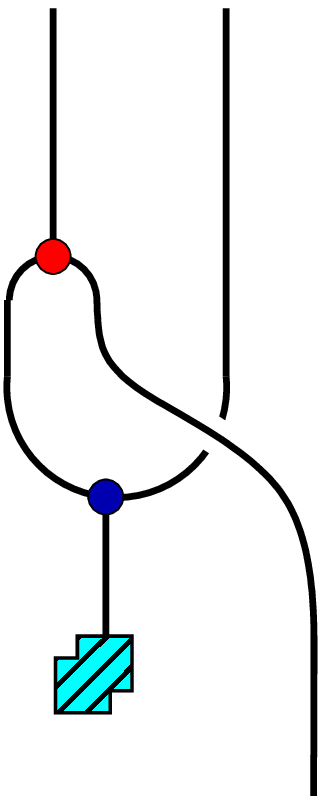}
   \put(1.2,92)  {\sse$ H $}
   \put(15.7,10) {\sse$ \Lambda $}
   \put(21.7,92) {\sse$ H $}
   \put(29.8,-8.5) {\sse$ H $}
   } }
\end{lem}

\begin{proof}
(i)\, The first equality holds by the defining properties of the antipode, unit and
counit of $H$. The second equality follows by associativity and coassociativity,
the third by the anti-coalgebra morphism property of the antipode, and the last
by the connecting axiom for product and coproduct of the bialgebra underlying $H$.
\nxl1
(ii)\, The first equality in \erf{Hopf_Frob_trick2} follows by composing
\erf{Hopf_Frob_trick} with $\id_H\oti\Lambda$ and using that $\Lambda$ is a left
integral. The second equality in \erf{Hopf_Frob_trick2} follows by composing the
left-right-mirrored version of \erf{Hopf_Frob_trick} (which is proven in the same
way as in (i)) with $\Lambda\oti\id_H$ and using that $\Lambda$ is a right integral.
\end{proof}

We will refer to the equality of the left and right hand sides of
\erf{Hopf_Frob_trick} as the \emph{Hopf-Frobenius trick}.

\begin{lem}\label{Delta=delp}
The map $\Delta\bico$ introduced in \erf{pic-Hb-Frobalgebra} can alternatively be
expressed as $\Delta\bico \eq \delp$ with
   \eqpic{Deltaprime} {133} {36} {
   \put(0,35)    {$ \delp ~:=$}
   \put(68,0) {\Includepichtft{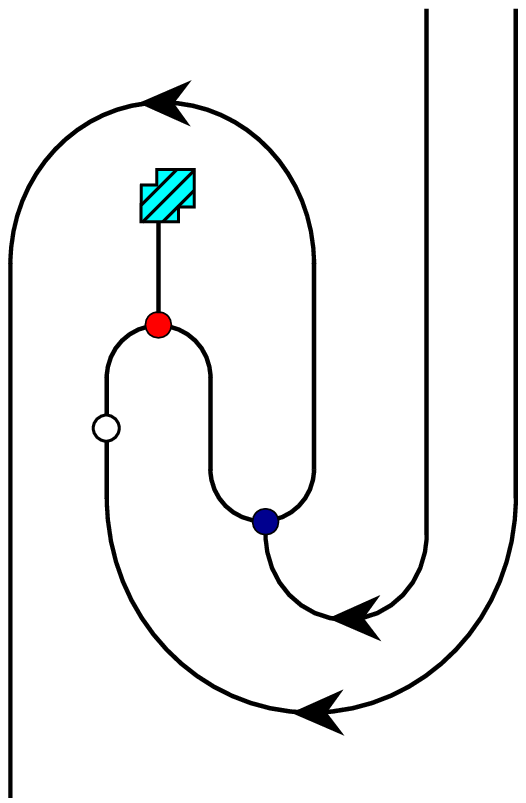}
   \put(-5,-8.5) {\sse$\Hss $}
   \put(21,62)   {\sse$\lambda$}
   \put(41,91)   {\sse$\Hss $}
   \put(53,91)   {\sse$\Hss $}
   } }
\end{lem}

\begin{proof}
Using the two equalities in \erf{fmap_inv} and coassociativity of $\Delta$ we obtain
   \eqpic{copr_equal} {250} {36} {
   \put(0,0)   {\Includepichtft{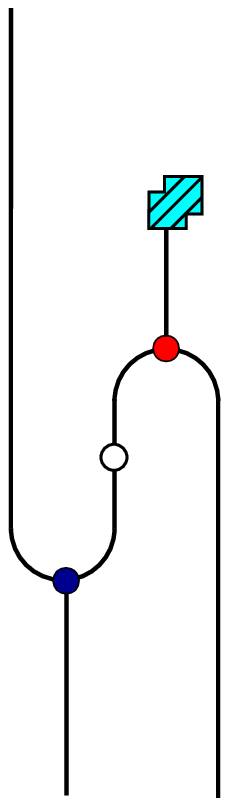}
   \put(-3,91)   {\sse$H$}
   \put(2,-8.5)  {\sse$H$}
   \put(19,-8.5) {\sse$H$}
   }
   \put(45,35)   {$=$}
   \put(70,0)  {\Includepichtft{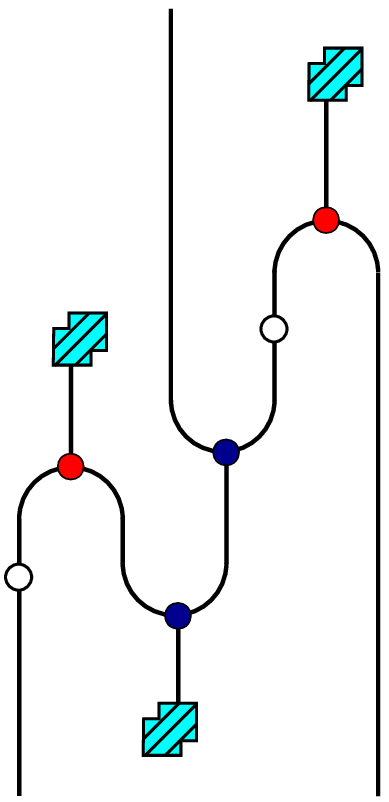}
   \put(-2.6,-8.5) {\sse$H$}
   \put(16,91)   {\sse$H$}
   \put(37,-8.5) {\sse$H$}
   }
   \put(130,35)  {$=$}
   \put(160,0) {\Includepichtft{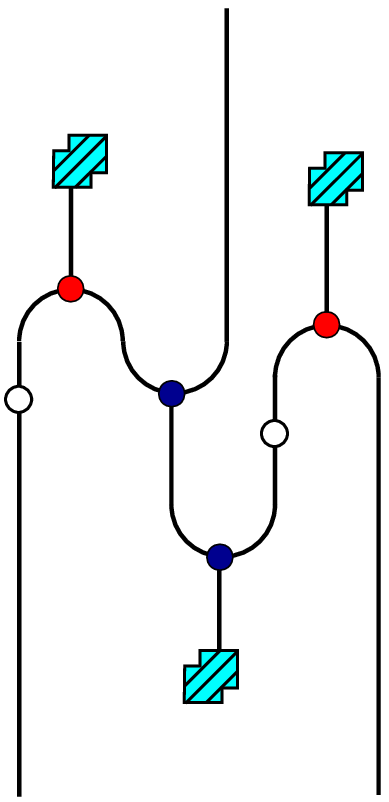}
   \put(-2.6,-8.5) {\sse$H$}
   \put(21,91)   {\sse$H$}
   \put(37,-8.5) {\sse$H$}
   }
   \put(220,35)  {$=$}
   \put(250,0) {\Includepichtft{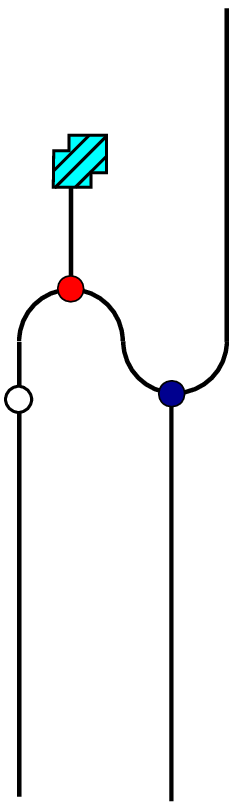}
   \put(-2.6,-8.5) {\sse$H$}
   \put(13,-8.5) {\sse$H$}
   \put(21,91)   {\sse$H$}
   } }
Dualizing the expressions on the left and right hand sides of \erf{copr_equal}
establishes the claimed equality.
\end{proof}

\begin{prop}
When \HBimod\ is endowed with the tensor product \erf{def-tp}, \ko\ is given
the structure of the trivial $H$-bimodule $\ko_\eps \eq (\ko,\eps,\eps)$
(the monoidal unit of \HBimod) and \Hs\ the $H$-bimodule structure \erf{rhorho},
then the maps \erf{def-Hb-Frobalgebra} are morphisms of $H$-bimodules.
\end{prop}

\begin{proof}
(i)~\,That $m\bico$ is a morphism of left $H$-modules is seen as
   \Eqpic{mb_left} {450} {47} { \put(0,3){
   \put(0,0)   {\Includepichtft{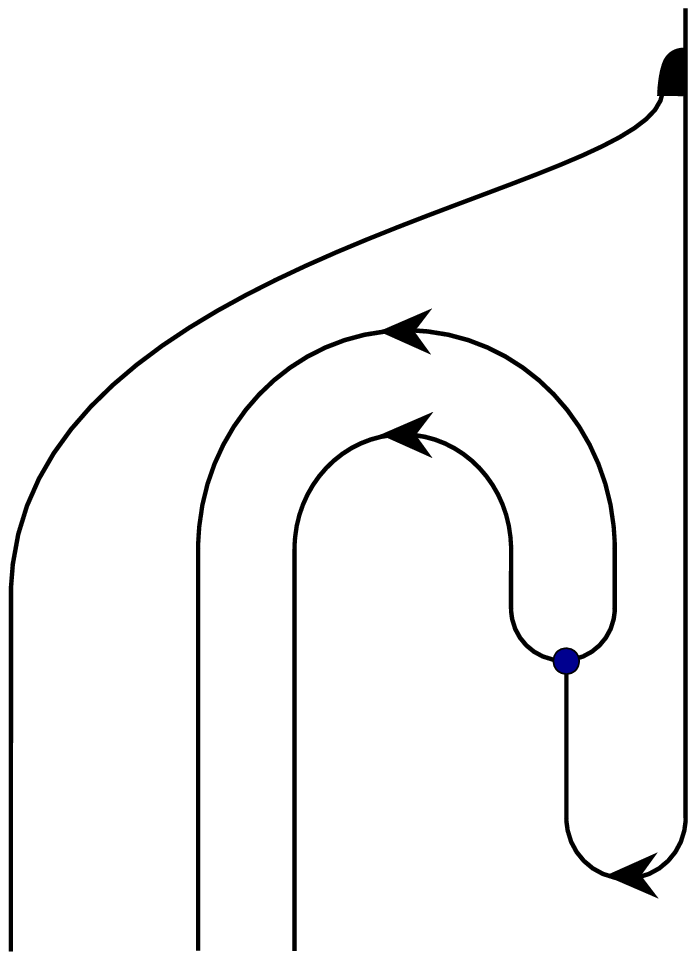}
   \put(-4,-8.5) {\sse$H$}
   \put(13.8,-8.5) {\sse$\Hss $}
   \put(27.8,-8.5) {\sse$\Hss $}
   \put(70.6,107){\sse$\Hss $}
   }
   \put(88,47)   {$=$}
   \put(93,0) { \Includepichtft{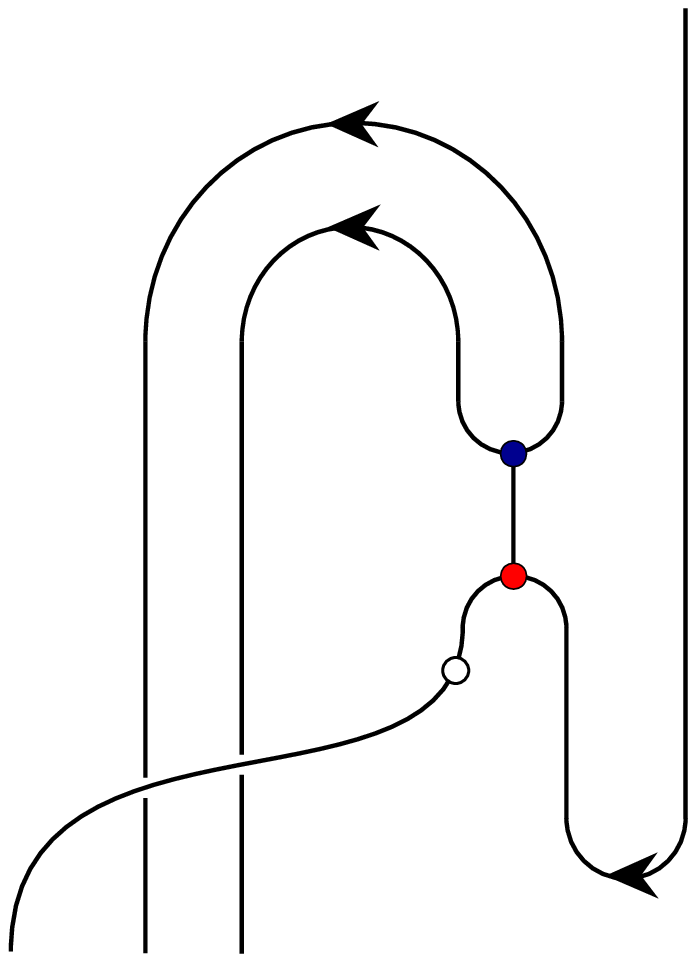}
   \put(-4,-8.5) {\sse$H$}
   \put(9.5,-8.5){\sse$\Hss $}
   \put(23,-8.5) {\sse$\Hss $}
   \put(70.6,107){\sse$\Hss $}
   }
   \put(184,47)  {$=$}
   \put(193,0)   {\Includepichtft{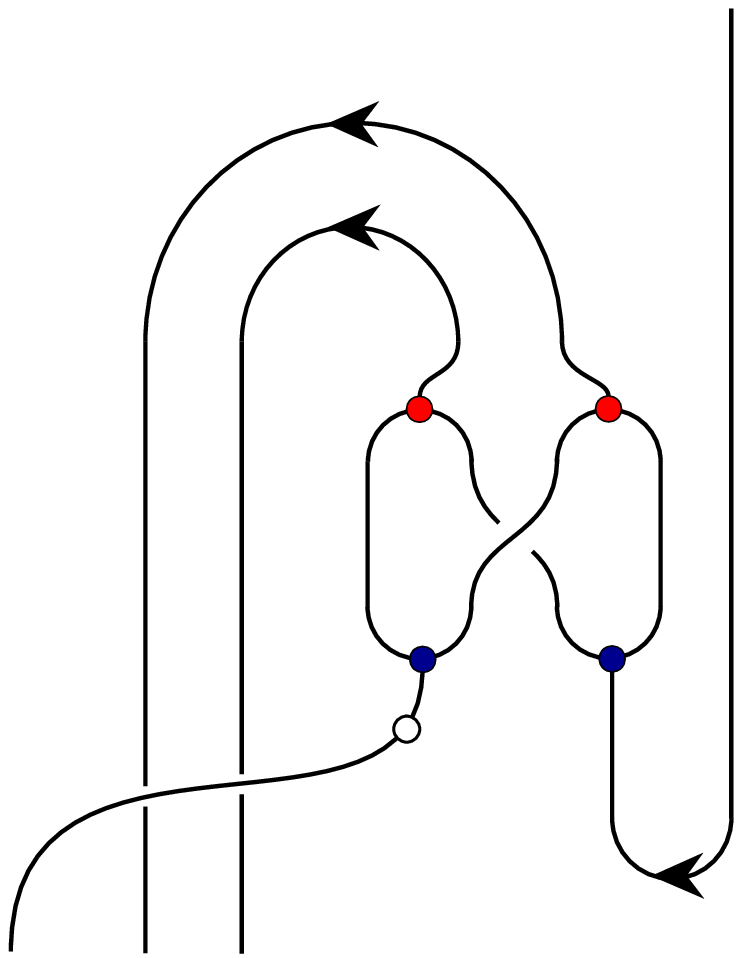}
   \put(-4,-8.5) {\sse$H$}
   \put(9.5,-8.5){\sse$\Hss $}
   \put(23,-8.5) {\sse$\Hss $}
   \put(75.6,107){\sse$\Hss $}
   }
   \put(284,47)   {$=$}
   \put(288,0) { \Includepichtft{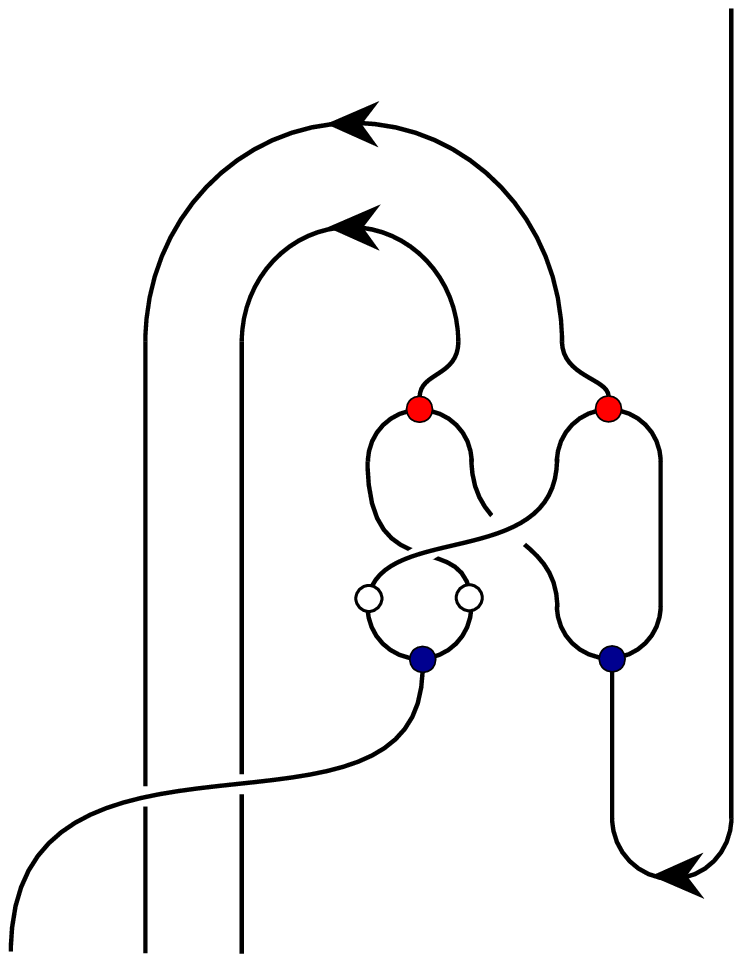}
   \put(-4,-8.5) {\sse$H$}
   \put(9.5,-8.5){\sse$\Hss $}
   \put(23,-8.5) {\sse$\Hss $}
   \put(75.8,107){\sse$\Hss $}
   }
   \put(383,47)   {$=$}
   \put(394,0) { \Includepichtft{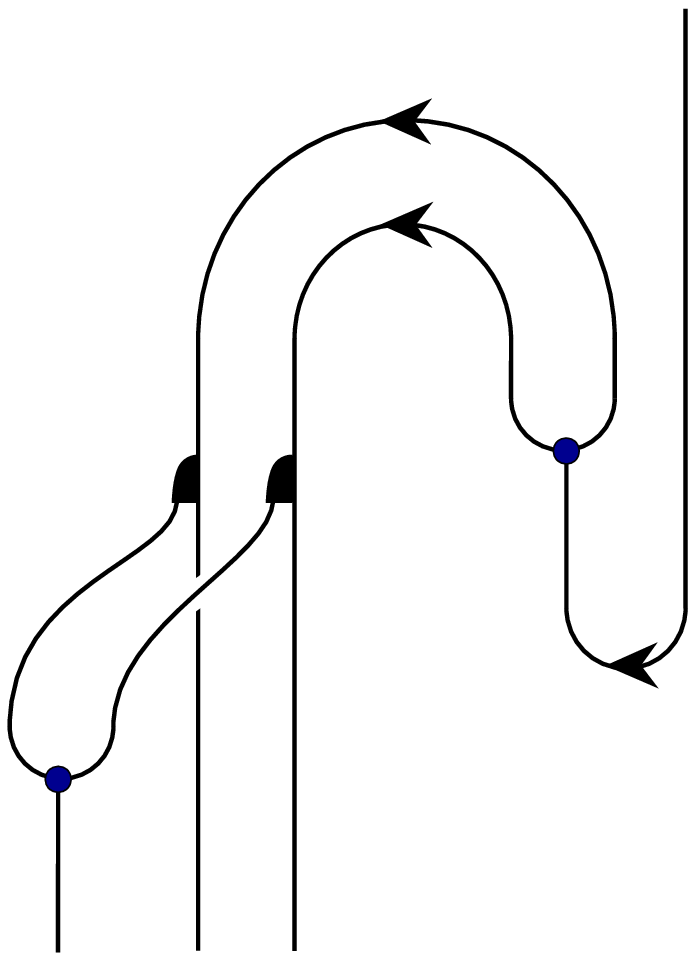}
   \put(1,-8.5)  {\sse$H$}
   \put(14.5,-8.5) {\sse$\Hss $}
   \put(28,-8.5) {\sse$\Hss $}
   \put(70.2,107){\sse$\Hss $}
   } } }
Here the first and last equalities just implement the definition \erf{rhoHb,ohrHb}
of the $H$-action, the second is the connecting axiom of $H$, and the third the
anti-algebra morphism property of the antipode.
\\
Similarly, that $m\bico$ is also a right module morphism follows as
   \Eqpic{mb_right} {450} {49} { \put(0,8){
   \put(0,0)   {\Includepichtft{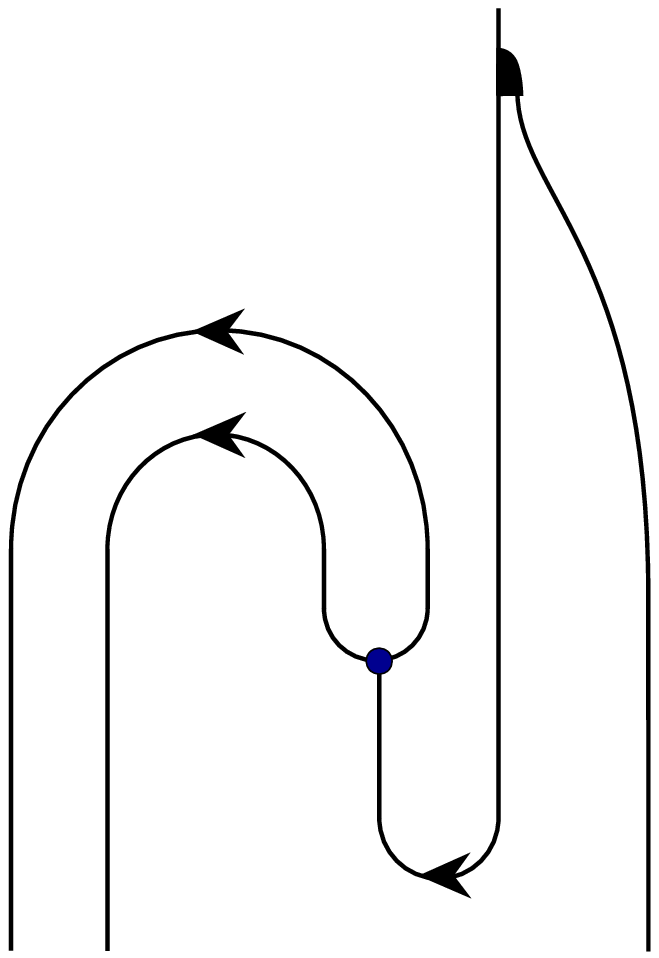}
   \put(-7,-8.5) {\sse$\Hss $}
   \put(6.6,-8.5){\sse$\Hss $}
   \put(66,-8.5) {\sse$H$}
   \put(50.2,107){\sse$\Hss $}
   }
   \put(84,47)   {$=$}
   \put(106,0) { \Includepichtft{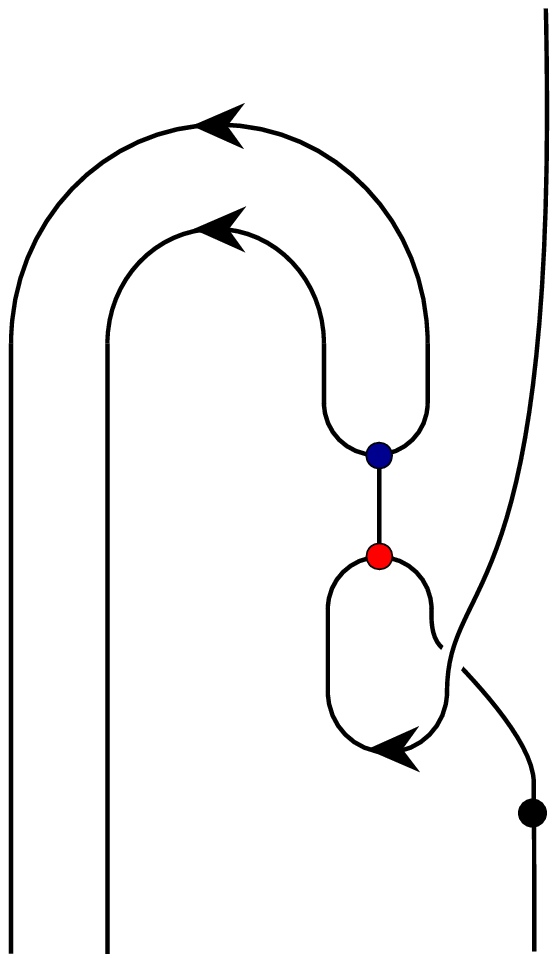}
   \put(-7,-8.5) {\sse$\Hss $}
   \put(6.6,-8.5){\sse$\Hss $}
   \put(53,-8.5) {\sse$H$}
   \put(55.8,107){\sse$\Hss $}
   }
   \put(181,47)  {$=$}
   \put(207,0)   {\Includepichtft{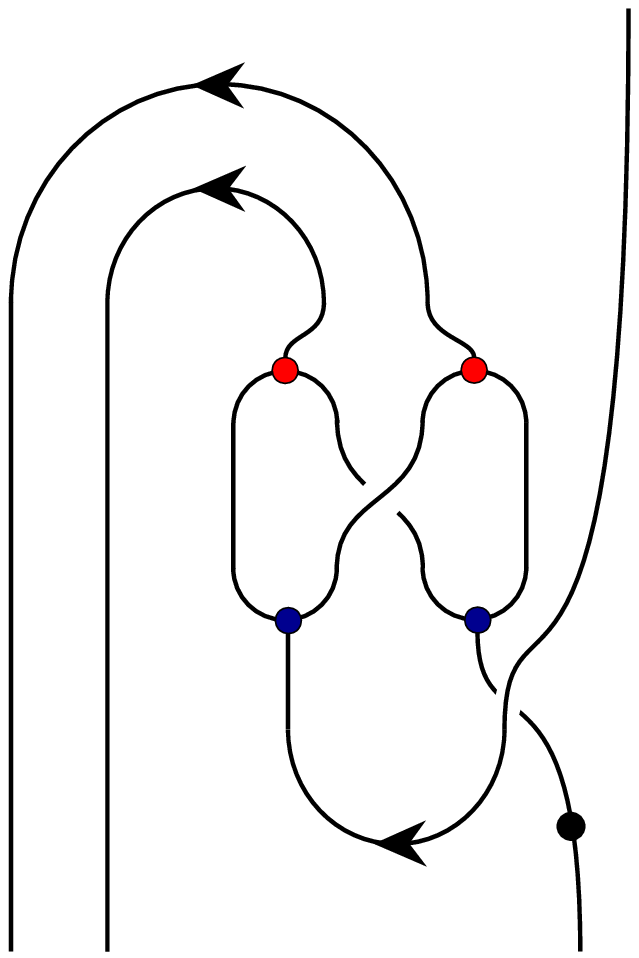}
   \put(-7,-8.5) {\sse$\Hss $}
   \put(6.6,-8.5){\sse$\Hss $}
   \put(58,-8.5) {\sse$H$}
   \put(65.2,107){\sse$\Hss $}
   }
   \put(286,47)  {$=$}
   \put(305,0) { \Includepichtft{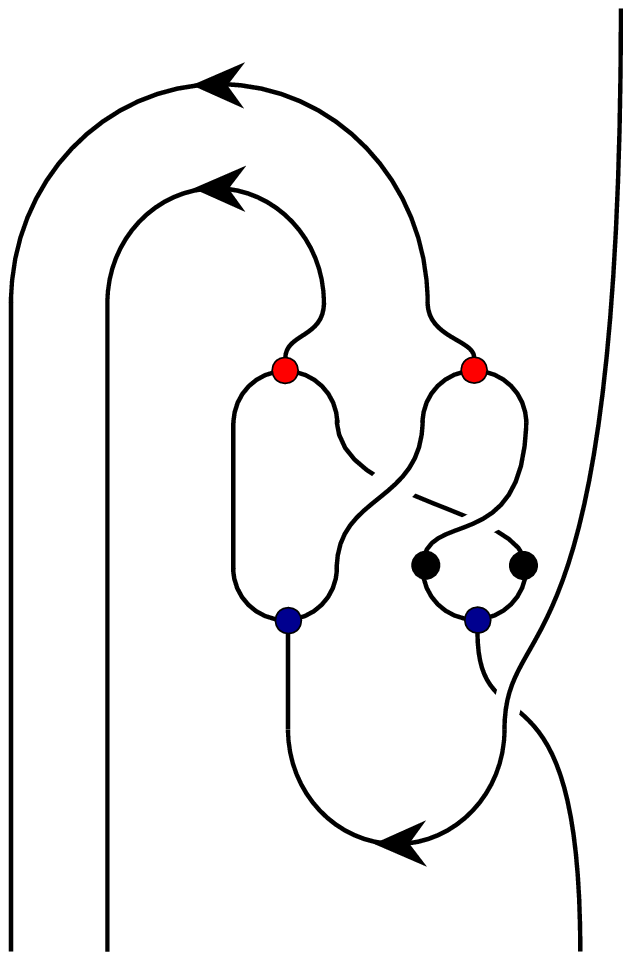}
   \put(-7,-8.5) {\sse$\Hss $}
   \put(6.5,-8.5){\sse$\Hss $}
   \put(58,-8.5) {\sse$H$}
   \put(63.2,107){\sse$\Hss $}
   }
   \put(388,47)  {$=$}
   \put(409,0) { \Includepichtft{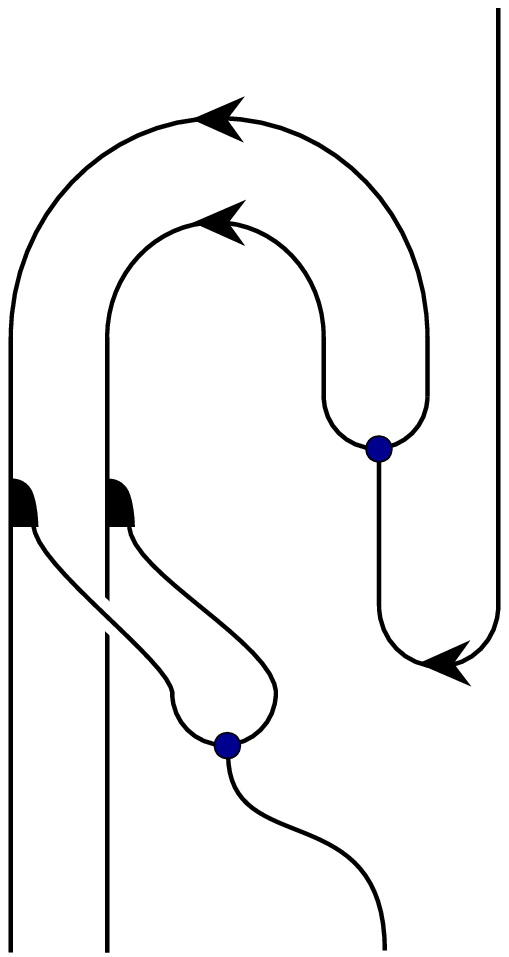}
   \put(-7,-8.5) {\sse$\Hss $}
   \put(6.5,-8.5){\sse$\Hss $}
   \put(37,-8.5) {\sse$H$}
   \put(50.2,107){\sse$\Hss $}
   } } }
(ii)~\,That $\eta\bico$ is a left and right module morphism follows
with the help of the properties $\eps\cir m \eq \eps\oti\eps$ and
$\eps\cir\apo\eq\eps$ of the antipode. We have
   \eqpic{etab_left} {270} {33} {
   \put(0,0)   {\Includepichtft{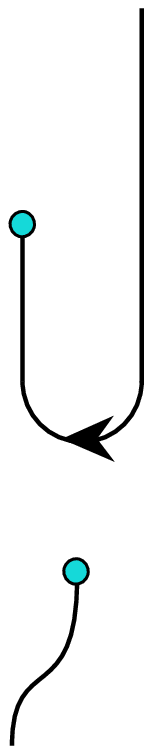}
   \put(-4,-8.5) {\sse$H$}
   \put(11.6,84) {\sse$ \Hss $}
   }
   \put(35,32)   {$=$}
   \put(60,0) { \Includepichtft{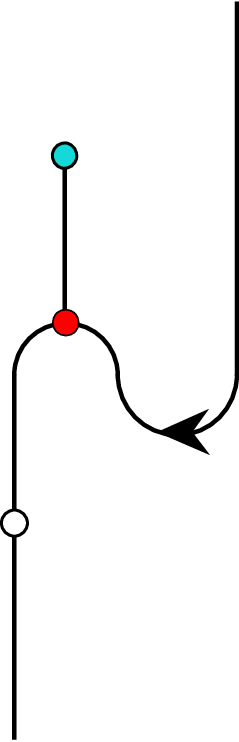}
   \put(-3,-8.5) {\sse$H$}
   \put(22,84)   {\sse$ \Hss $}
   }
   \put(112,32)  {$=$}
   \put(135,0)   {\Includepichtft{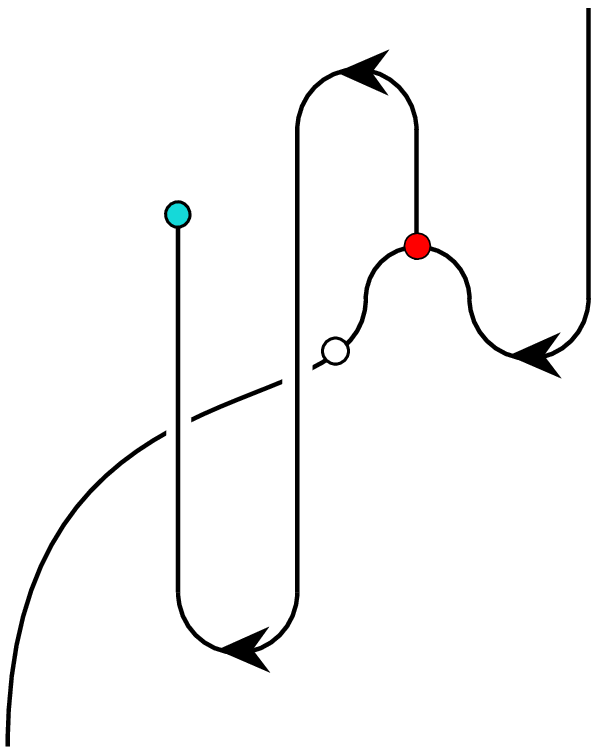}
   \put(-4,-8.5) {\sse$H$}
   \put(61,84)   {\sse$ \Hss $}
   }
   \put(218,32)   {$=$}
   \put(241,0) { \Includepichtft{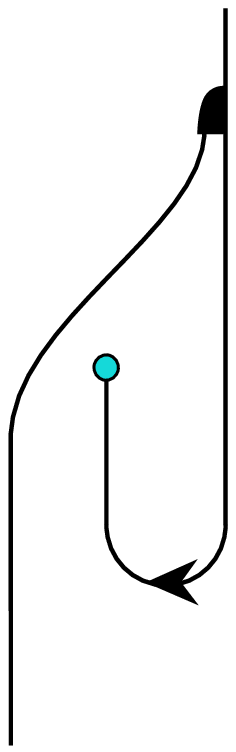}
   \put(-4,-8.5) {\sse$H$}
   \put(20,84)   {\sse$ \Hss $}
   \put(25,69)   {\sse$ \brho $}
   } }
and
   \eqpic{etab_right} {340} {32} {
   \put(0,0)   {\Includepichtft{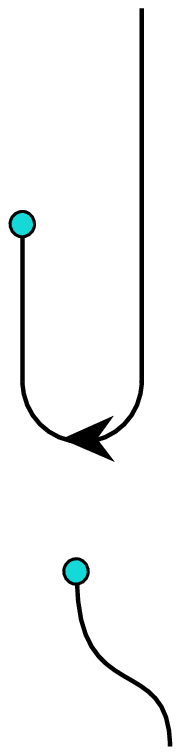}
   \put(11.7,84) {\sse$ \Hss $}
   \put(14,-8.5) {\sse$H$}
   }
   \put(37,32)   {$=$}
   \put(60,0) { \Includepichtft{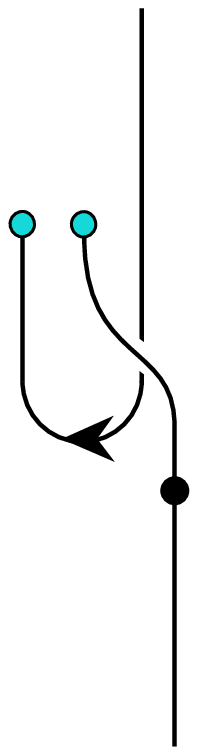}
   \put(11.7,84) {\sse$ \Hss $}
   \put(14,-8.5) {\sse$H$}
   }
   \put(107,32)  {$=$}
   \put(135,0)   {\Includepichtft{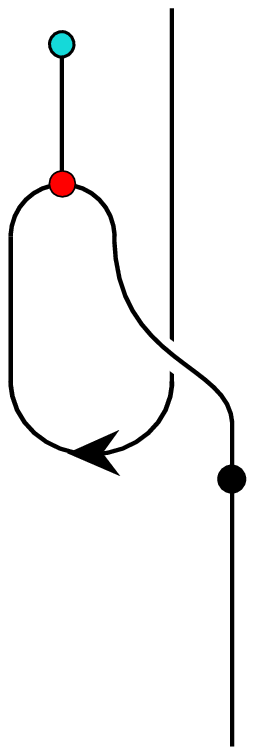}
   \put(14.7,84) {\sse$ \Hss $}
   \put(20.2,-8.5) {\sse$H$}
   }
   \put(181,32)  {$=$}
   \put(205,0) { \Includepichtft{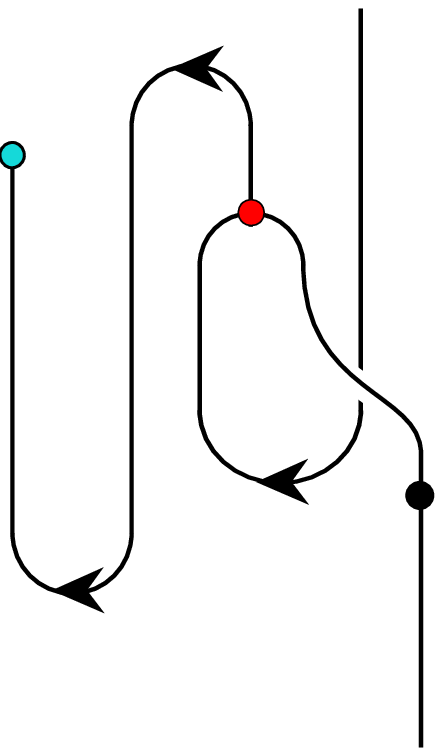}
   \put(36.2,84) {\sse$ \Hss $}
   \put(41.4,-8.5) {\sse$H$}
   }
   \put(279,32)  {$=$}
   \put(305,0) { \Includepichtft{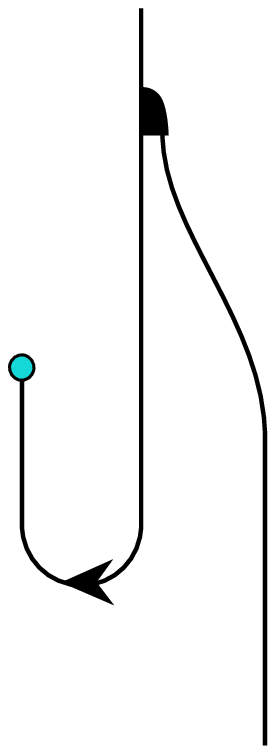}
   \put(11,84)   {\sse$ \Hss $}
   \put(18.7,69) {\sse$ \bohr $}
   \put(24.5,-8.5) {\sse$H$}
   } }
respectively. This uses in particular the homomorphism property of
the counit $\eps$ of $H$ and the fact that $\eps\cir\apo\eq\eps$.
\nxl1
(iii)~\,Next we apply the Hopf-Frobenius trick \erf{Hopf_Frob_trick},
which allows us to write
   \Eqpic{Deltab_left} {450} {45} {
   \put(0,0)   {\Includepichtft{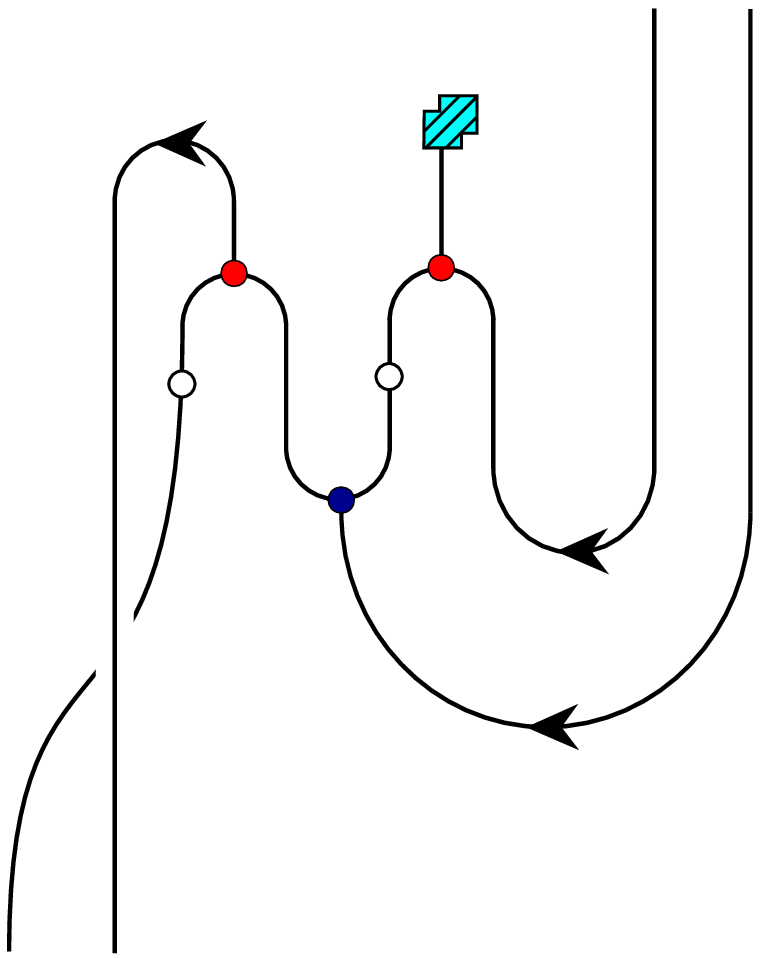}
   \put(-5,-8.5) {\sse$H$}
   \put(7,-8.5)  {\sse$\Hss$}
   \put(67,107)  {\sse$\Hss$}
   \put(78.5,107){\sse$\Hss$}
   }
   \put(105,47)  {$\equ{Hopf_Frob_trick}$}
   \put(130,0) { \Includepichtft{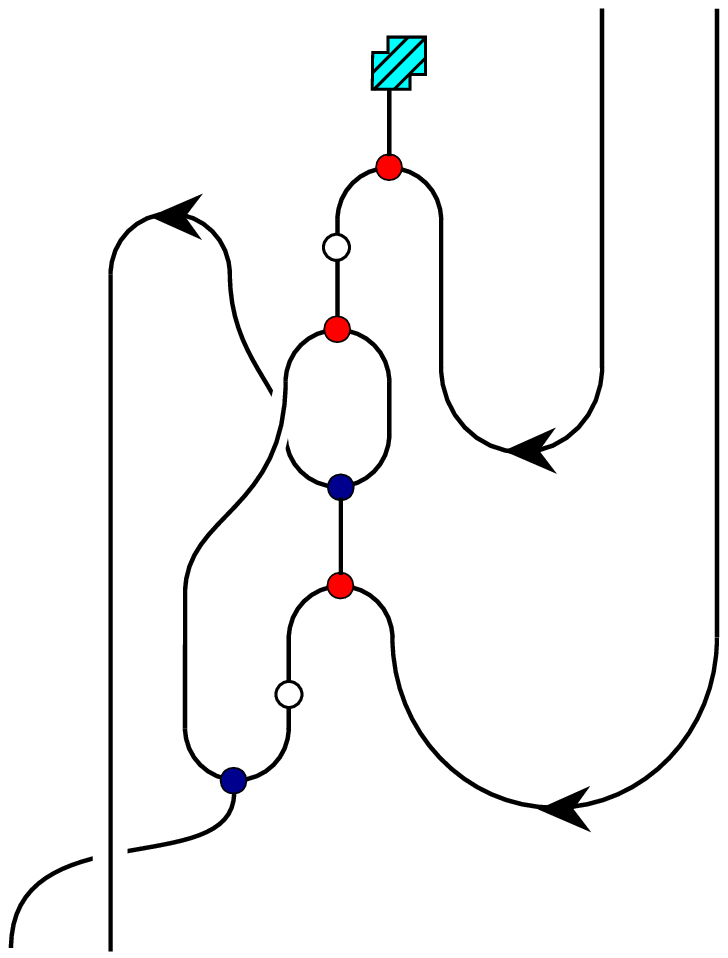}
   \put(-5,-8.5) {\sse$H$}
   \put(7,-8.5)  {\sse$\Hss$}
   \put(61,107)  {\sse$\Hss$}
   \put(74,107)  {\sse$\Hss$}
   }
   \put(234,47)  {$=$}
   \put(255,0)   {\Includepichtft{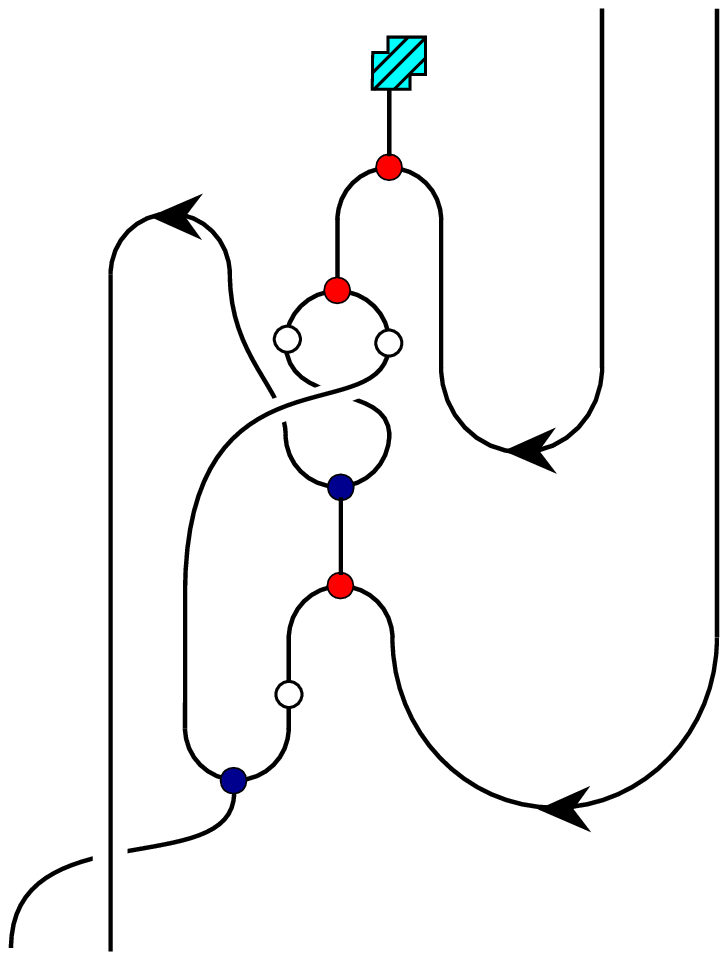}
   \put(-5,-8.5) {\sse$H$}
   \put(7,-8.5)  {\sse$\Hss$}
   \put(61,107)  {\sse$\Hss$}
   \put(74.6,107){\sse$\Hss$}
   }
   \put(357,47)  {$=$}
   \put(375,0) { \Includepichtft{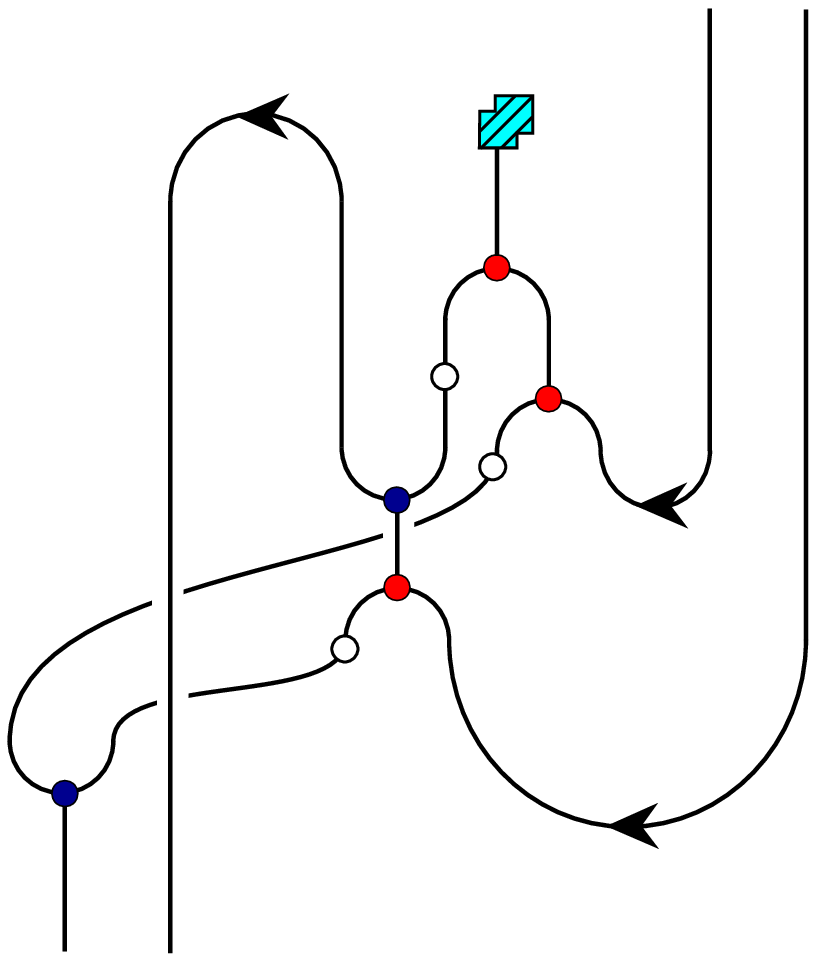}
   \put(2,-8.5)  {\sse$H$}
   \put(13,-8.5) {\sse$\Hss$}
   \put(73,107)  {\sse$\Hss$}
   \put(85,107)  {\sse$\Hss$}
   } }
This tells us that $\Delta\bico$ is a left module morphism.
\nxl1
(iv)~\,For establishing the right module morphism property of $\Delta\bico$ we
recall from Lemma \ref{Delta=delp} that $\Delta\bico \eq \delp$ with $\delp$
given by \erf{Deltaprime}. The following chain of equalities shows that $\delp$
is a morphism of right $H$-modules:
   \eqpic{Deltab_right} {430} {113} {
   \put(0,128) {\Includepichtft{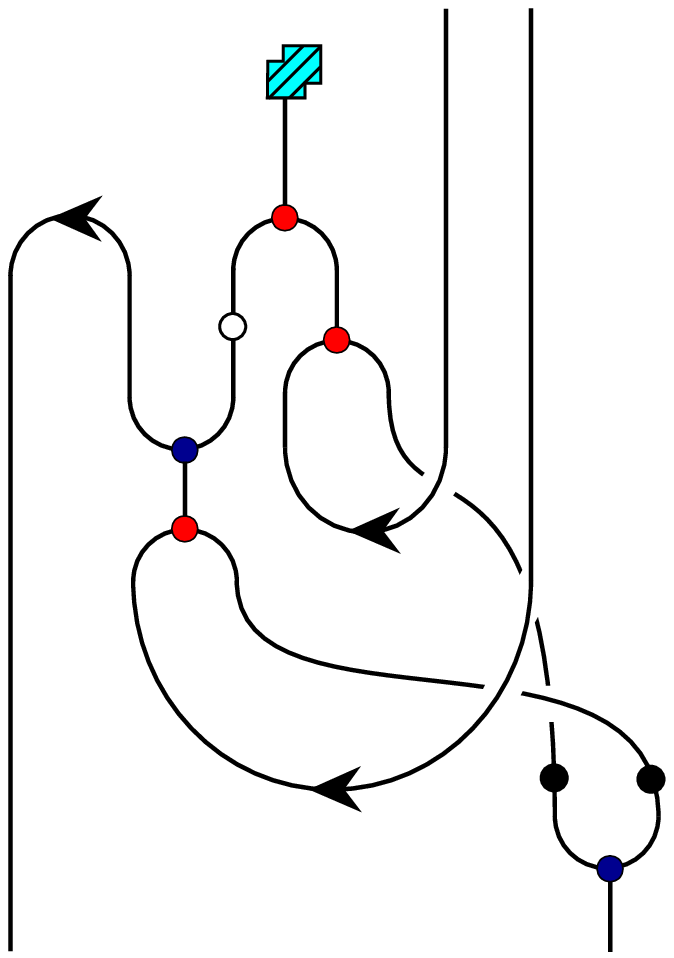}
   \put(-6,-8.5) {\sse$\Hss $}
   \put(43,107)  {\sse$\Hss $}
   \put(55,107)  {\sse$\Hss $}
   \put(62,-8.5) {\sse$H$}
   }
   \put(88,175)  {$=$}
   \put(120,128) { \Includepichtft{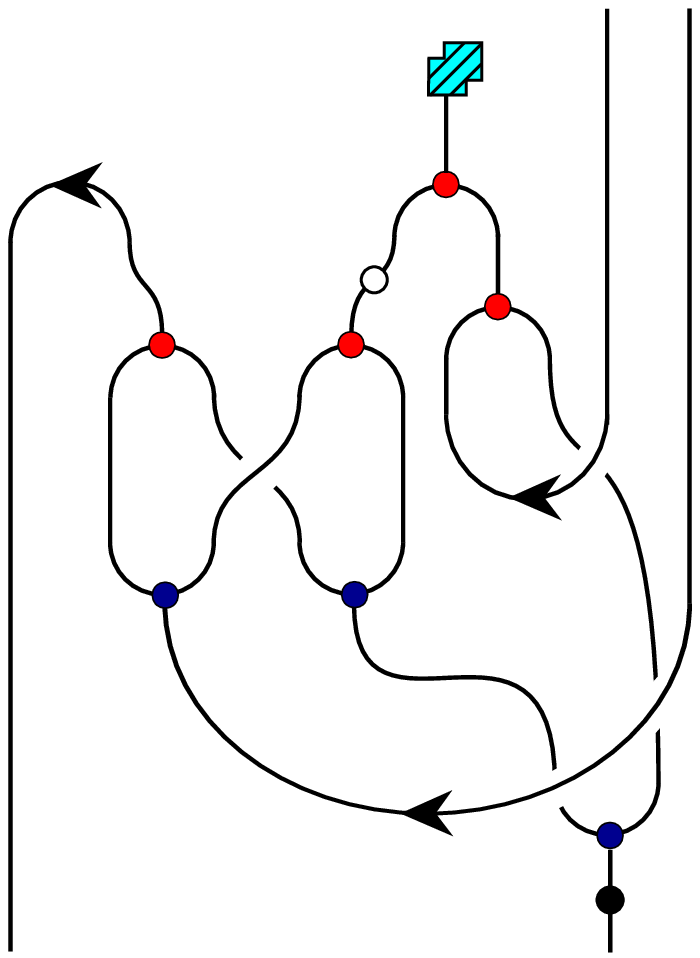}
   \put(-6,-8.5) {\sse$\Hss $}
   \put(62,-8.5) {\sse$H$}
   \put(61,107)  {\sse$\Hss $}
   \put(73,107)  {\sse$\Hss $}
   }
   \put(227,175) {$=$}
   \put(265,128) {\Includepichtft{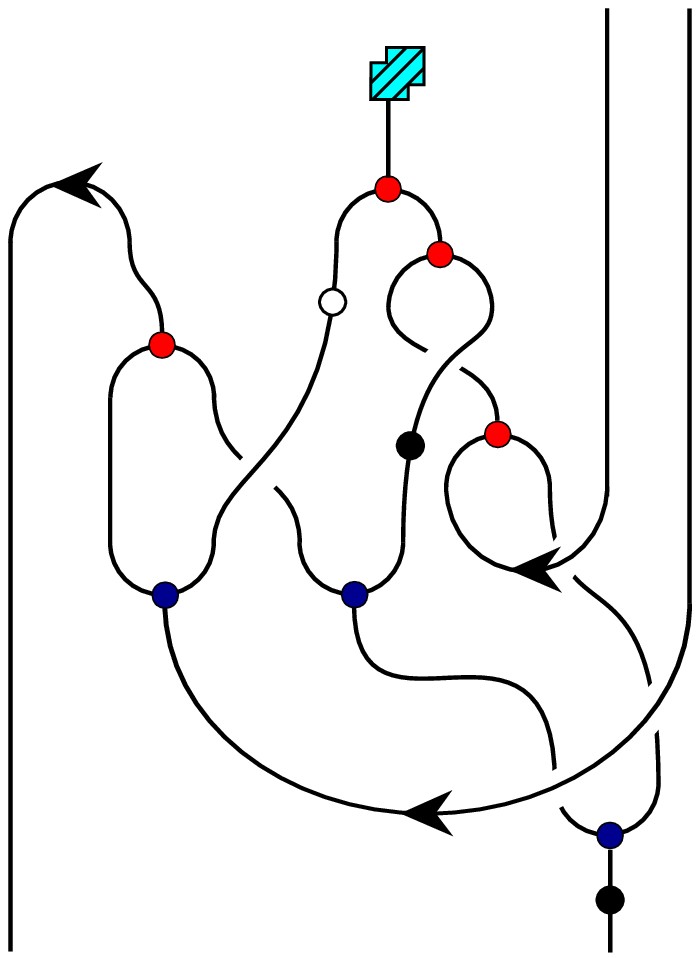}
   \put(-6,-8.5) {\sse$\Hss $}
   \put(62,-8.5) {\sse$H$}
   \put(61,107)  {\sse$\Hss $}
   \put(73,107)  {\sse$\Hss $}
   } \put(0,-2){
   \put(50,47)   {$=$}
   \put(75,0) { \Includepichtft{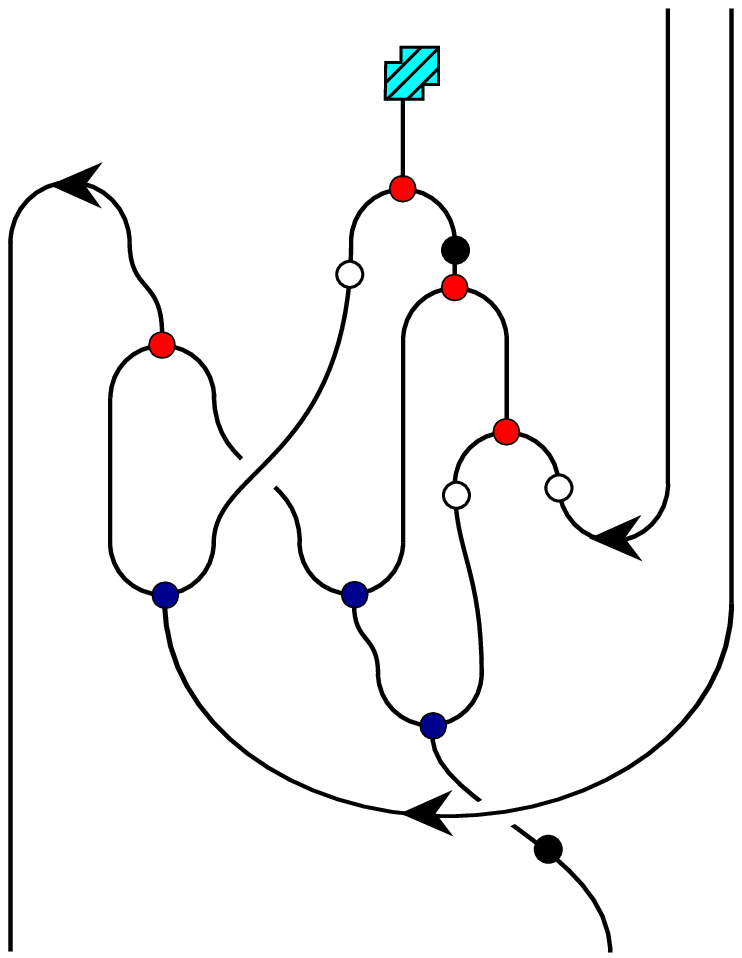}
   \put(-6,-8.5) {\sse$\Hss $}
   \put(62,-8.5) {\sse$H$}
   \put(66,107)  {\sse$\Hss $}
   \put(78,107)  {\sse$\Hss $}
   }
   \put(180,47)  {$=$}
   \put(207,0) { \Includepichtft{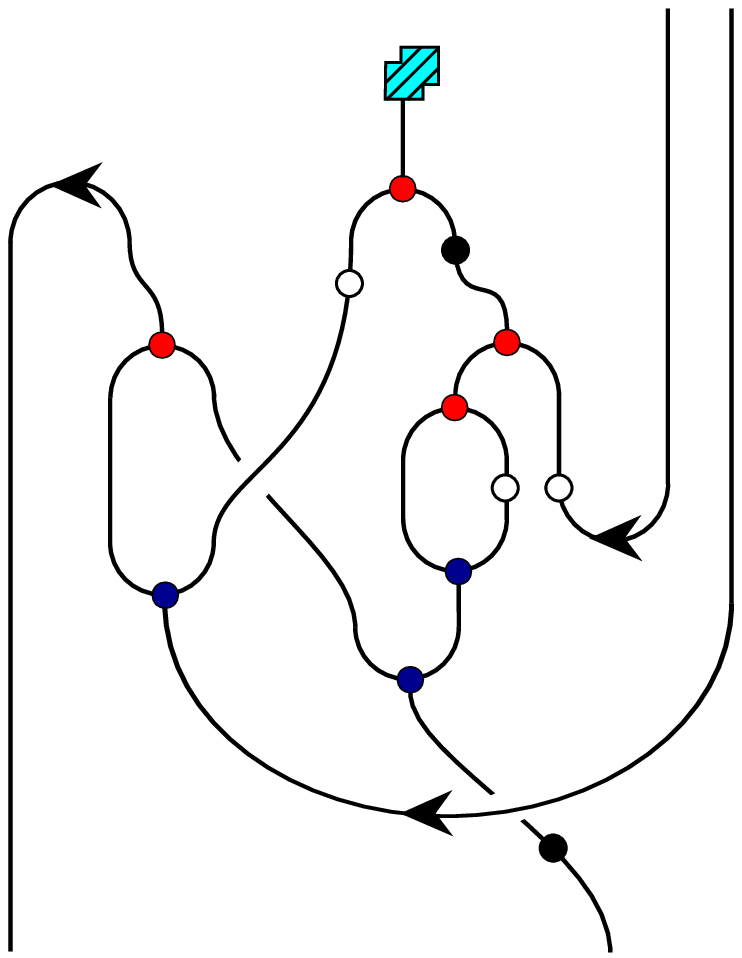}
   \put(-6,-8.5) {\sse$\Hs $}
   \put(62,-8.5) {\sse$H$}
   \put(66,107)  {\sse$\Hs$}
   \put(78,107)  {\sse$\Hs$}
   }
   \put(314,47)  {$=$}
   \put(340,0) { \Includepichtft{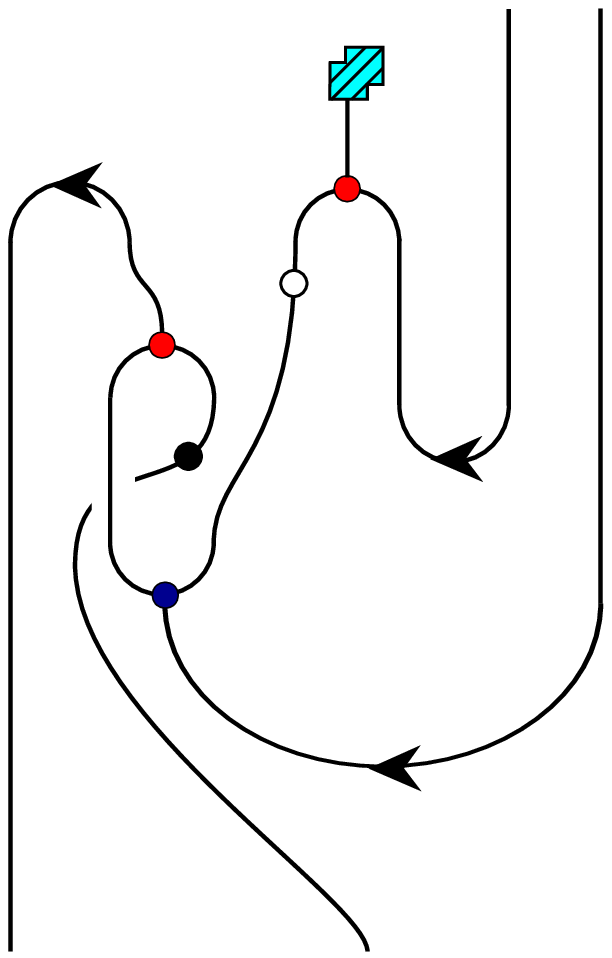}
   \put(-6,-8.5) {\sse$\Hs $}
   \put(35,-8.5) {\sse$H$}
   \put(50.5,107){\sse$\Hs$}
   \put(62.5,107){\sse$\Hs$}
   } } }
Here the first equality combines the anti-coalgebra morphism property of the
antipode and the connecting axiom, while the second equality uses that
$\lambda\cir m \eq \lambda\cir\tauHH\cir(\id_H\oti\apo^2)$ and hence
   \be
   \lambda\circ m\circ ((\apo\cir m)\oti\id_H)
   = \lambda\circ m\circ [\apo\oti(m\cir\tauHH\cir(\apoi\oti\id_H))] \,,
   \ee
which can be shown (see \cite[p.\,4306]{coWe6}) by using that $H$ is unimodular.
\nxl1
(v)~\,Finally, the proof of the bimodule morphism property for $\eps\bico$ is
similar to the one for $\eta\bico$. The sequence of equalities
   \eqpic{epsb_left} {230} {25} {
   \put(0,0)   {\Includepichtft{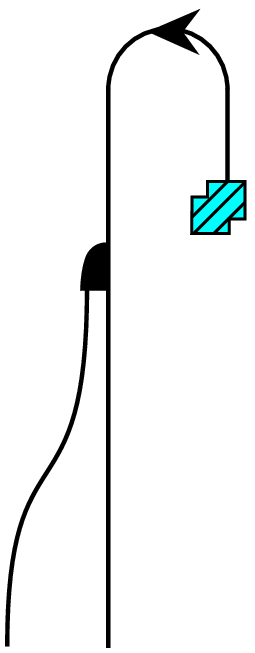}
   \put(-4,-8)   {\sse$ H $}
   \put(6,-8)    {\sse$ \Hs $}
   }
   \put(40,32)   {$=$}
   \put(60,0) { \Includepichtft{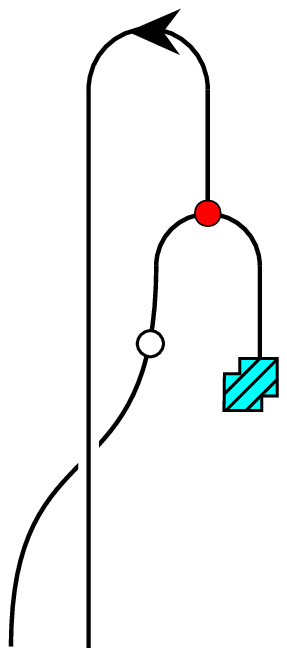}
   \put(-4,-8)   {\sse$ H $}
   \put(6,-8)    {\sse$ \Hs $}
   }
   \put(113,32)  {$=$}
   \put(137,0)   {\Includepichtft{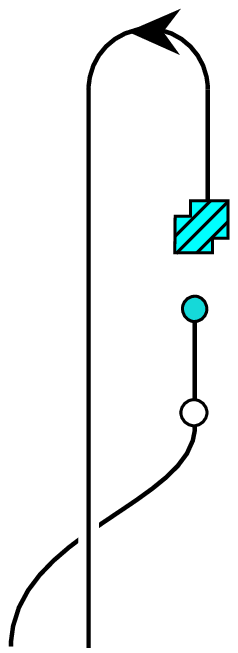}
   \put(-4,-8)   {\sse$ H $}
   \put(6,-8)    {\sse$ \Hs $}
   }
   \put(181,32)  {$=$}
   \put(205,0) { \Includepichtft{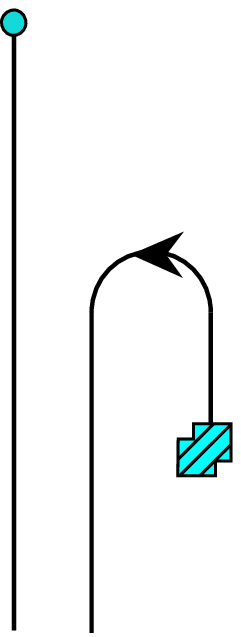}
   \put(-4,-8)   {\sse$ H $}
   \put(6,-8)    {\sse$ \Hs $}
   } }
shows that $\eps\bico$ is a morphism of left modules. Here the second equality 
holds because $\Lambda$ is a left integral. Using that $\Lambda$ is also a right 
integral, one shows analogously that $\eps\bico$ is a morphism of right modules. 
\end{proof}


\subsection{The Frobenius algebra structure of \Hb}

\begin{prop}
The morphisms \erf{def-Hb-Frobalgebra} endow the object $\Hb \eq (\Hs,\brho,\bohr)$
with the structure of a Frobenius algebra in \HBimod\ (with tensor product
\erf{def-tp}). That is, $(\Hb,m\bico,\eta\bico)$ is a (unital associative)
algebra, $(\Hb,\Delta\bico,\eps\bico)$ is a (counital coassociative) coalgebra,
and the two structures are connected by
   \be
   (\idHs\oti m\bico) \circ (\Delta\bico\oti\idHs) = \Delta\bico \circ m\bico
   = (m\bico\oti\idHs) \circ (\idHs\oti\Delta\bico) \,.
   \labl{FA}
\end{prop}

\begin{proof}
(i)~\,That $(\Hb,m\bico,\eta\bico) \eq (\Hs,\Delta^{\!\wee}_{},\eps^\wee)$ is a
unital associative algebra just follows from (and implies) the fact that
$(H,\Delta,\eps)$ is a counital coassociative coalgebra.
\nxl1
(ii)~\,It follows directly from the coassociativity of $\Delta$ that
   \be
   (\idHs\oti\delp)\circ\Delta\bico = (\Delta\bico\oti\idHs)\circ\delp \,.
   \ee
Since, as seen above, $\delp\eq\Delta\bico$, this shows that $\Delta\bico$ is
a coassociative coproduct.
\nxl1
(iii)~\,The coassociativity of $\Delta$ also implies directly the first of
the Frobenius properties \erf{FA}, as well as
   \be
   (m\bico\oti\idHs) \cir (\idHs\oti\delp) = \delp \cir m\bico \,.
   \labl{235}
In view of $\delp\eq\Delta\bico$, \erf{235} is the second of the equalities \erf{FA}.
\nxl1
(iv)~\,That $\eps\bico \eq \Lambda^*_{}$ is a counit for the coproduct
$\Delta\bico$ follows with the help of the invertibility \erf{fmap_inv} of the
Frobenius map: we have
   \Eqpic{Hb_rightunit} {420} {38} { \put(0,2){
   \put(-16,0)  {\Includepichtft{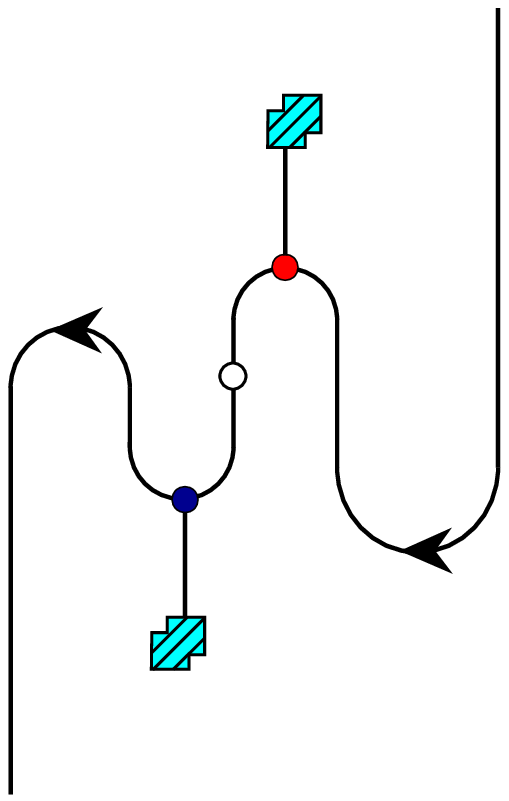}
   \put(-4,-8.5) {\sse$ \Hs $}
   \put(51,91)   {\sse$ \Hs $}
   }
   \put(59,35)   {$=$}
   \put(89,0)  {\Includepichtft{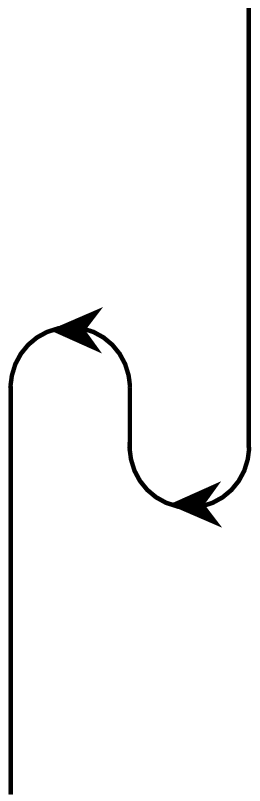}
   \put(-4,-8.5) {\sse$ \Hs $}
   \put(22.6,91) {\sse$ \Hs $}
   }
   \put(136,35)  {$=$}
   \put(166,0) {\Includepichtft{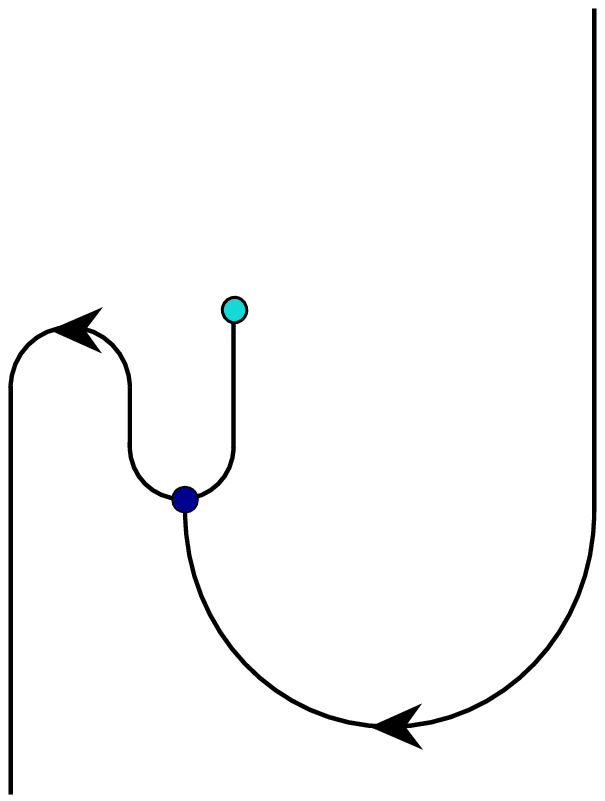}
   \put(-4,-8.5) {\sse$ \Hs $}
   \put(60.6,91) {\sse$ \Hs $}
   }
   \put(251,35)  {$=$}
   \put(281,0) {\Includepichtft{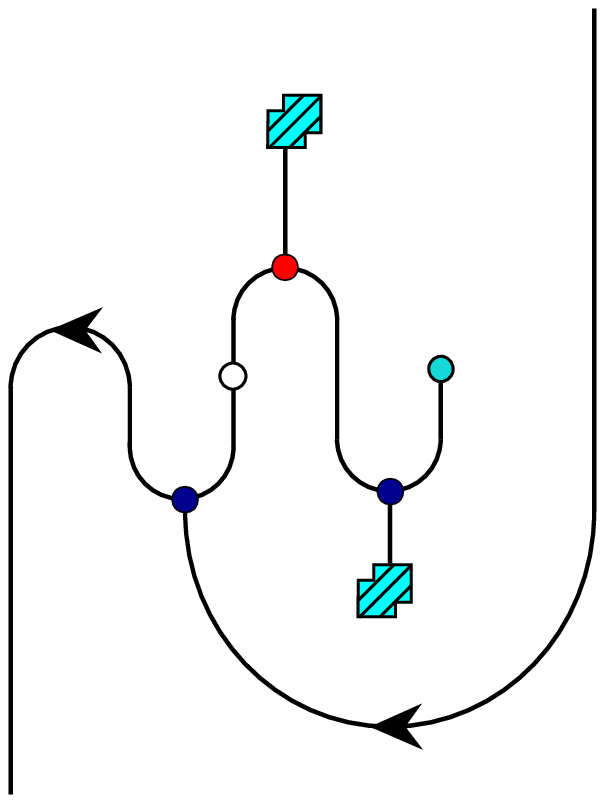}
   \put(-4,-8.5) {\sse$ \Hs $}
   \put(60.6,91) {\sse$ \Hs $}
   }
   \put(366,35)  {$=$}
   \put(396,0) {\Includepichtft{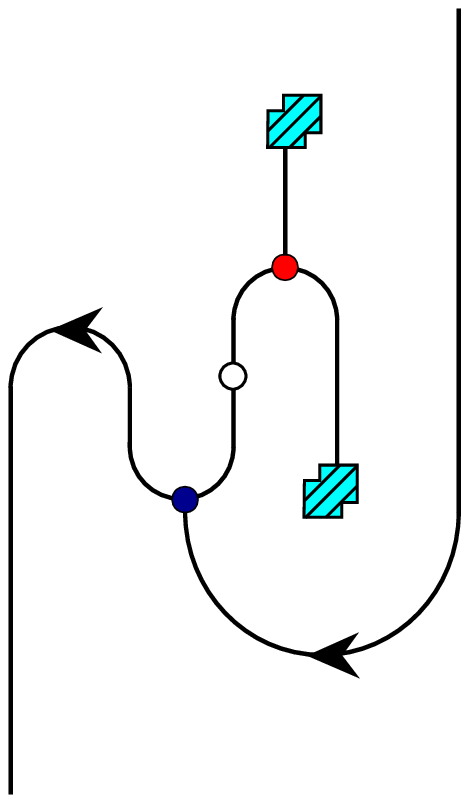}
   \put(-4,-8.5) {\sse$ \Hs $}
   \put(45.6,91) {\sse$ \Hs $}
   } } }
Here the left hand side is $(\idHs\oti\eps\bico)\cir\Delta\bico$, while the
right hand side is $(\eps\bico\oti\idHs)\cir\Delta\bico$.
\end{proof}

\begin{rem}\label{rem:vac}
The spaces of left- and right-module morphisms, respectively, from \Hb\ to $\one$
are given by $\ko\,\Lambda_{\rm r}$ and by $\ko\,\Lambda_{\rm l}$,  respectively,
with $\Lambda_{\rm r}$ and $\Lambda_{\rm l}$ non-zero left and right integrals
of $H$. Thus a non-zero bimodule morphism from \Hb\ to $\one$ exists iff $H$ is
unimodular, and in this case it is unique up to a non-zero scalar. In particular,
up to a non-zero scalar the Frobenius counit $\eps\bico$ is already
completely determined by the requirement that it is a morphism of bimodules.
In the situation at hand, the algebra \Hb\ being a Frobenius algebra is thus a
\emph{property} rather than the choice of a structure.
\end{rem}

       \newpage


\section{Commutativity}\label{sec-braid}

The conventional tensor product \erf{otimesHMod} of bimodules generically does
not admit a braiding. In contrast, the monoidal category \HBimod, with tensor
product as defined in \erf{def-tp}, over a quasitriangular Hopf algebra admits
braidings, and in fact can generically be endowed with several inequivalent ones.
Among these inequivalent braidings, one is distinguished from the
point of view of full local conformal field theory. We will select this
particular braiding $c$ and then show that with respect to this
braiding $c$ the algebra $(\Hb,m\bico,\eta\bico)$ is commutative.

The R-matrix $R\iN H\oti H$ is equivalent to a braiding $c^\HMod_{}$
on the category \HMod\ of left $H$-mo\-dules, consisting of a natural family of
isomorphisms in $\HomH(X\oti Y,Y\oti X)$ for each pair $(X,\rho_X),(Y,\rho_Y)$
of $H$-mo\-dules. These braiding isomorphisms are given by
   \be
   c^{H\text{-Mod}}_{X,Y} = \tau^{}_{X,Y} \circ (\rho_X \oti \rho_Y)
   \circ (\id_H \oti \tau^{}_{H,X} \oti \id_Y) \cir (R\oti\id_X\oti \id_Y) \,,
   \ee
where $\tau$ is the flip map. (When written in terms of elements $x\iN X$ and
$y\iN Y$, this amounts to
$x\oti y \,{\mapsto} \sum_i s_i\,y \oti r_i\,x$ for $R \eq \sum_i r_i\oti s_i$,
but recall that we largely refrain from working with elements.)
The inverse braiding is given by a similar formula, with $R$ replaced
by $R_{21}^{-1} \,{\equiv}\, \tauHH\cir R^{-1}_{}$.
Besides $R$, also the inverse $R_{21}^{-1}$ endows the category \HMod\ with the
structures of a braided tensor category; the two braidings are inequivalent
unless $R_{21}^{-1}$ equals $R$, in which case the category is symmetric.

Likewise one can act with $R$ and with $R_{21}^{-1}$ from the right to obtain
two different braidings on the category of right $H$-modules. As a consequence,
with respect to the chosen tensor product on \HBimod\ there are two inequivalent
natural braidings obtained by either using $R$ both on the left and on the right,
or else using (say) $R_{21}^{-1}$ on the left and $R$ on the right.
For our present purposes (compare Lemma \ref{lem-equiv}(iii)\hspace{.3pt}) the
second of these possibilities turns out to be the relevant braiding $c$.
Pictorially, describing the R-matrix and its inverse by
   \eqpic{def_R}{220}{7} { \put(0,-28){
   \put(0,36)    {$ R ~= $}
   \put(36,18) {\includepichopf{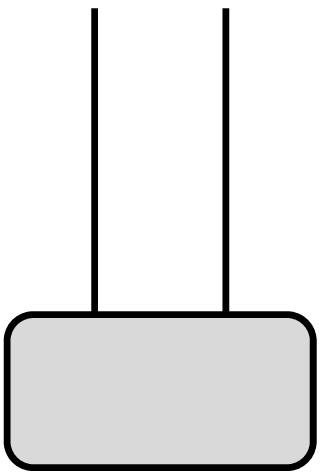}
   \put(4.5,44)  {\sse$ H $}
   \put(16,44)   {\sse$ H $}
   }
   \put(84,36)   {$ \qquand~~ R^{-1} ~= $}
   \put(207,18) {\includepichopf{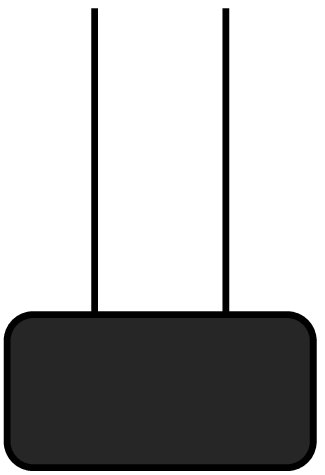}
   \put(4.5,44)  {\sse$ H $}
   \put(16,44)   {\sse$ H $}
   } } }
the braiding on \HBimod\ looks as follows:
   \eqpic{bibraid} {120} {43} {
   \put(0,45)    {$ c_{X,Y}^{} ~= $}
   \put(60,0)  {\Includepichtft{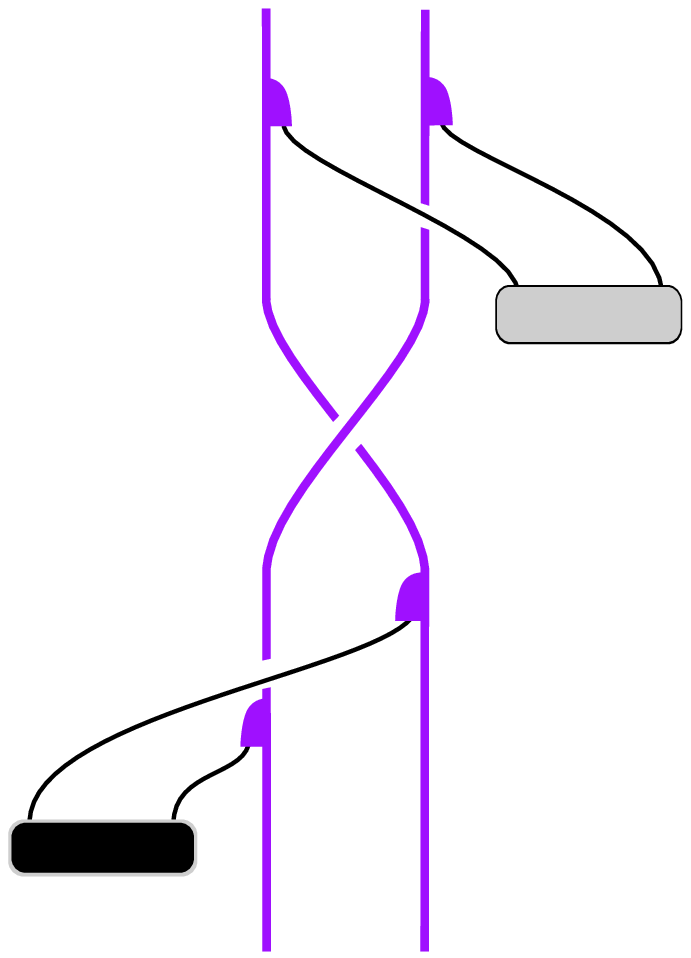}
   \put(24,-8.5) {\sse$ X $}
   \put(24.8,107){\sse$ X $}
   \put(42.3,-8.5) {\sse$ Y $}
   \put(43.5,107){\sse$ Y $}
   \put(-17,9)   {\sse$ R^{-1} $}
   \put(75,68)   {\sse$ R $}
   \put(42,56)   {\sse$ \tau_{X,Y}^{} $}
   } }
We are now in a position to state

\begin{prop}
The product $m\bico$ of the Frobenius algebra \Hb\ in \HBimod\ is commutative
with respect to the braiding \mbox{\erf{bibraid}}:
   \be
   m\bico \circ c^{}_{\Hb,\Hb} = m\bico \,.
   \ee
\end{prop}

\begin{proof}
We have
   \Eqpic{Hb_comm} {420} {106} { \put(-17,-6) {
   \put(0,181)   {$m\bico\circ \cHbHb ~= $}
   \put(88,134) { \Includepichtft{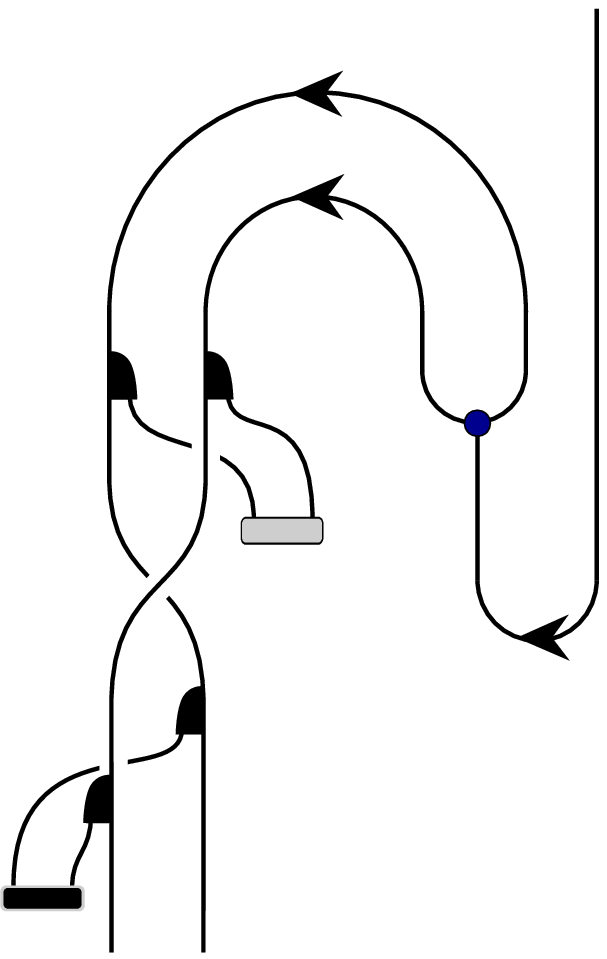}
   \put(5,-8.5)  {\sse$\Hs $}
   \put(19,-8.5) {\sse$\Hs $}
   \put(62,107)  {\sse$\Hs $}
   }
   \put(184,181)  {$=$}
   \put(221,134) {\Includepichtft{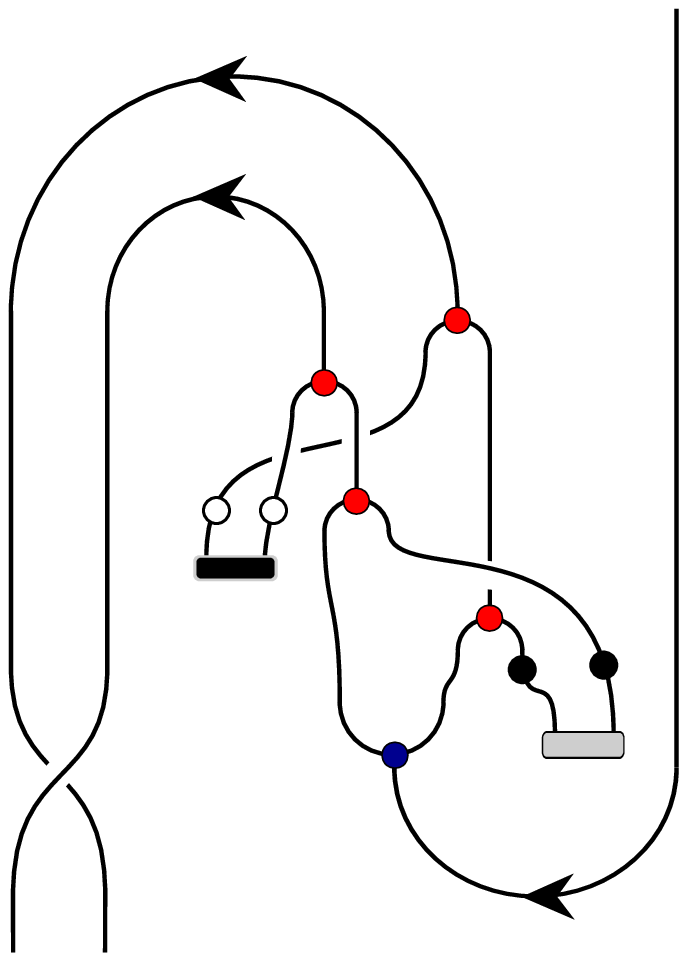}
   \put(-6,-8.5) {\sse$\Hss $}
   \put(8,-8.5)  {\sse$\Hss $}
   \put(70,107)  {\sse$\Hss $}
   }
   \put(59,0){
   \put(0,47)    {$=$}
   \put(26,0) { \Includepichtft{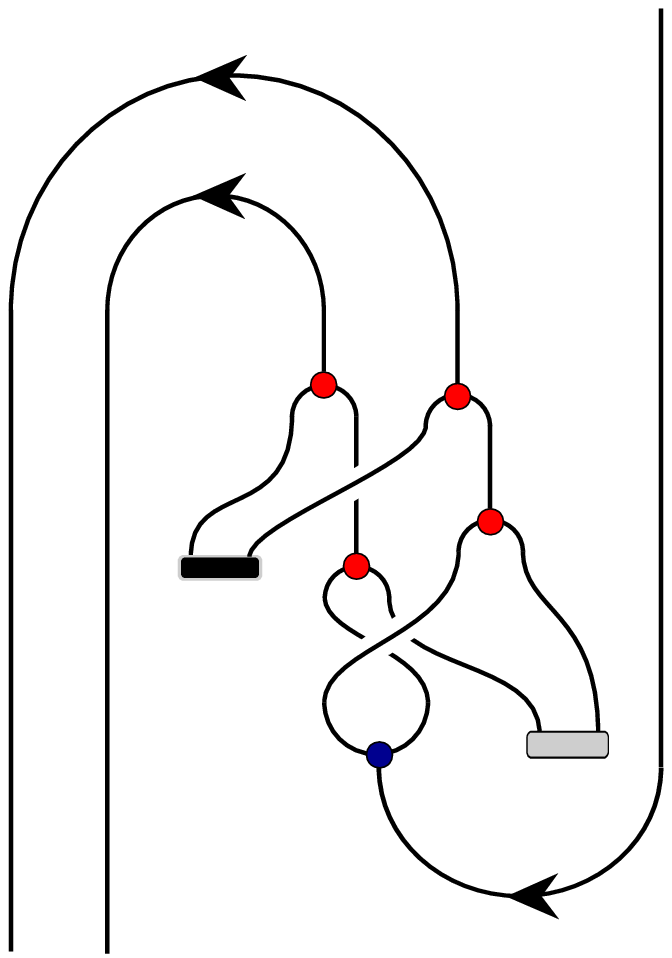}
   \put(-6,-8.5) {\sse$\Hss $}
   \put(8,-8.5)  {\sse$\Hss $}
   \put(68.7,107){\sse$\Hss $}
   }
   \put(125,47)  {$=$}
   \put(151,0) { \Includepichtft{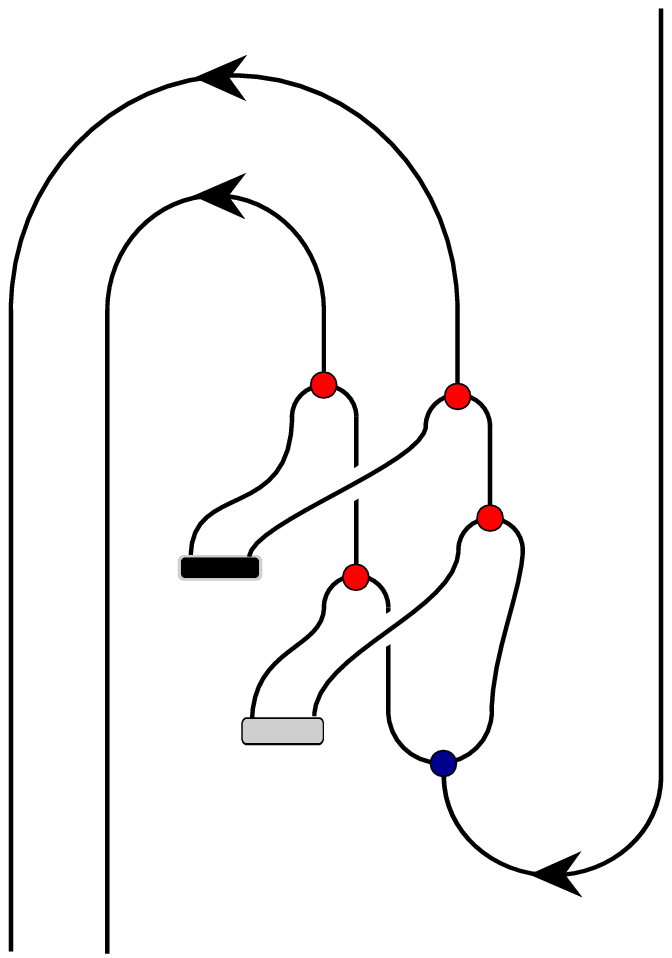}
   \put(-6,-8.5) {\sse$\Hss $}
   \put(8,-8.5)  {\sse$\Hss $}
   \put(68.7,107){\sse$\Hss $}
   }
   \put(250,47)  {$=$}
   \put(279,0) { \Includepichtft{79a}
   \put(-6,-8.5) {\sse$\Hss $}
   \put(8,-8.5)  {\sse$\Hss $}
   \put(50.7,107){\sse$\Hss $}
   }
   \put(359,47)  {$ = ~m\bico \,.$}
   } } }
Here in the second equality the definition \erf{rhorho} of the $H$-actions on
\Hb\ is inserted. The third equality holds because the R-matrix satisfies
   \be
   (\apo\oti\id_H)\circ R = R^{-1} = (\id_H\oti\apoi)\circ R \,,
   \ee
which implies $(\apo\oti\apo)\cir R^{-1} \eq R^{-1}$
as well as $(\apoi\oti\apoi)\cir R \eq R$. The fourth equality follows by the
defining property of $R$ to intertwine the coproduct and opposite coproduct of $H$.
\end{proof}


\section{Symmetry, specialness and twist}\label{sec-dual}

By combining the dualities of \Vectk\ with the antipode or its inverse, one
obtains left and right dualities on the category \HMod\ of left modules over
a \findim\ Hopf algebra $H$, and likewise for right $H$-modules.
In the same way we can define left and right dualities on \HBimod.
Since the monoidal unit of \HBimod\ (with our choice of tensor product) is
the ground field \ko, we can actually take for the evaluation and
coevaluation morphisms (and thus for the action of the functors on morphisms)
just the evaluation and coevaluation maps \erf{graphics2} of \Vectk, and
choose to define the action on objects $X \eq (X,\rho,\ohr)\iN\HBimod$ by
   \be
   X^\vee := (X^*,\rhov,\ohrv)  \qquand {}^{\vee}\!X:= (X^*,\rhoV,\ohrV\,)
   \labl{def-duals}
with
   \Eqpic{Hbim_dualactions} {420} {42} { \put(-22,8){
   \put(0,35)    {$\rhov ~:= $}
   \put(34,0)  {\Includepichtft{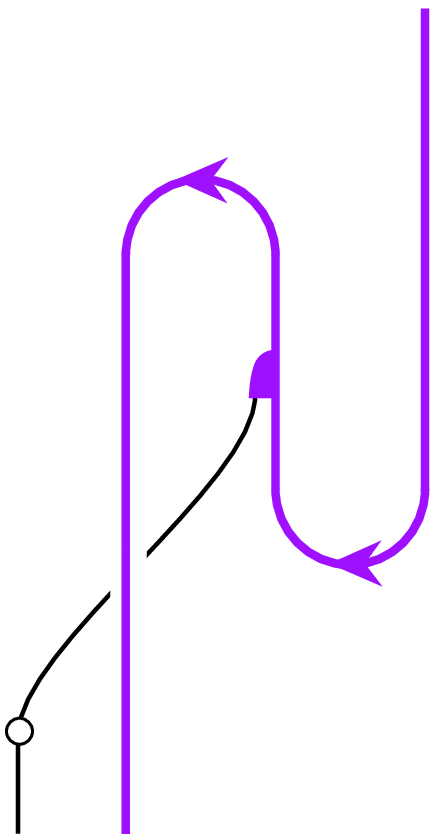}
   \put(-2.5,-9.2) {\sse$ H $}
   \put(10,-8.5) {\sse$ X^*_{} $}
   \put(31.3,49.8) {\sse$ \rho $}
   \put(42.2,94) {\sse$ X^*_{\phantom|} $}
   }
   \put(126,35)  {$\ohrv ~:= $}
   \put(175,0) {\Includepichtft{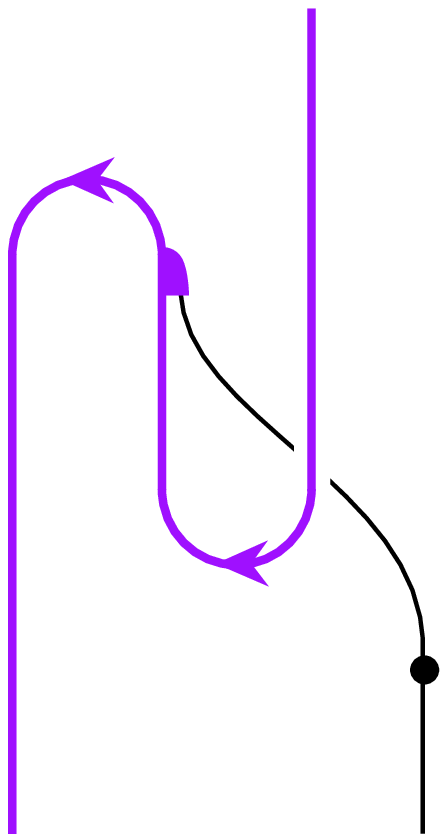}
   \put(-5,-8.5) {\sse$ X^*_{} $}
   \put(11,60.7) {\sse$ \ohr $}
   \put(29.8,94) {\sse$ X^*_{\phantom|} $}
   \put(41,-9.2) {\sse$ H $}
   }
   \put(259,35)  {$\rhoV ~:= $}
   \put(298,0) {\Includepichtft{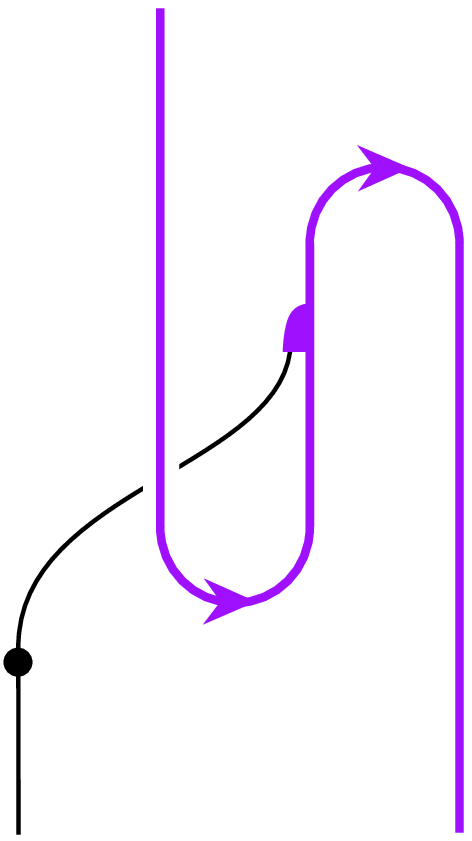}
   \put(-2.4,-9.2) {\sse$ H $}
   \put(13.6,94) {\sse$ X^*_{\phantom|} $}
   \put(35,54.7) {\sse$ \rho $}
   \put(44,-8.5) {\sse$ X^*_{} $}
   }
   \put(390,35)  {$\ohrV ~:= $}
   \put(429,0) {\Includepichtft{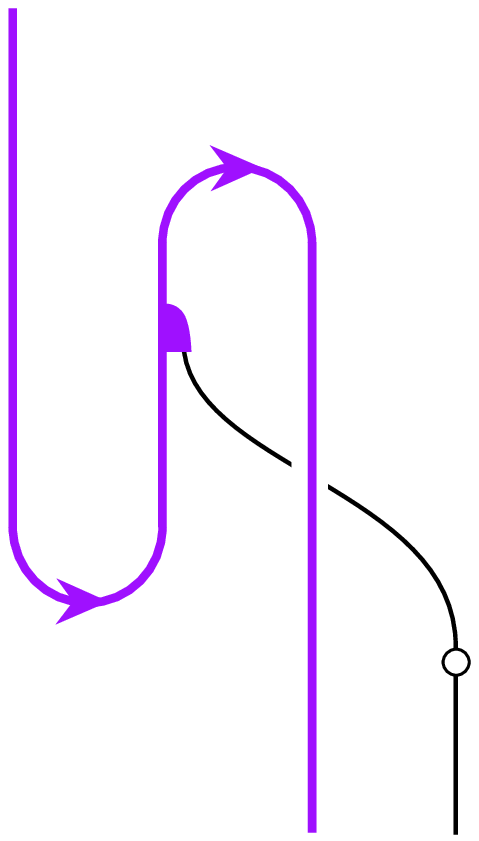}
   \put(11,54.1) {\sse$ \ohr $}
   \put(27,-8.5) {\sse$ X^*_{\phantom|} $}
   \put(-2.9,94) {\sse$ X^*_{\phantom|} $}
   \put(45,-9.2) {\sse$ H $}
   } } }
That the morphisms \erf{Hbim_dualactions} are (left respectively right) $H$-actions
follows from the fact that the antipode is an anti-algebra morphism, and that the
evaluations and coevaluations are bimodule morphisms follows from the defining property
$m\cir(\apo\oti\id_H) \cir \Delta \eq \eta\cir\eps \eq m\cir(\id_H\oti\apo) \cir
\Delta$ of the antipode. Note that with our definition of dualities
  \footnote{~The left and right duals of any object in a category with dualities are
  unique up to distinguished isomorphism. Our choice does not make use of the fact that
  $H$ is a ribbon Hopf algebra. Another realization of the dualities on \HMod\ (and
  analogously on \HBimod), which involves the special group-like element of $H$ and
  hence does use the ribbon structure, is described e.g.\ in \cite[Lemma\,4.2]{vire4}.}
we have
   \be
   {}^{\vee\!}_{}(X^\vee) = X = {({}^{\vee\!}X)}^\vee
   \ee
as equalities (not just isomorphisms) of $H$-bimodules.

\smallskip

The \emph{canonical element} (also called \emph{Drinfeld element}) $u\iN H$
of a quasitriangular Hopf algebra $(H,R)$ with invertible antipode is the element
   \be
   u := m \circ (\apo\oti\id_H) \circ \tauHH \circ R \,.
   \labl{u-R}
$u$ is invertible and satisfies $\apo^2 \eq \ad u$ \cite[Prop\,VIII.4.1]{KAss}.
We denote by $\uvi\iN H$ the inverse of the so-called special group-like element,
i.e.\ the product
   \be
   \uvi := u\,v^{-1} \equiv m \circ (u\oti v^{-1})
   \ee
of the Drinfeld element and the inverse of the ribbon element $v$. Since $v$
is invertible and central, we have $\ad{\uvi} \eq \apo^2$ and, as a consequence,
   \be
   m \circ (\apo\oti \uvi) = m \circ (\uvi\oti\apoi) \qquand
   m \cir (\apoi\oti \uvi^{-1}) \eq m\cir(\uvi^{-1}\oti\apo) \,.
   \labl{msu=musi}
Also, since $\uvi$ is group-like we have $\eps\cir \uvi \eq 1$ and
   \be
   \apo \circ \uvi = \uvi^{-1} = \apoi \circ \uvi \,.
   \labl{apo-t}

A \emph{sovereign structure} on a category with left and right dualities is a
choice of monoidal natural isomorphism $\piv$ between the left and right duality 
functors \cite[Def.\,2.7]{drab3}.\,%
  \footnote{~Equivalently \cite[Prop.\,2.11]{yett2} one may require the existence of
  monoidal natural isomorphisms between the (left or right) double dual functors
  and the identity functor. The latter is called a balanced structure (see e.g.\
  Section 1.7 of \cite{davy4}), or sometimes also
  a pivotal structure (see e.g.\ Section 3 of \cite{schau22}). }
The category \HMod\ of left modules over a ribbon Hopf algebra $H$ is sovereign iff
the square of the antipode of $H$ is inner \cite{bich2,drab3}. Similarly, we have

\begin{lemma}\label{lem-sovereign}
For a ribbon Hopf algebra $H$ with invertible antipode, the family
$\piv_X$ that is defined by 
   \eqpic{pivX} {190} {37} {
   \put(0,39)    {$ \piv_X ~:= $}
   \put(50,0)  {\Includepichtft{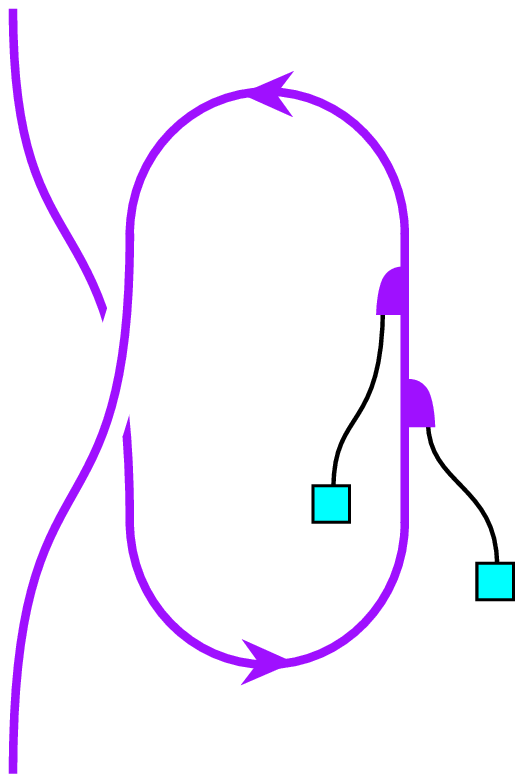}
   \put(-4.5,-8) {\sse$ \Xs $}
   \put(-3,88)   {\sse$ \Xs $}
   \put(28,28)   {\sse$ \uvi $}
   \put(58,20)   {\sse$ \uvi $}
   }
   \put(127,39)  {$\in \Homk(X^\wee_{\phantom|},X^\wee_{\phantom|}) $}
   }
for $X\iN\HBimod$ (with $X^*$ the vector space dual to $X$)
furnishes a sovereign structure on the category \HBimod\ of bimodules over $H$.
\end{lemma}

\begin{proof}
We must show that $\piv_X$ is an invertible bimodule intertwiner from $X^\vee$ to
$^{\vee\!}X$ (i.e.\ the dual bimodules as defined in \erf{def-duals}), that the family
$\{\piv_X\}$ is
natural, and that it is monoidal, i.e.\ $\piv_{X\otimes Y}\eq \piv_X\oti \piv_Y$.
\nxl1
(i)\,~That $\piv_X$ is a morphism in $\HomHH(X^\vee,{}^{\vee\!}X)$ is equivalent to
   \eqpic{pivX_lact_ract} {366} {40} {
   \put(0,0) {\Includepichtft{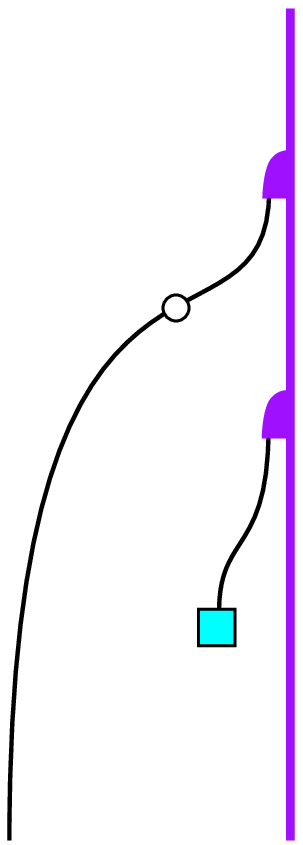}
   \put(-3.5,-8.5) {\sse$ H $}
   \put(16.6,21) {\sse$\uvi$}
   \put(25,-8.5) {\sse$ X $}
   \put(27.5,95) {\sse$ X $}
   }
   \put(57,38)   {$=$}
   \put(84,0) {\Includepichtft{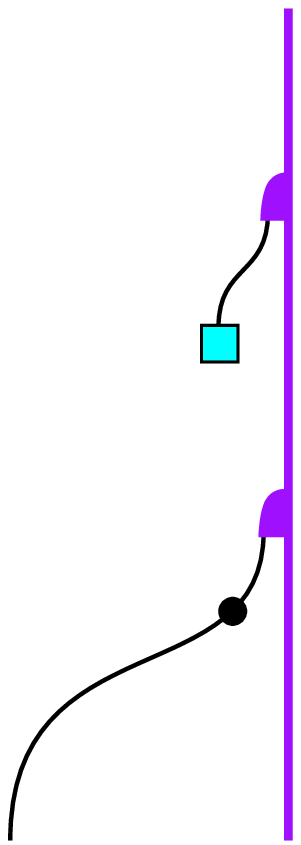}
   \put(-4.5,-8.5) {\sse$ H $}
   \put(16,52.5) {\sse$\uvi$}
   \put(26,-8.5) {\sse$ X $}
   \put(27.5,95) {\sse$ X $}
   }
   \put(163,38)  {and}
   \put(231,0)  {\Includepichtft{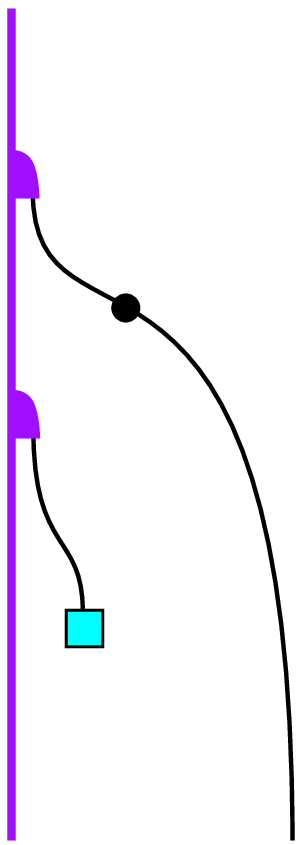}
   \put(-3.5,-8.5) {\sse$ X $}
   \put(-3.1,95) {\sse$ X $}
   \put(11.5,22) {\sse$\uvi$}
   \put(27,-8.5) {\sse$ H $}
   }
   \put(282,38)   {$=$}
   \put(316,0)  {\Includepichtft{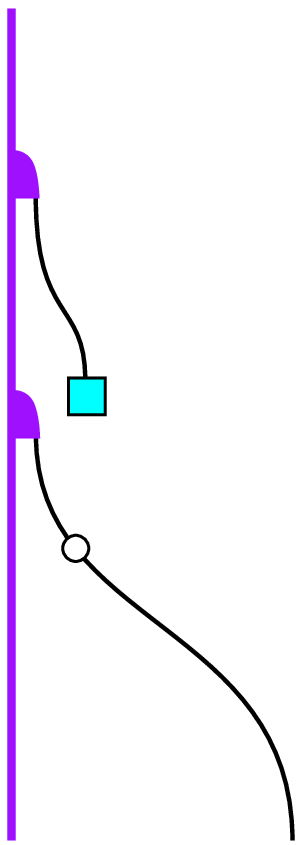}
   \put(-3.5,-8.5) {\sse$ X $}
   \put(-3.1,95) {\sse$ X $}
   \put(12,47)   {\sse$\uvi$}
   \put(27,-8.5) {\sse$ H $}
   } }
This in turn follows directly by combining \erf{msu=musi} and the (left,
respectively right) \rep\ properties.
\nxl1
(ii)\,~With the help of the defining properties of the evaluation
and coevaluation maps it is easily checked that
   \eqpic{pivX_inv} {190} {35} {
   \put(-6,38)   {$ \piv_X^{-1} ~:= $}
   \put(52,0)  {\Includepichtft{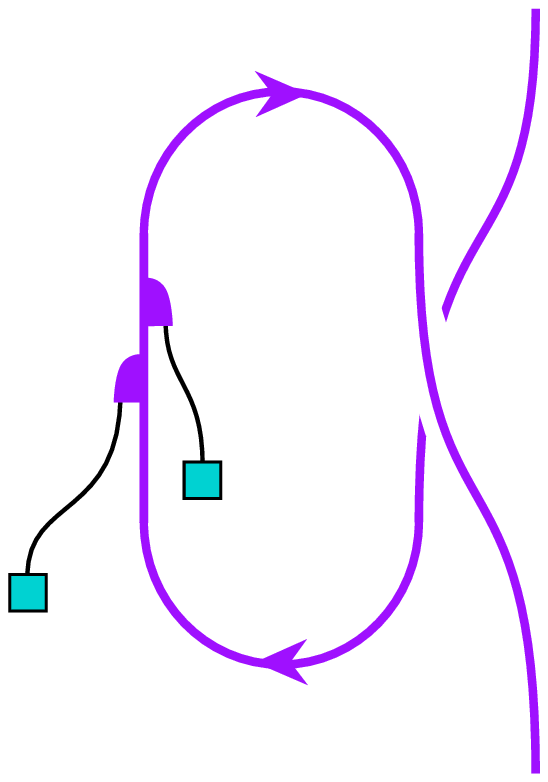}
   \put(53,-8.5) {\sse$ X^*_{\phantom|} $}
   \put(54,88)   {\sse$ X^*_{\phantom|} $}
   \put(24.9,30.7) {\sse$t^{-1}$}
   \put(-14,18.5){\sse$t^{-1}$}
   }
   \put(128,38)  {$\in \Homk(X^\wee_{\phantom|},X^\wee_{\phantom|})$}
   }
is a linear two-sided inverse of $\piv_X$. That $\piv_X^{-1}$ is a bimodule
morphism is then automatic.
\nxl1
(iii)\,~That the family $\{ \piv_X \}$ of isomorphisms furnishes a natural
transformation from the right to the left duality functor means that for
any morphism $f\iN\HomHH(X,Y)$ one has
   \be
   {}^{\vee\!\!}f \circ \piv_Y = \piv_X \circ f^\vee
   \labl{vf.piY=piX.fv}
as morphisms in $\HomHH(Y^\vee\!,{}^{\vee\!\!}X)$. Now by sovereignty of \Vectk\
we know that $^{\vee\!\!}f \eq f^\vee$ as linear maps from $Y^\wee$ to $X^\wee$,
and as a consequence \erf{vf.piY=piX.fv} is equivalent to
   \be
   f \circ \varphi_X = \varphi_Y \circ f
   \labl{f.vpX=vpY.f}
as morphisms in $\HomHH(X,Y)$, where $\varphi_X$ is the left action on $X$ with
$\uvi\iN H$ composed with the right action on $X$ with $\uvi$.
\erf{f.vpX=vpY.f}, in turn, is a direct consequence of the fact that $f$ is
a bimodule morphism.
\nxl1
(iv)\,~That $\piv_X$ is monoidal follows from the fact that $\uvi$ is group-like.
\end{proof}

\begin{defi} ~\nxl1
(a) An \emph{invariant pairing} on an algebra $A \eq (A,m,\eta)$ in a monoidal
category $(\C,\otimes,\one)$ is a morphism $\kappa\iN\HomC(A\oti A,\one)$ satisfying
$\kappa \cir (m \oti \idA) \eq \kappa \cir (\idA \oti m)$.
\nxl1
b) A \emph{symmetric} algebra $(A,\kappa)$ in a sovereign category \C\ is an algebra
$A$ in \C\ together with an invariant pairing $\kappa$ that is symmetric, i.e.\ satisfies
   \eqpic{pic_csp_14} {230} {22} {
   \put(0,0)   {\Includepichtft{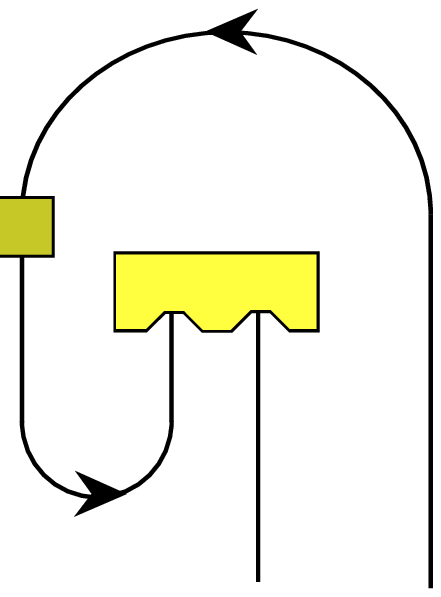}}
   \put(74,23)   {$ = $}
   \put(25.4,-8.5) {\sse$A$}
   \put(30.7,38.5) {\sse$\kappa$}
   \put(43.8,-8.5) {\sse$A$}
   \put(-13.5,40.9){\sse$\piv_{\!A}^{-1}$}
   \put(104,10) {\Includepichtft{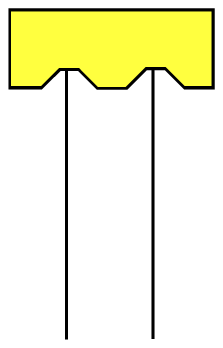}
   \put(2.4,-8.5)  {\sse$A$}
   \put(12.7,-8.5) {\sse$A$}
   \put(24.2,31.3) {\sse$\kappa$}
   }
   \put(148,23)  {$ = $}
   \put(179,0) {\Includepichtft{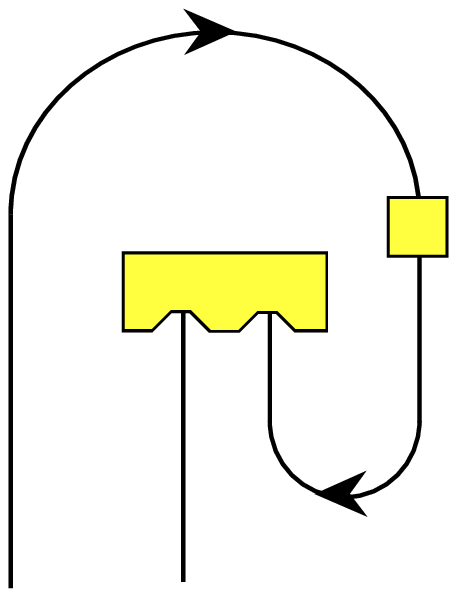}
   \put(-3.3,-8.5) {\sse$A$}
   \put(12.7,38.5) {\sse$\kappa$}
   \put(16,-8.5)   {\sse$A$}
   \put(50,39)     {\sse$\piv_{\!A}^{}$}
   }
   \put(260,44)    {\catpic}
   }
\end{defi}

\begin{rem}
(i)\,~Unlike the pictures used so far (and most of the pictures below), which
describe morphisms in \Vect, \erf{pic_csp_14} refers to morphisms in the category
\C\ rather than in \Vect; to emphasize this we have added the box
\,\raisebox{-.3em}{\catpicsm}\, to the picture. Also note that the morphisms
\erf{pic_csp_14} involve the left and right dualities of \C, but do not assume a
braiding. Thus the natural setting for the notion of symmetry of an algebra is
the one of sovereign categories \C; a braiding on \C\ is not needed.
\nxl1
(ii)\,~An algebra with an invariant pairing $\kappa$
is Frobenius iff $\kappa$ is non-degenerate, see e.g.\ \cite[Sect.\,3]{fuSt}.
\nxl1
(iii)\,~The two equalities in \erf{pic_csp_14} actually imply each other.
\end{rem}

In the case of the category \HBimod\ with sovereign structure $\pi$ as defined
in \erf{pivX}, the equalities \erf{pic_csp_14} read
   \eqpic{sym_Hbimod} {200} {20} { \put(12,0){
   \put(0,0)   {\Includepichtft{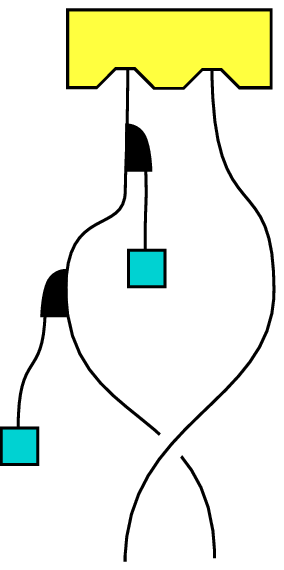}}
   \put(9.4,-8.5){\sse$A$}
   \put(31,55.5) {\sse$\kappa$}
   \put(19.8,-8.5) {\sse$A$}
   \put(-14,11)  {\sse$t^{-1}$}
   \put(12,24)   {\sse$t^{-1}$}
   }
   \put(68,23)   {$ = $}
   \put(100,10)  {\Includepichtft{99b}
   \put(2.1,-8.5){\sse$A$}
   \put(12.8,-8.5) {\sse$A$}
   \put(24.2,31.3) {\sse$\kappa$}
   }
   \put(146,23)  {$ = $}
   \put(179,0) {\Includepichtft{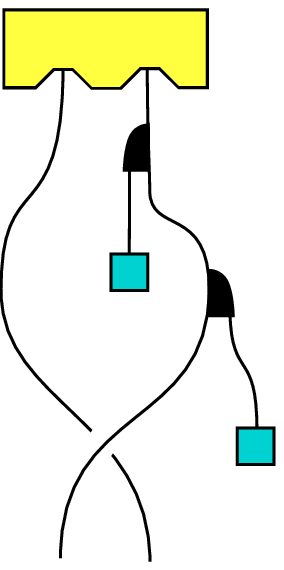}
   \put(2.1,-8.5){\sse$A$}
   \put(25,55.5) {\sse$\kappa$}
   \put(12.8,-8.5) {\sse$A$}
   \put(32,11)   {\sse$t$}
   \put(12,24)   {\sse$t$}
   } }

\begin{thm}
For any unimodular \findim\ ribbon Hopf algebra $H$ the pair $(\Hb,\kappa\bico)$
with \Hb\ the \bicom\ (with Frobenius algebra structure as defined above) and
   \be
   \kappa\bico := \eps\bico \circ m\bico = {(\Delta\cir\Lambda)}^\wee
   \ee
is a symmetric Frobenius algebra in \HBimod.
\end{thm}

\begin{proof}
That the pairing $\kappa\bico$ is invariant follows directly from the coassociativity of
$\Delta$. To establish that $\kappa\bico$ is symmetric, consider the following equalities:
   \eqpic{symm_Hb} {190} {37} { \put(0,-3){
   \put(0,13)  {\Includepichtft{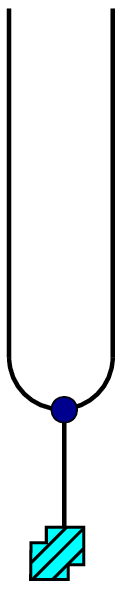}
   \put(-3,67)   {\sse$H$}
   \put(9,67)    {\sse$H$}
   \put(10,0)    {\sse$\Lambda$}
   }
   \put(40,42)   {$=$}
   \put(80,0) {\Includepichtft{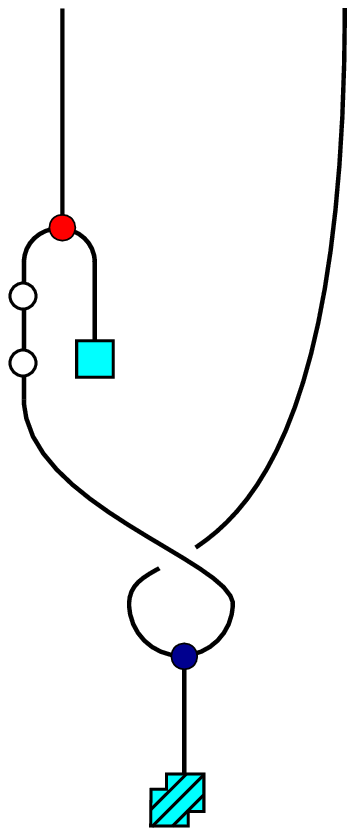}
   \put(2,93)    {\sse$H$}
   \put(33,93)   {\sse$H$}
   \put(13,50)   {\sse$g$}
   \put(24,0)    {\sse$\Lambda$}
   }
   \put(140,42)  {$=$}
   \put(180,0) {\Includepichtft{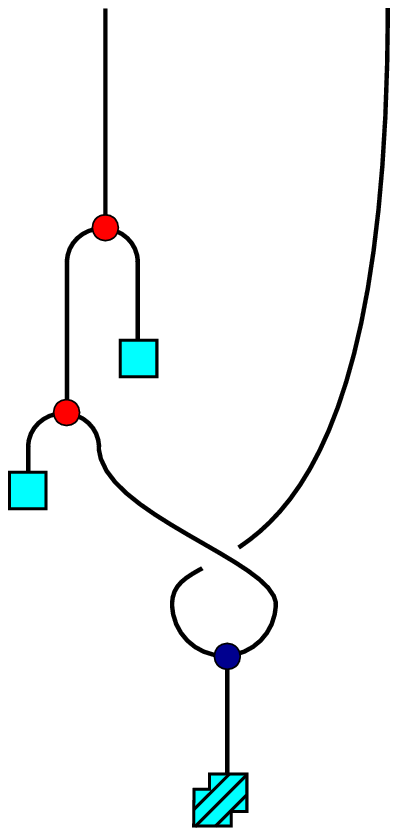}
   \put(7,93)    {\sse$H$}
   \put(38,93)   {\sse$H$}
   \put(-4.6,35) {\sse$\uvi$}
   \put(18.4,49) {\sse$\uvi$}
   \put(28,0)    {\sse$\Lambda$}
   } } }
The first equality is Theorem 3(d) of \cite{radf12}, and involves the right modular
element (also known as distinguished group-like element) $g$ of $H$, which by
definition satisfies $g\cir\lambda \eq (\id_H\oti\lambda)\cir\Delta$.  The second
equality uses that $g \eq \uvi^2$ (which holds by Theorem 2(a) and Corollary 1 of
\cite{radf10}, specialized to unimodular $H$) and $\apo^2\eq\ad t$.
\\
Using also the identity $\brho\cir(\uvi^{-1}\oti\idHs) \eq \big(m\cir(\uvi\oti\id_H)
{\big)}^*$ (which, in turn, uses \erf{apo-t}), it follows that the equality of the
left and right sides of \erf{symm_Hb} is nothing but the dualized version of the
first of the equalities \erf{sym_Hbimod} for the case $A \eq \Hb$ and
$\kappa\eq\kappa\bico$.
\end{proof}

Next we observe:

\begin{lemma}
The morphisms \erf{def-Hb-Frobalgebra} satisfy
   \be
   \eps\bico \cir \eta\bico = \eps\cir\Lambda \qquand
   m\bico \cir \Delta\bico = (\lambda\cir\eps)\, \idHs \,.
   \ee
\end{lemma}

\begin{proof}
We have
   \Eqpic{special} {420} {36} {
    \put(-15,30){
   \put(0,17)    {$\eps\bico\cir\eta\bico ~= $}
   \put(63,0) {\Includepichtft{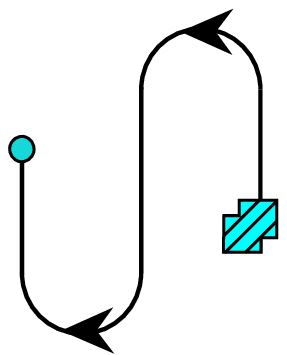}
   }
   \put(103,17)  {$=$}
   \put(123,9)   {\Includepichtft{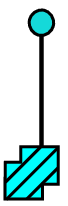}
   } }
    \put(208,12){
   \put(-55,35)  {and}
   \put(0,35)    {$m\bico\cir\Delta\bico ~= $}
   \put(74,0) {\Includepichtft{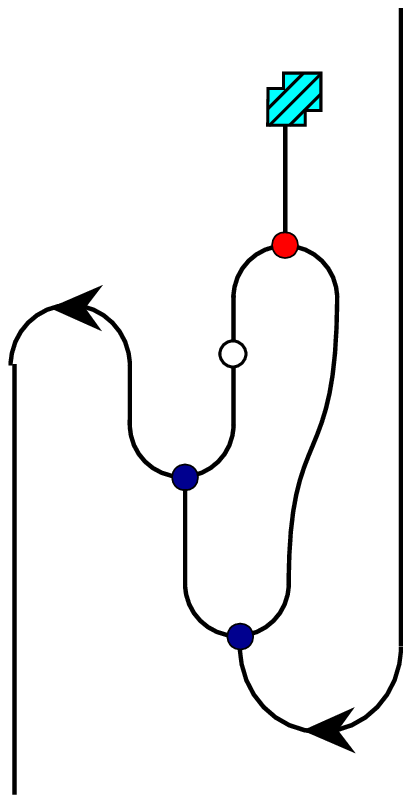}
   \put(-4,-8.5) {\sse$ \Hss $}
   \put(40,89.5) {\sse$ \Hss $}
   }
   \put(132,35)  {$=$}
   \put(155,0) {\Includepichtft{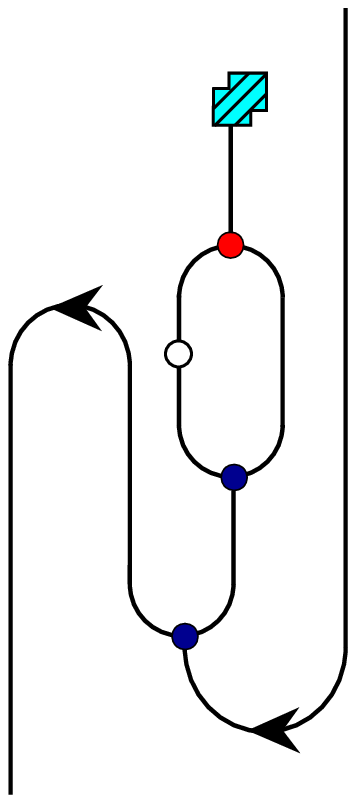}
   \put(-4,-8.5) {\sse$ \Hss $}
   \put(33,89.5) {\sse$ \Hss $}
   }
   \put(209,35)  {$=$}
   \put(231,0) {\Includepichtft{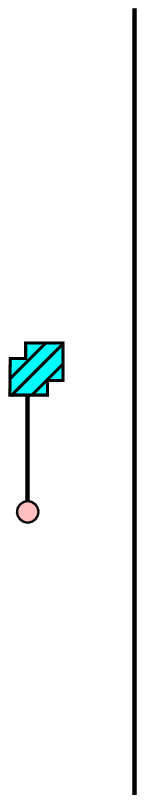}
   \put(8.5,-8.5){\sse$ \Hss $}
   \put(10,89.5) {\sse$ \Hss $}
   } } }
Here the last equality uses the defining property of the antipode $\apo$ of $H$.
\end{proof}

\begin{defi}\label{def-spec}
A Frobenius algebra $(A,m_A,\eta_A,\Delta_A,\eps_A)$ in a \ko-linear monoidal category
is called \emph{special} \cite{fuRs4,evpi} (or strongly separable \cite{muge8})
iff $\eps_A\cir\eta_A \eq \xi\,\id_\one$
and $m_A\cir\Delta_A \eq \zeta\,\id_A$ with $\xi,\zeta\iN\ko^\times$.
\end{defi}

It is known that a \findim\ Hopf algebra $H$ is semisimple iff the Maschke number
$\eps\cir\Lambda \iN \ko$ is non-zero, and it is cosemisimple iff
$\lambda\cir\eps \iN \ko$ is non-zero \cite{laSw}; also, in characteristic zero
cosemisimplicity is implied by semisimplicity \cite[Thm.\,3.3]{laRad2}. Thus we have

\begin{cor}
The Frobenius algebra \Hb\ in \HBimod\ is special iff $H$ is semisimple.
\end{cor}

\medskip

As already pointed out we do \emph{not}, however, assume that $H$ is semisimple.
We now note a consequence of the fact that \Hb\ is commutative and symmetric,
irrespective of whether $H$ is semisimple or not. We first observe:

\begin{lemma}
The braided monoidal category \HBimod\ of bimodules over a \findim\ ribbon Hopf
\ko-algebra $H$ is balanced. The twist endomorphisms are given by
   \eqpic{bimod-twist} {90} {31} {
     \put(0,35)    {$ \theta_X ~= $}
   \put(50,0) {\Includepichtft{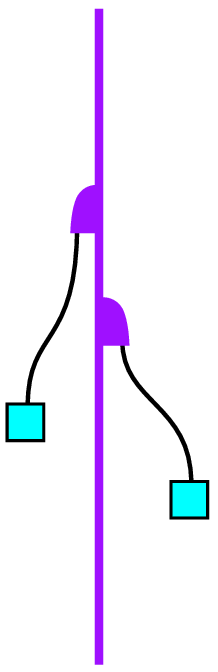}
   \put(6.1,-8.5) {\sse$X$}
   \put(6.9,76)   {\sse$X$}
   \put(-5.4,25.7){\sse$v$}
   \put(24.3,17)  {\sse$v^{-1}$}
   } }
with $v$ the ribbon element of $H$.
\end{lemma}

\begin{proof}
We have seen that \HBimod\ is braided, and according to Lemma
\ref{lem-sovereign} it is sovereign. Now a braided monoidal category with a (left or
right) duality is sovereign iff it is balanced, see e.g.\ Prop.\,2.11 of \cite{yett2}.
\\
For any sovereign braided monoidal category \C\ the twist endomorphisms $\theta_X$ can 
be obtained by combining the braiding, dualities and sovereign structure according to
   \eqpic{twist} {120} {37} {
   \put(0,40)      {$ \theta_X ~= $}
   \put(52,0)   {\Includepichtft{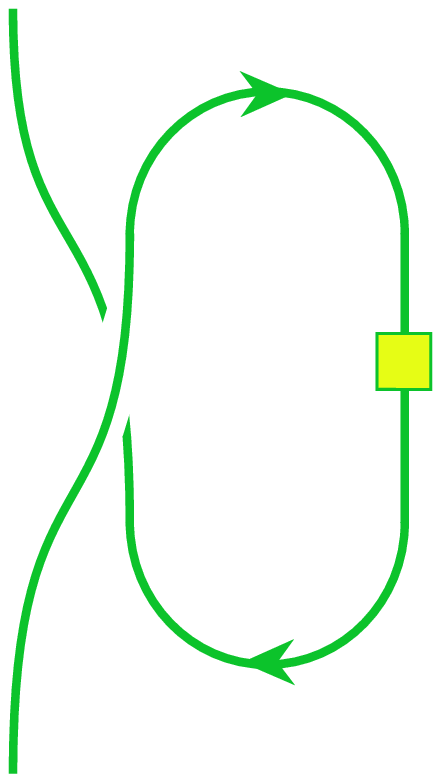}
   \put(-4.1,-8.5) {\sse$X$}
   \put(-2.5,88)   {\sse$X$}
   \put(15.6,34)   {\begin{turn}{90}\sse$c_{X,X}^{}$\end{turn}}
   \put(47.9,44.1) {\sse$\piv_{\!X}$}
   }
   \put(126,69)    {\catpic}
   }
With the explicit form \erf{bibraid} of the braiding and \erf{pivX} of the
sovereign structure, this results in
   \eqpic{twist_Hbimod} {105} {41} {
   \put(0,39)      {$\theta_X ~=$}
   \put(52,0)   {\Includepichtft{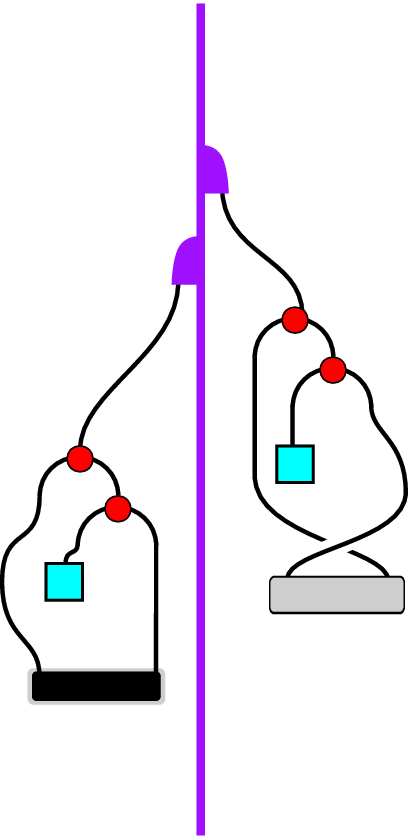}
   \put(17.8,-8.5) {\sse$ X $}
   \put(18.4,94.8) {\sse$ X $}
   \put(45,23.8)   {\sse$ R $}
   \put(-14.6,15.2){\sse$ R^{-1} $}
   \put(10.4,25.7) {\sse$ \uvi $}
   \put(36.2,38.9) {\sse$ \uvi $}
   } }
Using the relations $\uvi \eq v^{-1} u$ and $\apo^2\eq\ad t$, the fact
that $v$ is central and the relation \erf{u-R} between the canonical
element $u$ and the R-matrix then gives the formula \erf{bimod-twist}.
\end{proof}

\begin{rem}\label{rem-trivtwist}
(i)\,~By using that $v\iN H$ is central and satisfies $\apo\cir v \eq v$.
it follows immediately from \erf{bimod-twist} that
the Frobenius algebra \Hb\ has trivial twist,
   \be
   \theta_\Hb = \idHs \,.
   \labl{trivtwist}
(ii)\,~In fact, a commutative symmetric Frobenius algebra in any sovereign braided
category has trivial twist. This was shown in Prop.\,2.25(i) of \cite{ffrs}
for the case that the category is strictly sovereign (i.e.\ that
the sovereign structure is trivial in the sense that $\piv_X\eq\id_X$ for all
$X$), and the proof easily carries over to general sovereign categories.
Conversely, the fact that \Hb\ is a symmetric algebra can be derived by combining
the triviality \erf{trivtwist} of the twist with commutativity.
\nxl1
(iii)~\,That \Hb\ is commutative and symmetric implies \cite[Prop.\,2.25(iii)]{ffrs}
that it is cocommutative as well.
\end{rem}


\section{Modular invariance}\label{sec-modinv}

Our focus so far has been on a natural object in the sovereign braided finite
tensor category \HBimod, the symmetric Frobenius algebra \Hb. But in any such
category there exists another natural object \HK, which has been studied
by Lyubashenko. \HK\ is a Hopf algebra, and it plays a crucial role in the
construction of mapping class group actions.
The construction of these mapping class group actions relies on the presence of
several distinguished endomorphisms of \HK. The existence of these endomorphisms 
is a consequence of universal properties characterizing the Hopf algebra object \HK. 
More precisely, apart from the antipode $\apo_\HK$ of \HK, Lyubashenko obtains
invertible endomorphisms $\SK,\TK\iN \EndC(\HK)$ that obey the relations 
\cite[Thm\,2.1.9]{lyub6}
   \be
   (\SK\,\TK)^3 = \lambda\, \SK^2 \qquad\mbox{ and }\qquad \SK^2 = \apo_\HK^{-1}
   \labl{SK,TK}
with some scalar $\lambda$ that depends on the category \C\ in question.

These endomorphisms are the central ingredient for the construction of
representations of mapping class groups of punctured Riemann surfaces on
morphism spaces of \C.  In particular, Lyubashenko \cite[Sect.\,4.3]{lyub6} 
constructed a projective representation of the mapping class group $\Gamma_{\!1;m}$
of the $m$-punctured torus on morphism spaces of the form 
$\HomC(\HK,X_1 \oti X_2 \oti \cdots \oti X_m)$, for $(X_1,X_2,...\,,X_m)$ any $m$-tuple of 
objects of \C. Specializing to the case of one puncture, $m\eq1$, there is, for any object 
$X$ of \C, a projective \slze-action $\slze \Times \HomC(\HK,X) \To \HomC(\HK,X)$.
The mapping class group $\slze$ is generated by three generators $S$, $T$ and $D$,
where $D$  is the Dehn twist around the puncture. The representation of $\slze$ satisfies
   \be
   (S,f) \,\mapsto\, f \circ \SK^{-1}\,, ~~\quad (T,f) \,\mapsto\, f \circ \TK^{-1}
   ~\quad\text{ and }~\quad (D,f) \,\mapsto\, \theta_X \circ f
   \labl{slze-rep}
with $\theta$ the twist of \C.

These general constructions apply in particular to the finite tensor category
\HBimod. The main goal in this section is to use the coproduct of the Frobenius
algebra \Hb\ to construct an element in $\HomHH(\HK,\Hb)$ that is invariant
under the action of \slze. A similar construction allows one to derive an
invariant element in $\HomHH(\HK\oti\Hb,\one)$ from the product of \Hb.
For the motivation to detect such elements and for possible applications
in full local conformal field theory we refer to Appendix \ref{app:cft}.


  \subsection{Distinguished endomorphisms of coends}

A finite tensor category \cite{etos} is a \ko-linear abelian rigid monoidal
category with enough projectives and with finitely many simple objects up to
isomorphism, with simple tensor unit, and with every object having a composition
series of finite length. Both \HMod\ and \HBimod, for $H$ a \findim\ unimodular
ribbon Hopf algebra, belong to this class of categories, see \cite{lyma,lyub6}.

Let \C\ be a sovereign braided finite tensor category. As shown in \cite{lyub6,kerl5}
(compare also \cite{vire4} or, as a review, Sections 4.3 and 4.5 of \cite{fuSc17}),
there exists an object \HK\ in \C\ that carries a natural structure of a Hopf algebra
in \C. Moreover, there is a two-sided integral as well as a Hopf pairing for this
Hopf algebra \HK. Combining the integral of \HK\ and other structure of the category, 
one constructs \cite{lyma,lyub6,lyub8} distinguished morphisms $\SK$ and $\TK$ 
satisfying \erf{SK,TK} in $\EndC(\HK)$.

The Hopf algebra \HK\ can be characterized as the \emph{coend}\,%
  \footnote{~%
  The definition of the coend $\HK$ of a functor $G\colon\, \COPC \To \C$, including
  the associated dinatural family $\iHK$ of morphisms $\iHK_X \iN \HomC(F(X,X),\HK)$,
  will be recalled at the beginning of Appendix \ref{bicom-coend}.}    
   \be
   \HK = \Coend FX = \coend X
   \labl{def-HK}
of the functor $F$ that acts on objects as $(X,Y) \,{\mapsto}\, X^\vee{\otimes}\,Y$.
As described in some detail in Appendix \ref{bicoad-coend},
in the case $\C \eq \HBimod$ the object \HK\ is the \emph{\bicoad} \Haa. That is,
the underlying vector space is the tensor product $\Hs\otik\Hs$, and this space
is endowed with a left $H$-action by the coadjoint left action \erf{def_lads} on
the first factor, and with a right $H$-action by the coadjoint right action on the
second factor.

In the case of a general braided finite tensor category \C\ the
morphisms $\SK$ and $\TK$ in $\EndC(\HK)$ are 
defined with the help of the braiding $c$ and the twist $\theta$ of \C, respectively
\cite{lyma,lyub6,lyub8}. $\TK$ is given by the dinatural family
   \eqpic{p9} {92}{24} {
   \put(0,0)   {\Includepichtft{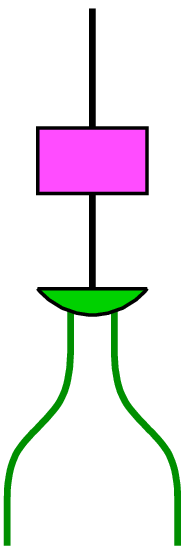}
   \put(-11.1,39.7) {\small$\TK$}
   \put(-5,-8.5)  {\sse$ X^{\!\vee} $}
   \put(17.4,25.7){\sse$ \iHK_X $}
   \put(15.3,-8.5){\sse$ X $}
   \put(6.8,63.3) {\sse$ \HK $}
   }
   \put(47,28) {$=$}
   \put(83,0)  {\Includepichtft{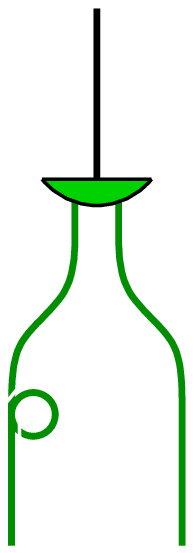}
   \put(-11.9,14) {\sse$ \theta_{\!X^{\!\vee}_{}}^{}$}
   \put(-3,-8.5)  {\sse$ X^{\!\vee} $}
   \put(8.4,63.3) {\sse$ \HK $}
   \put(18.9,37.7){\sse$ \iHK_X $}
   \put(17.4,-8.5){\sse$ X $}
   }
   \put(133,46)    {\catpic}
   }
Here it is used that a morphism $f$ with domain the coend \HK\ is uniquely
determined by the dinatural family $\{ f\cir \iHK_X\}$ of morphisms.
For $S$ one defines
   \be
   \SK := (\eps_\HK \oti \id_\HK) \circ \QQ_{\HK,\HK}
   \circ (\id_\HK \oti \Lambda_\HK) \,,
   \labl{S-HK}
where $\eps_\HK$ and $\Lambda_\HK$ are the counit and the two-sided integral of
the Hopf algebra \HK, respectively, while the morphism
$\QQ_{\HK,\HK} \iN \EndC(\HK\oti\HK)$ is determined through monodromies
$c_{Y^\vee_{\phantom,}\!,X} \cir c_{X,Y^\vee_{\phantom,}}$ according to
   \eqpic{QHH} {135} {47} {
   \put(0,0)  {\Includepichtft{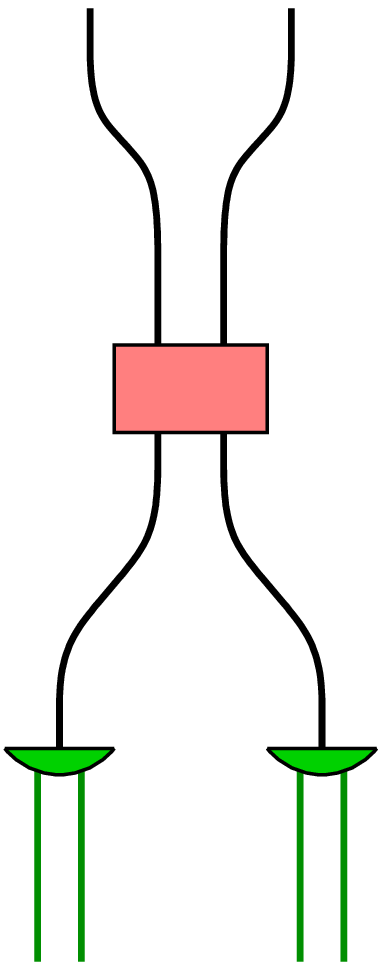}
   \put(-15.9,60.7){\small$ \QQ_{\HK,\HK} $}
   \put(-9.8,21)  {\sse$ \iHK_X $}
   \put(-4,-8.5)  {\sse$ X^{\!\vee} $}
   \put(5.3,109)  {\sse$ \HK $}
   \put(7,-8.5)   {\sse$ X $}
   \put(26.5,-8.5){\sse$ Y^{\!\vee} $}
   \put(28.8,109) {\sse$ \HK $}
   \put(38,-8.5)  {\sse$ Y $}
   \put(42.4,21)  {\sse$ \iHK_Y $}
   }
   \put(62,50)    {$ = $}
   \put(97,0) {\Includepichtft{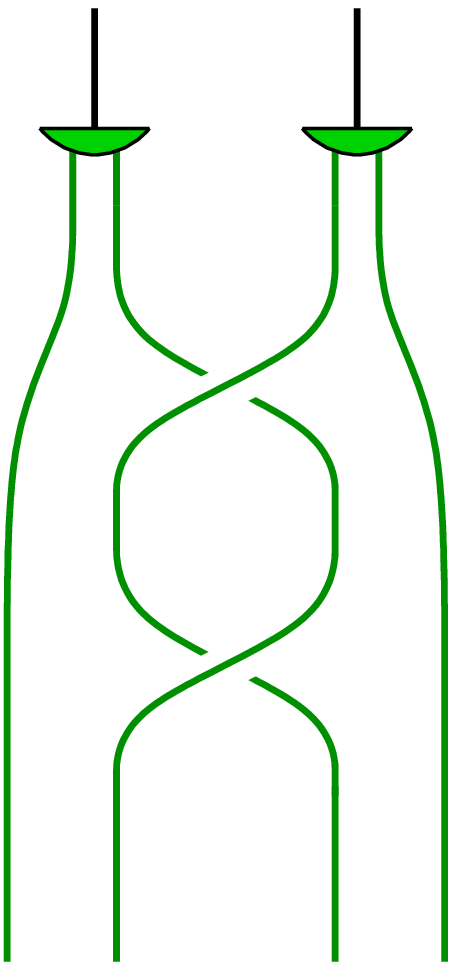}
   \put(-6,-8.5)  {\sse$ X^{\!\vee} $}
   \put(-5.8,89)  {\sse$ \iHK_X $}
   \put(8,-8.5)   {\sse$ X $}
   \put(6.1,109)  {\sse$ \HK $}
   \put(22.8,27.3){\sse$ c $}
   \put(22.8,58.4){\sse$ c $}
   \put(32,-8.5)  {\sse$ Y^{\!\vee} $}
   \put(36,109)   {\sse$ \HK $}
   \put(46.2,89)  {\sse$ \iHK_Y $}
   \put(46,-8.5)  {\sse$ Y $}
   }
   \put(186,91)    {\catpic}
   }

We also note that the Hopf algebra \HK\ is endowed with a Hopf pairing
$\omega_\HK$, given by \cite{lyub6}
   \be
   \omega_\HK^{} = (\eps_\HK \oti \eps_\HK ) \circ \QQ_{\HK,\HK} \,.
   \labl{HK-Hopa}


\subsection{The Drinfeld map}

Let us now specialize the latter formulas to the case of our interest, i.e.\
$\C \eq \HBimod$. Then the coend is $\HK \eq \Haa$, with dinatural family
$\iHK \eq \iHaa$ given by \erf{def_iHaa_X}, while the twist is given by
\erf{bimod-twist} and the braiding by \erf{bibraid}. Further, the structure
morphisms of the categorical Hopf algebra \Haa\ can be expressed through those
of the algebraic Hopf algebra $H$; in particular, the counit and integral are
$\eps\bicoa \eq \eta^\vee \oti \eta^\vee$ (see \erf{Haa-Hopf}) and
$\Lambda\bicoa \eq\lambda^\vee \oti \lambda^\vee$ (see \erf{Haa-integral}).
The monodromy morphism $\QQ_{\Haasm,\Haasm} \iN \EndHH(\Haa\oti\Haa)$ that was
introduced in \erf{QHH} then reads
   \eqpic{QQ_Haa} {230}{63} {
   \put(0,57)     {$ \QQ_{\Haasm,\Haasm} ~= $}
   \put(81,0)  { \Includepichtft{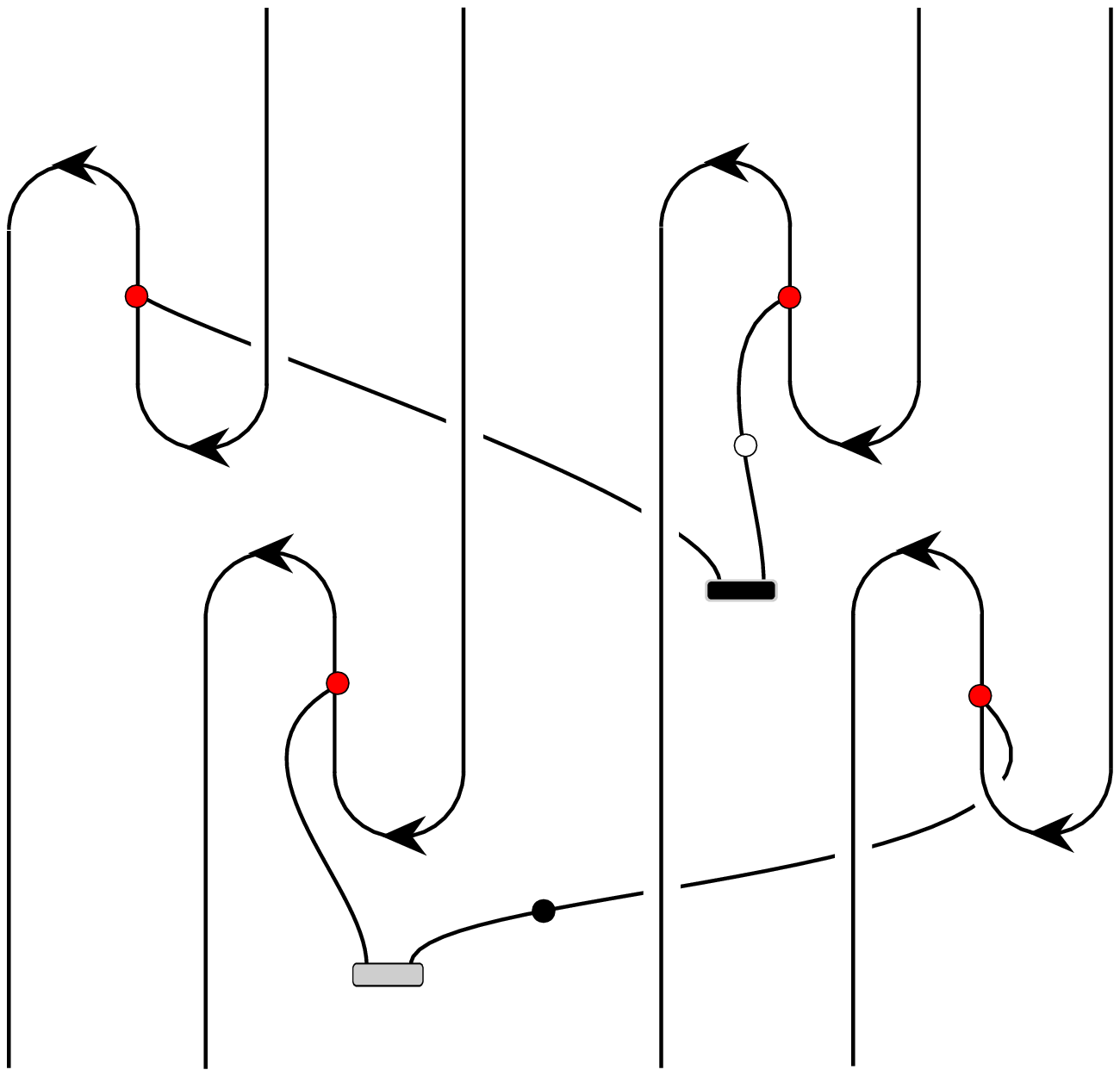}
   \put(-4.3,-8.5) {\sse$ \Hss $}
   \put(20.7,-8.5) {\sse$ \Hss $}
   \put(29.4,139)  {\sse$ \Hss $}
   \put(44.5,4.3)  {\sse$ Q $}
   \put(54.8,139)  {\sse$ \Hss $}
   \put(62.8,13.1) {\sse$ \apoi $}
   \put(78.7,-8.5) {\sse$ \Hss $}
   \put(86.6,53.5) {\sse$ Q^{-1} $}
   \put(87.6,77.9) {\sse$ \apo $}
   \put(102.9,-8.5){\sse$ \Hss $}
   \put(112.2,139) {\sse$ \Hss $}
   \put(137.7,139) {\sse$ \Hss $}
   } }
while the general formulas for $\TK$ and $\SK$ specialize to the morphisms
   \Eqpic{rho_Haa_S,T} {420}{39} {
   \put(4,36)    {$ T\bicoa ~= $}
   \put(55,0)  { \Includepichtft{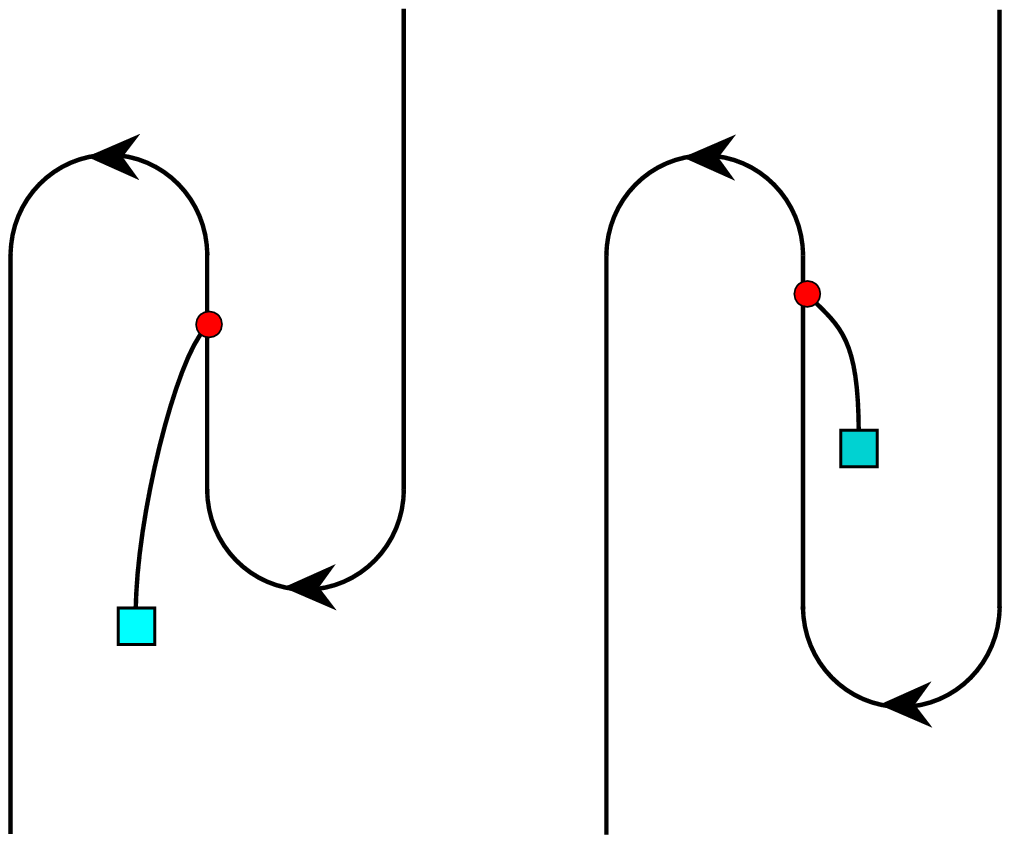}
   \put(-4.4,-8.5) {\sse$ \Hss $}
   \put(12,15)     {\sse$ v $}
   \put(39.6,94)   {\sse$ \Hss $}
   \put(61.2,-8.5) {\sse$ \Hss $}
   \put(89.7,34.5) {\sse$ v^{-1} $}
   \put(105.3,94)  {\sse$ \Hss $}
   }
   \put(215,36)    {and}
   \put(276,36)    {$ S\bicoa ~= $}
   \put(327,-6) { \Includepichtft{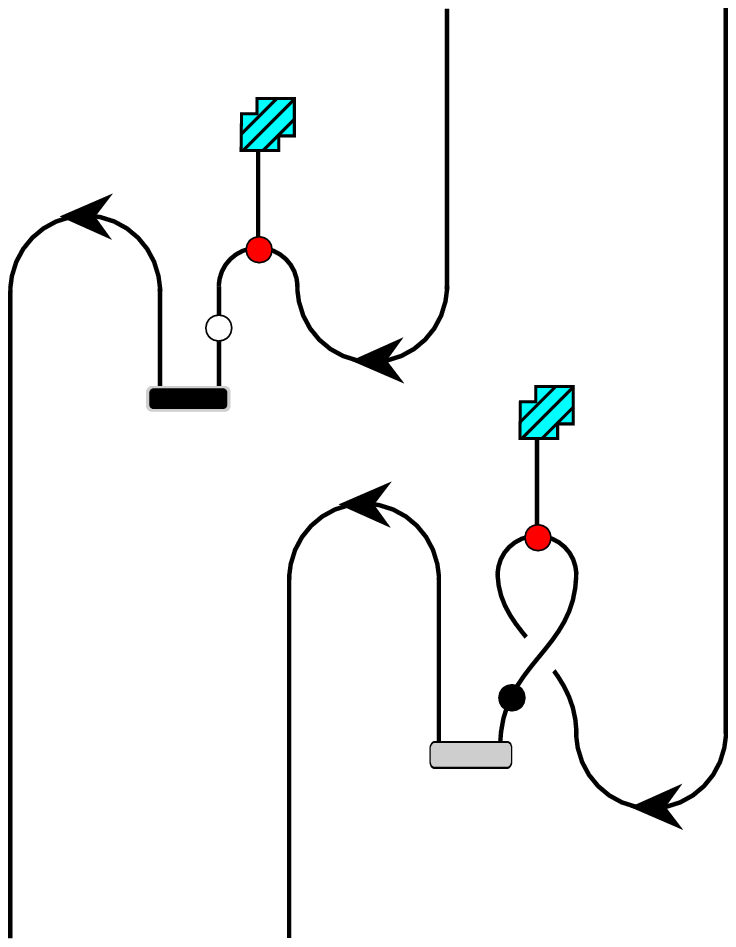}
   \put(-4.5,-8.5) {\sse$ \Hss $}
   \put(13,51.8)   {\sse$ Q^{-1} $}
   \put(26.5,-8.5) {\sse$ \Hss $}
   \put(32.5,87)   {\sse$ \lambda $}
   \put(44.6,106)  {\sse$ \Hss $}
   \put(47.4,12.1) {\sse$ Q $}
   \put(63.5,56)   {\sse$ \lambda $}
   \put(75.6,106)  {\sse$ \Hss $}
   } }
in $\EndHH(\Haa)$ .
Note that the morphism $S\bicoa$ is composed of (variants of) the
Frobenius map \erf{def_fmap} and the \emph{Drinfeld map}
   \be
   \drin := (d_H\oti \id_H) \circ (\idHs\oti Q) \,\in \Hom(\Hs,H) \,.
   \labl{def-drin}

In order that $S\bicoa$ is invertible, which is necessary for having projective
mapping class group \rep s, it is necessary and sufficient that the Drinfeld map
$\drin$ is invertible.

\begin{rem}
By the results for general \C, $S\bicoa$ is indeed a morphism in \HBimod.
But this is also easily checked directly: One just has to use that
the Drinfeld map intertwines the left coadjoint action $\rho\coa$ (see
\erf{def_lads}) of $H$ on \Hs\ and the left adjoint action $\rho_{\rm ad}$ 
of $H$ on itself \cite[Prop.\,2.5(5)]{coWe4}, i.e.\ that
   \be
   \drin \in \HomH(\Ha,H_{\rm ad}) \,,
   \labl{drin-intw}
together with the fact that the cointegral $\lambda$
satisfies (since $H$ is unimodular) \cite[Thm.\,3]{radf12}
   \be
   \lambda \circ m = \lambda \circ m \circ \tauHH \circ (\id_H\oti\apo^2) \,.
   \labl{lambda-Oapoo}
\end{rem}

\begin{rem}
The Drinfeld map $\drin$ of a \findim\ quasitriangular Hopf algebra $H$ is
invertible iff $H$ is factorizable. In the semisimple case, factorizability is
the essential ingredient for a Hopf algebra to be modular
\cite[Lemma\,8.2]{nitv}. Here, without any assumption of semisimplicity, we
see again a direct link between invertibility of $S$ and factorizability.
\end{rem}

\begin{rem}
The coadjoint $H$-module \Ha\ is actually the coend \erf{def-HK} for the
case that \C\ is the category \HMod\ of left $H$-modules. In this case the
endomorphism \erf{S-HK} is precisely the composition
   \be
   S\coa = \fmap \circ \drin
   \labl{S-coa}
of the Drinfeld and Frobenius maps (also called the quantum Fourier transform),
see e.g.\ \cite{lyma,fgst}. Further, the Drinfeld map $\drin$ is related to the
Hopf pairing $\omega_\Ha^{}$ from \erf{HK-Hopa} for the coend \Ha\ in \HMod\ by
   \be
   \drin = (\apo\oti\omega_\Ha^{}) \circ (\id_H \oti \tau_{\Hss,\Hss})
   \circ (b_H \oti \idHs) \,.
   \ee
In particular, since the antipode is invertible, the Drinfeld map of $H$ is
invertible iff the Hopf pairing of \Ha\ is non-degenerate.
\end{rem}

\begin{rem}
For factorizable $H$, the Drinfeld map $\drin$ maps any non-zero cointegral
$\lambda$ of $H$ to a non-zero integral $\Lambda$. Thus we may (and do) choose
$\lambda$ and $\Lambda$ (uniquely, up to a common factor $\pm1$) such that
besides $\lambda\cir\Lambda \eq 1$ we also have
   \be
   f_Q(\lambda) = \Lambda
   \labl{fQl=L}
(see \cite[Thm.\,2.3.2]{geWe} and \cite[Rem.\,2.4]{coWe5}). Together with
$\apo\cir\Lambda \eq \Lambda$ and the property \erf{lambda-Oapoo} of the cointegral
it then also follows that
   \be
   f_{Q^{-1}_{}}(\lambda) = \Lambda
   \qquad \mbox{as well as} \qquad f_{Q^{-1}_{}} \in \HomH(\Ha,H_{\rm ad}) \,,
   \labl{fQinv_intertw}
where $f_{Q^{-1}_{}}$ is the morphism \erf{def-drin} with the monodromy matrix
$Q$ replaced by its inverse $Q^{-1}$. Further, one has \cite[Lemma\,2.5]{coWe5}
   \be
   f_Q \circ \Psi \circ f_{Q^{-1}_{}} \circ \Psi = \id_H \,,
   \labl{fQ-Psi-fQi-Psi}
which in turn by comparison with \erf{S-coa} shows that
$\Psi \cir f_{Q^{-1}_{}} \eq S\coa^{-1}$. Further, the latter identity and
\erf{S-coa} are equivalent to the relations
   \Eqpic{fQS_Psi} {430} {41} { \put(0,-2){
   \put(0,0) { \Includepichtft{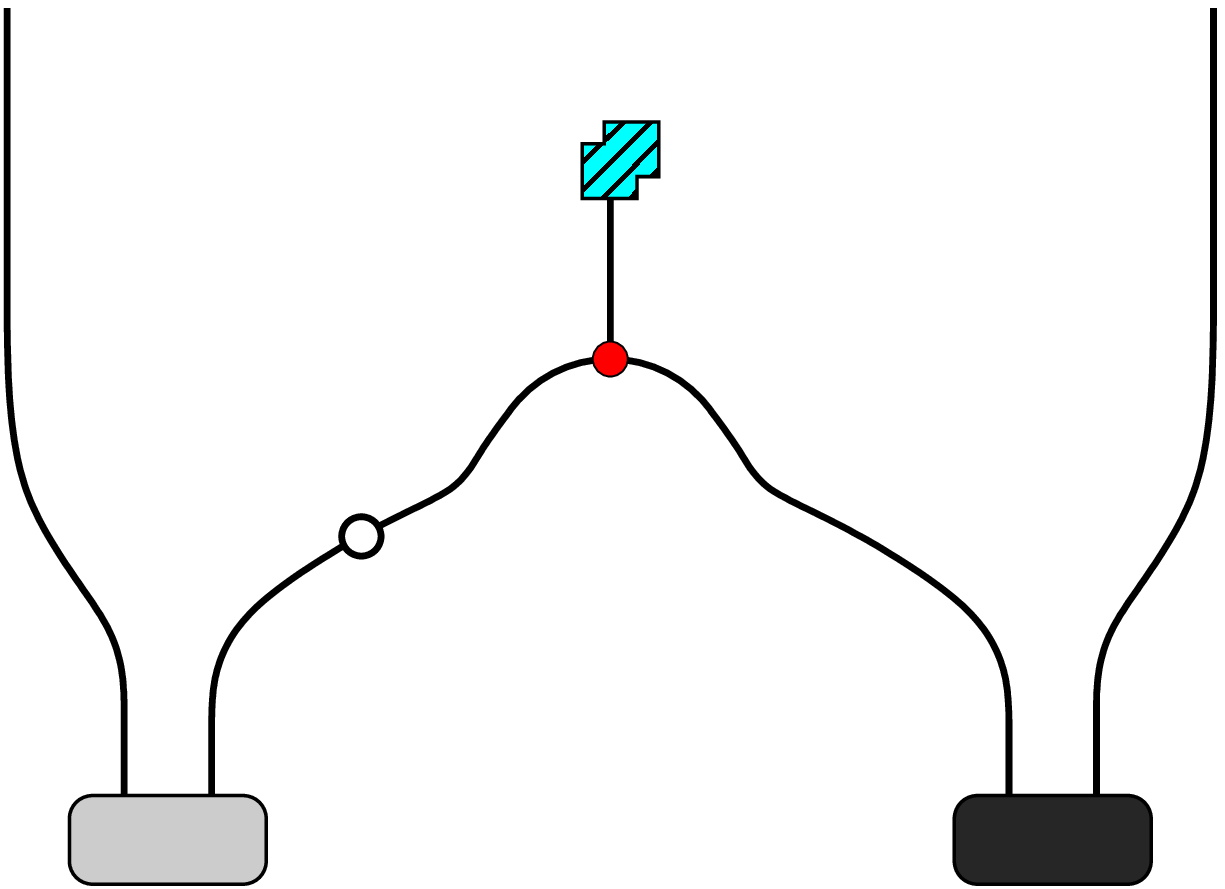}
   \put(-2.5,100){\sse$ H $}
   \put(14.5,3.4){\sse$ Q $}
   \put(43.6,35) {\sse$ \apo $}
   \put(73.1,75.5) {\sse$ \lambda $}
   \put(128.3,3.2) {\sse$ Q^{-1} $}
   \put(130,100) {\sse$ H $}
   }
   \put(167,46)  {$ = $}
   \put(193,13) { \Includepichtft{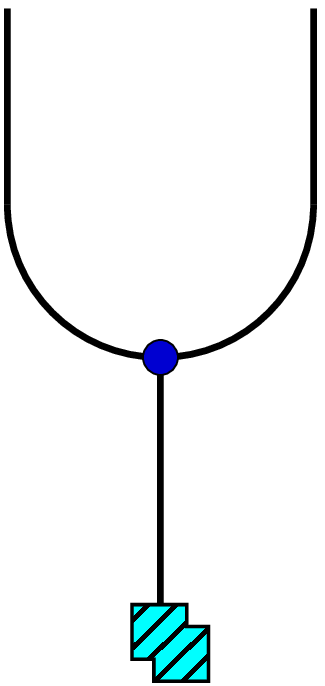}
   \put(-2.2,79) {\sse$ H $}
   \put(7.7,-1)  {\sse$ \Lambda $}
   \put(30.8,79) {\sse$ H $}
   }
   \put(256,46)  {$ = $}
   \put(294,0) { \Includepichtft{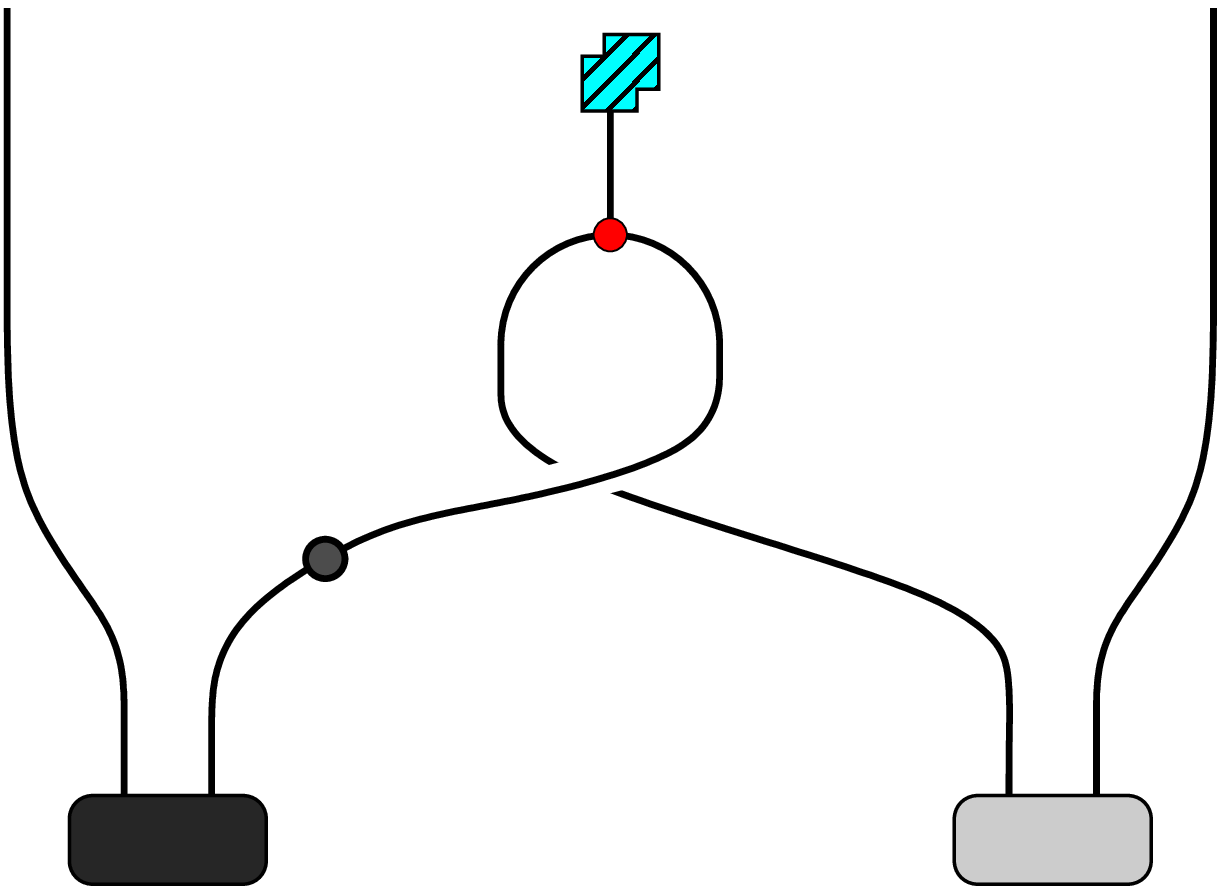}
   \put(-2.5,100){\sse$ H $}
   \put(31.1,3.2){\sse$ Q^{-1} $}
   \put(39.5,32.6) {\sse$ \apoi $}
   \put(73.7,85.4) {\sse$ \lambda $}
   \put(112.5,3.4) {\sse$ Q $}
   \put(130,100) {\sse$ H $}
   } } }
\end{rem}

\begin{rem}
According to the first of the formulas \erf{rho_Haa_S,T} we have $T\bicoa^{} 
\eq T\coa^{-1} \oti T\coa^{}$. Using \erf{lambda-Oapoo} and \erf{S-coa} it 
follows that $S\bicoa^{} \eq S\coa^{-1} \oti S\coa^{}$ as well.
As a consequence, the first of the relations \erf{SK,TK} is realized
with $\lambda \eq 1$. Thus in the case of the category \HBimod\ of our interest,
with coend \HK, the projective representation \erf{slze-rep} of the mapping
class group \slze\ of the one-punctured torus is actually a \emph{genuine}
\slze-representation.
\end{rem}


\subsection{Action of \slze\ on morphism spaces}

Consider now the representation $\rho_{\HK,Y}$
of \slze\ on the spaces $\HomC(\HK,Y)$ just mentioned. The group \slze\
is generated by three elements $S$, $T$ and $D$, where $D$ is the
Dehn twist around the puncture, while $S$ and $T$ are modular transformations
that act on the surface in the same way as in the absence of the puncture.
The generators are subject to the relations $(ST)^3 \eq S^2$ (like for the modular
group) and $S^4 \eq D$. Of particular interest to us are \emph{invariants} of the
\slze-action on the spaces $\HomC(\HK,Y)$, i.e.\ morphisms $g$ satisfying
   \be
   g \circ \rho_{\HK,Y}^{}(\gamma) = g
   \labl{morph_K_F}
for all $\gamma\iN\slze$.

Morphisms in $\HomC(\HK,Y)$ can in particular be obtained by defining a linear map
$\Wiha$ from $\HomC(X,Y\oti X)$ to $\HomC(\HK,Y)$ as a universal partial trace,
to which we refer as the \emph{\pqt}. Thus for $f \iN \HomC(X,Y\oti X)$ we set
   \eqpic{def_wiha} {210} {67} {
   \put(-6,70)   {$ \wiha f ~:= $}
   \put(75,0)  {\Includepichtft{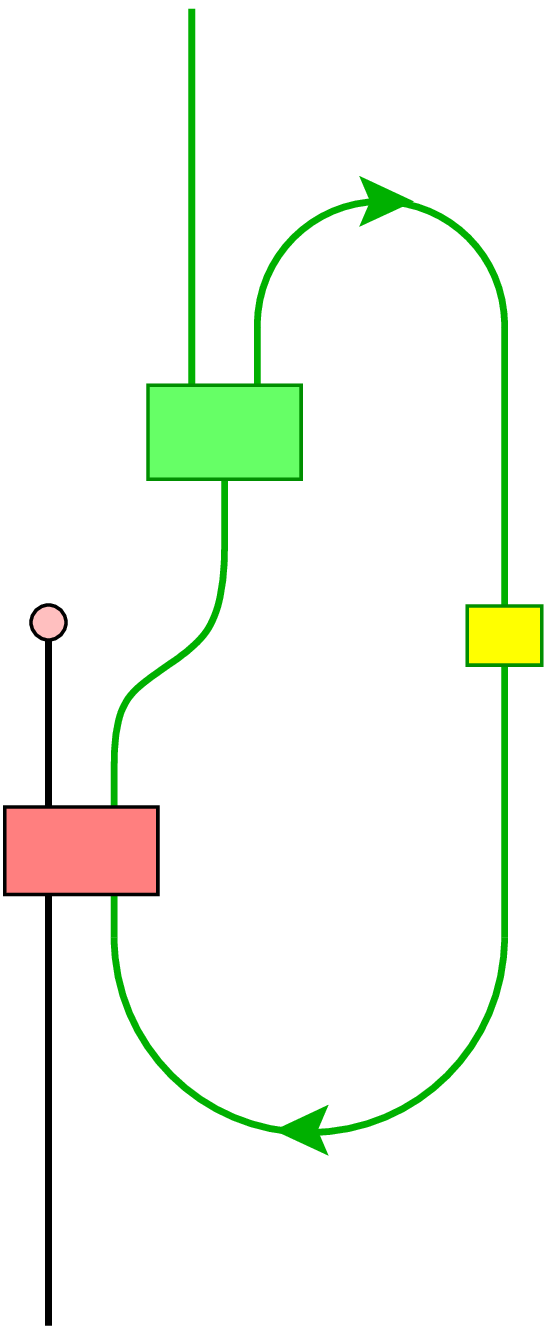}
   \put(-23.2,50.5){\sse$ \Qq_{\HK,X} $}
   \put(-8.4,77) {\sse$ \eps_\HK^{} $}
   \put(1.1,-8.5){\sse$ \HK $}
   \put(18.8,149){\sse$ Y $}
   \put(23.1,97) {\sse$ f $}
   \put(24,76)   {\sse$ X $}
   \put(29.5,109){\sse$ X $}
   \put(61.3,75) {\sse$ \pi_X $}
   }
   \put(156,130) {\catpic}
   \put(165,70)  {$ \in ~\HomC(\HK,Y)~, $}
   }
where $\Qq_{\H,Y} \iN \EndC(\HK\oti Y)$ is defined by
   \eqpic{QHX} {110} {46} {
   \put(0,0)  {\Includepichtft{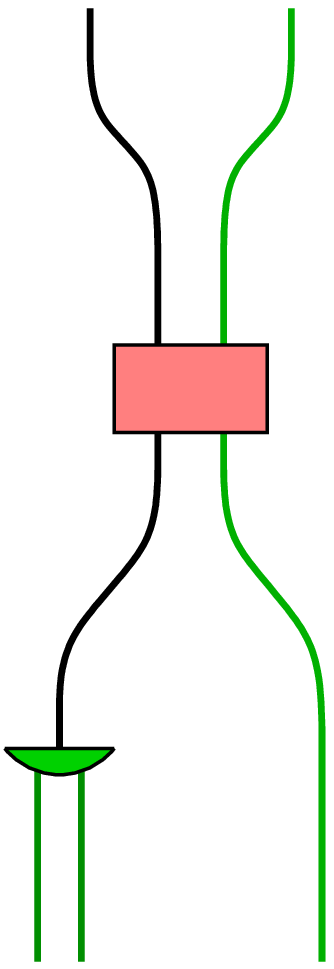}
   \put(-15,60)  {$ \Qq_{\HK,Y} $}
   \put(-4,-8.5) {\sse$ X^{\!\vee} $}
   \put(5.3,109) {\sse$ \HK $}
   \put(7,-8.5)  {\sse$ X $}
   \put(31.3,-8.5){\sse$ Y $}
   \put(30.4,109){\sse$ Y $}
   }
   \put(60,50)   {$ := $}
   \put(94,0) {\Includepichtft{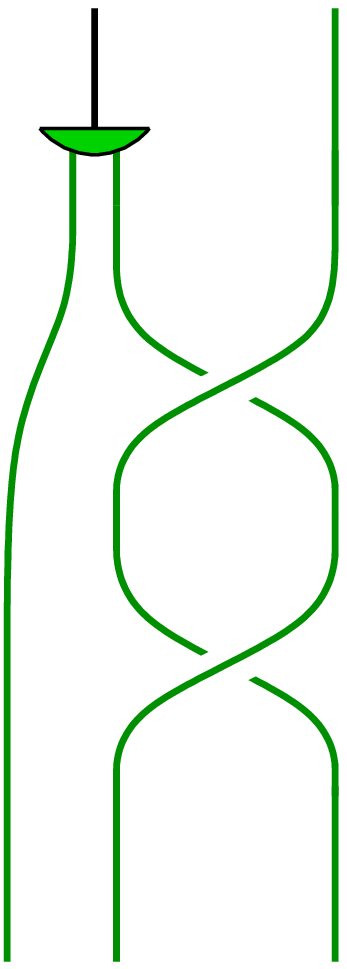}
   \put(-6,-8.5) {\sse$ X^{\!\vee} $}
   \put(6.3,109) {\sse$ \HK $}
   \put(8,-8.5)  {\sse$ X $}
   \put(22.5,27.3) {\sse$ c $}
   \put(22.5,58.4) {\sse$ c $}
   \put(34,-8.5) {\sse$ Y $}
   \put(34.3,109){\sse$ Y $}
   }
   \put(167,91)    {\catpic}
   }
Note that, by the naturality of the braiding, the morphisms \erf{QHX} are
natural in $Y$.

\begin{rem}\label{rem:rep}
(i)~\,It follows from elementary properties of the braiding that for any
object $X$ of \C\ the morphism $(\eps_\HK^{} \oti \id_X) \cir \Qq_{\HK,X} $
endows $X$ with the structure of a \HK-module internal to \C.
\nxl1
(ii)~\,If \C\ is a (semisimple, strictly sovereign) modular tensor category, then
the invariance property \erf{morph_K_F} for $\gamma \eq S$ and $g \eq \wiha f$ is
equivalent to the definition of $S$-invariance of morphisms in $\HomC(Y\oti X,X)$
that is given in \cite[Def.\,3.1(i)]{koRu2}.
\end{rem}

\smallskip

Specializing now to $\,\C\eq\HBimod$ and $X \eq Y$ being the Frobenius algebra 
\Hb\ in \HBimod, we can state one of the main results of this paper:

\begin{thm}
The \pqt\ of the coproduct $\Delta\bico$ is \slze-invariant.
\end{thm}

We will prove this statement by establishing, in Lemma \ref{lemmaT} and Lemma 
\ref{lemmaS} below, separately invariance under the two generators $T$ and $S$ of
\slze. Note that this implies in particular invariance under the Dehn twist $D$;
the latter can also be directly deduced from the fact that $D\bicoa \eq \theta_\Hb$
together with the result \erf{trivtwist} that \Hb\ has trivial twist.
Before investigating the action of $S$ and $T$, let us first present the
\pqt\ $\wiha {\Delta\bico}$ in a convenient form. To this end
we first note that, invoking the explicit form \erf{bibraid} of the braiding
and \erf{pivX} of the sovereign structure of \HBimod, we have
   \Eqpic{Sinv-1-2} {420} {49} { \put(-10,0){
   \put(29,8)  {\Includepichtft{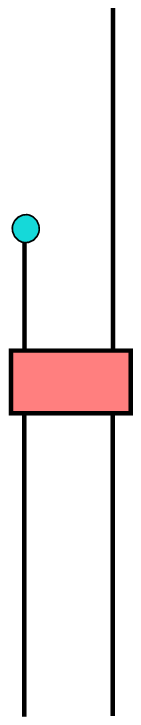}
   \put(-23,35.5){\sse$ \Qq_{\HK,\Hb} $}
   \put(-3.6,-8.5) {\sse$ \HK $}
   \put(7.8,-8.5){\sse$ \Hb $}
   \put(8.1,81.5){\sse$ \Hb $}
   }
   \put(63,43)   {$=$}
   \put(91,0)  {\Includepichtft{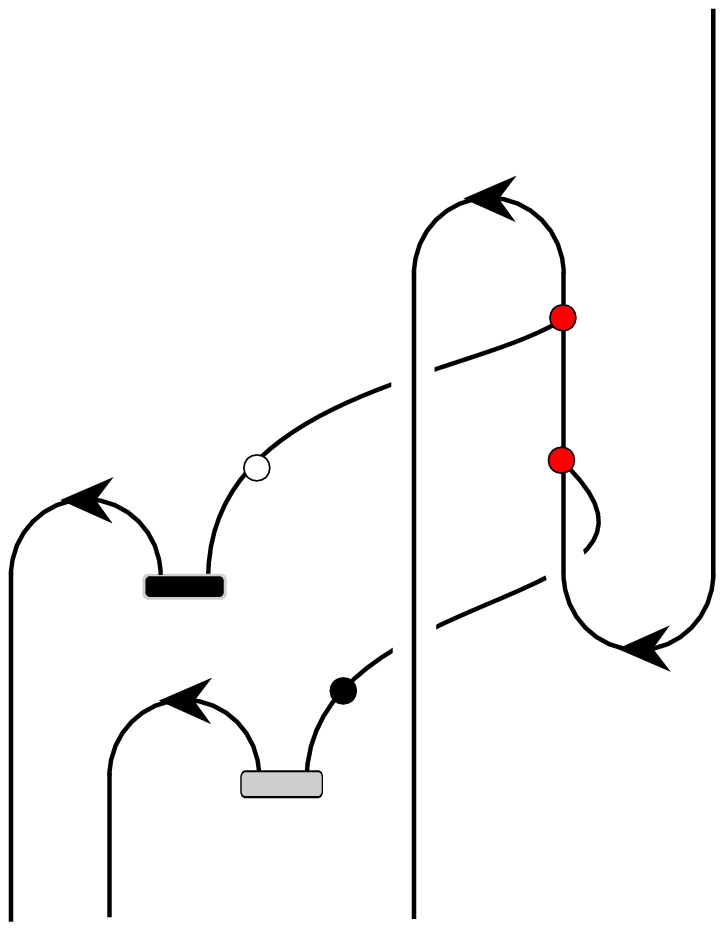}
   \put(-5,-8.5) {\sse$ \Hss $}
   \put(7.8,-8.5){\sse$ \Hss $}
   \put(24,32)   {\sse$ Q^{-1} $}
   \put(26.5,7.2){\sse$ Q $}
   \put(40,-8.5) {\sse$ \Hss $}
   \put(73.5,104){\sse$ \Hss $}
   }
   \put(202,43)  {and hence}
   \put(280,43)  {$\wiha f~= $}
   \put(342,0) {\Includepichtft{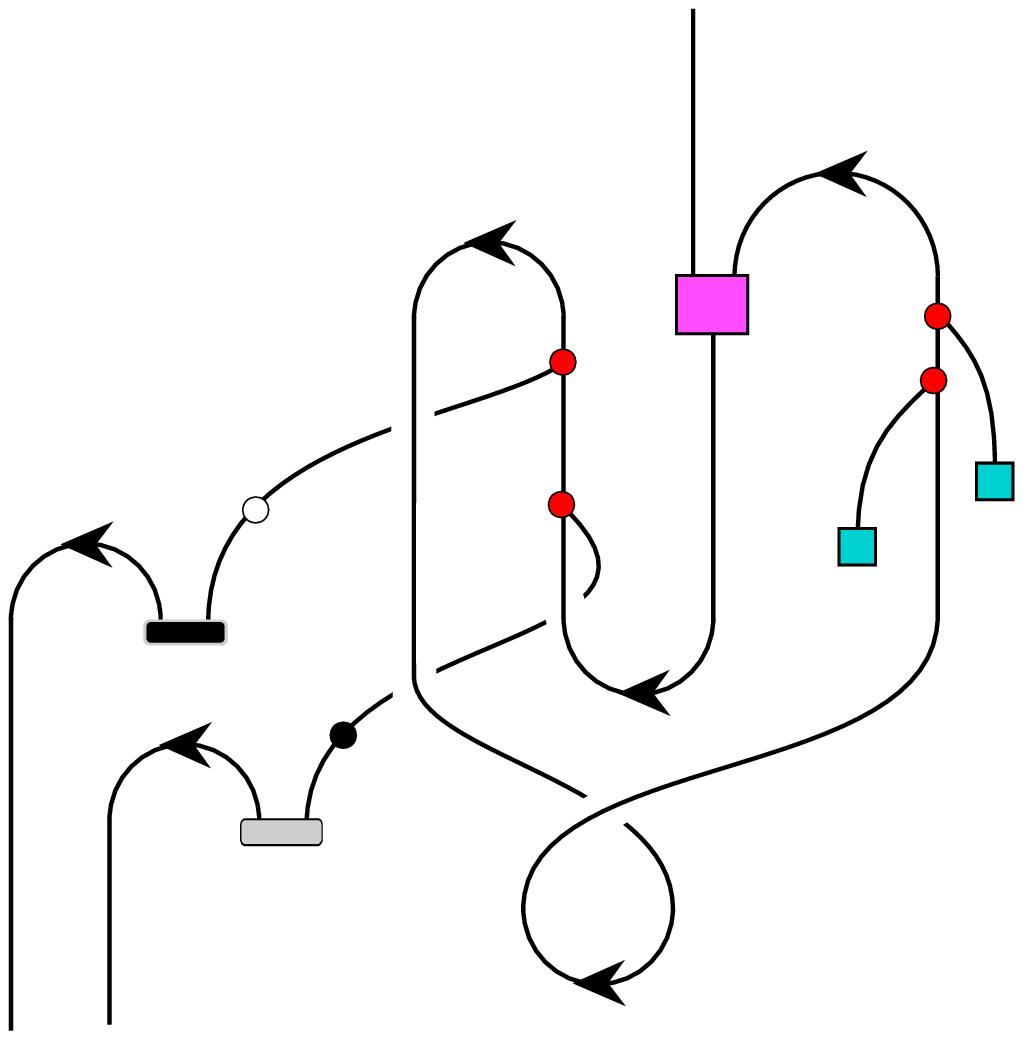}
   \put(-6,-8.5) {\sse$ \Hss $}
   \put(7,-8.5)  {\sse$ \Hss $}
   \put(25,39)   {\sse$ Q^{-1} $}
   \put(26.6,14.2) {\sse$ Q $}
   \put(82,78)   {\sse$ f $}
   \put(87.1,44.5) {\sse$ t^{-1} $}
   \put(71.4,116){\sse$ \Hss $}
   \put(112.2,58){\sse$ t^{-1} $}
   } } }
for any $f \iN \HomHH(\Hb,\Hb\oti\Hb)$. Specializing \erf{Sinv-1-2} to the \pqt\
of the coproduct $\Delta\bico$, i.e.\ inserting $\Delta\bico$ from
\erf{Deltaprime}, yields
   \eqpic{Sinv-3} {160} {63} {
   \put(0,62)    {$ \wiha {\Delta\bico}~= $}
   \put(72,0)  {\Includepichtft{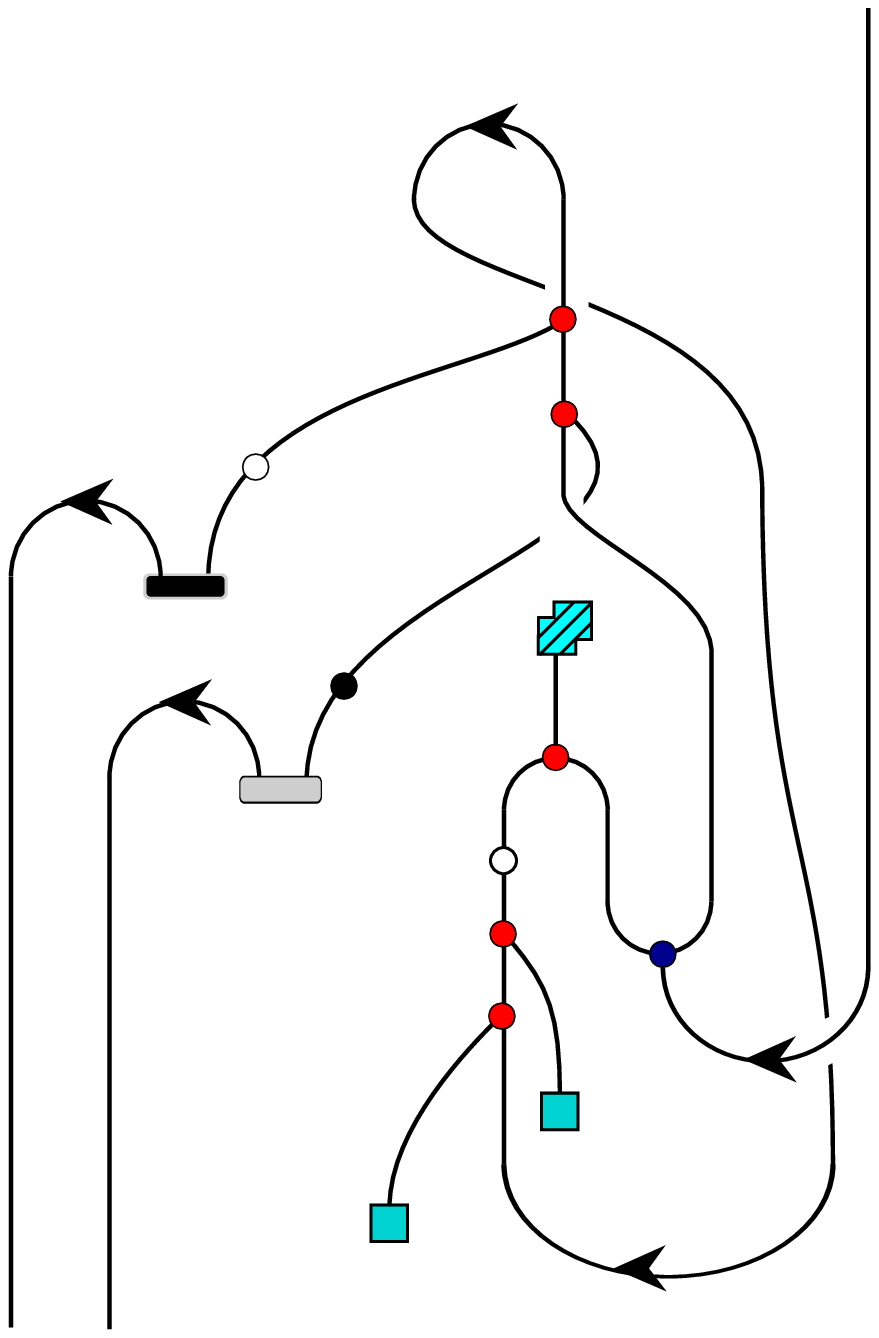}
   \put(-6,-8.5) {\sse$ \Hss $}
   \put(7,-8.5)  {\sse$ \Hss $}
   \put(24.9,78) {\sse$ Q^{-1} $}
   \put(26.6,51.6) {\sse$ Q $}
   \put(65,74)   {\sse$ \lambda $}
   \put(36.3,2.5){\sse$ t^{-1} $}
   \put(64.3,22) {\sse$ t^{-1} $}
   \put(90,149)  {\sse$ \Hss $}
   } }
This can be rewritten as follows:
   \eqpic{Sinv-4} {390} {61} {
   \put(-5,59)   { $\wiha {\Delta\bico}~=$ }
   \put(73,0)  {\Includepichtft{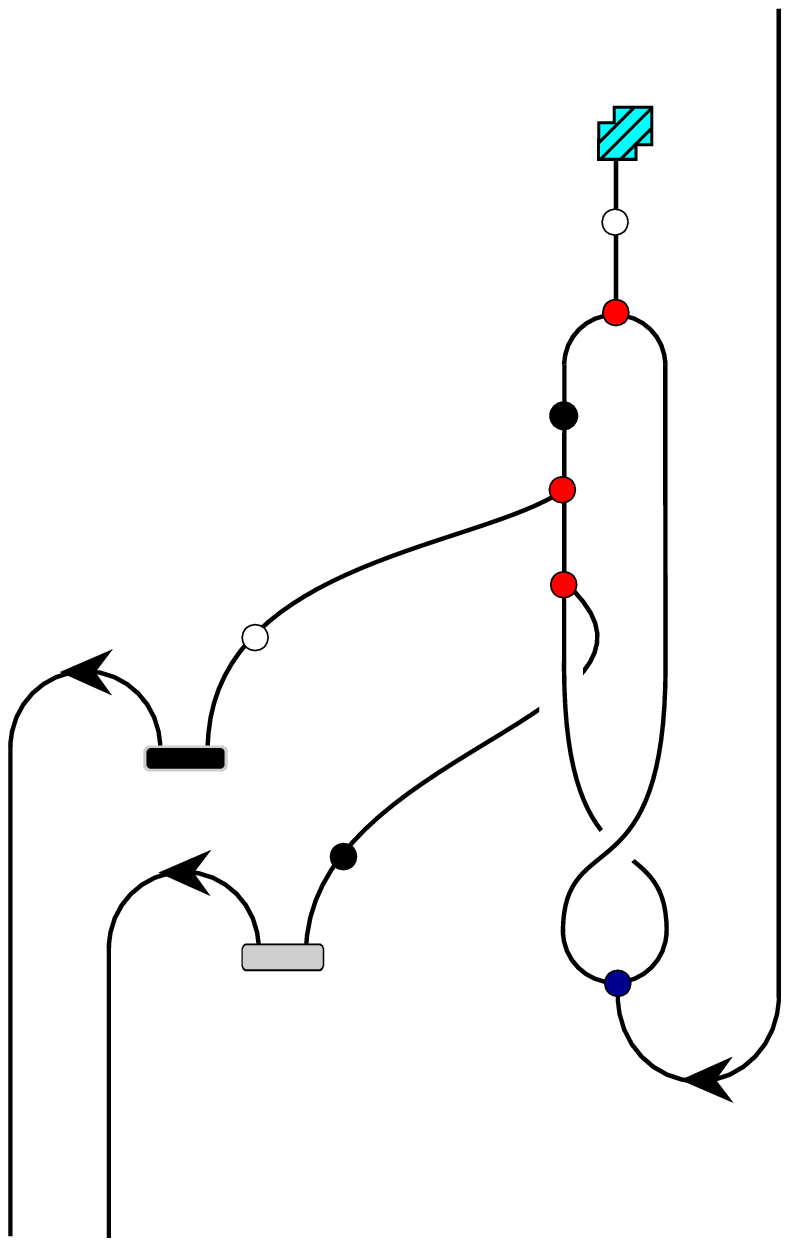}
   \put(-5.5,-8.5) {\sse$ \Hss $}
   \put(7.4,-8.5){\sse$ \Hss $}
   \put(25.2,49) {\sse$ Q^{-1} $}
   \put(26.5,23) {\sse$ Q $}
   \put(72,118.7){\sse$ \lambda $}
   \put(81,139)  {\sse$ \Hss $}
   }
   \put(186,59)  {$=$}
   \put(219,0) {\Includepichtft{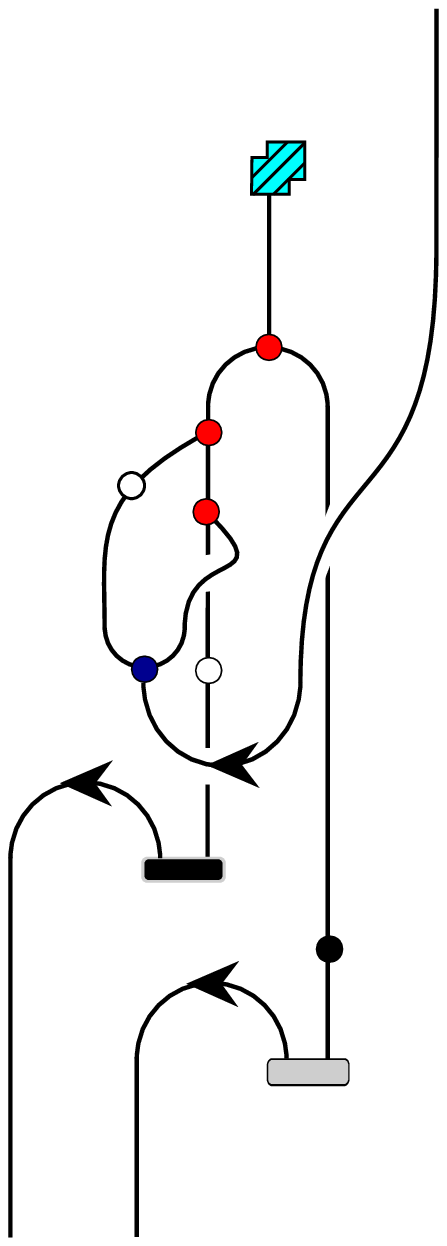}
   \put(-4.5,-8.5) {\sse$ \Hss $}
   \put(10,-8.5) {\sse$ \Hss $}
   \put(11.5,33) {\sse$ Q^{-1} $}
   \put(29.5,10.8) {\sse$ Q $}
   \put(34.1,115){\sse$ \lambda $}
   \put(44,139)  {\sse$ \Hss $}
   }
   \put(284,59)  {$=$}
   \put(317,0) {\Includepichtft{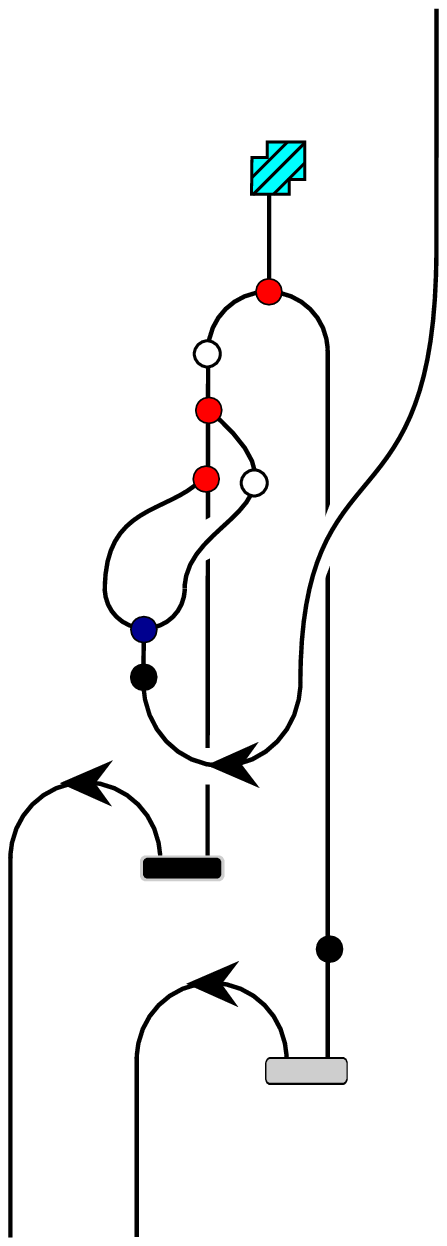}
   \put(-4.5,-8.5) {\sse$ \Hss $}
   \put(10,-8.5) {\sse$ \Hss $}
   \put(11.5,33) {\sse$ Q^{-1} $}
   \put(29.5,10.8) {\sse$ Q $}
   \put(34.1,115){\sse$ \lambda $}
   \put(44,139)  {\sse$ \Hss $}
   } }
Here in the first step it is used that $\apo^2\eq\ad t$
and that $\lambda \cir m \cir (g \oti \id_H) \eq \lambda \cir \apo$ (which is
a \emph{left} cointegral), while the second equality follows by the fact that
the antipode is an anti-algebra morphism and by associativity of $m$.
We can now use the fact (see \erf{fQinv_intertw}) that the morphism 
$f_{Q^{-1}_{}}$ intertwines the coadjoint and adjoint actions; we then have
   \eqpic{Sinv-5} {360} {49} {
   \put(-3,49)    {$ \wiha {\Delta\bico}~= $}
   \put(73,0)  {\Includepichtft{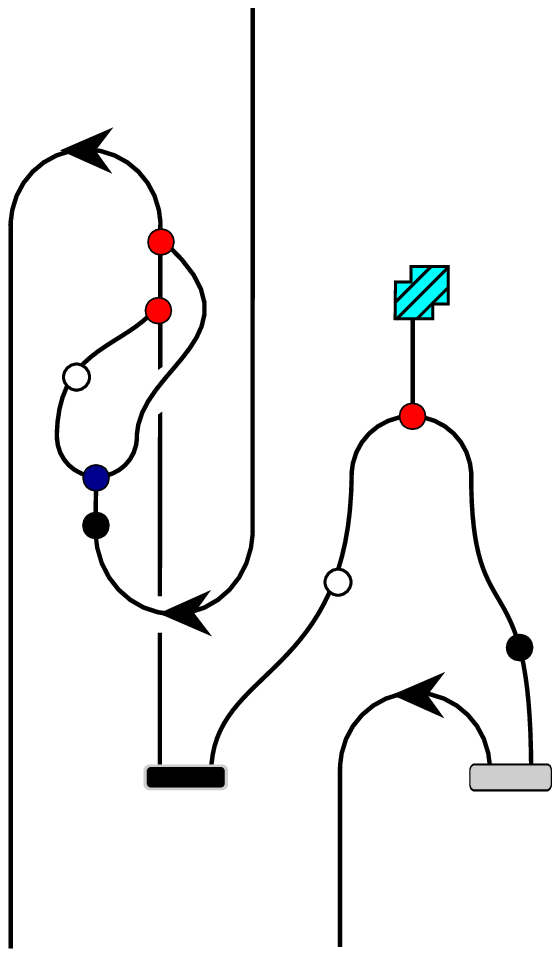}
   \put(-5,-8.5) {\sse$ \Hss $}
   \put(32,-8.5) {\sse$ \Hss $}
   \put(12,11)   {\sse$ Q^{-1} $}
   \put(51.6,11) {\sse$ Q $}
   \put(49.7,69.5) {\sse$ \lambda $}
   \put(23,107)  {\sse$ \Hss $}
   }
   \put(153,49)  {$=$}
   \put(186,0) {\Includepichtft{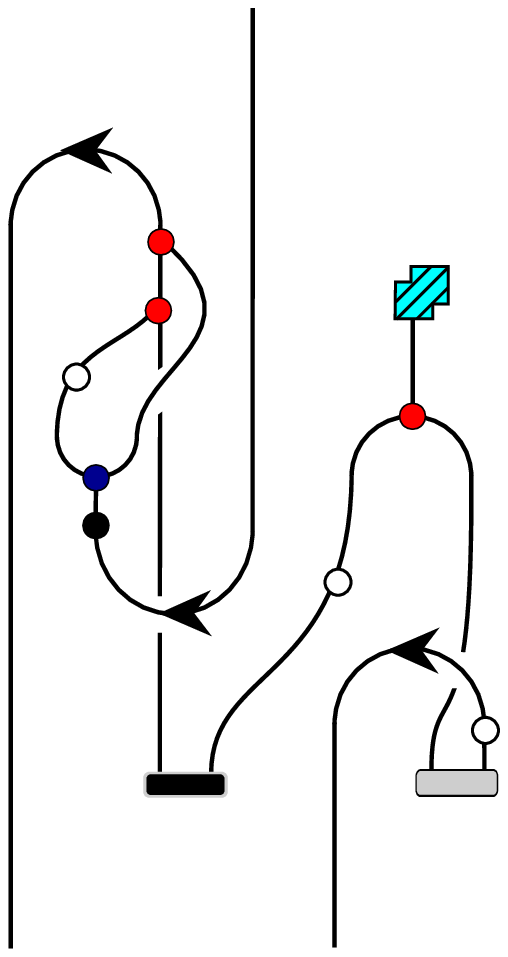}
   \put(-5,-8.5) {\sse$ \Hss $}
   \put(31,-8.5) {\sse$ \Hss $}
   \put(12,11)   {\sse$ Q^{-1}$}
   \put(45.6,11) {\sse$ Q $}
   \put(48.9,69.2) {\sse$ \lambda $}
   \put(23,107)  {\sse$ \Hss $}
   }
   \put(262,49)  {$=$}
       \put(295,0) {\Includepichtft{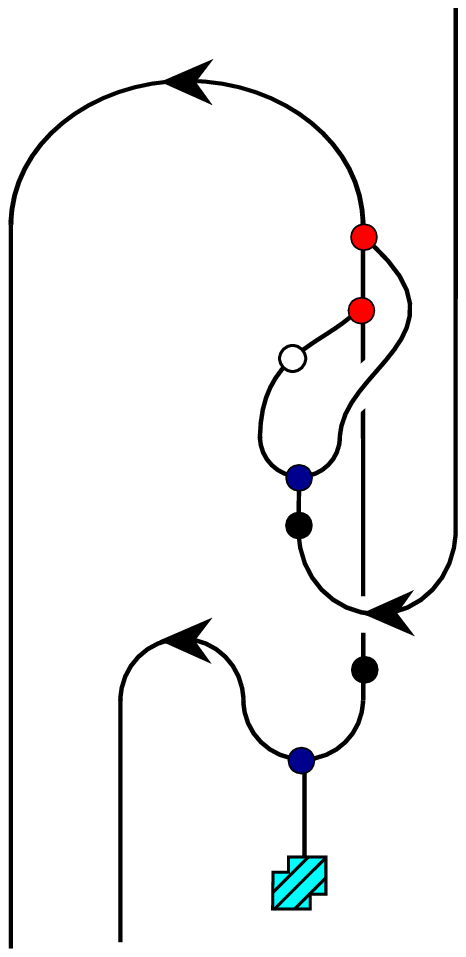}
   \put(-5,-8.5) {\sse$ \Hss $}
   \put(8.2,-8.5){\sse$ \Hss $}
   \put(36,3.5)  {\sse$ \Lambda $}
   \put(46,107)  {\sse$ \Hss $}
   } }
where the last equality uses the first of the identities \erf{fQS_Psi}
together with the fact that the antipode is an anti-coalgebra morphism
and that $\apo\cir\Lambda \eq \Lambda$.

\begin{lemma}\label{lemmaT}
The morphism $\wiha{\Delta\bico}$ is $T$-invariant, i.e.\ satisfies
$\wiha{\Delta\bico} \circ T\bicoa \eq \wiha{\Delta\bico}$.
\end{lemma}

\begin{proof}
Invoking the expressions for $T\bicoa$ given in \erf{rho_Haa_S,T}
and for $\wiha{\Delta\bico}$ given on the right hand side of \erf{Sinv-5}, and
using the centrality of the ribbon element $v\iN H$, we have
   \eqpic{T-inv} {190} {47} {
   \put(10,52)   {$ \wiha{\Delta\bico} \circ T\bicoa ~= $}
   \put(118,0) {\Includepichtft{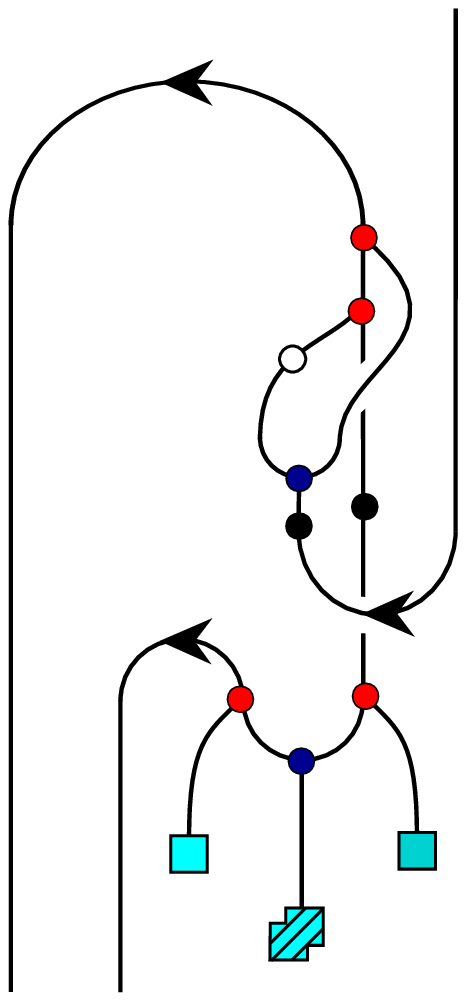}
   \put(-5,-8.5) {\sse$ \Hss $}
   \put(8.5,-8.5){\sse$ \Hss $}
   \put(23.5,13.5) {\sse$ v $}
   \put(49,13.5) {\sse$ v^{-1} $}
   \put(36,2.5)  {\sse$ \Lambda $}
   \put(45.2,112){\sse$ \Hss $}
   } }
Recalling now the identity \erf{Hopf_Frob_trick2}, the central elements $v$ and
$v^{-1}$ cancel each other, hence \erf{T-inv} equals $\wiha{\Delta\bico}$.
\end{proof}

\begin{lemma}\label{lemmaS}
The morphism $\wiha{\Delta\bico}$ is $S$-invariant, i.e.\ satisfies
$\wiha{\Delta\bico} \circ S\bicoa \eq \wiha{\Delta\bico}$.
\end{lemma}

\begin{proof}
We will show that $\wiha{\Delta\bico}$ invariant under $S^{-1}$.
Applying definition \erf{rho_Haa_S,T}, we can use the identities \erf{fQS_Psi} to
obtain
   \eqpic{Sinv-6} {340} {64} {
   \put(7,63)  {$\wiha{\Delta\bico} \circ S\bicoa^{-1} ~=$}
       \put(117,0) {\Includepichtft{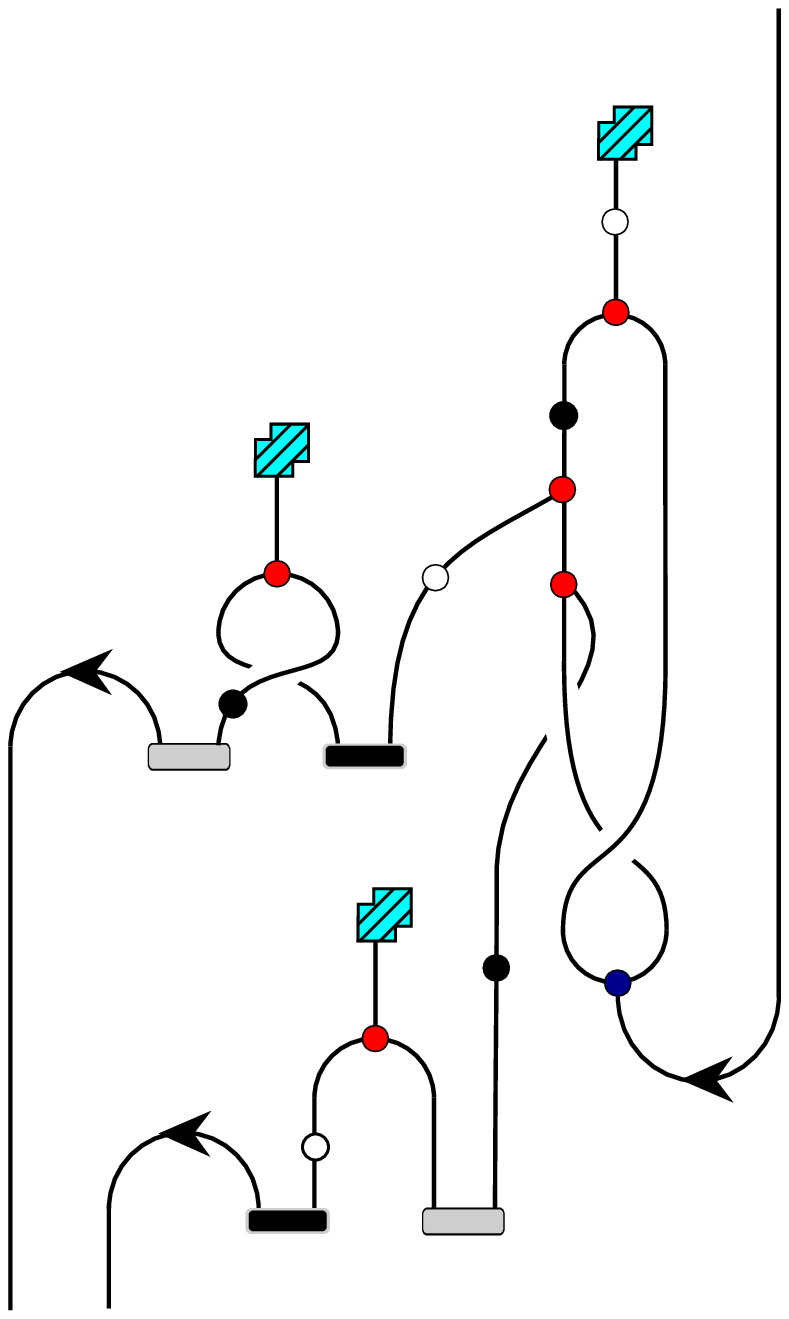}
   \put(-5.4,-8.5) {\sse$ \Hss $}
   \put(7,-8.5)  {\sse$ \Hss $}
   \put(16.3,52.8) {\sse$ Q $}
   \put(31.9,53.2) {\sse$ Q^{-1} $}
   \put(23.7,2.2){\sse$ Q^{-1} $}
   \put(47.1,2.3){\sse$ Q $}
   \put(32.3,39) {\sse$ \lambda $}
   \put(34.5,91) {\sse$ \lambda $}
   \put(72,126)  {\sse$ \lambda $}
   \put(81,147)  {\sse$ \Hss $}
   }
   \put(225,63)  {$=$}
       \put(257,0) {\Includepichtft{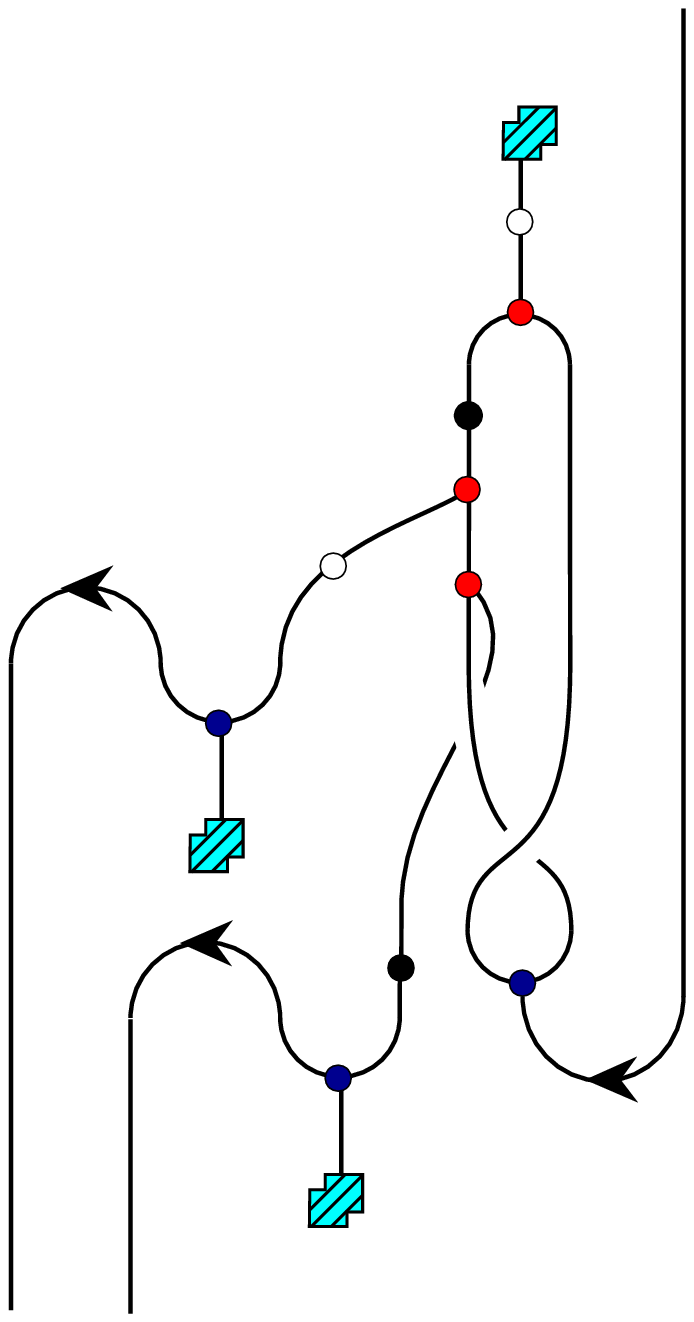}
   \put(-5,-8.5) {\sse$ \Hss $}
   \put(9,-8.5)  {\sse$ \Hss $}
   \put(26.5,48) {\sse$ \Lambda $}
   \put(40,9)    {\sse$ \Lambda $}
   \put(60.4,126){\sse$ \lambda $}
   \put(70.5,147){\sse$ \Hss $}
   } }
Further we note that owing to the identities \erf{lambda-Oapoo},
$\apo^{-2}\cir\Lambda \eq \Lambda$ and \erf{fmap_inv} we can write
   \eqpic{Sinv-7} {360} {46} {
   \put(0,0) {\Includepichtft{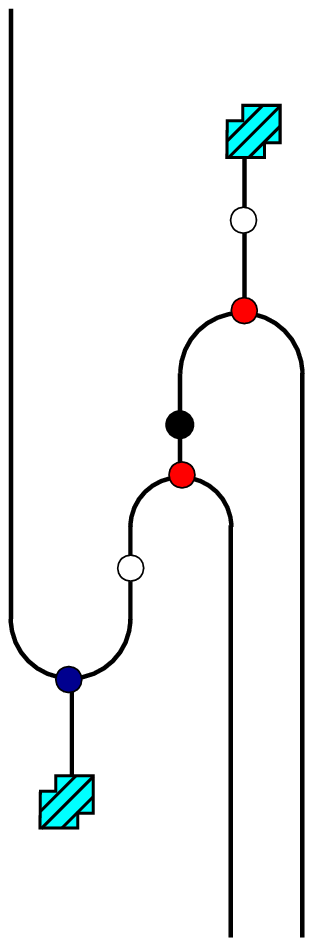}
   \put(18,-8)   {\sse$ \Hss $}
   \put(30,-8)   {\sse$ \Hss $}
   \put(11,12)   {\sse$ \Lambda $}
   \put(31.5,86) {\sse$ \lambda $}
   \put(-3,106)  {\sse$ \Hss $}
   }
   \put(53,49)   {$ = $}
   \put(85,0)  {\Includepichtft{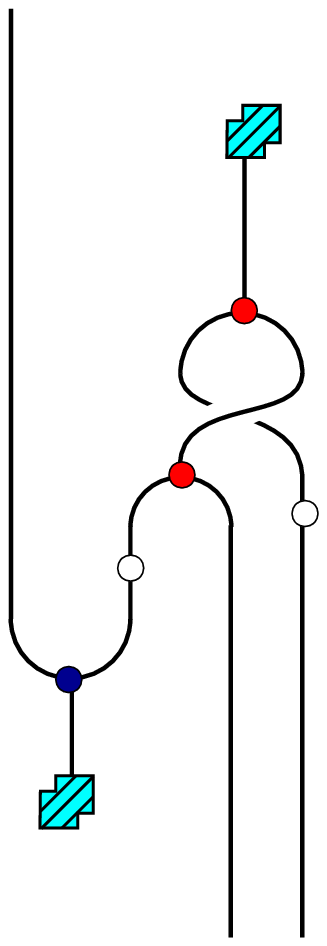}
   \put(18,-8.5) {\sse$ \Hss $}
   \put(30,-8.5) {\sse$ \Hss $}
   \put(11,12)   {\sse$ \Lambda $}
   \put(31.5,86) {\sse$ \lambda $}
   \put(-3,106)  {\sse$ \Hss $}
   }
   \put(138,49)  {$=$}
   \put(170,0) {\Includepichtft{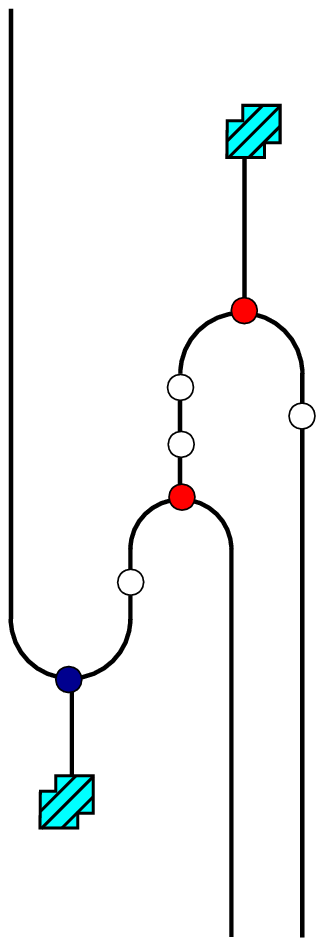}
   \put(18,-8.5) {\sse$ \Hss $}
   \put(30,-8.5) {\sse$ \Hss $}
   \put(11,12)   {\sse$ \Lambda $}
   \put(31.5,86) {\sse$ \lambda $}
   \put(-3,106)  {\sse$ \Hss $}
   }
   \put(223,49)  {$=$}
   \put(255,0) {\Includepichtft{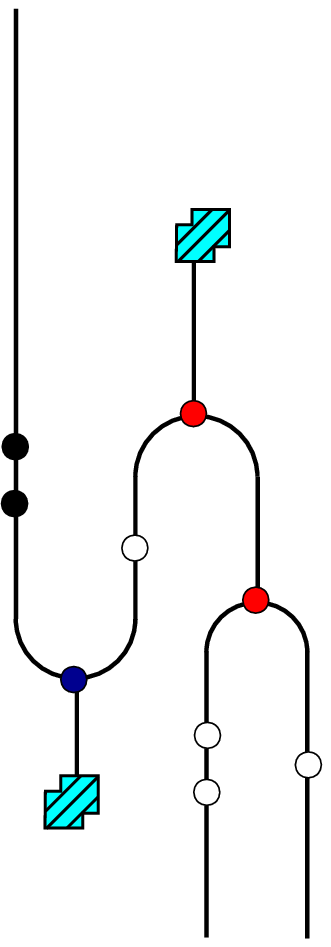}
   \put(17.5,-8.5) {\sse$ \Hss $}
   \put(30,-8.5) {\sse$ \Hss $}
   \put(11.4,11) {\sse$ \Lambda $}
   \put(26.5,74) {\sse$ \lambda $}
   \put(-1.4,106){\sse$ \Hss $}
   }
   \put(307,49)  {$ = $}
   \put(338,16) {\Includepichtft{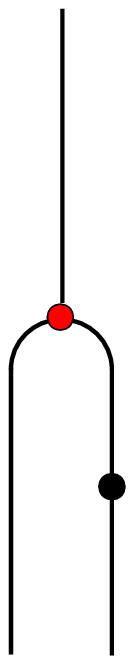}
   \put(-5,-8.5) {\sse$ \Hss $}
   \put(8,-8)    {\sse$ \Hss $}
   \put(3,75)    {\sse$ \Hss $}
   } }
It follows that
   \eqpic{Sinv-8} {340} {63} {
   \put(4,67)  {$ \wiha{\Delta\bico} \circ S\bicoa^{-1} ~= $}
   \put(113,0) {\Includepichtft{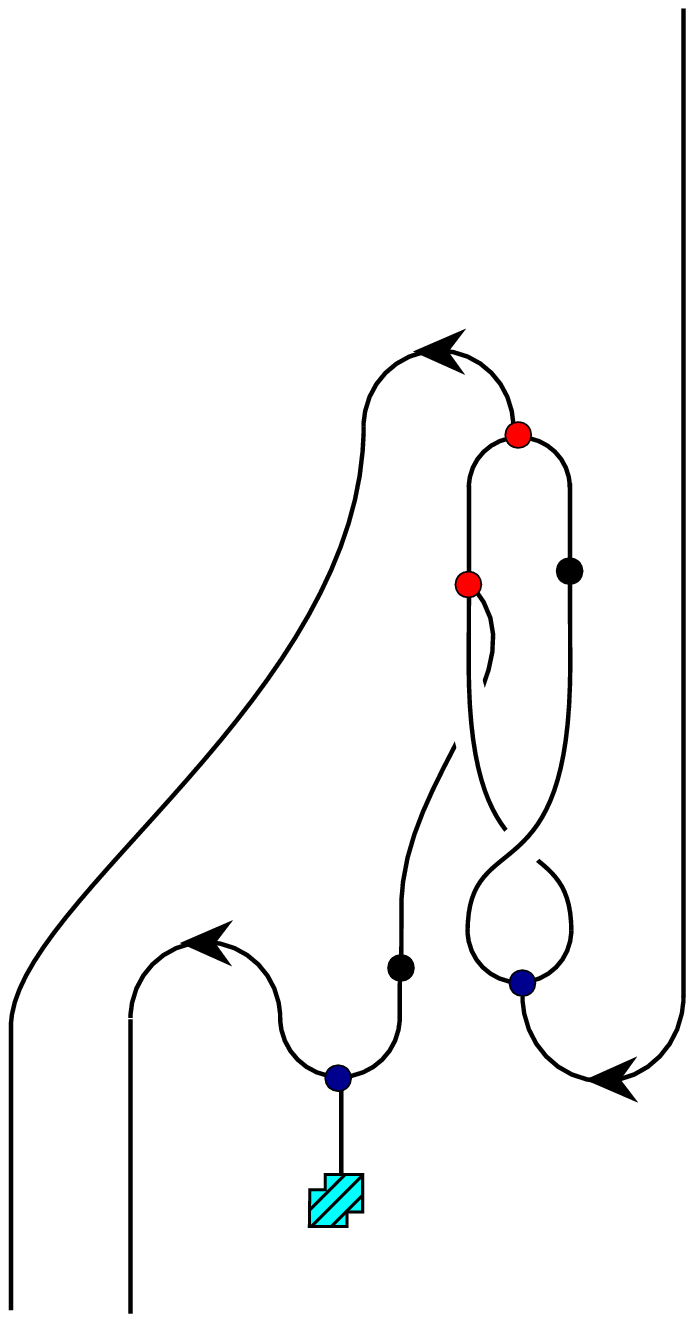}
   \put(-5,-8.5) {\sse$ \Hss $}
   \put(9,-8.5)  {\sse$ \Hss $}
   \put(40,9)    {\sse$ \Lambda $}
   \put(70.5,147){\sse$ \Hss $}
   }
   \put(220,67)  {$ = $}
   \put(251,0)   {\Includepichtft{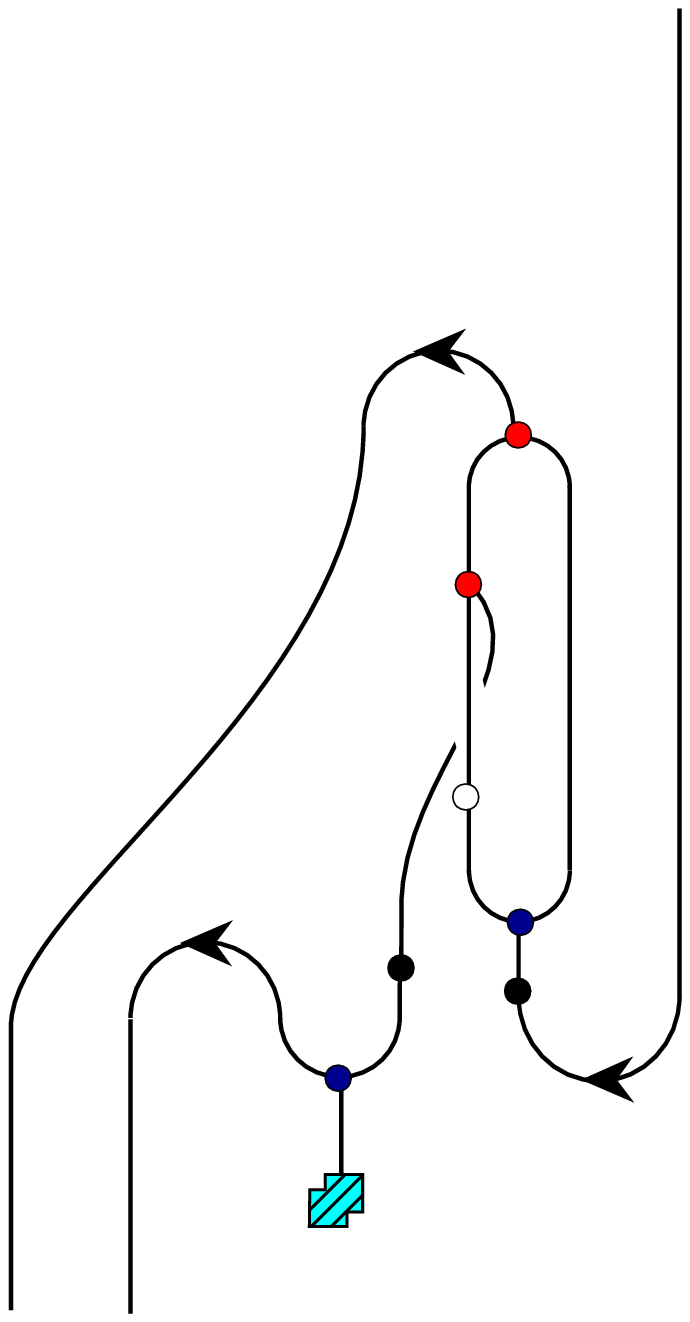}
   \put(-5,-8.5) {\sse$ \Hss $}
   \put(9,-8.5)  {\sse$ \Hss $}
   \put(40,9)    {\sse$ \Lambda $}
   \put(70.5,147){\sse$ \Hss $}
   } }
This coincides with the right hand side of \erf{Sinv-5} and thus with
$\wiha{\Delta\bico}$.
\end{proof}

To neatly summarize the results above we state

\begin{defi}\label{def:Fmodinv}
A coalgebra $(C,\Delta_C,\eps_C)$ in \HBimod\ is called \emph{modular invariant} iff
the morphism $\wiha{\Delta_C} \iN \HomHH(\HK,C)$ is \slze-invariant.
\end{defi}

Thus what we have shown can be rephrased as

\begin{cor}
The Frobenius algebra $\Hb\iN\HBimod$ introduced in \erf{def-Hb-Frobalgebra}
is modular invariant.
\end{cor}

\begin{proof}
Invariance of $\wiha{\Delta_C}$ under the action of the generators $T$ and $S$
of \slze\ has been shown in Lemma \ref{lemmaT} and \ref{lemmaS}. Invariance
under the action of the generator $D$ of \slze\ follows immediately from
the fact that \Hb\ has trivial twist.
\end{proof}

\begin{rem}\label{Cmatrix}
(i)\,~ The morphism \,~$\wiha{\Delta\bico}$ is non-zero. Indeed one can show that 
$\eps\bico \cir \wiha{\Delta\bico}$ can be expanded as a bilinear form in
the simple characters of Lyubashenko's Hopf algebra $\H\iN\HMod$, with coefficients 
given by the Cartan matrix of $H$, i.e.\ by the matrix that describes the
composition series of indecomposable projectives, and these \H-characters are non-zero. 
In conformal field theory terms, this means that the Cartan matrix -- which is a 
quantity directly associated to the category -- is the right substitute of the 
charge conjugation matrix in the non-semisimple case. However, establishing this 
result requires methods different from those on which our focus is in this paper.
\nxl1
(ii)\,~In the same way as the \pqt\ \erf{def_wiha} associates a morphism in $\HomC(\HK,Y)$
to a morphism in $\HomC(X,Y\oti X)$, one may introduce another \pqt\ $\tauQ$ that
maps morphisms in $\HomC(X\oti Y,X)$ linearly to morphism in $\HomC(\HK\oti Y,\one)$,
say as
   \eqpic{pic128} {155} {33} {
   \put(0,37)    {$ \tauQ(f) ~:= $}
   \put(87,0)  { \Includepichtft{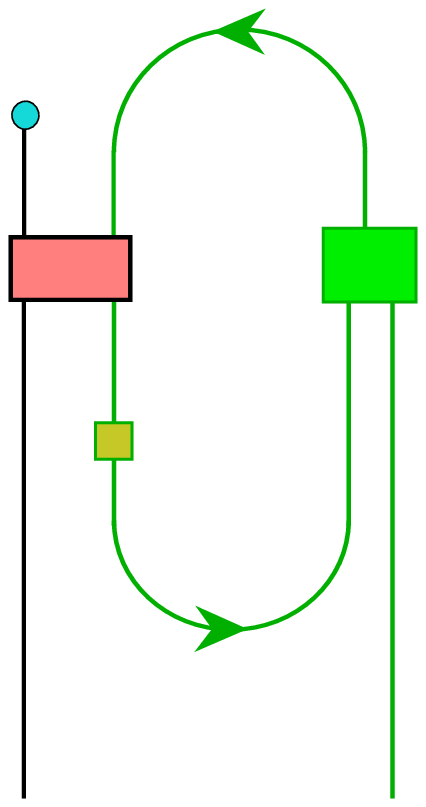}
   \put(-2,-8.5) {\sse$ \HK $}
   \put(39,-8.5) {\sse$ Y $}
   \put(-23,58)  {\sse$ \Qq_{\HK,Y} $}
   \put(15.3,38) {\sse$ \pi_X^{-1} $}
   \put(37.7,57.4) {\sse$ f $}
   \put(41,69.4) {\sse$ X $}
   } }
It is not difficult to check that the morphism $\tauQ(m\bico)$ obtained this way
from the product of the Frobenius algebra \Hb\ is modular invariant in the sense that
$\tauQ(m\bico) \cir \rho_{\HK\otimes F,\one}(\gamma) \eq \tauQ(m\bico)$
for all $\gamma\iN\slze$.
Indeed, $\tauQ(m\bico)$ is related to $\wiha{\Delta\bico}$ by
   \eqpic{tauQm} {150} {43} {
   \put(0,45)    {$ \tauQ(m_{\bico}) ~= $}
   \put(77,0)  {\Includepichtft{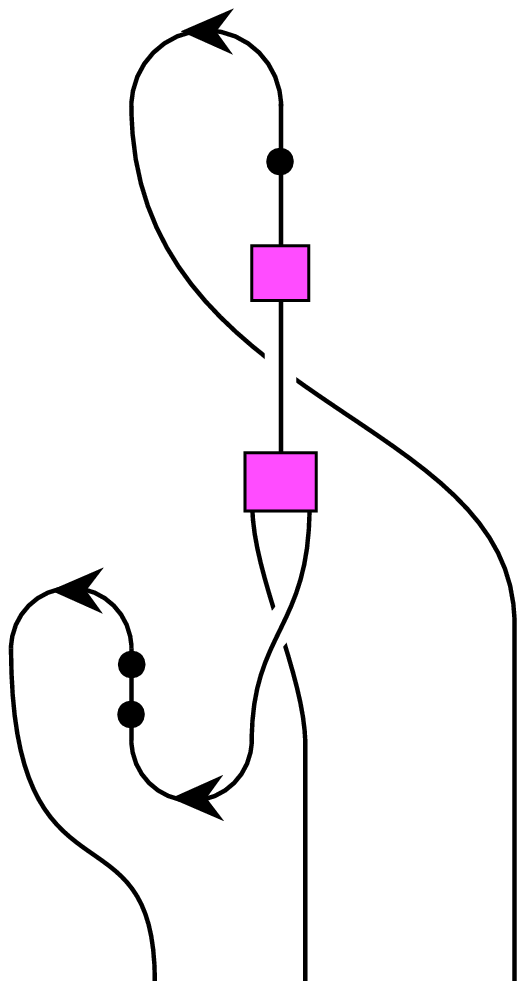}
   \put(11,-8.5) {\sse$ \Hss $}
   \put(27.5,-8.5){\sse$ \Hss $}
   \put(51,-8.5) {\sse$ \Hss $}
   \put(-4.5,53) {\sse$ \wiha{\Delta_{\bico}} $}
   \put(35,75.7) {\sse$ \fmap^{-1} $}
   } }
with $\fmap$ the Frobenius map, and as a consequence (using that
$(\apo^{-2})^\wee\oti\idHs$ commutes with the action of \slze\ and that
$\tau_{\Hss,\Hss}\cir S\bicoa^{} \eq S\bicoa^{-1}$
and $\tau_{\Hss,\Hss}\cir T\bicoa^{} \eq T\bicoa^{-1}$)
modular invariance of $\tauQ(m\bico)$ is equivalent
to modular invariance of $\wiha{\Delta\bico}$. Accordingly, from the perspective
of \HBimod\ alone we could as well have referred to algebras rather than coalgebras
in Def.\ \ref{def:Fmodinv}. Indeed, this is the option that was chosen for
the semisimple case in \cite[Def.\,3.1(ii)]{koRu2}. Our preference for coalgebras
derives from the fact that, as described in Appendix \ref{app:cft}, the morphism
space $\HomC(\HK,\Hb)$ plays a more direct role than $\HomC(\HK\oti\Hb,\one)$ in
the motivating context of modular functors and conformal field theory.
\nxl1
(iii)\,~More generally, for any non-negative integers $m$ and $n$, the mapping class
group \slzmn\ of the $(m{+}n)$-punctured torus acts on the morphism space
$\HomC(K,\Hb^{\otimes m+n})$ \cite[Sect.\,4.3]{lyub6} and thus, using
the canonical isomorphism of the Frobenius algebra \Hb\ with its dual,
also on $\HomC(K\oti \Hb^{\otimes m}, \Hb^{\otimes n})$. By suitably composing
the morphisms \erf{Sinv-3} and \erf{pic128} with product and coproduct morphisms
of \Hb, one easily constructs a morphism in
$\HomC(K\oti \Hb^{\otimes m}, \Hb^{\otimes n})$ that, owing to commutativity
and cocommutativity of \Hb\ and to the triviality of its twist,
is invariant under this action of \slzmn.
\end{rem}

         \newpage


\section{The case of non-trivial Hopf algebra automorphisms}\label{sec-auto}

A \emph{Hopf algebra automorphism} of a Hopf algebra $H$ is a linear map from
$H$ to $H$ that is both an algebra and a coalgebra automorphism and commutes with
the antipode. For $H$ a ribbon Hopf algebra with R-matrix $R$ and ribbon
element $v$, an automorphism $\omega$ of $H$ is said to be a \emph{ribbon Hopf
algebra automorphism} iff $(\omega\oti\omega)(R) \eq R$ and $\omega(v) \eq v$.
For any $H$-bi\-mo\-dule $(X,\rho,\ohr)$ and any pair of algebra automorphisms
$\omega,\omega'$ of $H$ there is a corresponding
$(\omega,\omega')$-\emph{twisted bimodule} ${}^{\omega\!}X^{\omega'}
\eq (X,\rho\cir(\omega\oti\id_X),\ohr\cir(\id_X\oti\omega'))$. If $\omega$ and
$\omega'$ are \emph{Hopf} algebra automorphisms, then the twisting is compatible
with the monoidal structure of \HBimod, and if they are even \emph{ribbon} Hopf
algebra automorphisms, then it is compatible with the ribbon structure of \HBimod.

In this section we observe that to any \findim\ factorizable ribbon Hopf algebra
$H$ and any ribbon Hopf algebra automorphism of $H$ there is again associated a
Frobenius algebra in \HBimod, which moreover shares all the properties, in particular
modular invariance, of the Frobenius algebra \Hb\ that we obtained in the previous
sections. The arguments needed to establish this result are simple modifications
of those used previously. Accordingly we will be quite brief.

\begin{prop}
{\rm (i)}\,~For $H$ a \findim\ factorizable ribbon Hopf algebra over \ko\ and $\omega$
a Hopf algebra automorphism of $H$, the bimodule $\Hbo \,{:=}\,\, {}^{\idsm_H\!}
(\Hb)^{\omega}_{}$ carries the structure of a Frobenius algebra.
The structure morphisms of \Hbo\ as a Frobenius algebra are given by the formulas
\erf{def-Hb-Frobalgebra} (thus as linear maps they are the same as for
$\Hb \,{\equiv}\, \Hbi$).
\nxl1
{\rm (ii)}\,~\Hbo\ is commutative and symmetric, and it is special iff $H$ is semisimple.
\nxl1
{\rm (iii)}\,~If $\omega$ is a ribbon Hopf algebra automorphism, then \Hbo\ is
modular invariant.
\end{prop}

\begin{proof}
The proofs of all statements are completely parallel to those in the case
$\omega \eq \id_H$. The only difference is that the various morphisms one
deals with, albeit coinciding as linear maps with those encountered before,
are now morphisms between different $H$-bimodules than previously.
That they do intertwine the relevant bimodule structures follows by
combining the simple facts that (since $\omega$ is compatible with the ribbon
structure of \HBimod) $(\Hb\oti\Hb)^\omega \eq \Hbo \oti \Hbo$ as a bimodule and
that a linear map $f \iN \Hom(X,Y)$ for $X,Y\iN\HBimod$ lies in the subspace
$\HomHH(X,Y)$ iff it lies in the subspace $\HomHH(X^\omega,Y^\omega)$.
\end{proof}

Furthermore, again the Frobenius algebra \Hbo\ is canonically associated with
\HBimod\ as an abstract category. Indeed, analogously as in Proposition
\ref{prop-F=coend}, one sees that \Hbo\ can be constructed as a coend, namely
the one of the functor $\Fbxo \colon \HMod\op \Times \HMod \To \HBimod$
that acts on morphisms as $f\Times g\mapsto f^\vee\otik g$ and on objects
by mapping $(X,\rho_X) \Times (Y,\rho_Y)$ to
   \be
   \big( X^\wee\otik Y\,,\, [\rho_{X^\vee_{\phantom:}}^{}
   \cir (\omega^{-1}\oti\idXs)] \oti\id_Y\,,\,
   \idXs \oti(\rho_Y\cir \tau_{Y,H}^{} \cir (\id_Y\oti\apo^{-1})) \big)
   \ee
(or, in other words, $\Fbxo \eq (?^{\omega^{-1}\!}_{} {\times}\, \Id) \cir \Fbx$
with the functor $\Fbx$ (whose coend is \Hb) given by \erf{Fbx}):

\begin{prop}
The $H$-bimodule \Hbo\ together with the dinatural family of morphisms
   \be
   \iHbo_X := (\omega^{-1})^* \circ \iHb_X \,,
   \ee
with $\iHb_X$ as defined in \erf{def-iHbX}, is the coend of the functor $\Fbxo$.
\end{prop}

\begin{proof}
Again the proof is parallel to the one for the case $\omega\eq\id_H$, the
difference being that the automorphisms $\omega^{\pm1}$ need to be inserted at
appropriate places. For instance, the equalities
   \eqpic{biregIToml} {420} {46} { \put(-86,4){
   \put(121,0)  {\Includepichtft{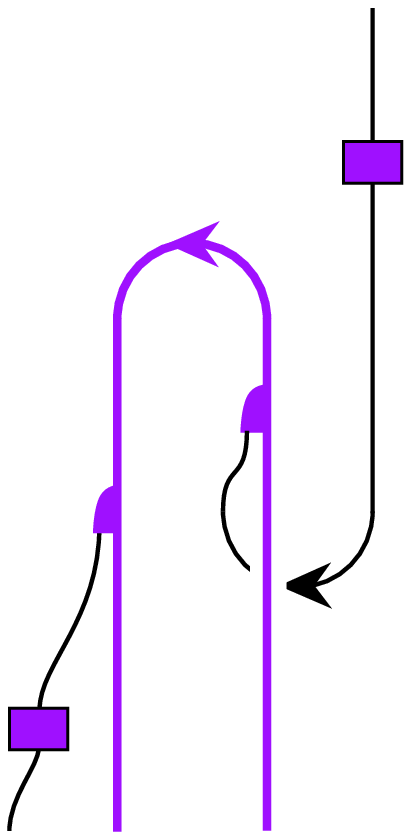}
   \put(-5,-8.5) {\sse$ H $}
   \put(-2.1,38.7) {\sse$ \rho_{\!X^{\!\vee}_{}}^{} $}
   \put(16,48.5) {\sse$ \rho_{\!X}^{} $}
   \put(6,-8.5)  {\sse$\Xs $}
   \put(24,-8.5) {\sse$ X $}
   \put(36.5,95) {\sse$ \Hss $}
   \put(-10.5,12){\sse$\omegi$}
   \put(45.3,72) {\sse$\omegis$}
   }
   \put(184,44)  {$=$}
   \put(208,0)   {\Includepichtft{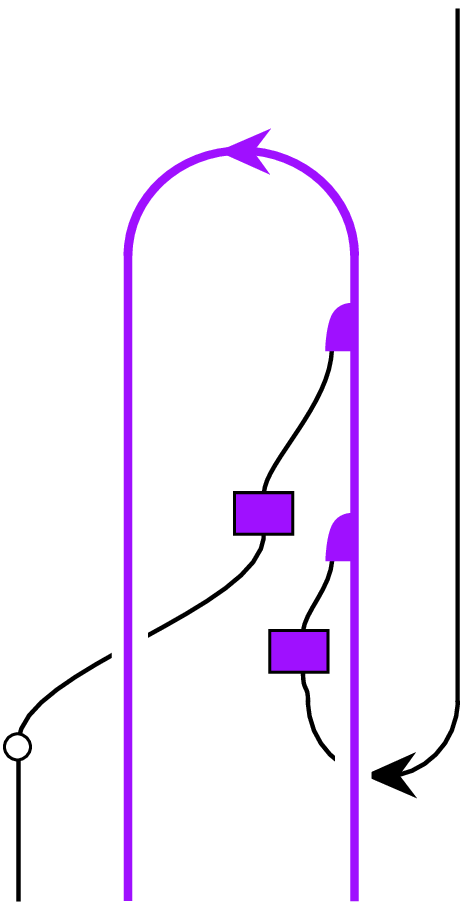}
   \put(-4,-8.5) {\sse$ H $}
   \put(8,-8.5)  {\sse$ \Xs $}
   \put(34,-8.5) {\sse$ X $}
   \put(47,101)  {\sse$ \Hss $}
   \put(16,44.5) {\sse$\omegi$}
   \put(19.5,19) {\sse$\omegi$}
   }
   \put(279,44)   {$=$}
   \put(297,0) { \Includepichtft{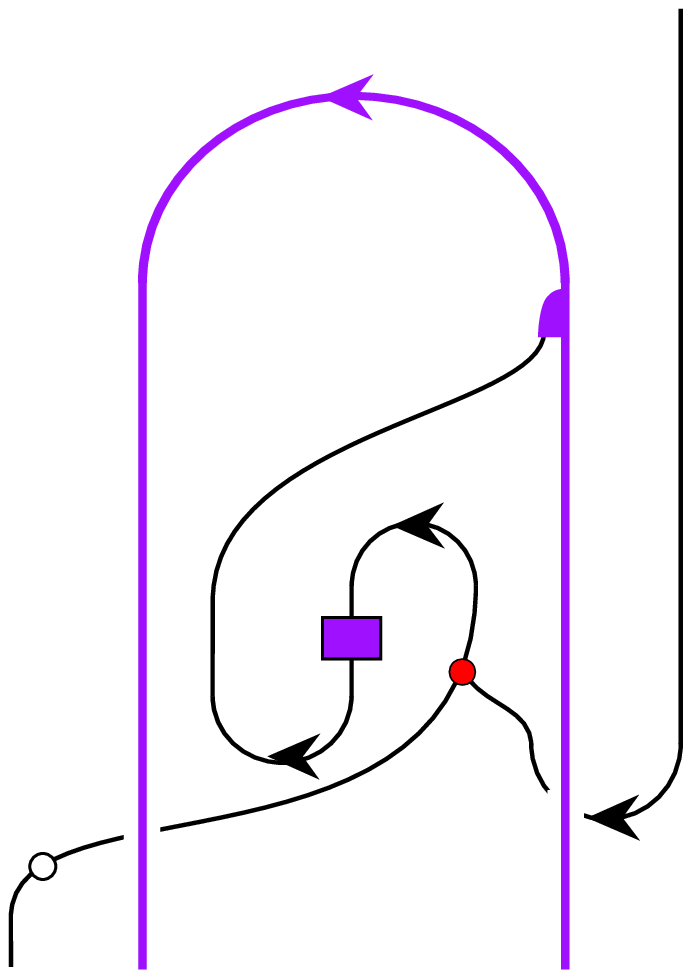}
   \put(-4,-8.5) {\sse$ H $}
   \put(10.6,-8.5) {\sse$ \Xs $}
   \put(51,-8.5) {\sse$ X $}
   \put(71,109)  {\sse$ \Hss $}
   \put(26,29)   {\sse\begin{turn}{90}$\omegis$\end{turn}}
   }
   \put(395,44)   {$=$}
   \put(420,0) { \Includepichtft{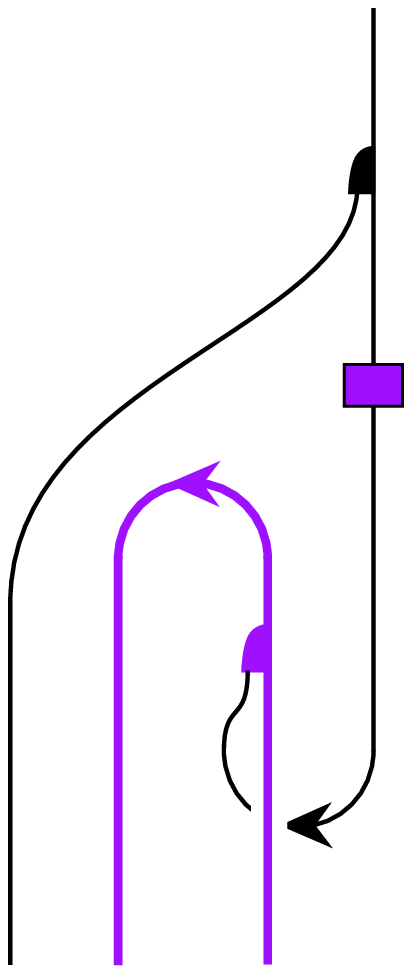}
   \put(-4,-8)   {\sse$H $}
   \put(8,-8.5)  {\sse$\Xs$}
   \put(16.3,36.5) {\sse$ \rho_{\!X}^{} $}
   \put(24,-8.5) {\sse$X$}
   \put(37,109)  {\sse$ \Hss $}
   \put(42.5,86.7) {\sse$ \brho $}
   \put(45.3,62) {\sse$\omegis$}
   } } }
and
   \eqpic{biregITomr} {340} {42} {
   \put(0,0)   {\Includepichtft{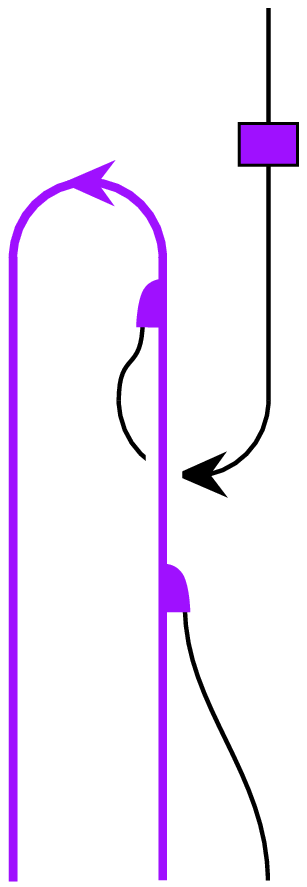}
   \put(-5,-8.5) {\sse$\Xs $}
   \put(13,-8.5) {\sse$X$}
   \put(25,-8.5) {\sse$H$}
   \put(25,100)  {\sse$ \Hss $}
   \put(33.9,80) {\sse$\omegis$}
   }
   \put(52,38)   {$=$}
   \put(85,0) { \Includepichtft{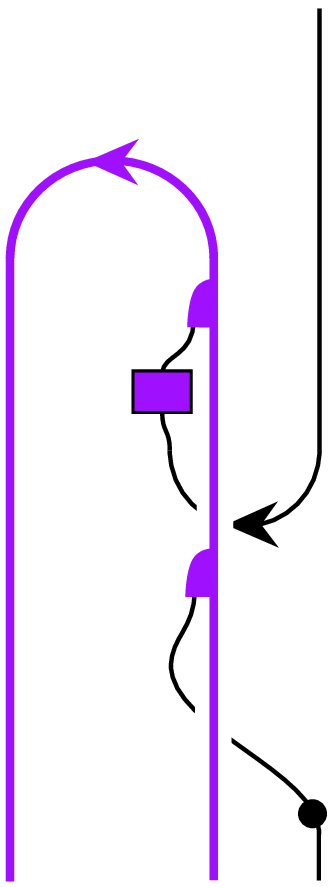}
   \put(-5,-8.5) {\sse$\Xs $}
   \put(18,-8.5) {\sse$X$}
   \put(30,-8.5) {\sse$H$}
   \put(31,100)  {\sse$ \Hss $}
   \put(3.7,55)  {\sse$\omegi$}
   }
   \put(149,38)  {$=$}
   \put(188,0)   {\Includepichtft{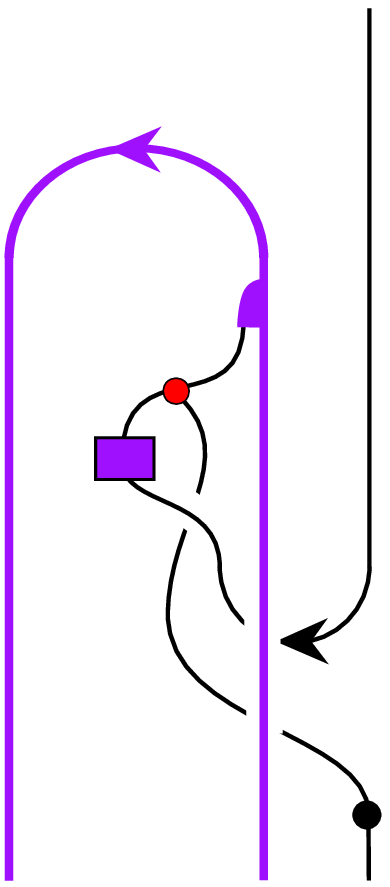}
   \put(-5,-8.5) {\sse$\Xs $}
   \put(24,-8.5) {\sse$X$}
   \put(36,-8.5) {\sse$H$}
   \put(36,100)  {\sse$ \Hss $}
   \put(3.5,37.1){\sse$\omegi$}
   }
   \put(252,38)   {$=$}
   \put(281,0) { \Includepichtft{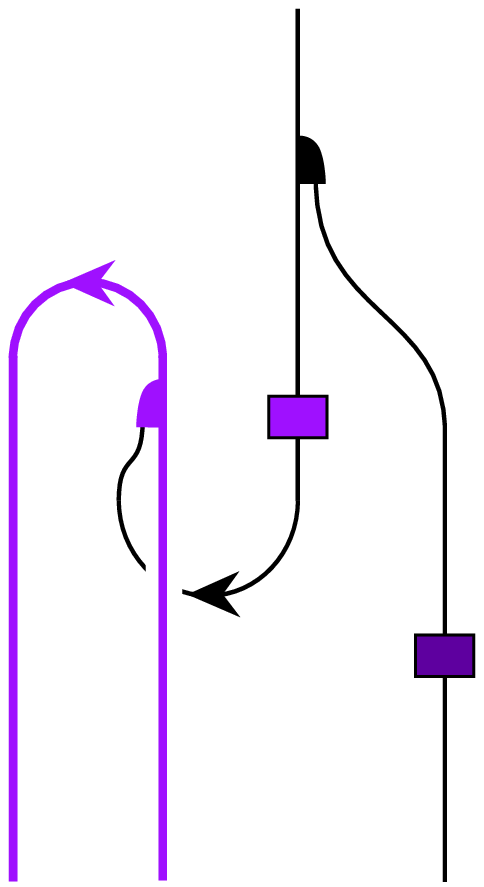}
   \put(-5,-8.5) {\sse$\Xs $}
   \put(13,-8.5) {\sse$ X $}
   \put(44,-8.5) {\sse$ H $}
   \put(35,79.1) {\sse$ \bohr $}
   \put(28,100)  {\sse$ \Hss $}
   \put(34.4,45) {\sse\sse\begin{turn}{90}$\omegis$\end{turn}}
   \put(53.2,23.5) {\sse$\omega$}
   } }
which generalize the relations \erf{biregITl} and \erf{biregITr}, respectively,
demonstrate that the linear maps $\iHbo_X \iN \Hom(\Xs\otik X,\Hs)$ are indeed
bimodule morphisms in $\HomHH(\Fbxo(X,X),\Hbo)$.
\end{proof}

\begin{rem}\label{rem:vaco}
(i)\,~When discussing twists of \Hb\ we can restrict to the case that only the,
say, right module structure is twisted, because the bimodule ${}^{\omega\!}
H^{\omega'}$ is isomorphic to ${}^{\idsm_H}\!H^{\omega^{-1}\circ\omega'}\!$.
\nxl2
(ii)\,~It follows from the automorphism property of $\omega$ that together with
$\Lambda$ also $\omega(\Lambda)$ is a non-zero two-sided integral of $H$. As a
consequence, just like in the case $\omega\eq\id_H$ considered in Remark
\ref{rem:vac}, the counit $\eps\bico$
of \Hbo\ is uniquely determined up to a non-zero scalar.
\end{rem}

      \newpage
\appendix

\section{Coend constructions}\label{app:A}

\subsection{The \bicom\ as a coend}\label{bicom-coend}

A \emph{dinatural transformation} $F\,{\Rightarrow}\,B$ from a functor
$F\colon\, \CopC\To\D$, to an object $B\iN\D$ is a family of morphisms $\varphi
\eq \{ \varphi_X\colon F(X,X)\To B \}_{\!X\in\C}^{}$ such that the diagram
   \bee4010{
   \xymatrix @R+8pt{
   F(Y,X)\ \ar^{F(\idsm_Y,f)}[rr]\ar_{F(f,\idsm_X)}[d]&&\ F(Y,Y)\ar^{\varphi_Y^{}}[d]
   \\ F(X,X)\ \ar^{\varphi_X^{}}[rr] && \, B
   } }
commutes for all $f\iN\Hom(X,Y)$.
For instance, the family $\{d_X\}$ of evaluation morphisms of a rigid monoidal
category \C\ forms a dinatural transformation from the functor that acts as
$X\Times Y \,{\mapsto}\,X^\vee\oti Y$ to the monoidal unit $\one\iN\C$.

Dinatural transformations from a given functor $F$ to an object of $\D$
form a category, with the morphisms from $(F\,{\Rightarrow}\,B,\varphi)$ to
$(F\,{\Rightarrow}\,B',\varphi')$ being given by morphisms $f\iN\Hom_\D(B,B')$
satisfying $f \cir \varphi_X \eq \varphi'_X$ for all $X\iN\C$. A \emph{coend} 
$(A,\iota)$ for the functor $F$ is an initial object in this category.
If the coend of $F$ exists, then it is unique up to unique isomorphism; one 
denotes it by $\Coend FX $. A morphism with domain $\Coend FX$ and codomain $Y$ 
is equivalent to a family $\{f_X\}_{\!X\in\C}^{}$ of morphisms from 
$F(X,X)$ to $Y$ such that $(Y,f)$ is a dinatural transformation.

\medskip

For $H$ a \findim\ Hopf algebra over \ko, endow the categories \HMod\ and
\HBimod\ of left $H$-modules and of $H$-bimodules, respectively, with the tensor
products \erf{otimesHMod} and \erf{def-tp} and with the dualities described at
the beginning of Section \ref{sec-dual}.
Consider the tensor product (bi)functor
   \be
   \Fbx :~~ H\Mod\op \Times H\Mod \,\to\, H\Bimod
   \labl{Fbx}
that acts on objects as
   \be
   (X,\rho_X) \Times (Y,\rho_Y) ~\stackrel\Fbx\longmapsto~
   \big( X^\wee\otik Y\,,\, \rho_{X^\vee_{}}\oti\id_Y\,,\, \idXs
   \oti(\rho_Y\cir \tau_{Y,H}^{} \cir (\id_Y\oti\apo^{-1})) \big)
   \ee
and on morphisms as $f\Times g\,{\mapsto}\, f^\vee\otik g$.
Pictorially, the action on objects is
   \eqpic{Fboxmorph} {430} {42} { \put(-9,3){
   \put(0,0)   {\Includepichtft{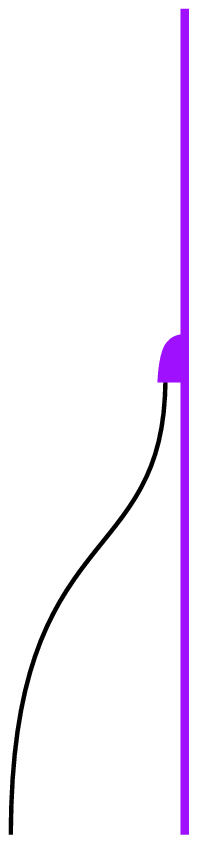}
   \put(-4,-8.8) {\sse$H $}
   \put(15.1,-8.8) {\sse$ X $}
   \put(16.2,95.5) {\sse$ X $}
   \put(22,52)   {\sse$\rho_X^{}$}
   }
   \put(48,38)   {$ \times$}
   \put(73,0) { \Includepichtft{66a}
   \put(-4,-8.8) {\sse$H $}
   \put(16.1,-8.8) {\sse$ Y $}
   \put(17.2,95.5) {\sse$ Y $}
   \put(22,52)   {\sse$\rho_Y^{}$}
   }
   \put(156,38)  {$\stackrel\Fbx\longmapsto$}
   \put(220,0)   {\Includepichtft{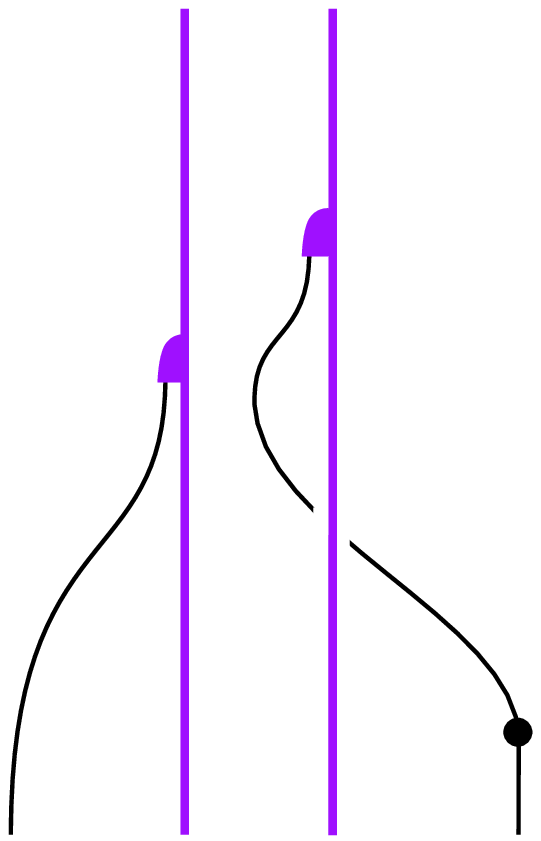}
   \put(-4,-8.8) {\sse$H $}
   \put(13,-8.8) {\sse$X^\vee$}
   \put(14,95.5) {\sse$X^\vee$}
   \put(32,-8.8) {\sse$ Y $}
   \put(33,95.5) {\sse$ Y $}
   \put(51.5,-8.8) {\sse$ H $}
   \put(59,10.1) {\sse$ \apo^{-1} $}
   }
   \put(298,38)   {$ \equiv $}
   \put(328,0) { \Includepichtft{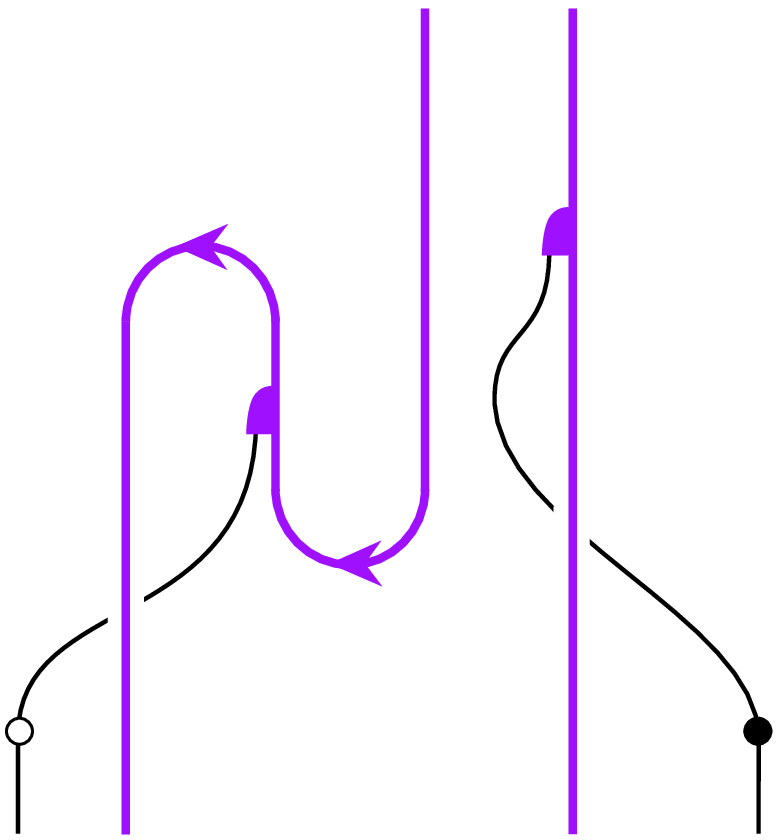}
   \put(-4.7,9.8){\sse$ \apo $}
   \put(-3,-8.8) {\sse$ H $}
   \put(8.5,-8.8){\sse$ \Xs $}
   \put(31,46)   {\sse$ \rho_X^{} $}
   \put(41.6,95.5) {\sse$ \Xs $}
   \put(59,-8.8) {\sse$ Y $}
   \put(60,95.5) {\sse$ Y $}
   \put(63.7,65.6) {\sse$ \rho_Y^{} $}
   \put(79,-8.8) {\sse$ H $}
   \put(87,10.1) {\sse$ \apo^{-1} $}
   } } }

\begin{rem}\label{rem1}
The category $H\Mod\op \Times H\Mod$ is naturally endowed with a tensor product,
acting on objects as $(X\Times Y) \Times (X'\Times Y') \,{\mapsto}\,
(X\,{\otimes^\HMod}\,X') \Times (Y'\,{\otimes^\HMod}\,Y)$. With respect to this
tensor product and the tensor product \erf{def-tp} on \HBimod, \Fbx\
together with the associativity constraints from \Vectk\ is a monoidal functor.
\end{rem}

In this appendix we show that the coregular $H$-bimodule \Hb\ introduced in Def.\
\ref{rhorho} is the coend of the
functor \Fbx. We first present the appropriate dinatural family.

\begin{lemma}\label{lem-dinat}
The family $(\iHb_X)$ of morphisms
   \be
   \iHb_X := (d_X \oti \idHs) \cir [ \id_{X^*_{}} \oti (\rho_X^{}\cir\tau_{X,H})
   \oti \idHs] \cir ( \id_{X^*_{}} \oti \id_X \oti b_H)
   \labl{def-iHbX}
in \HBimod, pictorially given by
   \eqpic{def_iA_X_bi} {100} {35} {
   \put(-4,0)  { \Includepichtft{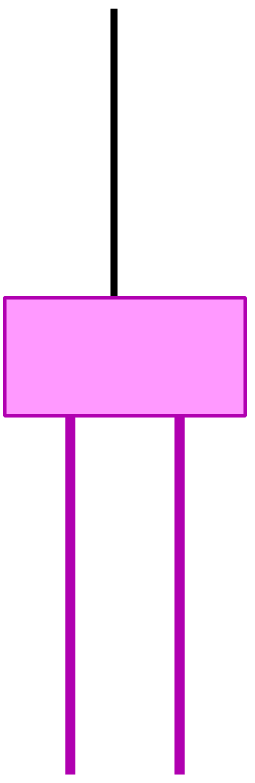}}
   \put(2,-8.5)  {\sse$ \Xs $}
   \put(8.9,44.9){\sse$ \iHb_X $}
   \put(9.2,89)  {\sse$ \Hss $}
   \put(16,-8.5) {\sse$ X $}
   \put(49,38)   {$ = $}
   \put(77,0) { \Includepichtft{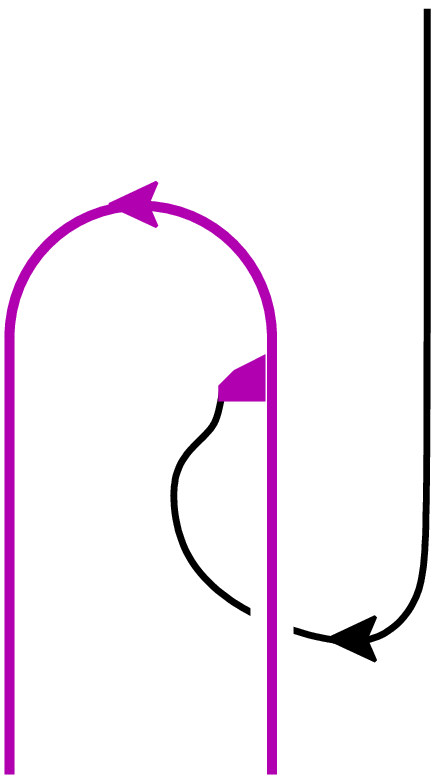}
   \put(-4,-8.5) {\sse$ \Xs $}
   \put(25,-8.5) {\sse$ X $}
   \put(31,43)   {\sse$ \rho_{\!X}^{} $}
   \put(43,89)   {\sse$ \Hss $}
   } }
is dinatural for the functor \Fbx, i.e.
   \be
   \iHb_Y \circ \Fbx(\id_Y,f) = \iHb_x \circ \Fbx(f,\id_X)
   \ee
for any $f\iN\HomH(X,Y)$.
\end{lemma}

\begin{proof}
(i)\,~First note that the maps \erf{def-iHbX}
are a priori just linear maps in $\Homk(X^*_{}\otik X,\Hs)$. However, when \Hs\
is endowed with the $H$-bimodule structure \erf{rhorho} and $X^*_{}\otik X$ with
the one implied by \erf{Fbx}, we have the chain of equalities
   \Eqpic{biregITl} {420} {50} { \put(-13,9){
   \put(-3,0)  {\Includepichtft{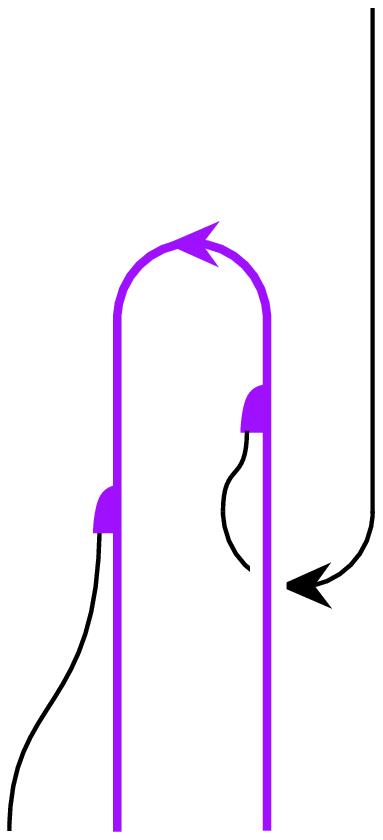}
   \put(-4,-8.5) {\sse$ H $}
   \put(-1.8,38.4) {\sse$ \rho_{\!X^{\!\vee}_{}}^{} $}
   \put(16,48.5) {\sse$ \rho_{\!X}^{} $}
   \put(6,-8.5)  {\sse$\Xs $}
   \put(24,-8.5) {\sse$ X $}
   \put(37,95)   {\sse$ \Hss $}
   }
   \put(60,44)   {$=$}
   \put(82,0) { \Includepichtft{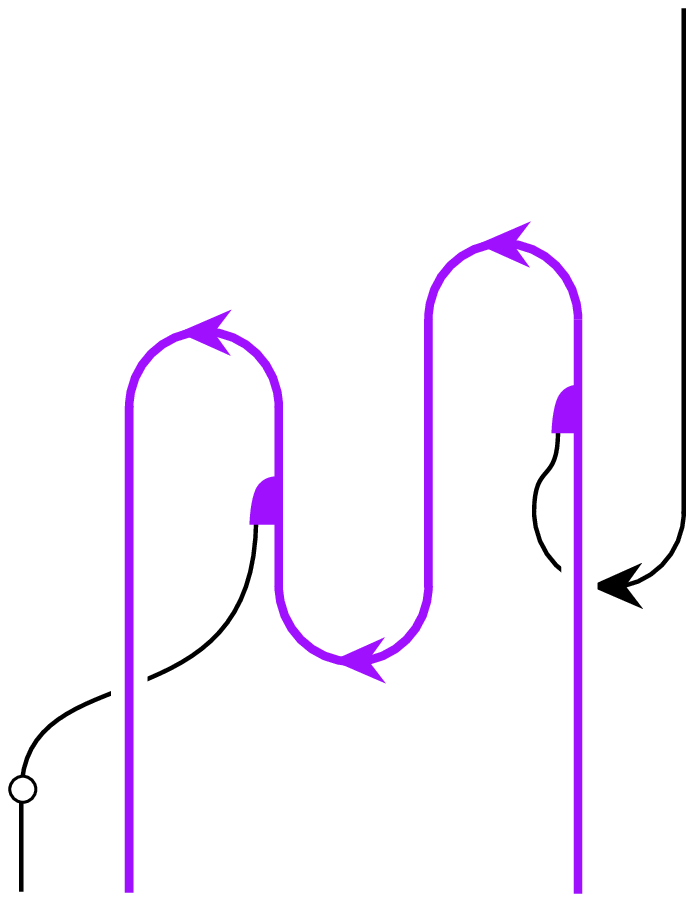}
   \put(-4,-8.5) {\sse$ H $}
   \put(9,-8.5)  {\sse$ \Xs $}
   \put(16.7,44.5) {\sse$ \rho_{\!X}^{} $}
   \put(58,-8.5) {\sse$ X $}
   \put(71,101)  {\sse$ \Hss $}
   }
   \put(184,44)  {$=$}
   \put(208,0)   {\Includepichtft{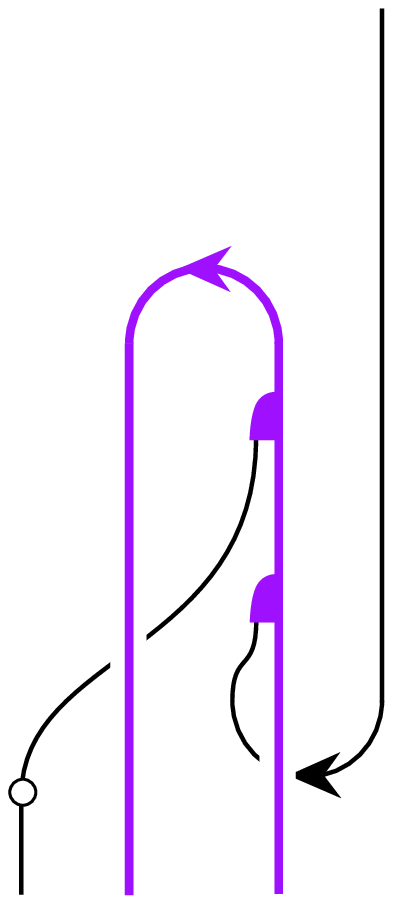}
   \put(-4,-8.5) {\sse$ H $}
   \put(8,-8.5)  {\sse$ \Xs $}
   \put(26,-8.5) {\sse$ X $}
   \put(37,101)  {\sse$ \Hss $}
   }
   \put(273,44)   {$=$}
   \put(292,0) { \Includepichtft{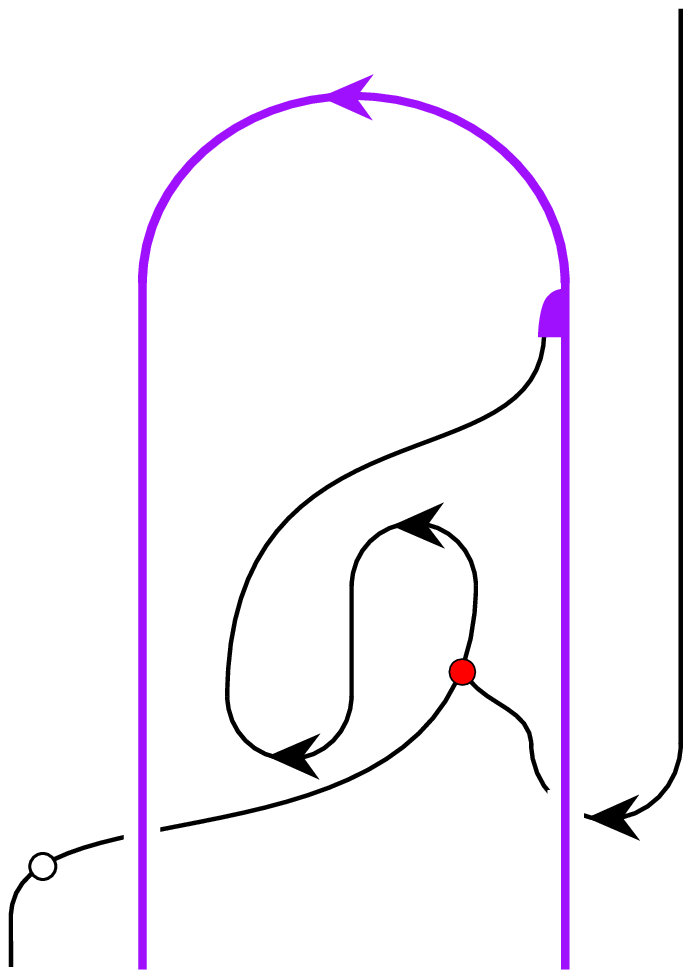}
   \put(-4,-8.5) {\sse$ H $}
   \put(10.6,-8.5) {\sse$ \Xs $}
   \put(57,-8.5) {\sse$ X $}
   \put(70,109)  {\sse$ \Hss $}
   }
   \put(393,44)   {$=$}
   \put(419,0) { \Includepichtft{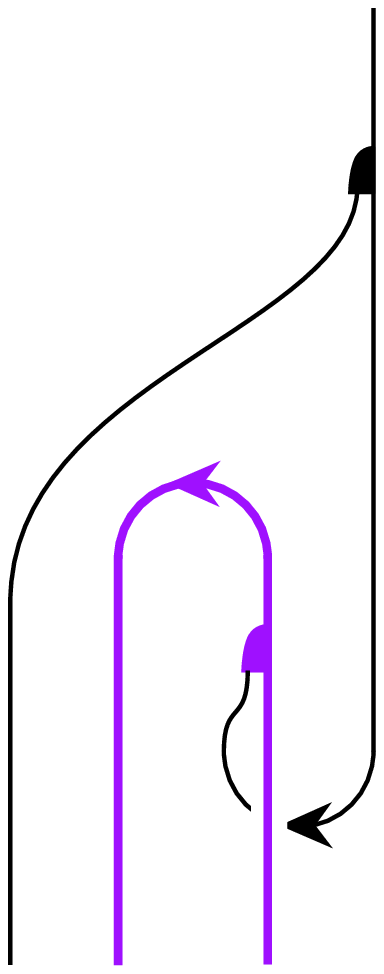}
   \put(-4,-8)   {\sse$H $}
   \put(8,-8.5)  {\sse$\Xs$}
   \put(16.3,41.5) {\sse$ \rho_{\!X}^{} $}
   \put(24,-8.5) {\sse$X$}
   \put(37,109)  {\sse$ \Hss $}
   \put(42.5,86.7) {\sse$ \brho $}
   } } }
showing that \erf{def_iA_X_bi} intertwines the left action of $H$, and
   \eqpic{biregITr} {320} {43} {
   \put(0,0)   {\Includepichtft{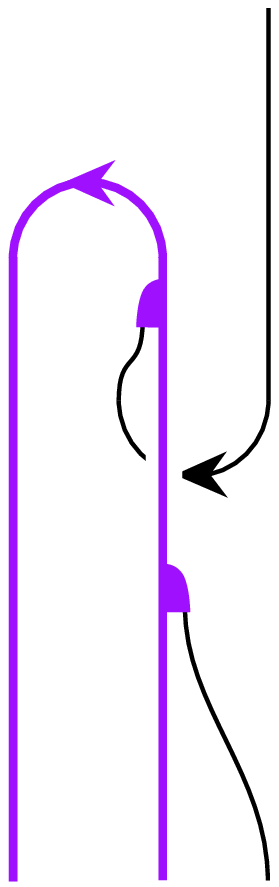}
   \put(-5,-8.5) {\sse$\Xs $}
   \put(13,-8.5) {\sse$X$}
   \put(25,-8.5) {\sse$H$}
   \put(25,100)  {\sse$ \Hss $}
   }
   \put(52,38)   {$=$}
   \put(85,0) { \Includepichtft{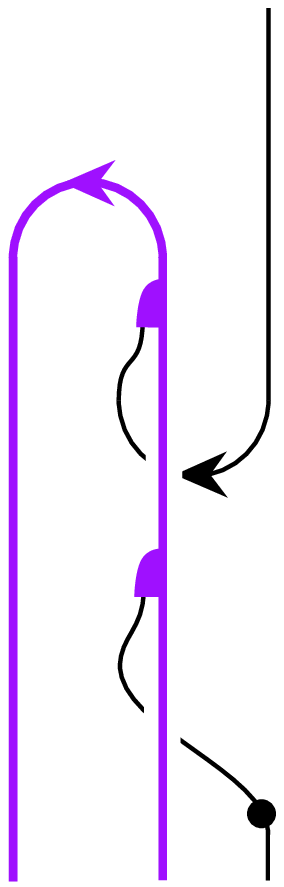}
   \put(-5,-8.5) {\sse$\Xs $}
   \put(13,-8.5) {\sse$X$}
   \put(25,-8.5) {\sse$H$}
   \put(25,100)  {\sse$ \Hss $}
   }
   \put(146,38)  {$=$}
   \put(184,0)   {\Includepichtft{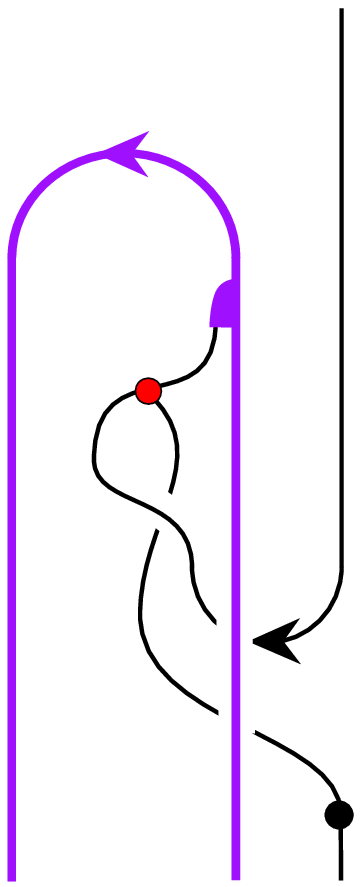}
   \put(-5,-8.5) {\sse$\Xs $}
   \put(20,-8.5) {\sse$X$}
   \put(32,-8.5) {\sse$H$}
   \put(32.5,100)  {\sse$ \Hss $}
   }
   \put(246,38)   {$=$}
   \put(275,0) { \Includepichtft{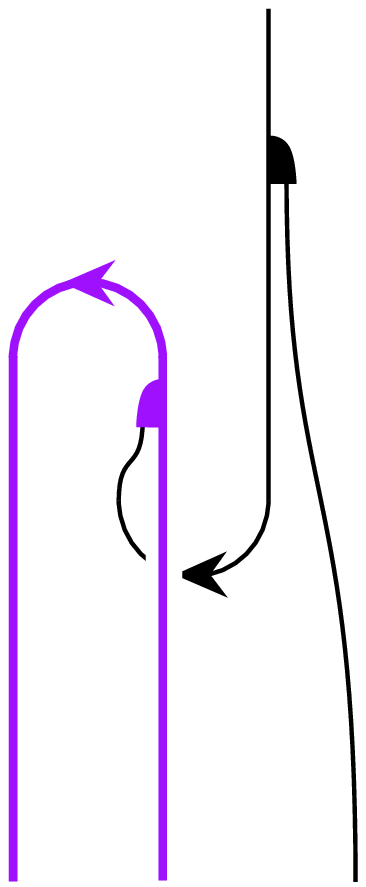}
   \put(-5,-8.5) {\sse$\Xs $}
   \put(13,-8.5) {\sse$ X $}
   \put(34,-8.5) {\sse$ H $}
   \put(33,79.1) {\sse$ \bohr $}
   \put(25,100)  {\sse$ \Hss $}
   } }
showing that it also intertwines the right action.
\nxl1
(ii) The dinaturalness property amounts to the equality of the left and right
hand sides of
   \eqpic{dinat_iA_X} {250} {44} {
   \put(0,0)   {\Includepichtft{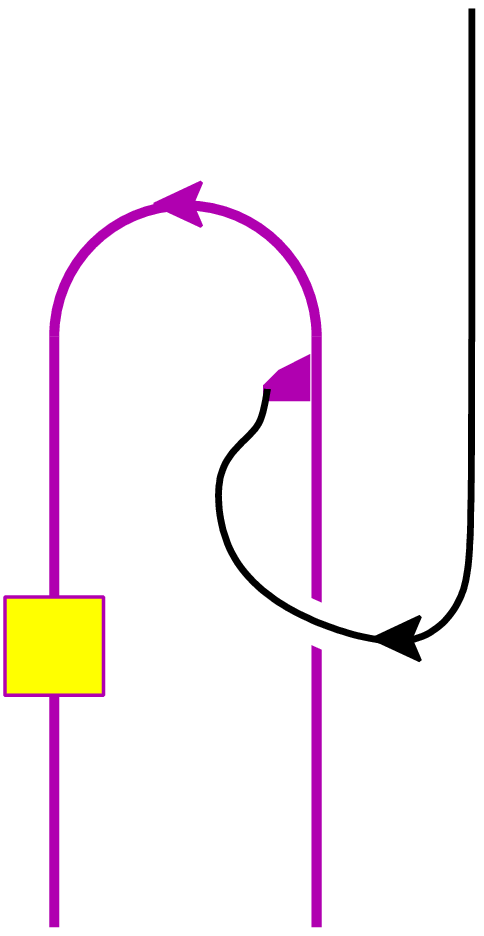}}
   \put(-7.5,50){\sse$ \Xs $}
   \put(0,-8.5) {\sse$ Y^* $}
   \put(1.2,30.4) {\sse$ f^\wee $}
   \put(30,-8.5){\sse$ X $}
   \put(48,106) {\sse$ \Hs $}
   \put(75,47)  {$ = $}
   \put(104,0) { {\Includepichtft{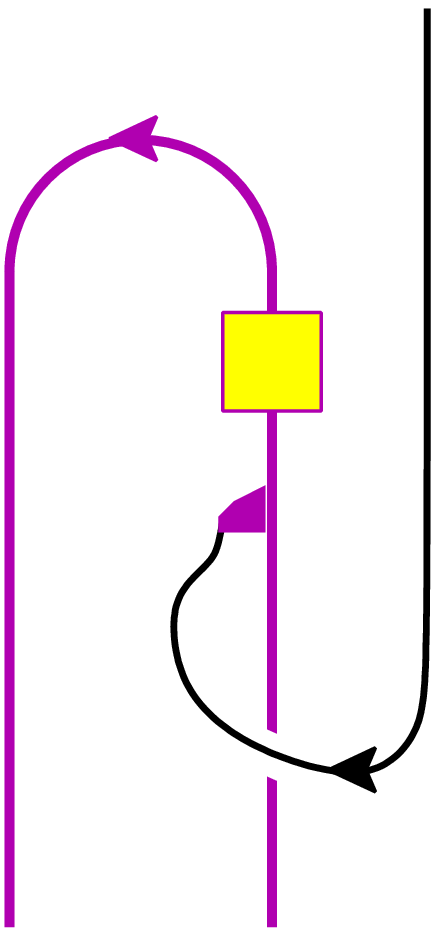}}
   \put(-5,-8.5){\sse$ Y^* $}
   \put(25,-8.5){\sse$ X $}
   \put(27.3,60.8){\sse$ f $}
   \put(43,106) {\sse$ \Hs $}
   }
   \put(176,47) {$ = $}
   \put(204,0) { {\Includepichtft{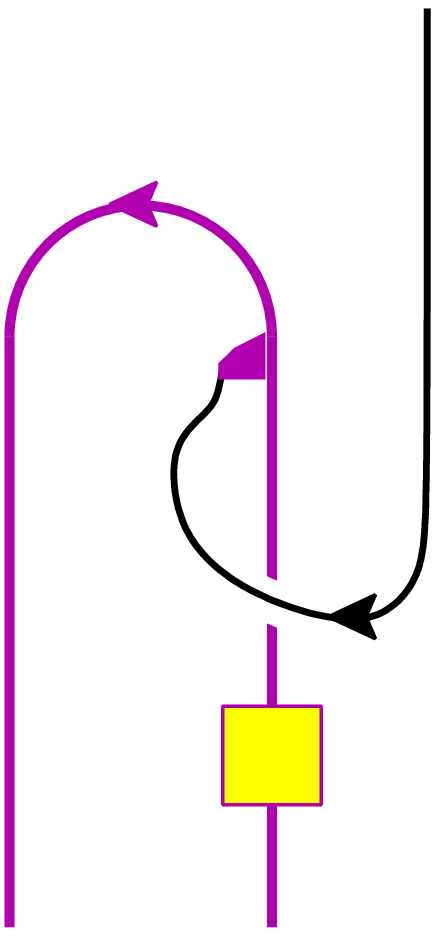}}
   \put(-5,-8.5){\sse$ Y^* $}
   \put(25,-8.5){\sse$ X $}
   \put(27.3,17.5){\sse$ f $}
   \put(31.3,49){\sse$ Y $}
   \put(43,106) {\sse$ \Hs $}
   } }
for any module morphism $f$ from $X$ to $Y$. Now the first equality in
\erf{dinat_iA_X} holds by definition of $f^\wee$, and the second equality holds
because $f$ is a module morphism.
\end{proof}

\begin{prop}\label{prop-F=coend}
The $H$-bimodule \Hb\ together with the dinatural family $(\iHb_X)$
given by \erf{def-iHbX} is the coend of the functor \Fbx,
   \be
   (\Hb,\iHb) = \coen X \Fbx(X,X) \,.
   \labl{F=coend}
\end{prop}

\begin{proof}
We have to show that $(\Hb,\iHb)$ is an initial object in the category of
dinatural transformations from \Fbx\ to a constant.
\nxl1
(i)~\,Let $j^Z$ be a dinatural transformation from \Fbx\ to $Z\iN H\Bimod$.
Given any $X\iN H\Mod$ and any $x_\circ\iN \Homk(\ko,X)$ (i.e.\ element of
$X$), applying the dinaturalness property of $j^Z$ to the morphism
$f_{x_\circ} \,{:=}\, \rho_X \cir (\id_H \oti x_\circ) \iN \HomH(H,X)$
(with $H$ regarded as an $H$-module via the regular left action) yields
$j^Z_X \cir (\idXs \oti x_\circ) \eq j^Z_H  \cir (\iHb_X \oti \eta)
\cir (\idXs \oti x_\circ)$. Namely, we have
   \eqpic{rZ_etc_2B} {300} {38} { \setlength\unitlength{.8pt}
   \put(0,0)   { \includepichtft{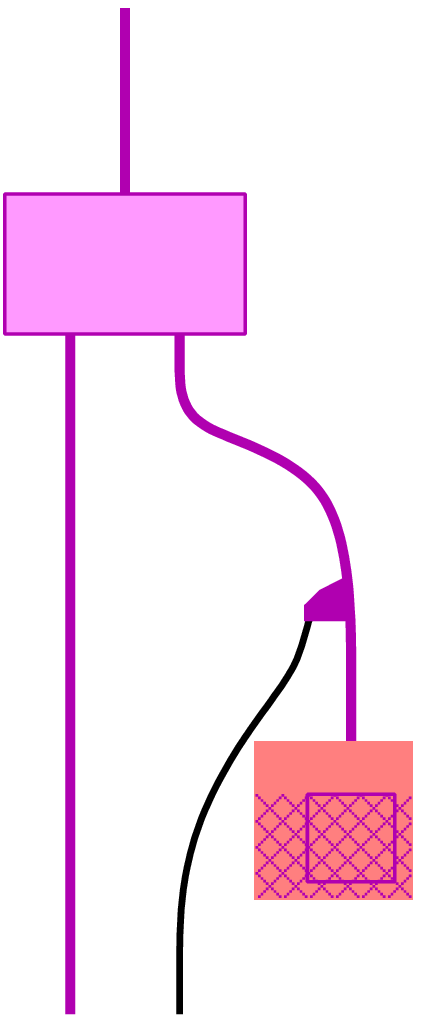}
   \put(0.2,-8.5) {\sse$ \Xs $}
   \put(8.8,81.1) {\sse$ j^Z_{\!X^{}} $}
   \put(10.6,115) {\sse$ Z $}
   \put(16,-8.5)  {\sse$ H $}
   \put(45.3,18)  {\sse$ x_\circ $}
   }
   \put(77,51)   {$ \equiv $}
   \put(108,0) { \includepichtft{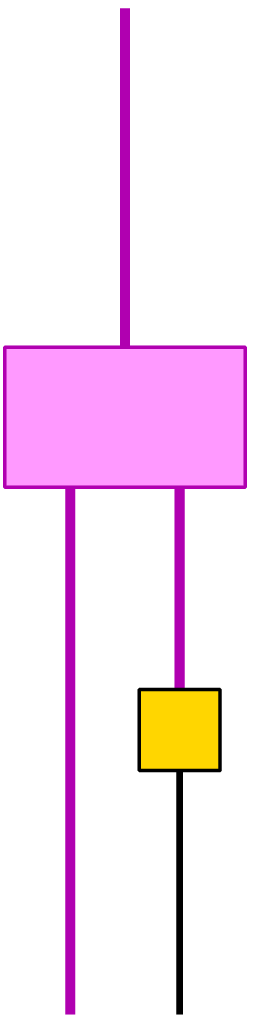}
   \put(0.2,-8.5) {\sse$ \Xs $}
   \put(9.2,64.1) {\sse$ j^Z_{\!X^{}} $}
   \put(10.6,115) {\sse$ Z $}
   \put(16,-8.5)  {\sse$ H $}
   \put(25.4,30)  {\sse$ f_{x_\circ} $}
   }
   \put(174,51)  {$ = $}
   \put(207,0) { \includepichtft{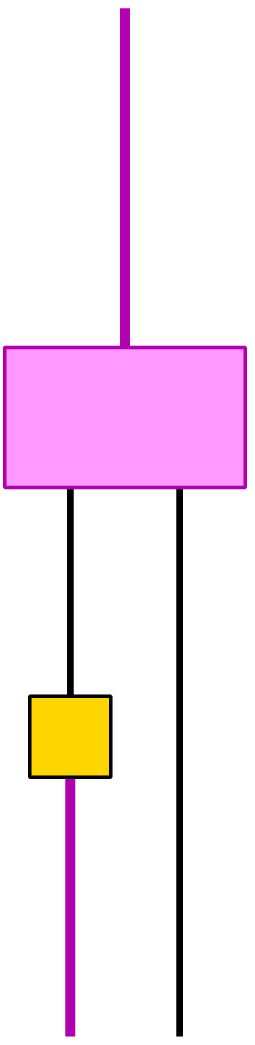}
   \put(-12.6,32.3) {\sse$ f_{x_\circ}^\wee $}
   \put(0.2,-8.5) {\sse$ \Xs $}
   \put(8.8,67.1) {\sse$ j^Z_{\!H^{}} $}
   \put(10.6,117) {\sse$ Z $}
   \put(16,-8.5)  {\sse$ H $}
   }
   \put(270,51)  {$ \equiv $}
   \put(303,0) { \includepichtft{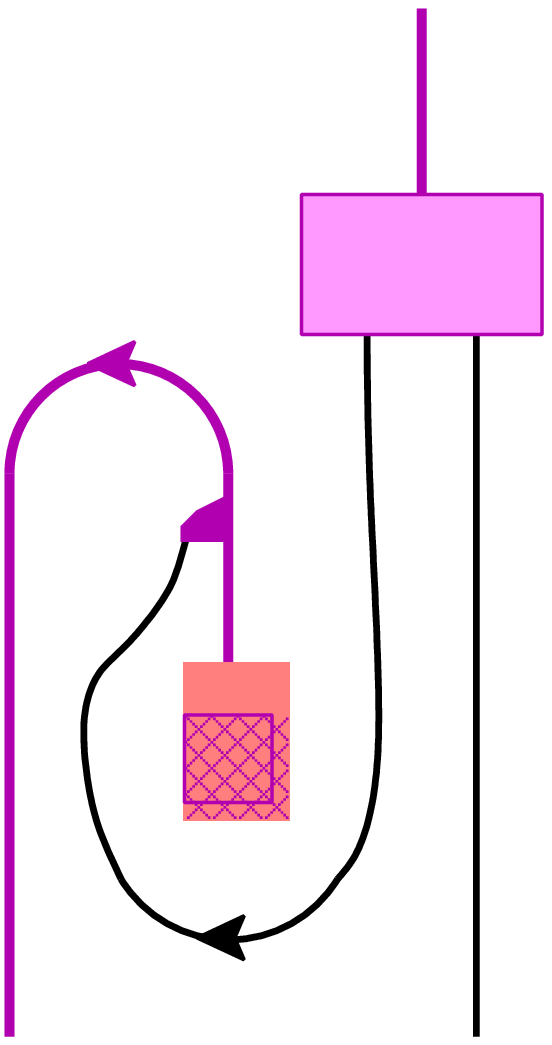}
   \put(-5,-8.5)  {\sse$ \Xs $}
   \put(20.7,20.9){\sse$ x_\circ $}
   \put(41.5,83.4){\sse$ j^Z_{\!H^{}} $}
   \put(43.2,117) {\sse$ Z $}
   \put(49,-8.5)  {\sse$ H $}
   } }
and thus, after composition with $\idXs\oti\eta$,
   \eqpic{rZ_etc_3B} {420} {38} { \setlength\unitlength{.8pt}
   \put(-18,51)   {$ j^Z_X\circ(\idXs\oti x_0)~ = $}
   \put(116,0) { \includepichtft{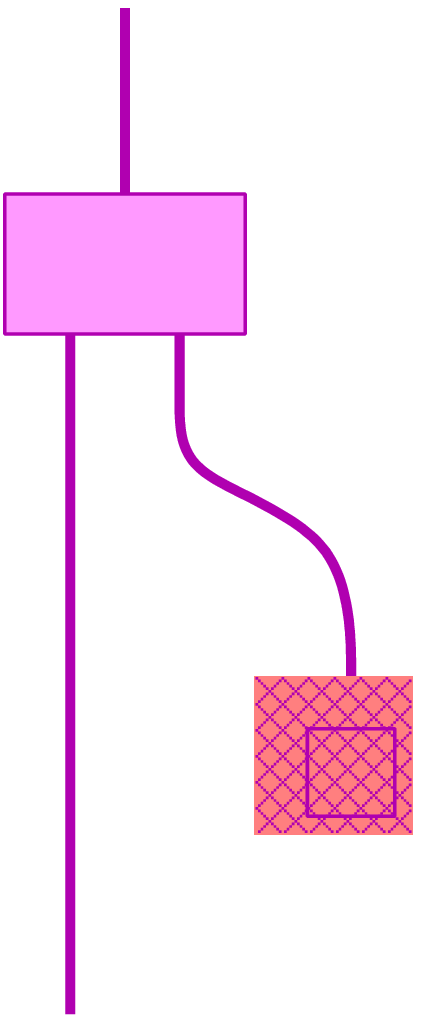}
   \put(2,-8.5)   {\sse$ \Xs $}
   \put(8.8,81.1) {\sse$ j^Z_{\!X^{}} $}
   \put(10.6,115) {\sse$ Z $}
   }
   \put(185,51)  {$ = $}
   \put(216,0) { \includepichtft{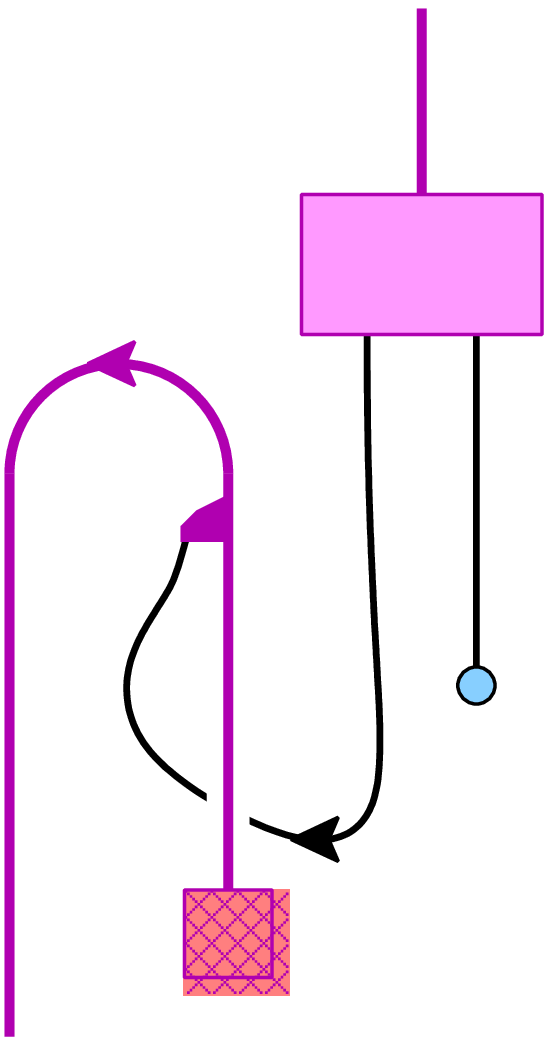}
   \put(-4,-8.5)  {\sse$ \Xs $}
   \put(41.5,83.7){\sse$ j^Z_{\!H^{}} $}
   \put(43.6,117) {\sse$ Z $}
   }
   \put(306,51)  {$ =~ j^Z_H  \cir (\iHb_X \oti \eta) \cir (\id_{X^*_{}} \oti x_\circ) $}
   }
with $\iHb_X$ from \erf{def-iHbX}.
Since $x_\circ\iN \Homk(\ko,X)$ is arbitrary, we actually have
   \be
   j^Z_X = j^Z_H \cir (\iHb_X \oti \eta)
   \labl{jZ_iAB}
for any bimodule $Z$ and dinatural transformation $j^Z$ from \Fbx\ to $Z$.
\nxl1
(ii)~\,Now consider the linear map
   \be
   \kappa^Z := j^Z_H \circ (\idHs \oti \eta)
   \ee
from \Hs\ to $Z$. This is in fact a bimodule morphism from \Hb\ to $Z$:
Compatibility with the left $H$-action follows directly from the fact that
$j_H^Z$ is a morphism of bimodules, and thus in particular of left modules,
while compatibility with the right $H$-action is seen as follows:
   \eqpic{kappazbim} {430} {37} {
   \put(3,0)   {\Includepichtft{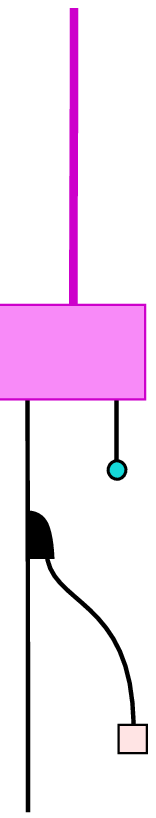}
   \put(-9.3,30) {\sse$\bohr$}
   \put(-2.3,-8.5) {\sse$\Hs $}
   \put(3.4,49)  {\sse$\jHZx$}
   \put(5,92.5)  {\sse$ Z $}
   \put(17,6)    {\sse$h$}
   }
   \put(52,38)   {$=$}
   \put(79,0) { \Includepichtft{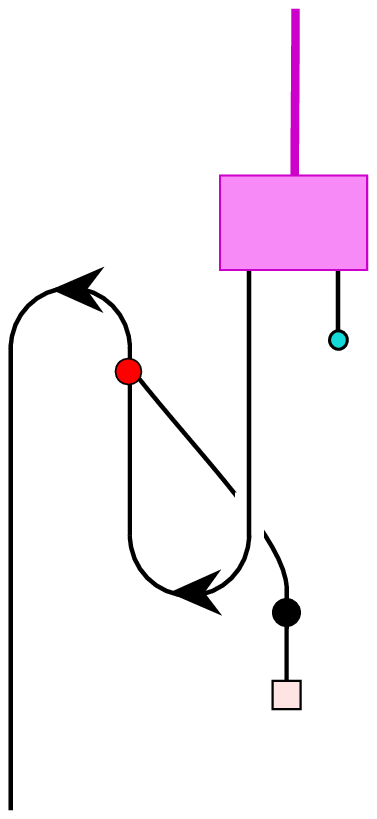}
   \put(-5,-8.5) {\sse$ \Hs $}
   \put(28.7,92.5) {\sse$ Z $}
   \put(33.5,11) {\sse$h$}
   \put(27.4,62.8) {\sse$\jHZx$}
   }
   \put(146,38)  {$=$}
   \put(187,0) {\Includepichtft{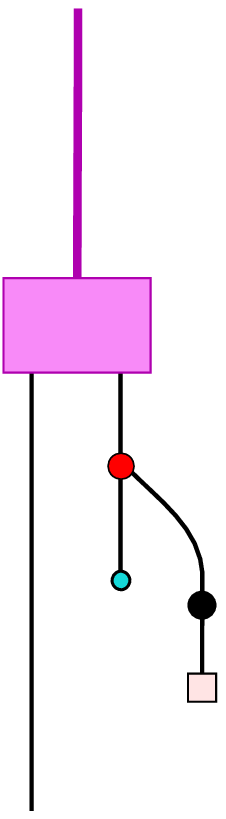}
   \put(-2,-8.5) {\sse$\Hs $}
   \put(5.4,92.5){\sse$ Z $}
   \put(24.5,12) {\sse$h$}
   \put(3.7,51.2){\sse$\jHZx$}
   }
   \put(238,38)  {$=$}
   \put(280,0) {\Includepichtft{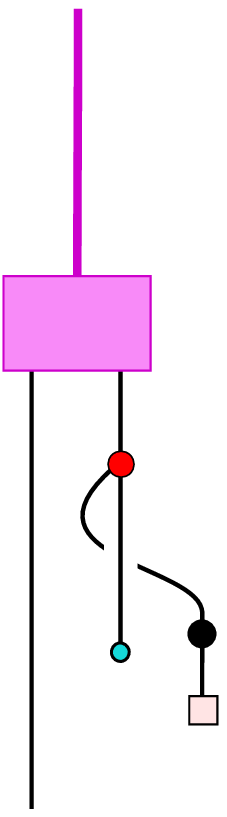}
   \put(-2,-8.5) {\sse$\Hs $}
   \put(5,92.5)  {\sse$ Z $}
   \put(24.5,9)  {\sse$h$}
   \put(3.7,51.2){\sse$\jHZx$}
   }
   \put(334,38)  {$=$}
   \put(378,0) {\Includepichtft{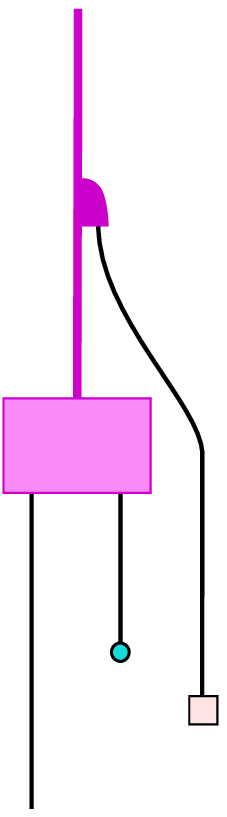}
   \put(-2,-8)   {\sse$\Hs $}
   \put(5.7,92.5){\sse$ Z $}
   \put(13,66)   {\sse$ \ohr_Z $}
   \put(24.5,9)  {\sse$h$}
   \put(-11,38.2){\sse$j_H^Z$}
   } }
Here the element $h\iN \Homk(\ko,H)$ is arbitrary; the second equality invokes the
dinaturalness of $j_{}^Z$ for the map $m\cir(\id_H \oti (\apoi\cir h)) \iN \EndH(H)$.
\nxl1
(iii)\,~In terms of the morphism $\kappa^Z$, \erf{jZ_iAB} amounts to
   \be
   j^Z_X = \kappa^Z \circ \iHb_X \,.
   \labl{jZ_rZ_iAB}
This establishes existence of the morphism from \Hb\ to $Z$ that is required for the
universal property of the coend.
\nxl1
(iv)\,~It remains to show that $\kappa^Z$ is uniquely determined.
This just follows by specializing \erf{jZ_rZ_iAB} to the case $X\eq H$ and
observing that $\iHb_H$ has a right-inverse. The latter property holds because of
$\iHb_H \cir (\idHs\oti\eta) \eq (d_H\oti\idHs)\cir (\idHs\oti b_H) \eq \idHs$.
\end{proof}


\subsection{Some equivalences of braided monoidal categories}\label{ssec:equiv}

We note the following equivalences, where as usual $H\op$ is $H$ with
opposite product $m\cir\tauHH$ (and with the same coproduct), and $H\coop$
is $H$ with opposite coproduct $\tauHH\cir\Delta$ (and with the same product).

\begin{lemma}\label{lem-equiv}
(i)\,~For any Hopf algebra $H$ there are equivalences
   \be
   \mbox{\rm $ \HBimod \;\simeq\, (H\Otik H)\Mod \;\simeq\, (H\Otik H\op)\Mod $}
   \labl{2equiv}
as abelian categories.
\nxl1
(ii)\,~The equivalences \erf{2equiv} extend to equivalences
   \be
   \mbox{\rm $ \HBimod \;\simeq\, (H\Otik H\coop)\Mod \;\simeq\, (H\Otik H\op)\Mod $}
   \labl{2equiv2}
as monoidal categories, with respect to the tensor
products \erf{otimesHMod} on \HMod\ and \erf{def-tp} on \HBimod.
The constraint morphisms for the tensor functor structures of the equivalence
functors are all identities.
\nxl1
(iii)\,~If the Hopf algebra $H$ is quasitriangular with R-matrix $R$, then
the equivalences \erf{2equiv2} extend to equivalences
   \be
   \mbox{\rm $ \HBimod \;\simeq\, (\overline H\Otik H\coop)\Mod
   \;\simeq\, (\overline H\Otik H\op)\Mod $}
   \labl{opequiv}
as braided monoidal categories, where \HBimod\ is endowed with the braiding
\erf{bibraid} and $\overline H$ is $H$ with R-matrix $R_{21}^{-1}$.
(Also, $H\op$ is endowed with the natural quasitriangular structure inherited from
$H$, i.e.\ has R-matrix $R_{21}$.)
\end{lemma}

\begin{proof}
(i)\,~We derive each of the equivalences in a somewhat more general context.
\nxl1
For any two Hopf algebras $H$ and $H'$ there is an equivalence
$H\mbox-H'\Bimod \,{\simeq}\, (H\Otik H')\Mod$ as abelian categories. The
equivalence is furnished by the two functors which on morphisms are the identity
and which map objects according to
   \eqpic{HHbimiso} {390} {36} {
   \put(0,0)   {\Includepichtft{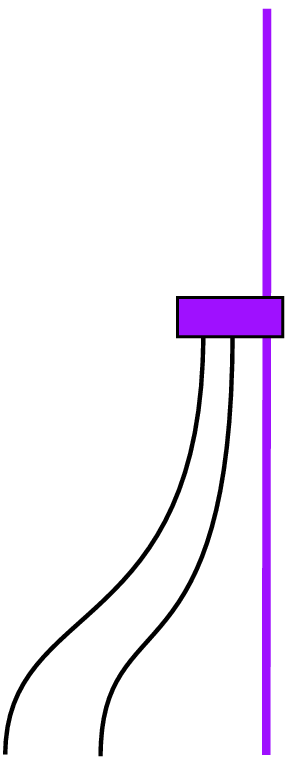}
   \put(-4,-8.5) {\sse$ H $}
   \put(6,-8)    {\sse$ H'$}
   \put(24.1,-8.5) {\sse$ M $}
   \put(24.8,86) {\sse$ M $}
   \put(-5.5,47) {\sse$ \rho^{H\Oti H'} $}
   }
   \put(58,34)   {$ \mapsto$}
   \put(100,0) { \Includepichtft{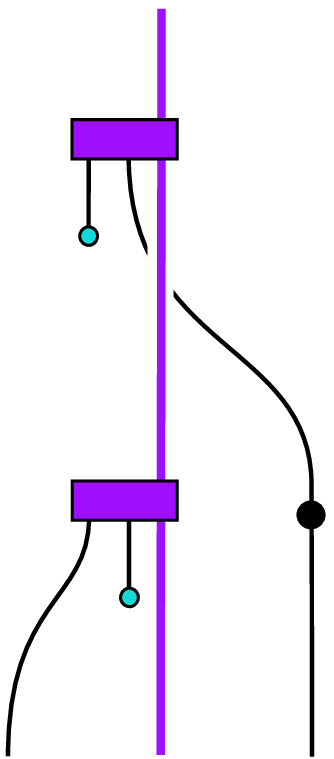}
   \put(-4,-8.5) {\sse$ H $}
   \put(12.3,-8.5) {\sse$ M $}
   \put(12.9,86) {\sse$ M $}
   \put(29,-8.5) {\sse$ H' $}
   \put(36,25)   {\sse$ \apo'^{-1} $}
   \put(-17.2,26.4){\sse$ \rho^{H\Oti H'} $}
   \put(-17.2,66.4){\sse$ \rho^{H\Oti H'} $}
   }
   \put(185,34)  {and}
   \put(240,0) {\Includepichtft{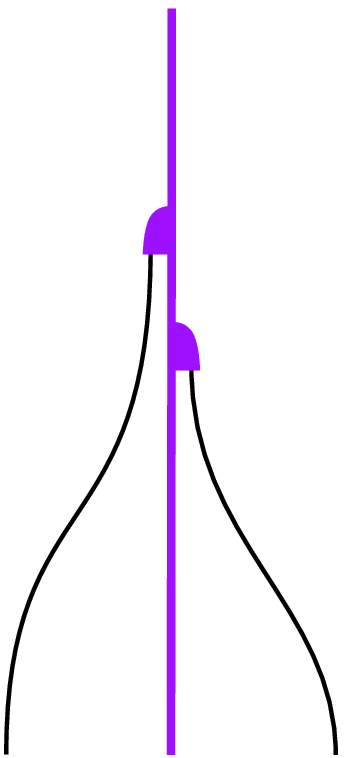}
   \put(-4,-8.5) {\sse$H $}
   \put(5.2,55.4){\sse$ \rho^H $}
   \put(13,-8.5) {\sse$ M $}
   \put(13.8,86) {\sse$ M $}
   \put(22.2,43.4){\sse$ \ohr^{\!H'} $}
   \put(32,-8.5) {\sse$ H' $}
   }
   \put(300,34)   {$ \mapsto$}
   \put(332,0) { \Includepichtft{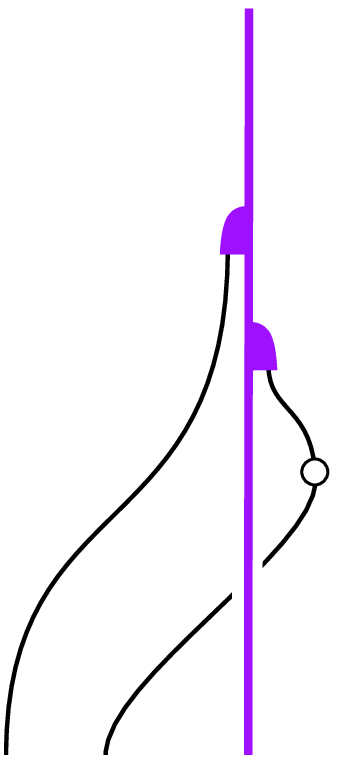}
   \put(-4,-8.5) {\sse$ H $}
   \put(7,-8.5)  {\sse$ H'$}
   \put(22,-8.5) {\sse$ M $}
   \put(22.8,86) {\sse$ M $}
   \put(36,29.5) {\sse$ \apo' $}
   } }
respectively.
\nxl1
Similarly, an equivalence
$H\mbox-H'\Bimod \,{\simeq}\, (H\Otik H'{}^\text{op})\Mod$ as abelian categories
is furnished by functors that differ from those in \erf{HHbimiso} by just
omitting the (inverse) antipode (compare e.g.\ \cite[Prop.\,4.6]{fuRs4}).
\nxl2
(ii)\,~For the first equivalence, compatibility with the tensor product follows
for the second functor in \erf{HHbimiso} as
   \eqpic{HHbimisotens} {400} {38} {
   \put(0,0) {\Includepichtft{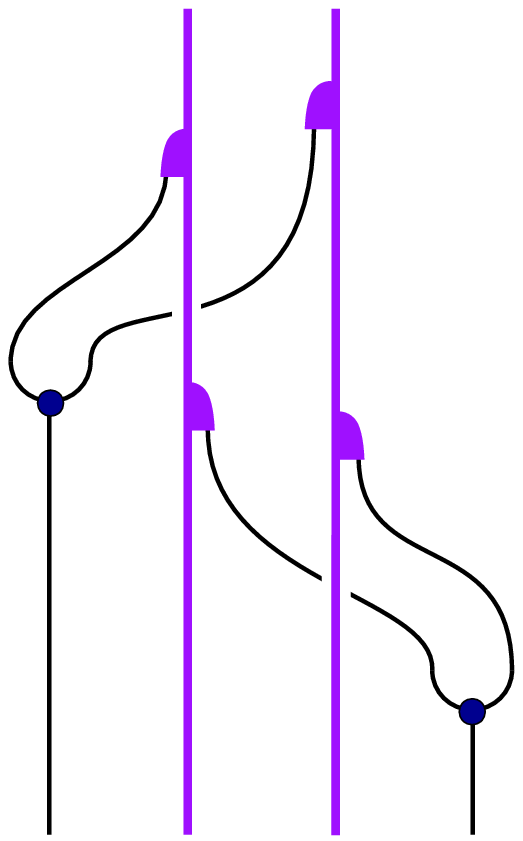}
   \put(0,-8.5)  {\sse$ H $}
   \put(14.6,-8.5) {\sse$ M $}
   \put(15.5,95) {\sse$ M $}
   \put(31.7,-8.5) {\sse$ N $}
   \put(32.7,95) {\sse$ N $}
   \put(46,-8.5) {\sse$ H' $}
   }
   \put(85,38)   {$\mapsto$}
   \put(130,0) { \Includepichtft{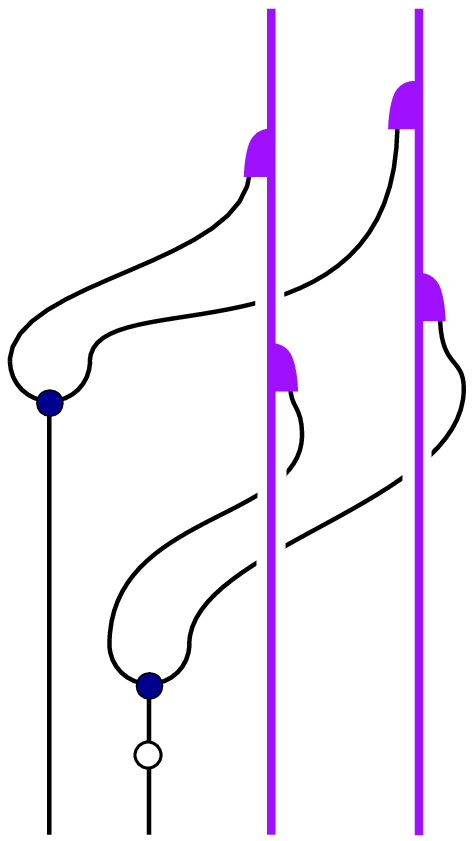}
   \put(0,-8.5)  {\sse$ H $}
   \put(11,-8.5) {\sse$ H' $}
   \put(24,-8.5) {\sse$ M $}
   \put(24.8,95) {\sse$ M $}
   \put(41,-8.5) {\sse$ N $}
   \put(41.8,95) {\sse$ N $}
   }
   \put(204,38)  {$=$}
   \put(232,0) {\Includepichtft{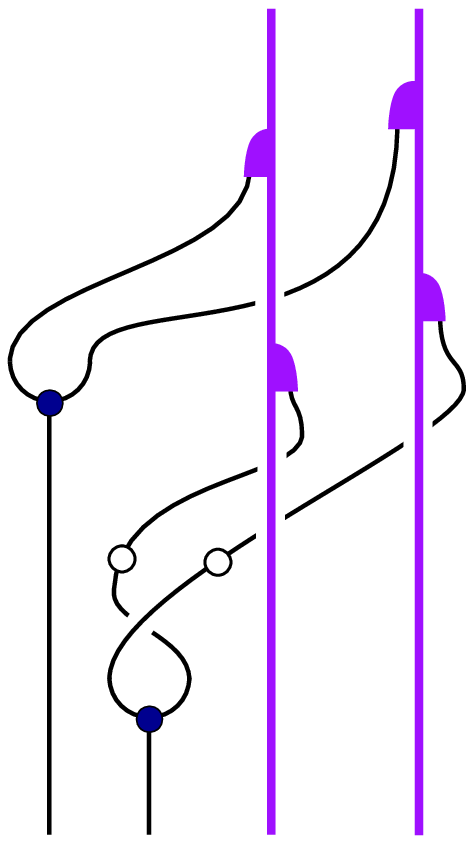}
   \put(0,-8.5)  {\sse$ H $}
   \put(11,-8.5) {\sse$ H' $}
   \put(24,-8.5) {\sse$ M $}
   \put(24.8,95) {\sse$ M $}
   \put(41,-8.5) {\sse$ N $}
   \put(41.8,95) {\sse$ N $}
   }
   \put(302,38)  {$ =$}
   \put(328,0) { \Includepichtft{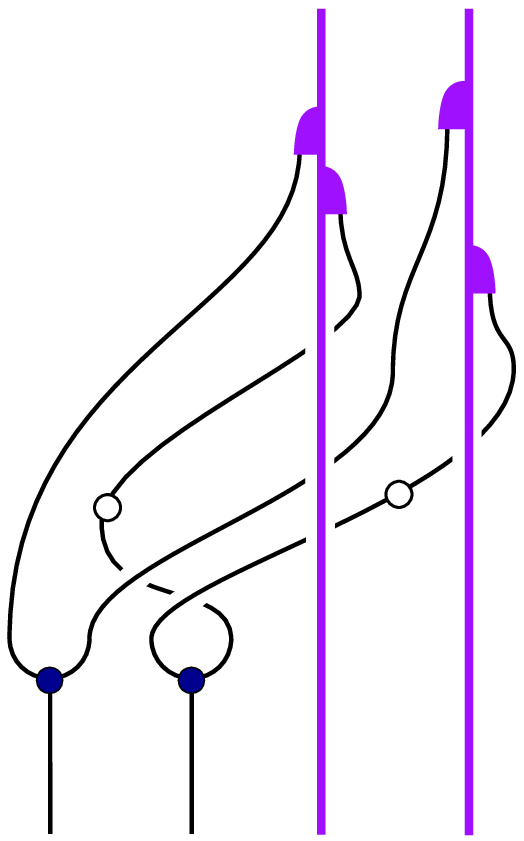}
   \put(0,-8.5)  {\sse$ H $}
   \put(15,-8.5) {\sse$ H' $}
   \put(29,-8.5) {\sse$ M $}
   \put(29.8,95) {\sse$ M $}
   \put(46,-8.5) {\sse$ N $}
   \put(46.8,95) {\sse$ N $}
   } }
and analogously for the first functor, as well as for the second equivalence.
\nxl2
(iii)\,~The Hopf algebras $\overline H \otik H\coop$ and $\overline H \otik H\op$
have natural quasitriangular structures, with R-matrices given by
$(\id_H\oti c_{H,H} \oti \id_H) \cir (R_{21}^{-1} \oti R^{-1}_{})$.
By direct calculation one checks that the two functors given in
\erf{HHbimiso} (with $H' \eq H$), respectively the ones with the occurences of the
antipode removed, not only furnish an equivalence between \HBimod\ and
$(\overline H\Otik H\coop)\Mod$, respectively $(\overline H\Otik H\op)\Mod$, as abelian
monoidal categories, but map the braidings of these categories to each other as well.
\\
Also note that the R-matrix furnishes an equivalence between $H\coop\Mod$ and
\HMod\ as monoidal categories, so that in the equivalences \erf{opequiv} we
could as well use $H$ instead of $H\coop$.
\end{proof}

\begin{rem}
In view of Lemma \ref{lem-equiv}, Prop.\ \ref{prop-F=coend} is implied by
Theorem 7.4.13 of \cite{KEly}.
\end{rem}

The significance of the coend \erf{F=coend} and of the equivalences in Lemma
\ref{lem-equiv} actually transcends the framework of the (bi)module categories
considered in this paper. Namely, one can consider the situation that \HMod\ is
replaced by a more general ribbon
category \C, while the role of \HBimod\  is taken over by the Deligne product
of \C\ with itself. Recall \cite[Sect.\,5]{deli} that the Deligne tensor product
of two \ko-linear abelian categories \C\ and \D\ that are locally finite, i.e.\
all morphism spaces of which are \findim\ and all objects of which have finite
length, is a category $\C\boti\D$ together with a bifunctor
$\boxtimes\colon \C\Times\D \To \C\boti\D$ that is right exact and \ko-linear in
both variables and has the following universal property: for any bifunctor $G$
from $\C\Times\D$ to a \ko-linear abelian category $\mathcal E$ being right exact
and \ko-linear in both variables there exists a unique right exact \ko-linear
functor $G_{\!\Box}\colon \C\boti\D \To \mathcal E$ such that $G \cong G_{\!\Box}
\cir \boxtimes$. In short, bifunctors from $\C\Times\D$ become functors from
$\C\boti\D$. The category $\C\boti\D$ is again \ko-linear abelian and locally finite.

By the universal property of the Deligne product, there is a unique functor
   \be
   \Fbb : \quad H\Mod \boti H\Mod \,\longrightarrow\, H\Bimod
   \labl{equivFbb}
such that the bifunctor \erf{Fbx} can be written as the composition
$\Fbx \eq \Fbb \circ (?^\vee \boti \Id)$, with the functor
$?^\vee \boti \Id \eq \boxtimes \cir (?^\vee \Times \Id)$ acting as
$X\Times Y \,{\mapsto}\, X^\vee\boti Y$ and $f\Times g \,{\mapsto}\, f^\vee\boti g$.
On objects of $\C\boti\D$ that are of the form $U\boti V$ with $U\iN\C$ and
$V\iN\D$, the functor \Fbb\ acts as
   \be
   (X,\rho_X) \boti (Y,\rho_Y) ~\stackrel\Fbb\longmapsto~
   \big( X\oti Y\,,\, \rho_X\oti\id_Y\,,\, \id_X\oti(\rho_Y\cir \tau_{Y,H}
   \cir (\id_Y\Oti\apo^{-1})) \big) \,.
   \ee

Now by combining Prop.\ 5.3 of \cite{deli} with the first equivalence in
\erf{2equiv}, one sees (compare also e.g.\ \cite[Ex.\,7.10]{franI}) that the
functor \Fbb\ is an equivalence of abelian categories. Further,
$H\Mod \boti H\Mod$ has a natural monoidal structure \cite[Prop.\,5.17]{deli}
as well as a braiding (which on objects of the form $U\boti V$ acts as
${(c^{H\text{-Mod}}_{U',U})}^{-1} {\otimes_\ko}\,c^{H\text{-Mod}}_{V,V'}$).
With respect to these the equivalence \erf{equivFbb} can be endowed with the
structure of an equivalence of braided monoidal categories.

Observations analogous to those made here for the category $\C_H \eq \HMod$
in fact apply to any locally finite \ko-linear abelian ribbon category \C.
Hereby the Frobenius algebra \Hb\ in \HBimod\ can be understood as a particular
case of the coend
   \be
   F_\C := \coen X X^\vee \boti X
   \ee
of the functor $?^\vee\boti\Id \colon \CopC\to\CbC$ (where $\overline{\mathcal C}$
is \C\ with opposite braiding), which exists for any such category \C.
This coend $F_\C$ has already been considered in \cite{kerl5}
and \cite[Sect\,5.1.3]{KEly}.
It is natural to expect that also in this general setting the coend $F_\C$
still carries a natural Frobenius algebra structure. However, so far we only
know that $F_\C$ is naturally a unital associative algebra in \CbC.

For $\C \eq \HMod$, the category \CbC\ is also equivalent, as a ribbon category,
to the center of \C, and thus to the category of
Yetter-Drinfeld modules over $H$. Hence instead of with $H$-bimodules we could
equivalently work with Yetter-Drinfeld modules over $H$.
In particular, the Frobenius algebra \Hb\ can be recognized as the so-called
\cite{ffrs,davy20} full center of the tensor unit of \HMod; in the
Yetter-Drinfeld setting, this is described in Example 5.5 of \cite{davy22}.


\subsection{The coend \Haalg\ in \HBimod}\label{bicoad-coend}

The coadjoint left and right actions $\rho\coa \iN \Hom(H\oti\Hs,\Hs)$ and
$\ohr\coar \iN \Hom(\Hs\oti H,\Hs)$ of $H$ on its dual $\Hs$ are by
definition the morphisms
   \eqpic{def_lads} {370} {49} {
   \put(0,52)      {$ \rho\coa ~= $}
   \put(42,0) { \includepichopf{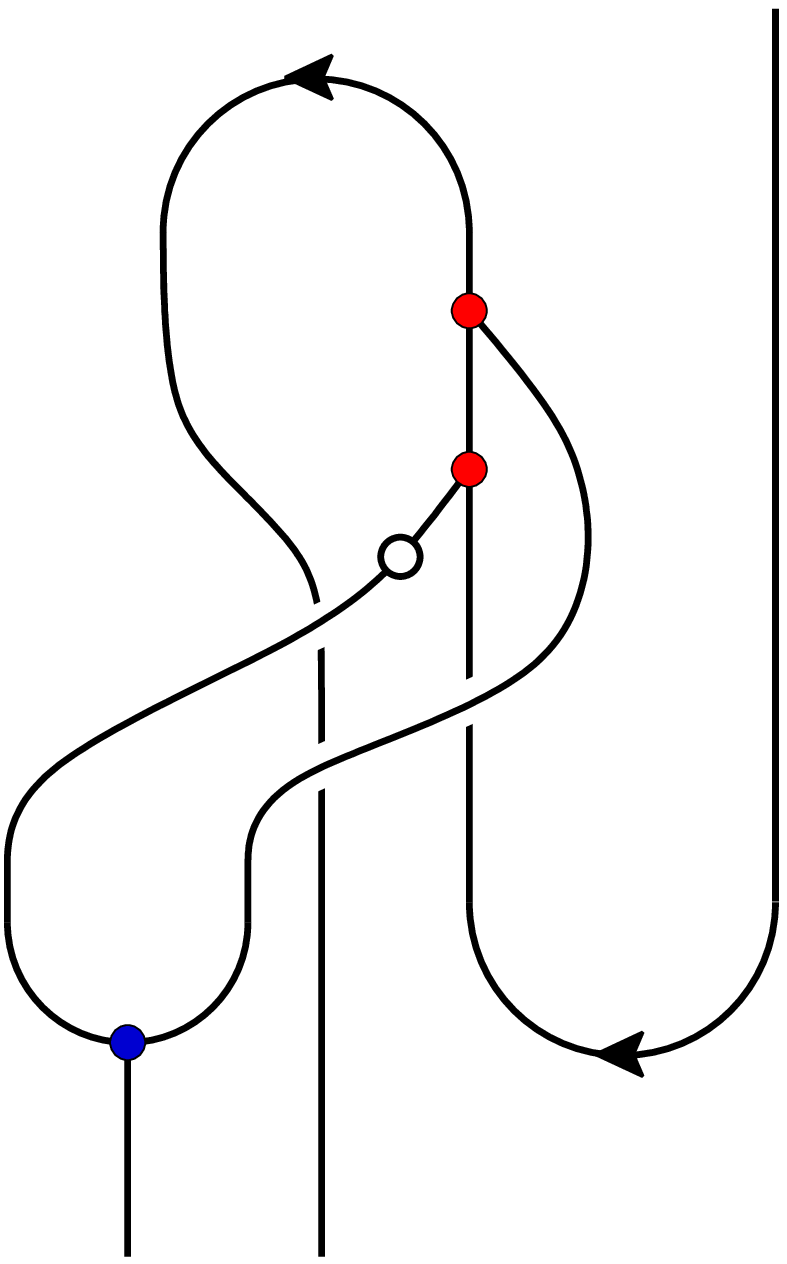}
   \put(6.5,-8.5){\sse$ H $}
   \put(23.4,-8.5) {\sse$ \Hss $}
   \put(32.3,65.6) {\sse$ \apo $}
   \put(64.5,114)  {\sse$ \Hss $}
   }
   \put(167,52)  {and}
   \put(228,52)    {$ \ohr\coar ~= $}
   \put(278,0) { \includepichtft{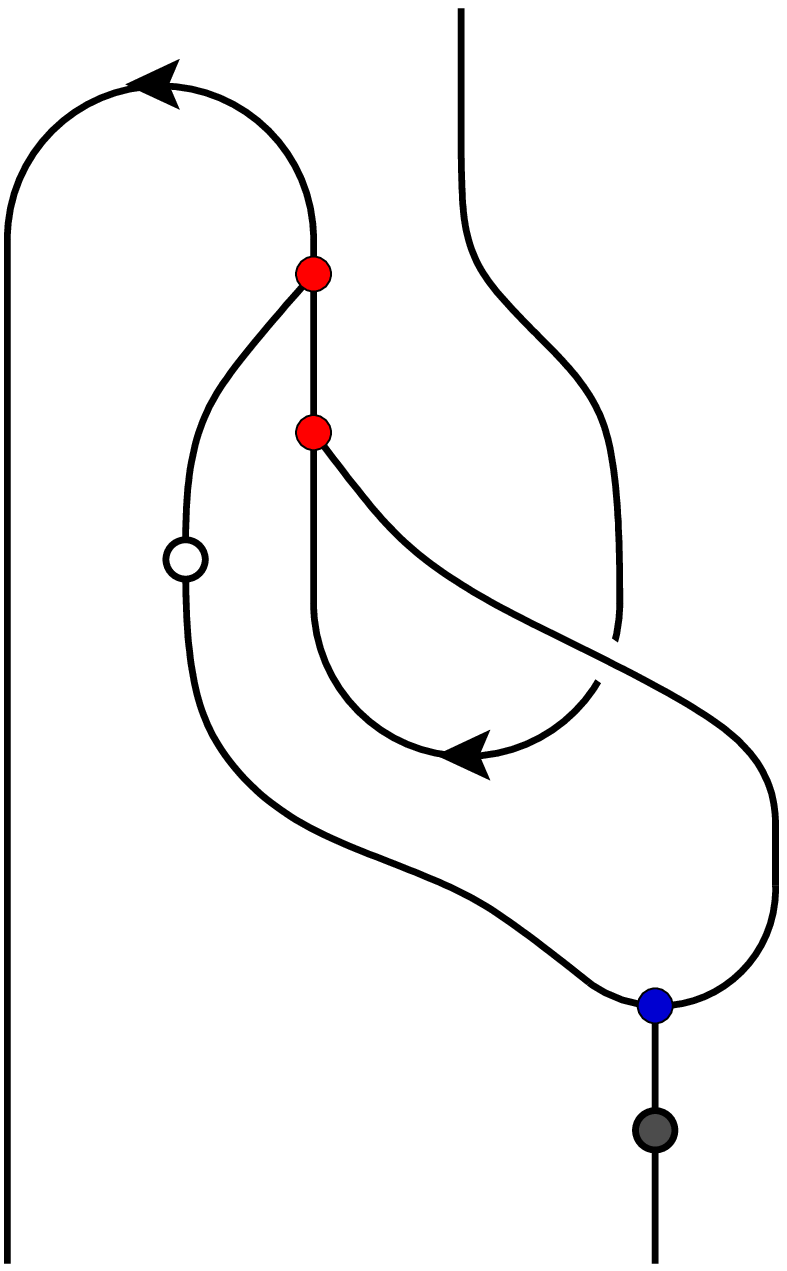}
   \put(-4,-8.5)   {\sse$ \Hss $}
   \put(8.8,60.5)  {\sse$ \apo $}
   \put(37,114)    {\sse$ \Hss $}
   \put(53.4,-8.5) {\sse$ H $}
   \put(61.4,10.5) {\sse$ \apoi $}
   } }
We call the $H$-bimodule that consists of the vector space $\Hs \otik \Hs$,
endowed with the coadjoint left $H$-action on the first tensor factor and with
the coadjoint right $H$-action on the second factor, the \emph{\bicoad} and
denote it by \Haa. That is,
   \be
   \Haa = (\Hs \Otik \Hs,\rho\coa\oti\idHs,\idHs\oti\rho\coar) \,.
   \labl{bicoad}

We will now show that this bimodule arises as the coend of the functor
   \be
   \otimes \cir(?^\vee{\times}\,\Id) :\quad \HBimod\op \Times \HBimod \to \HBimod \,,
   \labl{F-lyub}
where $\otimes$ and $?^\vee$ are the tensor product \erf{def-tp} and right duality
\erf{def-duals} of \HBimod. A crucial input is the braided monoidal equivalence
described in Lemma \ref{lem-equiv}(iii).

\begin{prop}\label{bilyub-coend}
Let $H$ be a \findim\ ribbon Hopf \ko-algebra. Then the
$H$-bimodule \Haa\ together with the family $\iHaa$ of morphisms
   \eqpic{def_iHaa_X} {120} {43} {
   \put(0,0)   {\Includepichtft{11a}}
   \put(-1.6,89)  {\sse$ \Hss\Oti\Hss $}
   \put(2,-8.5)   {\sse$ X^{\!\vee} $}
   \put(9.2,45.2) {\sse$ \iHaa_X $}
   \put(16,-8.5)  {\sse$ X $}
   \put(48,44)    {$ := $}
   \put(80,0) { \Includepichtft{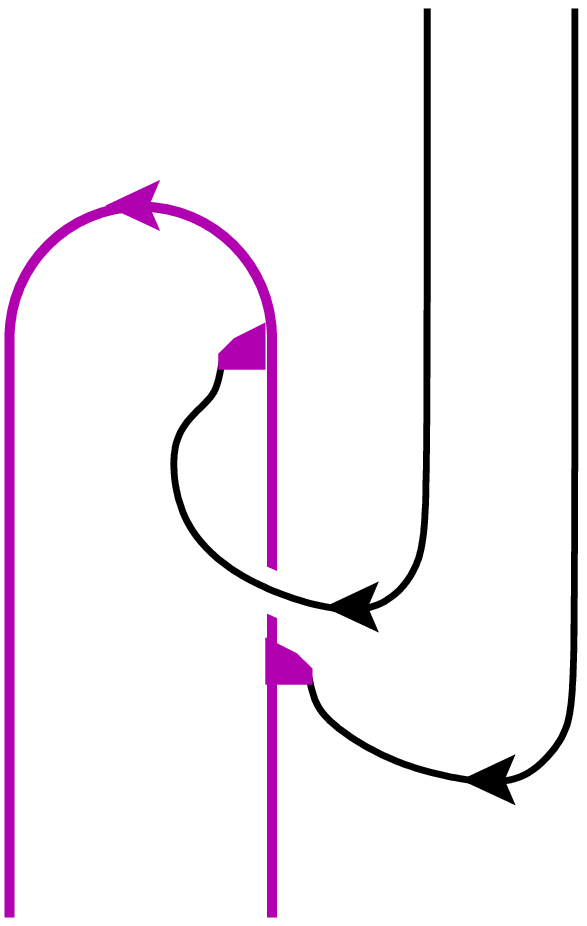}
   \put(-4.1,-8.5){\sse$ \Xs $}
   \put(19.7,28)  {\sse$ \ohr_{\!X}^{} $}
   \put(25.3,-8.5){\sse$ X $}
   \put(31.2,62)  {\sse$ \rho_{\!X}^{} $}
   \put(42.6,104) {\sse$ \Hss $}
   \put(59.4,104) {\sse$ \Hss $}
   } }
from $X^\vee{\otimes}\,X$ to \Haa, for $X\eq(X,\rho_X,\ohr_X)\iN \HBimod$, is the
coend for the functor \erf{F-lyub}:
   \be
   (\Haa,\iHaa) = \coend X \,.
   \ee
\end{prop}

\begin{proof}
The statement follows from the results of \cite[Sect.\,1.2]{lyub6} and
\cite[Sect.\,4.5]{vire4} for the coend of the functor
$\otimes \cir(?^\vee{\times}\,\Id)$ from $H'\Mod\op \Times H'\Mod$ to $H'\Mod$,
with the Hopf algebra $H' \eq \overline H\Oti H\op$,
by transporting them via the equivalence \erf{opequiv} to \HBimod.
\nxl1
We omit the details, but find it instructive to compare a few aspects of a
direct proof to the corresponding parts of the proof of Lemma \ref{lem-dinat}
and of Proposition \ref{prop-F=coend}:
\nxl1
First, dinaturalness follows by an argument completely parallel to the one used
in \erf{dinat_iA_X} to show dinaturalness of the family \erf{def-iHbX}. Second,
the role of the morphism $f_{x_\circ}$ (that is, left action of $H$ on an element
$x_\circ$ of $X$) that appears in formula \erf{rZ_etc_2B} is taken over by the map
   \eqpic{rZ_etcHaa} {140} {47} {
   \put(48,0) { {\Includepichtft{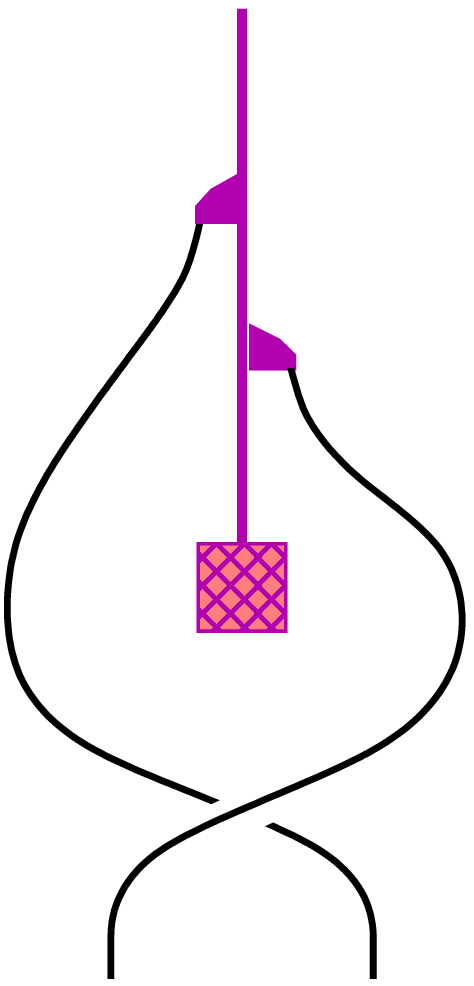}}
   \put(7.8,-8.6){\sse$ H $}
   \put(22.9,111){\sse$ X $}
   \put(33.2,42) {\sse$ x_\circ $}
   \put(36.1,-8.5) {\sse$ H $}
   } }
This map is a bimodule morphism from $H\Otik H$ -- regarded as an $H$-bimodule
$(H\Otik H)_\text{reg}$ via the regular left and right actions on the second
and first factor, respectively -- to $X$. Analogously as in \erf{rZ_etc_3B} one
shows that for any dinatural transformation $j^Z$ from the functor \erf{F-lyub}
to $Z\iN \HBimod$ one has $j^Z_X\cir(\idXs\oti x_\circ) \eq \kappa^Z \cir
\iHaa_X \cir (\idXs\oti x_\circ)$, with the map $\kappa^Z$ defined by $\kappa^Z
\,{:=}\, j^Z_{{(H\Otik H^{})}_\text{reg}} {\circ}\, (\idHs\oti\idHs\oti\eta\oti\eta)$.
And again $\kappa^Z$ is a bimodule morphism, so that the existence part of the
universal property of the coend is established.
Uniqueness follows by specializing to the case $X\eq (H\Otik H)_\text{reg}$ and
observing that $\iHaa_{{(H\Otik H^{})}_\text{reg}}$ is an epimorphism (as is e.g.\
seen from $\iHaa_{{(H\Otik H^{})}_\text{reg}} {\circ}\, (\idHs\oti\idHs\oti\eta\oti\eta)
\eq \idHs\oti\idHs$).
\end{proof}

\begin{cor}
The $H$-bimodule \Haa\ carries the structure of a Hopf algebra, with structure
morphisms given as follows. The unit, counit and coproduct are
   \be
   \bearl
   \eta\bicoa = \eps^\vee \oti \eps^\vee \,, \qquad
   \eps\bicoa = \eta^\vee \oti \eta^\vee \,,
   \Nxl4
   \Delta\bicoa = (\idHs \oti \tau_{\Hss,\Hss}\oti\idHs)
                  \circ \big( {(m\op)}^\vee \oti m^\vee \big) \,,
   \eear
   \labl{Haa-Hopf}
the product is
   \Eqpic{mHaa1,2} {420} {86} {
   \put(-16,87)   {$ m\bicoa ~= $}
   \put(45,0)  {\Includepichtft{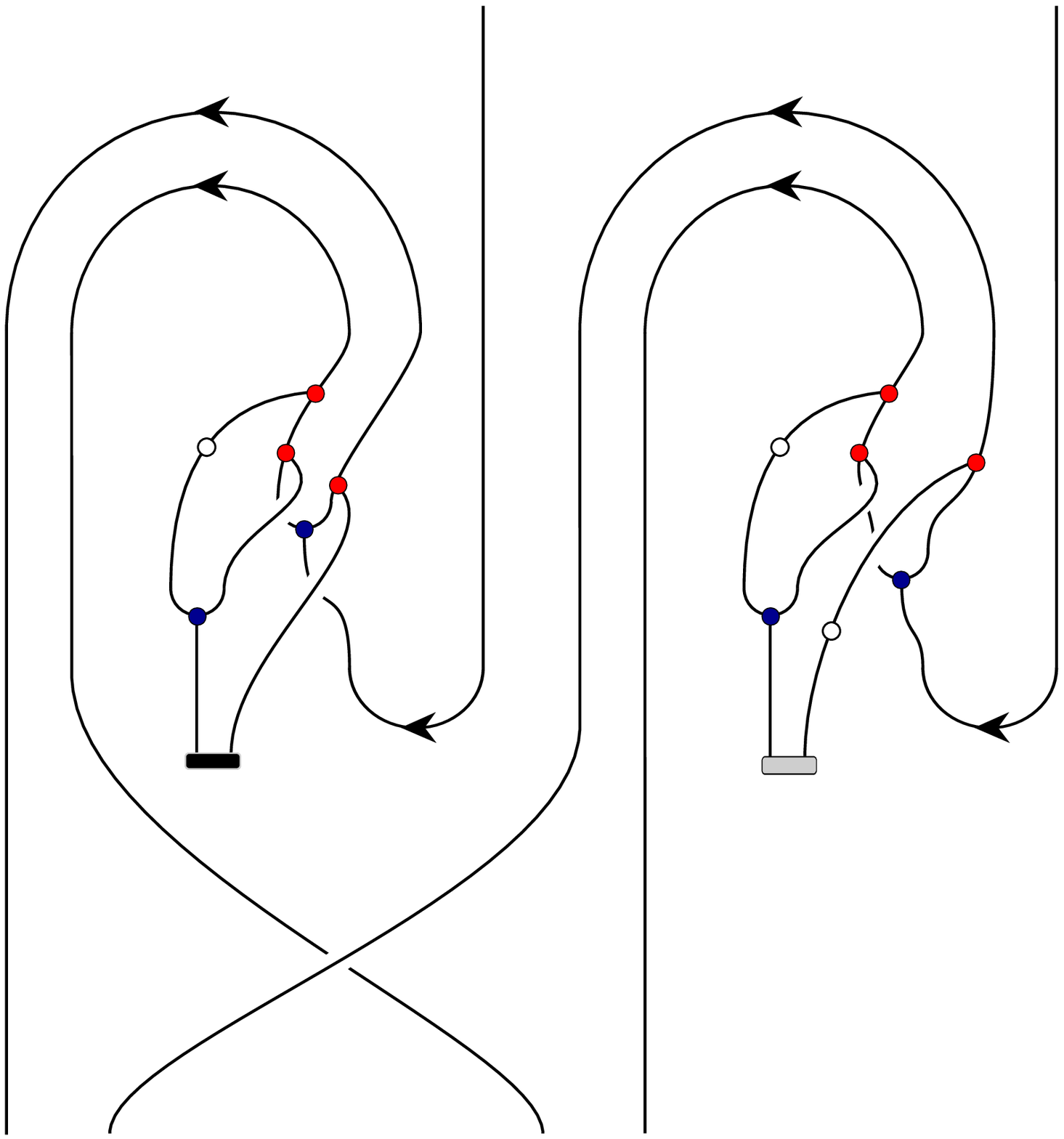}
   \put(-5,-8.5) {\sse$ \Hss $}
   \put(11.7,-8.5) {\sse$ \Hss $}
   \put(27.4,53.3) {\sse$ R^{-1} $}
   \put(75,189.6){\sse$ \Hss $}
   \put(83,-8.5) {\sse$ \Hss $}
   \put(100,-8.5){\sse$ \Hss $}
   \put(124.9,52.1){\sse$ R $}
   \put(169,189.6) {\sse$ \Hss $}
   }
   \put(242,87)    {$ = $}
   \put(277,0)  {\Includepichtft{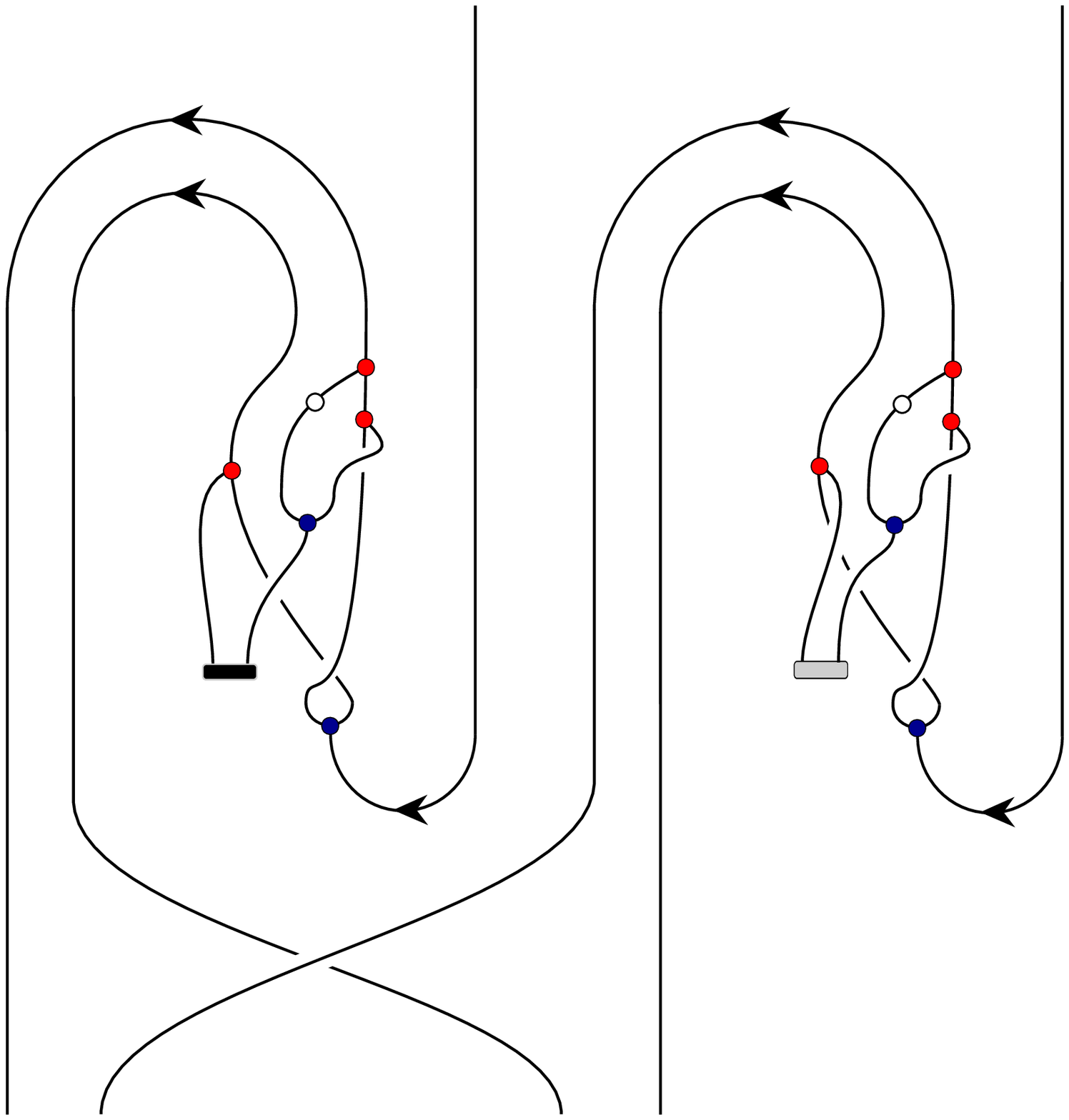}
   \put(-5.4,-8.5) {\sse$ \Hss $}
   \put(11.2,-8.5) {\sse$ \Hss $}
   \put(30,66.4) {\sse$ R^{-1} $}
   \put(75,189.6){\sse$ \Hss $}
   \put(87.7,-8.5) {\sse$ \Hss $}
   \put(104,-8.5){\sse$ \Hss $}
   \put(132.1,66.4){\sse$ R $}
   \put(173,189.6) {\sse$ \Hss $}
   } }
and the antipode is
   \eqpic{apo_Haa} {200} {38} {
   \put(0,41)    {$ \apo_\bicoaa ~= $}
   \put(52,0)  {\Includepichtft{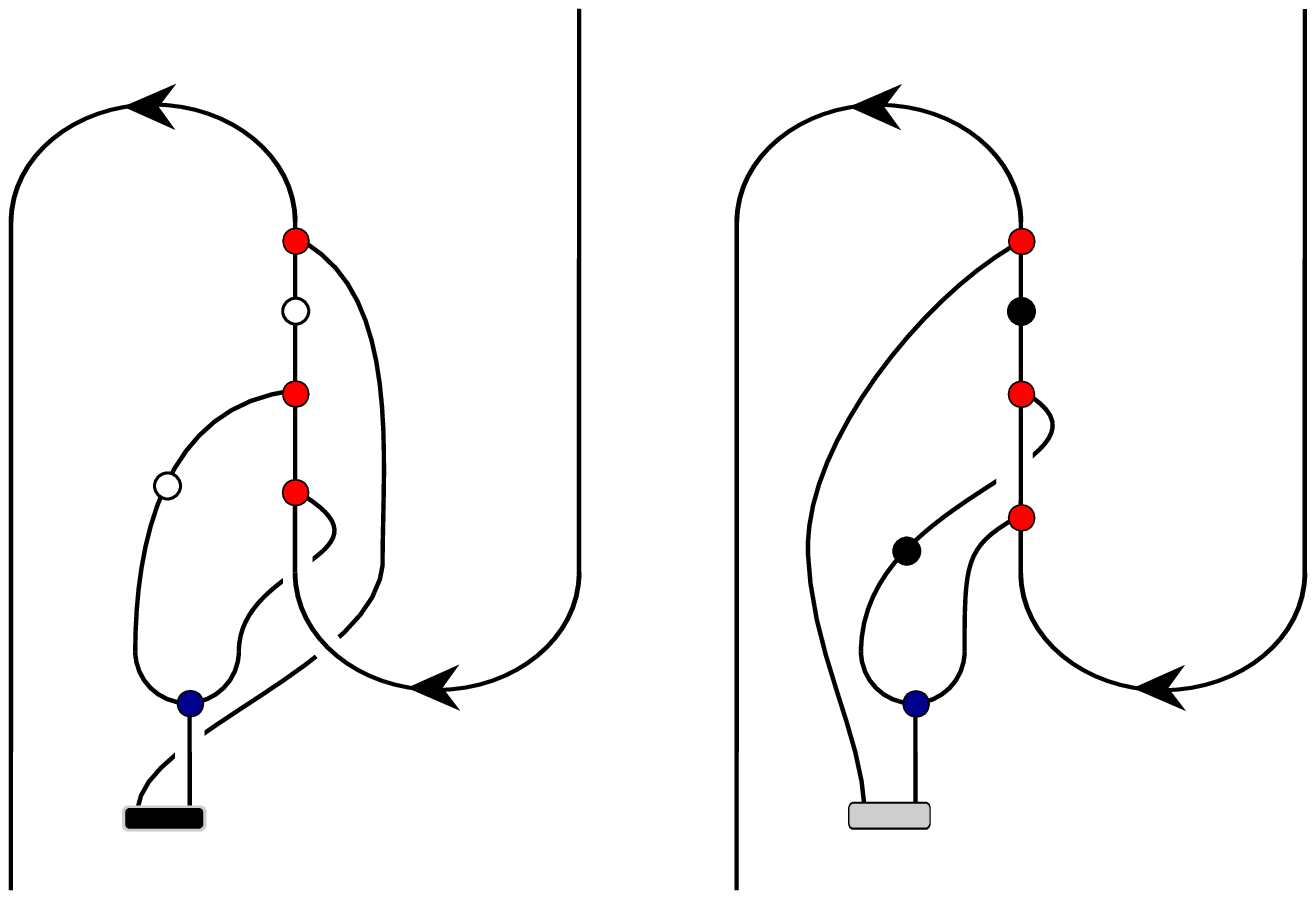}
   \put(-5,-8.5) {\sse$ \Hss $}
   \put(59,100.4){\sse$ \Hss $}
   \put(10,-.5)  {\sse$ R^{-1} $}
   \put(75,-8.5) {\sse$ \Hss $}
   \put(138.8,100.5) {\sse$ \Hss $}
   \put(93,-.5)  {\sse$ R $}
   } }
\end{cor}

\begin{proof}
We just have to specialize the general results of \cite{lyub8}, which apply to the
coend of the functor $\otimes \cir(?^\vee\oti\Id)\colon \CopC \To \C$ in any
\ko-linear abelian ribbon category
\C, to the case $\C \eq \HBimod$. The calculations are straightforward,
and except for the multiplication and the antipode they are very short.
\\
Let us just mention that the first equality in \eqref{mHaa1,2} follows from the
general results (see \cite[Prop.\,2,3]{lyub8}, as well as \cite[Sect.\,1.6]{vire4}
or \cite[Sect.\,4.3]{fuSc17}) together with \eqref{Hbim_dualactions} and the
defining relation \erf{deqf-qt} of the R-matrix. The second equality in
\eqref{mHaa1,2} follows with the help of standard manipulations from the fact
that the R-matrix intertwines the coproduct and the opposite coproduct.
\end{proof}

\begin{prop}
(i) If $\Lambda$ is a two-sided integral of $H$, then
   \be
   \lambda\bicoa := \Lambda^\vee \oti \Lambda^\vee
   \ee
is two-sided cointegral of
the Hopf algebra $(\Haa,m\bicoa,\eta\bicoa,\Delta\bicoa,\eps\bicoa,\apo\bicoa)$.
\nxl1
(ii) If $\lambda$ is a right cointegral of $H$, then
   \be
   \Lambda\bicoa := \lambda^\vee \oti \lambda^\vee
   \labl{Haa-integral}
is a two-sided integral of
$(\Haa,m\bicoa,\eta\bicoa,\Delta\bicoa,\eps\bicoa,\apo\bicoa)$.
\end{prop}

\begin{proof}
(i) Inserting the definitions one has
   \be
   \bearll
   (\lambda\bicoa \oti \idHsHs) \circ \Delta\bicoa
   = \big( m \cir (\Lambda \oti \idHs) {\big)}^*
   \otimes \big( m \cir (\idHs \oti \Lambda) {\big)}^*  \qquand
   \Nxl4
   (\idHsHs \oti \lambda\bicoa) \circ \Delta\bicoa
   = \big( m \cir (\idHs \oti \Lambda) {\big)}^*
   \otimes \big( m \cir (\Lambda \oti \idHs) {\big)}^* .
   \eear
   \ee
Since $\Lambda$ is a two-sided integral of $H$, both of these expressions are
equal to $\eta\bicoa \cir \lambda\bicoa$.
\nxl2
(ii) That $\Lambda\bicoa$ is a left integral readily follows from the first
expression for the product in \eqref{mHaa1,2} together with the fact that
$\lambda$ is a right cointegral and that it satisfies \erf{lambda-Oapoo}.
That $\Lambda\bicoa$ is also a right cointegral follows in the same way by
using instead the second expression in \eqref{mHaa1,2} for the product.
\end{proof}

         \newpage


\section{Motivation from conformal field theory}\label{app:cft}

A major motivation for the mathematical results of this paper comes from
structures that originate in full local two-dimensional conformal field theory.
In this appendix we briefly describe some of these structures.

In representation theoretic approaches to conformal field theory the starting
point is a chiral symmetry algebra together with its category \C\ of
representations. For
any mathematical structure that formalizes the physical concept of chiral
symmetry algebra, the category \C\ can be endowed with a lot of additional
structure. In particular, in many cases it leads to a so-called modular functor.
A modular functor actually consists of a collection of functors.
Namely, to any compact Riemann surface $\Sigma_{g,n}$ of genus $g$ and with
a finite number $n$ of marked points it assigns a functor
   \be
   F_{\Sigma_{g,n}} :\quad \C^{{\scriptscriptstyle\boxtimes} n} \to \Vect
   \ee
from $\C^{{\scriptscriptstyle\boxtimes} n}$ to the category \Vect\ of
finite-dimensional complex vector spaces. This collection of functors
is required to obey a system of compatibility conditions, which in
particular expresses factorization constraints and accommodates actions
of mapping class groups of surfaces.
Thus, selecting for a genus-$g$ surface $\Sigma_{g,n}$ with $n$ marked points
any $n$-tuple $(V_1,V_2,...\,, V_n)$ of objects of the category \C, we obtain
a finite-dimensional complex vector space $F_{\Sigma_{g,n}}(V_1,V_2,...\,,V_n)$
which carries an action of the mapping class group of $\Sigma_{g,n}$. In chiral
conformal field theory, this space plays the role of the space of conformal
blocks with chiral insertion of type $V_i$ at the $i$th marked point of
$\Sigma_{g,n}$.

\medskip

In the particular case that the category \C\ is finitely semisimple,
the structure of a modular functor
is reasonably well understood. Specifically, precise conditions are known under
which the representation category of a vertex algebra $\mathscr V$ is a modular
tensor category. In this case the Reshetikhin-Tu\-ra\-ev construction allows one
to obtain a modular functor just on the basis of \C\ as an abstract category.
In a remarkable development, Lyubashenko and others (see \cite{KEly} and references
cited there) have extended many aspects of this story to a larger class of
monoidal categories that are not necessarily semisimple any longer. In particular,
given an abstract monoidal category with adequate additional properties, one can
still construct representations of mapping class groups.

Representation categories that are not semisimple are of considerable physical
interest; they arise in particular in various systems of statistical mechanics.
The corresponding models of conformal field theory have been termed ``logarithmic''
conformal field theories. A complete characterization of this class of models has
not been achieved yet, but a necessary requirement ensuring tractability is that
the category \C, while being non-semisimple, still possesses certain finiteness
properties, e.g.\ each object should have a composition series of finite length.

In the present paper we consider an even more restricted, but non-empty, subclass,
namely the one for which the monoidal category \C\ is equivalent to the
representation category of a finite-dimensional factorizable ribbon Hopf algebra.
Finite-dimensional Hopf algebras $H_\text{KL}$ have indeed been associated, via a
Kazhdan-Lusztig correspondence, to certain classes of logarithmic conformal field 
theories. These Hopf algebras $H_\text{KL}$ do not have an $R$-matrix, albeit 
they do have a monodromy matrix that is even factorizable \cite{fgst} 
(so that in particular the \pqt s which we introduced in section \ref{sec-modinv} 
can still be defined).  Accordingly our results do not perfectly match the presently
available conformal field theoretic proposals. On the other hand, it is
apparent that the Hopf algebras $H_\text{KL}$ are not quite
the appropriate algebraic structures: their representation categories, albeit being
equivalent to the representation categories of the relevant vertex algebras
as abelian categories, are not equivalent to them as monoidal categories.\,%
  \footnote{~Also, constructing algebras with the help of the Kazhdan-Lusztig
  correspondence involves some arbitrariness. It has been suggested \cite{seTi4}
  that one should better work with Hopf algebras in a category of Yetter-Drinfeld
  modules built from braided vector spaces, rather than Hopf algebras in \Vect.}

\medskip

The Riemann surface of interest to us is $\Sigma_{1,1}$, a one-punctured torus.
This surface is distinguished by the fact \cite{yett6} that it carries a
natural Hopf algebra structure in the category of three-cobordisms.
For general reasons, the functor $F_{\Sigma_{1,1}}$ is representable:
   \be
   F_{\Sigma_{1,1}} \cong\, \HomC(K_\C,-) \,.
   \ee
  In this way we obtain for the category \C\ a distinguished object, and
  this object is actually a Hopf algebra in \C. The construction
  in \cite{lyub6} turns the logic around: it starts with a
  Hopf algebra object $K_\C\iN\C$ canonically associated to the braided
  category and constructs the functor as $F_{\Sigma_{1,1}}(V) \,{=}\, \HomC(K_\C,V)$
  for any object $V{\in}\;\C$.

{}From the point of view of two-dimensional conformal field theory, the
one-punctured torus is the surface relevant for partition functions.
We are interested in this paper in a candidate for the partition function of the
space of bulk fields and thus in one-point functions of bulk fields on the torus.
The space $\mathcal H_\text{bulk}$ of bulk fields carries the structure of a
bimodule over the chiral symmetry algebra $\mathscr V$. In the case that the
category \C\ is semisimple, a particularly simple solution is given by the bulk
state space $\bigoplus_i S_i^\vee \,{\otimes_\complex^{}}\, S_i$, where the (finite)
summation is over all isomorphism classes $[S_i]$ of simple $\mathscr V$-modules.
The corresponding partition function is the so-called \emph{charge conjugation}
modular invariant. It has been conjectured \cite{qusc4,garu2} that this type of
bulk state space exists in the non-semisimple case as well, and that
as a left $\mathscr V$-module it decomposes as
   \be
   \mathcal H_\text{bulk} \,\cong\, \bigoplus_i P_i^\vee \,{\otimes_\complex^{}}\, S_i \,,
   \labl{bulkspace}
with $P_i$ the projective cover of the simple $\mathscr V$-module $S_i$.

According to the principle of holomorphic factorization, a correlation function
for a conformal real surface is an element in the space of conformal blocks
associated to the oriented double of the surface. Thus a one-point function on
the torus is a specific element in the space of conformal blocks associated to
the double of the torus (as a real surface), that is, of the disconnected sum
of two copies of $\Sigma_{1,1}$ with opposite orientation.
For any selection of a pair $(V_1,V_2)$ of objects of \C\ at the two points
on the double cover that lie over the one insertion
point on the torus, this space of conformal blocks is the tensor product
$\HomC(K_\C,V_1)^*_{} \,{\otimes_\complex}\, \HomC(K_\C,V_2)$.
More compactly, this space can be written as a morphism space of another braided
tensor category $\D \,{:=}\, \overline\C\boti\C$, which has its own canonical
Hopf algebra object $K_\D$. As we have noted in Section \ref{ssec:equiv}, if \C\
is the category \HMod\ of left modules over a finite-dimensional factorizable
ribbon Hopf algebra $H$, then \D\ can be identified with the category of
bimodules over $H$, with a tensor product derived from the coproduct on $H$.

Compatibility of \erf{bulkspace} with short exact sequences implies that the 
character of the projective cover $P_i$ is a linear combination of simple 
characters, with coefficients given by the entries of the Cartan matrix of the 
category. Thus in the charge conjugation case the Cartan matrix provides the 
coefficients in the bilinear combination of characters that, owing to 
holomorphic factorization, describes the bulk partition function. Indeed, as 
mentioned in Remark \ref{Cmatrix}(iii), the same structure is seen when 
expressing the morphism $\eps\bico \cir \wiha {\Delta\bico}$, which (as follows 
from Remark \ref{rem:rep}(i)) is nothing but the character of \Hb\ for the 
algebra $K_\D$, as a bilinear combination of characters for the algebra $K_\C$.

\medskip

Correlation functions in conformal field theory should be invariant under the 
relevant mapping class group. For the one-point correlation function on the
torus we are thus interested in finding an object $F\iN\D$ corresponding to the space
of bulk fields as well as a vector
   \be
   Z_F \in \Hom_\D(\HK_\D, F)
   \ee
that is invariant under the action of the mapping class group \slze\ of the
one-punctured torus. Moreover, comparison with
the semisimple situation, in which \C\ is a modular tensor category,
indicates that the object $F$ should possess a structure of a commutative symmetric
Frobenius algebra in \D. The partition function, given by $\eps_F \cir Z_F$, is 
then invariant under the modular group \slz.

This is precisely what the present paper achieves
for the case $\C \,{\simeq}\, \HMod$:\ given a ribbon Hopf algebra
automorphism $\omega$ of $H$, we obtain a commutative symmetric Frobenius algebra
\Hbo\ in the category \HBimod. As an object, \Hbo\ is the twisted \bicom\
${}^{\idsm_H\!}(\Hb)^{\omega}_{}$, so that e.g.\ its decomposition as a left
$H$-module precisely reproduces the decomposition \erf{bulkspace} above.
(The conjecture \erf{bulkspace} has only been made for the case corresponding to
trivial automorphism $\omega\eq\id_H$, though.) Also note that according to Remark
\ref{rem:vaco}(ii) the counit of \Hbo\ is unique up to a non-zero scalar; in the
conformal field theory context this amounts to uniqueness of the vacuum state.
The \pqt\ \erf{Sinv-3} of the coproduct $\Delta\colon \Hbo \To \Hbo\oti\Hbo$
furnishes a \slze-invariant morphism $Z_\omega \iN \Hom(\HK_\D,\Hbo)$. The
morphism $\eps_{\Hbo} \cir Z_\omega$, associated to $H$ and $\omega$, is a natural candidate 
for a modular invariant partition function on the torus.

  \vskip 5.5em

\noindent{\sc Acknowledgments:}
JF is largely supported by VR under project no.\ 621-2009-3993.
CSc is partially supported by the Collaborative Research Centre 676 ``Particles,
Strings and the Early Universe - the Structure of Matter and Space-Time'' and
by the DFG Priority Programme 1388 ``Representation Theory''.
We thank Jens Fjelstad for discussions.
JF is grateful to Hamburg university, and in particular to CSc and Astrid
D\"orh\"ofer, for their hospitality during the time when this study was initiated.

\newpage

  \newcommand\wb{\,\linebreak[0]} \def\wB {$\,$\wb}
  \newcommand\Bi[2]    {\bibitem[#2]{#1}}
  \newcommand\inBo[9]  {{\em #9}, in:\ {\em #1}, {#2}\ ({#3}, {#4} {#5}), p.\ {#6--#7} }
  \newcommand\inBO[9]  {{\em #9}, in:\ {\em #1}, {#2}\ ({#3}, {#4} {#5}), p.\ {#6--#7} {{\tt [#8]}}}
  \renewcommand\J[7]   {{\em #7}, {#1} {#2} ({#3}) {#4--#5} {{\tt [#6]}}}
  \newcommand\JO[6]    {{\em #6}, {#1} {#2} ({#3}) {#4--#5} }
  \newcommand\JP[7]    {{\em #7}, {#1} ({#3}) {{\tt [#6]}}}
  \newcommand\BOOK[4]  {{\em #1\/} ({#2}, {#3} {#4})}
  \newcommand\PhD[2]   {{\em #2}, Ph.D.\ thesis #1}
  \newcommand\prep[2]  {{\em #2}, preprint {\tt #1}}
  \def\adma  {Adv.\wb Math.}
  \def\amjm  {Amer.\wb J.\wb Math.}
  \def\ajse  {Arabian Journal for Science and Engineering}
  \def\coia  {Com\-mun.\wB in\wB Algebra}
  \def\coma  {Con\-temp.\wb Math.}
  \def\comp  {Com\-mun.\wb Math.\wb Phys.}
  \def\jhep  {J.\wb High\wB Energy\wB Phys.}
  \def\joal  {J.\wB Al\-ge\-bra}
  \def\jopa  {J.\wb Phys.\ A}
  \def\joms  {J.\wb Math.\wb Sci.}
  \def\jktr  {J.\wB Knot\wB Theory\wB and\wB its\wB Ramif.}
  \def\jpaa  {J.\wB Pure\wB Appl.\wb Alg.}
  \def\momj  {Mos\-cow\wB Math.\wb J.}
  \def\nupb  {Nucl.\wb Phys.\ B}
  \def\pams  {Proc.\wb Amer.\wb Math.\wb Soc.}
  \def\plms  {Proc.\wB Lon\-don\wB Math.\wb Soc.}
  \def\ruma  {Revista de la Uni\'on Matem\'atica Argentina}
  \def\slnm  {Sprin\-ger\wB Lecture\wB Notes\wB in\wB Mathematics}
  \def\taia  {Top\-o\-lo\-gy\wB and\wB its\wB Appl.}

\end{document}